\documentclass[10pt,a4paper,twoside,headinclude=true]{article} 

\usepackage{lipsum}
\usepackage{url}
\usepackage[T1]{fontenc}
\usepackage[utf8]{inputenc}
\usepackage[english]{babel}
\usepackage{mathrsfs}
\usepackage{rsfso}
\usepackage{amsmath}
\usepackage{amsthm,amssymb}
\usepackage{thmtools}
\usepackage{geometry}
\usepackage{enumitem}
\usepackage{hyperref}
\usepackage[nochapters, dottedtoc]{classicthesis} 
\usepackage{titling} 
\usepackage{comment}
\usepackage{tikz-cd}
\usepackage{imakeidx, multind}
\makeindex[name=symbol,title=Index of notations and symbols]
\makeindex[name=notion,title=Glossary]

\swapnumbers

\declaretheoremstyle[
spaceabove=\topsep, spacebelow=\topsep,
headfont=\scshape,
notefont=\mdseries, notebraces={(}{)},
bodyfont=\itshape,
headpunct=.~---,
postheadspace=0.5em,
qed=$\blacktriangle$
]{mythmstyle}
\declaretheorem[style=mythmstyle, name=Theorem, numberwithin=subsection]{thm}
\declaretheorem[style=mythmstyle, name=Lemma, sibling=thm]{lem}
\declaretheorem[style=mythmstyle, name=Proposition, sibling=thm]{prop}
\declaretheorem[style=mythmstyle, name=Proposition-Definition, sibling=thm]{propdef}
\declaretheorem[style=mythmstyle, name=Corollary, sibling=thm]{cor}

\declaretheoremstyle[
spaceabove=\topsep, spacebelow=\topsep,
headfont=\scshape,
notefont=\mdseries, notebraces={(}{)},
bodyfont=\normalfont,
headpunct=.~---,
postheadspace=0.5em,
qed=$\blacktriangle$
]{mydefstyle}
\declaretheorem[style=mydefstyle, name=Definition, sibling=thm]{defin}

\declaretheorem[style=mydefstyle, name=Examples, sibling=thm]{exs}
\declaretheorem[style=mydefstyle, name=Remark, sibling=thm]{rk}
\declaretheorem[style=mydefstyle, name=Remarks, sibling=thm]{rks}
\declaretheorem[style=mydefstyle, name=Notation, sibling=thm]{nb}
\declaretheorem[style=mydefstyle, name=Notations, sibling=thm]{nbs}
\declaretheorem[style=mydefstyle, name=Convention, sibling=thm]{conv}

\declaretheoremstyle[
spaceabove=\topsep, spacebelow=\topsep,
headfont=\scshape,
notefont=\mdseries, notebraces={(}{)},
bodyfont=\normalfont,
headpunct=.,
postheadspace=0.5em,
qed=$\blacktriangle$
]{mynoheadstyle}
\declaretheorem[style=mynoheadstyle, name={}, sibling=thm]{noh}

\numberwithin{equation}{section}

\newcommand{\M}{\mathcal{M}}
\newcommand{\B}{\mathcal{B}}
\newcommand{\K}{\mathcal{K}}
\newcommand{\Lin}{\mathcal{L}}
\newcommand{\s}[1]{\mathscr{{#1}}}
\newcommand{\f}[1]{\mathfrak{{#1}}}
\newcommand{\GR}{{\mathbb R}}
\newcommand{\GC}{{\mathbb C}}
\newcommand{\GN}{{\mathbb N}}
\newcommand{\QG}{\mathbb{G}}
\newcommand{\QH}{\mathbb{H}}
\newcommand{\Cstar}{\mathrm{C}^*}

\newcommand{\rmc}{{\rm C}}

\newcommand{\tens}{\otimes}

\newcommand{\cst}[5]{{#1}\underset{{#3}}{_{#2}\tens_{#4}}{#5}}
\newcommand{\reltens}[4]{{#1}{\,_{#2}\tens_{#3}\,}{#4}}

\newcommand{\fprod}[4]{{#1}{\,_{#2}\star_{#3}\,}{#4}}

\newcommand{\nsf}{n.s.f.~}
\newcommand{\GNS}{G.N.S.~}

\newcommand{\ext}{\mathrm{ext}}
\newcommand{\id}{\mathrm{id}}

\newcommand{\restr}[2]{{#1}\!\!\restriction_{#2}}
\newcommand{\ind}{{\rm Ind}_{\QG_1}^{\QG_2}}
\newcommand{\iind}{{\rm Ind}_{\QG_2}^{\QG_1}}

\newcommand{\er}{{\cal E}_{A,R}}

\newcommand{\kk}{{\sf KK}}
\newcommand{\der}{{\rm Der}}

\makeatletter
\newcommand{\raisemath}[1]{\mathpalette{\raisem@th{#1}}}
\newcommand{\raisem@th}[3]{\raisebox{#1}{$#2#3$}}
\makeatother

\pdfmapfile{+rsfso.map}
\DeclareSymbolFont{rsfso}{U}{rsfso}{m}{n}
\DeclareSymbolFontAlphabet{\mathscr}{rsfso}

\begin{document}

\pagestyle{scrheadings}

\clearscrheadfoot
\cohead{\footnotesize{{\sc MEASURED QUANTUM GROUPOIDS ON A FINITE BASIS AND EQUIVARIANT K-THEORY}}}
\rohead{\thepage}
\cehead{\footnotesize{{\sc J.\ CRESPO}}}
\lehead{\thepage}


\setcounter{tocdepth}{2} 

    \title{\rmfamily\normalfont\spacedallcaps{Measured quantum groupoids on a finite basis and equivariant Kasparov theory}}
    \author{\small{by}\\[1em] \spacedlowsmallcaps{jonathan crespo}}
    \date{\small{\today}}
  
    \maketitle
    
	\thispagestyle{empty}
	
    \begin{abstract}
        \noindent In this article, we generalize to the case of measured quantum groupoids on a finite basis some important results concerning equivariant Kasparov theory for actions of locally compact quantum groups \cite{BS1,BS2}. To every pair $(A,B)$ of C*-algebras continuously acted upon by a regular measured quantum groupoid on a finite basis $\cal G$, we associate a $\cal G$-equivariant Kasparov theory group $\kk_{\cal G}(A,B)$. The Kasparov product generalizes to this setting. By applying recent results concerning actions of regular measured quantum groupoids on a finite basis \cite{BC,C2}, we obtain two canonical homomorphisms $J_{\cal G}:\kk_{\cal G}(A,B)\rightarrow\kk_{\widehat{\cal G}}(A\rtimes{\cal G},B\rtimes{\cal G})$ and $J_{\widehat{\cal G}}:\kk_{\widehat{\cal G}}(A,B)\rightarrow\kk_{\cal G}(A\rtimes\widehat{\cal G},B\rtimes\widehat{\cal G})$ inverse of each other through the Morita equivalence coming from a version of the Takesaki-Takai duality theorem \cite{BC,C2}.
       We investigate in detail the case of colinking measured quantum groupoids. In particular, if $\QG_1$ and $\QG_2$ are two monoidally equivalent regular locally compact quantum groups, we obtain a new proof of the canonical equivalence of the associated equivariant Kasparov categories \cite{BC}.

\medbreak        
        
        \noindent{\bf Keywords} Locally compact quantum groups, measured quantum groupoids, monoidal equivalence, equivariant Kasparov theory.
    \end{abstract}
     
    \tableofcontents

    
\section*{Introduction}\addcontentsline{toc}{section}{Introduction}

The notion of monoidal equivalence of compact quantum groups has been introduced by Bichon, De Rijdt and Vaes in \cite{BRV}. Two compact quantum groups $\QG_1$ and $\QG_2$ are said to be monoidally equivalent if their categories of representations are equivalent as monoidal $\Cstar$-categories. They have proved that $\QG_1$ and $\QG_2$ are monoidally equivalent if, and only if, there exists a unital $\Cstar$-algebra equipped with commuting continuous ergodic actions of full multiplicity of $\QG_1$ on the left and of $\QG_2$ on the right. Among the applications of monoidal equivalence to the geometric theory of free discrete quantum groups, we mention the contributions to randow walks and their associated boundaries \cite{VV,RV}, CCAP property and Haagerup property \cite{DFY}, the Baum-Connes conjecture and K-amenability \cite{V1,VeVo}.

\medskip

In his Ph.D.~thesis \cite{DC}, De Commer has extended the notion of monoidal equivalence to the locally compact case. Two locally compact quantum groups $\QG_1$ and $\QG_2$ (in the sense of Kustermans and Vaes \cite{KV2}) are said to be monoidally equivalent if there exists a von Neumann algebra equipped with a left Galois action of $\QG_1$ and a right Galois action of $\QG_2$ that commute. He proved that this notion is completely encoded by a measured quantum groupoid (in the sense of Enock and Lesieur \cite{E08}) on the basis $\GC^2$. Such a groupoid is called a colinking measured quantum groupoid.\hfill\break
The measured quantum groupoids  have been introduced and studied by Lesieur and Enock (see \cite{E08,Le}). Roughly speaking, a measured quantum groupoid (in the sense of Enock-Lesieur) is an octuple ${\cal G}=(N,M,\alpha,\beta,\Gamma,T,T',\nu)$, where $N$ and $M$ are von Neumann algebras (the basis $N$ and $M$ are the algebras of the groupoid corresponding respectively to the space of units and the total space for a classical groupoid), $\alpha$ and $\beta$ are faithful normal *-homomorphisms from $N$ and $N^{\rm o}$ (the opposite algebra) to $M$ (corresponding to the source and target maps for a classical groupoid) with commuting ranges, $\Gamma$ is a coproduct taking its values in a certain fiber product, $\nu$ is a normal semi-finite weight on $N$ and $T$ and $T'$ are operator-valued weights satisfying some axioms. \hfill\break
In the case of a finite-dimensional basis $N$, the definition has been greatly simplified by De Commer \cite{DC2,DC} and we will use this point of view in this article. Broadly speaking, we can take for $\nu$ the non-normalized Markov trace on the $\Cstar$-algebra $N=\bigoplus_{1\leqslant l\leqslant k}{\rm M}_{n_l}(\GC)$\index[symbol]{ma@${\rm M}_{n_l}(\GC)$, square matrices of order $n_l$ with entries in $\GC$}. The relative tensor product of Hilbert spaces (resp.\ the fiber product of von Neumann algebras) is replaced by the ordinary tensor product of Hilbert spaces (resp.\  von Neumann algebras). The coproduct takes its values in $M\tens M$ but is no longer unital.

\medskip

In \cite{BC}, the authors introduce a notion of (strongly) continuous actions on $\Cstar$-algebras of measured quantum groupoids on a finite basis. They extend the construction of the crossed product, the dual action and give a version of the Takesaki-Takai duality generalizing the Baaj-Skandalis duality theorem \cite{BS2} in this setting.\hfill\break
If a colinking measured quantum groupoid ${\cal G}$, associated with a monoidal equivalence of two locally compact quantum groups $\QG_1$ and $\QG_2$, acts (strongly) continuously on a $\Cstar$-algebra $A$, then $A$ splits up as a direct sum $A=A_1\oplus A_2$ of $\Cstar$-algebras and the action of ${\cal G}$ on $A$ restricts to an action of $\QG_1$ (resp.\ $\QG_2$) on $A_1$ (resp.\ $A_2$).\hfill\break
They also extend the induction procedure to the case of monoidally equivalent regular locally compact quantum groups. To any continuous action of $\QG_1$ on a $\Cstar$-algebra $A_1$, they associate canonically a $\Cstar$-algebra $A_2$ endowed with a continuous action of $\QG_2$. As important consequences of this construction, we mention the following:
\begin{itemize}
	\item a one-to-one functorial correspondence between the continuous actions of the quantum groups $\QG_1$ and $\QG_2$, which generalizes the compact case \cite{RV} and the case of deformations by a 2-cocycle \cite{NT14};
	\item a complete description of the continuous actions of colinking measured quantum groupoids;
	\item the equivalence of the categories $\kk_{\QG_1}$ and $\kk_{\QG_2}$, which generalizes to the regular locally compact case a result of Voigt \cite{V1}.
\end{itemize}
The proofs of the above results rely crucially on the regularity of the quantum groups $\QG_1$ and $\QG_2$. They prove that the regularity of $\QG_1$ and $\QG_2$ is equivalent to the regularity of the associated colinking measured quantum groupoid in the sense of \cite{E05} (see also \cite{Tim1,Tim2}). 

\medskip

In \cite{C2}, the author generalizes to the case of (semi-)regular measured quantum groupoid on a finite basis some important properties of (semi-)regular locally compact quantum groups \cite{BS2,Ba95}, which then allow him to generalize some crucial results of \cite{BSV} concerning actions of (semi-)regular locally compact quantum groups. More precisely, if $\cal G$ is a regular measured quantum groupoid on a finite basis then any weakly continuous action of $\cal G$ is necessarily continuous in the strong sense.\newline
Let $\cal G$ be a measured quantum groupoid on a finite basis. The author provides a notion of action of $\cal G$ on Hilbert C*-modules in line with the corresponding notion for quantum groups \cite{BS1}. By using the previous result, if $\cal G$ is regular then any action of $\cal G$ on a Hilbert C*-module is necessarily continuous. The author defines the notion of $\cal G$-equivariant Morita equivalence between $\cal G$-C*-algebras. By applying the version of the Takesaki-Takai duality theorem obtained in \cite{BC}, the author finally obtains that any $\cal G$-C*-algebra $A$ is $\cal G$-equivariantly Morita equivalent to its double crossed product $(A\rtimes{\cal G})\rtimes\widehat{\cal G}$ in a canonical way.

\medskip

In this article, we generalize to the setting of measured quantum groupoid on a finite basis some crucial results concerning equivariant Kasparov theory for actions of quantum groups \cite{BS1}. More precisely, we define the equivariant Kasparov groups $\kk_{\cal G}(A,B)$ for any pair of $\cal G$-C*-algebras $(A,B)$ and extend the functorial properties and the Kasparov product in this framework. For all pair of $\cal G$-C*-algebras (resp.\ $\widehat{\cal G}$-C*-algebras), we build a homomorphism $J_{\cal G}:\kk_{\cal G}(A,B)\rightarrow\kk_{\widehat{\cal G}}(A\rtimes{\cal G},B\rtimes{\cal G})$ (resp.\ $J_{\widehat{\cal G}}:\kk_{\widehat{\cal G}}(A,B)\rightarrow\kk_{\cal G}(A\rtimes\widehat{\cal G},B\rtimes\widehat{\cal G})$). We also prove that $J_{\cal G}$ and $J_{\widehat{\cal G}}$ are inverse of each other through the Morita equivalences obtained in \cite{C2}. The rest of the paper is dedicated to  the applications of the above theory to monoidal equivalence. In particular, we provide a new proof of the equivalence of the equivariant Kasparov categories $\kk_{\QG_1}$ and $\kk_{\QG_2}$ when $\QG_1$ and $\QG_2$ are monoidally equivalent regular locally compact quantum groups \cite{BC} (see also \cite{V1} for the compact case). It should be mentioned that the equivariant Kasparov theory for actions of locally compact topological groupoid has been studied by Le Gall in \cite{LG}.

\medskip

This article is organized as follows.\newline
$\bullet$ {\it Chapter 1.} We recall the general conventions and notations used throughout this paper.\newline
$\bullet$ {\it Chapter 2.} We make an overview of the theory of locally compact quantum groups (cf.\ \cite{KV2} and \cite{BS2}). We recall the construction of the Hopf $\Cstar$-algebra associated with a locally compact quantum group and the notion of action of locally compact quantum groups in the $\Cstar$-algebraic setting. We also recall the notion of equivariant Hilbert C*-modules (cf.\ \cite{BS1}).\newline
$\bullet$ {\it Chapter 3.} We make a very brief survey of the theory of measured quantum groupoid (cf.\ \cite{Le,E08}) and we recall the simplified definition and the associated C*-algebraic structure in the case where the basis is finite-dimensional (cf.\ \cite{DC2,DC}). In the last section, we make an outline of the reflection technique across a Galois object provided by De Commer (cf.\ \cite{DC,DC3}), the construction and the structure of the colinking measured quantum groupoid associated with monoidally equivalent locally compact quantum groups. We also recall the precise description of the C*-algebraic structure of colinking measured quantum groupoids (cf.\ \cite{BC}).\newline
$\bullet$ {\it Chapter 4.} In the first section of this chapter, we recall the definitions and the main results of \cite{BC,C2} concerning the notion of continuous action of measured quantum groupoids on a finite basis on C*-algebras. We also recall the version of the Takesaki-Takai duality theorem obtained in \cite{BC} in this framework. The second section is dedicated to a brief overview of  C*-algebras acted upon by a colinking measured quantum groupoid (cf.\ \cite{BC}).\newline
$\bullet$ {\it Chapter 5.} In the first section, we recall the notion of action of measured quantum groupoids on a finite basis on Hilbert C*-modules (cf.\ \cite{C2}). In the second section, we study the crossed product construction in this setting and we state a version of the Takesaki-Takai duality theorem. The last section begins with a reminder of the case of a colinking measured quantum groupoid (cf.\ \cite{BC,C2}). The structure of the double crossed product is investigated at the end of this section.\newline
$\bullet$ {\it Chapter 6.} In this chapter, we give the definition and some properties of equivariant Kasparov groups by a regular measured quantum groupoid on a finite basis. We generalize to our setting the Kasparov technical theorem, which allows us to build the Kasparov product. In the last section, we build the so-called ``descent morphisms'' $J_{\cal G}$ and $J_{\widehat{\cal G}}$ and prove that they are inverse of each other up to Morita equivalences.\newline
$\bullet$ {\it Chapter 7.} We apply the previous results to the case of a colinking measured quantum groupoid $\cal G$ associated with two monoidally equivalent regular locally compact quantum groups $\QG_1$ and $\QG_2$. We obtain canonical equivalences between the equivariant Kasparov categories associated with $\cal G$, $\QG_1$ and $\QG_2$. In particular, we provide a new proof of the isomorphism obtained in \S 4.5 \cite{BC}.\newline
$\bullet$ {\it Chapter 8.} In the appendix of this article, we have assembled a very brief review of the main notions and notations of the non-commutative measure and integration theory. We can also find some notations and important results used throughout this paper.

\begin{center}
{\bf Acknowledgements}
\end{center}

The author is indebted to his advisor Prof.\ S.\ Baaj for helpful comments and advice. He is also very grateful to Prof.\ K.\ De Commer for sharing his expertise in various topics concerning measured quantum groupoids and for the financial support of the F.W.O.

\section{Preliminary notations}\label{sectionNotations}

We specify here some elementary notations and conventions used in this article. For more notations, we refer the reader to the appendix and the index of this article.

\medskip

$\bullet$ For all subset $X$ of a normed vector space $E$, we denote $\langle X\rangle$ (resp.\ $[X]$) the linear span (resp.\ closed linear span) of $X$ in $E$. If $X$ and $Y$ are two sets, we define the set $XY:=\{xy\,;\, x\in X,\, y\in Y\}$, where $xy$ denotes the product/composition of $x$ and $y$ or the evaluation of $x$ at $y$ (when these operations make sense). If $X$ is a subset of a *-algebra $A$, we denote by $X^*$ the subset $\{x^*\,;\, x\in X\}$ of $A$. If $X$ is a subset of a C*-algebra $A$, we denote by ${\rm C}^*(X)$ the C*-subalgebra of $A$ generated by $X$.\hfill\break
$\bullet$ We denote by $\tens$ the tensor product of Hilbert spaces, the tensor product of von Neumann algebras, the minimal tensor product of C*-algebras or the external tensor product of Hilbert C*-modules. We also denote by $\odot$ (resp.\ $\odot_A$) the algebraic tensor product over the field of complex numbers $\GC$ (resp.\ an algebra $A$).\newline
$\bullet$ Let $A$ and $B$ be C*-algebras. We denote by $\M(A)$ (resp.\ $\widetilde{A}$) the C*-algebra consisting of the multipliers of $A$ (resp.\ the C*-algebra obtained from $A$ by adjunction of a unit element). If $J$ is a closed two-sided ideal of $A$, let $\M(A;J)$ be the strictly closed C*-subalgebra of $\M(A)$ consisting of the elements $m$ of $\M(A)$ such that the relations $mA\subset J$ and $Am\subset J$ hold. Note that the restriction homomorphism from $\M(A)$ to $\M(J)$ identifies $\M(A;J)$ to a C*-subalgebra of $\M(J)$. We denote by $\widetilde{\M}(A\tens B):=\M(\widetilde{A}\tens B;A\tens B)$ (or $\widetilde{\M}_B(A\tens B)$ in case of ambiguity) the $B$-relative multiplier algebra of $A$ and identify $\widetilde{\M}(A\tens B)$ to a stricly closed C*-subalgebra of $\M(A\tens B)$ (\S 1 \cite{BS1}). \newline
Let $\pi:A\rightarrow\M(B)$ be a (possibly degenerate) *-homomorphism. For all C*-algebra $D$, there exists a unique strictly continuous *-homomorphism $\pi\tens\id_D:\widetilde{\M}(A\tens D)\rightarrow\M(B\tens D)$ satisfying the relation
$
(\pi\tens\id_D)(x)(1_B\tens d)=(\pi\tens\id_D)(x(1_A\tens d))
$
for all $x\in\widetilde{\M}(A\tens D)$ and $d\in D$. 
Indeed, denote by $\widetilde{\pi}$ the unital extension of $\pi$ to $\widetilde{A}$. The non-degenerate *-homomorphism $\widetilde{\pi}\tens\id_D:\widetilde{A}\tens D\rightarrow\M(B\tens D)$ uniquely extends to $\M(\widetilde{A}\tens D)$. By restricting to $\widetilde{\M}(A\tens D)$, we obtain the desired extension of $\pi\tens\id_D$ (\S 1 \cite{BS1}).\newline
$\bullet$ If $x$ and $y$ are two elements of an algebra $A$, we denote by $[x,\,y]$ the commutator of $x$ and $y$, {\it i.e.} the element of $A$ defined by $[x,\,y]:=xy-yx$.

\medskip

All Hilbert spaces considered in this article are complex and separable. All inner products are assumed to be anti-linear in the first variable and linear in the second variable. Let $\s H$ and $\s K$ be Hilbert spaces.\newline
$\bullet$ Let $\B(\s H,\s K)$ (resp.\ $\K(\s H,\s K)$) be the Banach space of bounded (resp.\ compact) linear operators from $\s H$ to $\s K$. For $\xi\in\s K$ and $\eta\in\s H$, we denote by $\theta_{\xi,\eta}$ the rank-one operator defined by $\theta_{\xi,\eta}(\zeta):=\langle\eta,\,\zeta\rangle\xi$ for all $\zeta\in\s H$. We have $\K(\s H,\s K)=[\theta_{\xi,\eta}\,;\,\xi\in\s K,\,\eta\in\s K]$. We denote by $\B(\s H):=\B(\s H,\s H)$ (resp.\ $\K(\s H):=\K(\s H,\s H)$) the C*-algebra of bounded (resp.\ compact) linear operators on $\s H$. Recall that $\K(\s H)$ is a closed two-sided ideal of $\B(\s H)$ and $\B(\s H)=\M(\K(\s H))$.\newline 
$\bullet$ We denote by $\Sigma_{\s K\tens\s H}$ (or simply $\Sigma$) the flip map, that is to say the unitary operator $\s K\tens\s H\rightarrow\s H\tens\s K\,;\, \xi\tens\eta\mapsto\eta\tens\xi$.\hfill\break
$\bullet$ For $u\in\B(\s H)$, we denote by ${\rm Ad}(u)$ the bounded operator on $\B(\s H)$ defined for all $x\in\B(\s H)$ by ${\rm Ad}(u)(x):=uxu^*$.

\medskip

In this article, we will use the notion of (right) Hilbert C*-module over a C*-algebra and their tensor products (internal and external). All the definitions and conventions are those of \cite{Kas1}. In particular, let $\s E$ and $\s F$ be two Hilbert C*-modules over a C*-algebra $A$.\hfill\break
\noindent$\bullet$ We denote by $\Lin(\s E,\s F)$ the Banach space consisting of all adjointable operators from $\s E$ to $\s F$ and $\Lin(\s E)$ the C*-algebra $\Lin(\s E,\s E)$.\hfill\break
$\bullet$ For $\xi\in\s F$ and $\eta\in\s E$, we denote by $\theta_{\xi,\eta}$ the elementary operator of $\Lin(\s E,\s F)$ defined by $\theta_{\xi,\eta}(\zeta):=\xi\langle\eta,\,\zeta\rangle_A$ for all $\zeta\in\s E$. Let $\K(\s E,\s F):=[\theta_{\xi,\eta}\,;\,\xi\in\s F,\,\eta\in\s E]$ be the Banach space of ``compact'' adjointable operators. Denote by $\K(\s E)$ the C*-algebra $\K(\s E,\s E)$ consisting of the ``compact'' adjointable operators of $\Lin(\s E)$. Recall that $\K(\s E)$ is a closed two-sided ideal of $\Lin(\s E)$ and $\Lin(\s E)=\M(\K(\s E))$.\hfill\break
$\bullet$ Let $\s E^*:=\K(\s E,A)$. We have $\s E^*=\{T\in\Lin(\s E,A)\,;\,\exists\,\xi\in\s E,\,\forall\,\eta\in\s E,\, T(\eta)=\langle\xi,\,\eta\rangle_A\}$. We will identify $\s E=\K(A,\s E)\subset\Lin(A,\s E)$. Up to this identification, for $\xi\in\s E$ the operator $\xi^*\in\s E^*$ satisfies $\xi^*(\eta)=\langle\xi,\,\eta\rangle_A$ for all $\eta\in\s E$. We recall that $\s E^*$ is a Hilbert $\K(\s E)$-module for the inner product defined by $\langle T,\, T'\rangle_{\K(\s E)}:=T^*\circ T'$ for $T,T'\in\s E^*$ and the right action defined by composition of maps.

\medskip

In this article, we will also use the leg numbering notation. Let $\s H$ be a Hilbert space and $T\in\B(\s H\tens\s H)$. We define the operators $T_{12},T_{13},T_{23}\in\B(\s H\tens\s H)$ by setting $T_{12}:=T\tens 1$, $T_{23}:=1\tens T$ and $T_{13}:=(\Sigma\tens 1)(1\tens T)(\Sigma\tens 1)$. We can generalize the leg numbering notation for operators acting on any tensor product of Hilbert spaces and for adjointable operators acting on any external tensor product of Hilbert C*-modules over possibly different C*-algebras.

\section{Locally compact quantum groups}

\setcounter{thm}{0}

\numberwithin{thm}{section}
\numberwithin{prop}{section}
\numberwithin{lem}{section}
\numberwithin{cor}{section}
\numberwithin{propdef}{section}
\numberwithin{nb}{section}
\numberwithin{nbs}{section}
\numberwithin{rk}{section}
\numberwithin{rks}{section}
\numberwithin{defin}{section}
\numberwithin{ex}{section}
\numberwithin{exs}{section}
\numberwithin{noh}{section}
\numberwithin{conv}{section}

For the notations and conventions used in this article concerning the non-commutative integration theory and the canonical objects of the Tomita-Takesaki theory, we refer the reader to the appendix of this article (cf.\ \S\ref{integration}).
\begin{defin}\cite{KV2}
A locally compact quantum group\index[notion]{locally compact quantum group} is a pair $\QG=({\rm L}^{\infty}(\QG),\Delta)$, where ${\rm L}^{\infty}(\QG)$ is a von Neumann algebra and $\Delta:{\rm L}^{\infty}(\QG)\rightarrow{\rm L}^{\infty}(\QG)\tens{\rm L}^{\infty}(\QG)$ is a unital normal *-homomorphism satisfying the following conditions:
\begin{enumerate}
\item $(\Delta\tens\id)\Delta=(\id\tens\Delta)\Delta$;
\item there exist \nsf weights $\varphi$ and $\psi$ on ${\rm L}^{\infty}(\QG)$ such that:
\begin{enumerate}
\item $\varphi$ is left invariant, {\it i.e.}\ $\varphi((\omega\tens\id)\Delta(x))=\varphi(x)\omega(1)$, for all $\omega\in{\rm L}^{\infty}(\QG)_*^+$ and $x\in{\f M}_{\varphi}^+$,
\item $\psi$ is right inveriant, {\it i.e.}\ $\psi((\id\tens\omega)\Delta(x))=\psi(x)\omega(1)$, for all $\omega\in{\rm L}^{\infty}(\QG)_*^+$ and $x\in{\f M}_{\psi}^+$.
\end{enumerate}
\end{enumerate}
A left (resp.\ right) invariant \nsf weight on ${\rm L}^{\infty}(\QG)$ is called a left (resp.\ right) Haar weight on $\QG$.
\end{defin}

\begin{noh}
For a locally compact quantum group $\QG$, there exists a unique left (resp.\ right) Haar weight on $\QG$ up to a positive scalar \cite{KV2}. Let us fix a locally compact quantum group $\QG:=({\rm L}^{\infty}(\QG),\Delta)$. Let us fix a left Haar weight $\varphi$ on $\QG$. Let $({\rm L}^2(\QG),\pi,\Lambda)$ be the \GNS construction for $({\rm L}^{\infty}(\QG),\varphi)$. The left regular representation of $\QG$ is the multiplicative unitary \cite{KV2,BS2} $W\in\B({\rm L}^2(\QG)\tens {\rm L}^2(\QG))$ defined by
\[
W^*(\Lambda(x)\tens\Lambda(y))=(\Lambda\tens\Lambda)(\Delta(y)(x\tens 1)), \quad \text{for all } x,\, y\in{\f N}_{\varphi}.
\]
By identifying ${\rm L}^{\infty}(\QG)$ with its image by the \GNS representation $\pi$, we obtain:
\begin{itemize}
\item ${\rm L}^{\infty}(\QG)$ is the strong closure of the algebra $\{(\id\tens\omega)(W)\,;\, \omega\in\B({\rm L}^2(\QG))_*\}$;
\item $\Delta(x)=W^*(1\tens x)W$, for all $x\in{\rm L}^{\infty}(\QG)$.\qedhere
\end{itemize}
\end{noh}

\begin{noh}
The Hopf-von Neumann algebra $({\rm L}^{\infty}(\QG),\Delta)$ admits \cite{KV2} a unitary antipode map $R_{\QG}:{\rm L}^{\infty}(\QG)\rightarrow{\rm L}^{\infty}(\QG)$ and we can choose for right Haar weight on $\QG$ the weight $\psi$ defined by $\psi(x):=\varphi(R_{\QG}(x))$, for all $x\in {\rm L}^{\infty}(\QG)_+$. The Connes cocycle derivative \cite{Co2,Vaes1} of $\psi$ with respect to $\varphi$ is given by 
\[
(D\psi\,:\,D\varphi)_t:=\nu^{\,{\rm i}t^2/2}d^{\,{\rm i}t},\quad \text{for all } t\in\GR,
\] 
where $\nu>0$ is the scaling constant of $\QG$ and the operator $d\eta M$ is the modular element of $\QG$ \cite{KV2}. Let 
$
{\f N}_{\varphi}^d:=\{x\in M\,;\, xd^{1/2} \text{ is bounded and its closure } \overline{x d^{1/2}} \text{ belongs to } {\f N}_{\varphi}\}.
$
The \GNS construction \cite{Vaes1} for $({\rm L}^{\infty}(\QG),\psi)$ is given by $({\rm L}^2(\QG),\id,\Lambda_{\psi})$, where $\Lambda_{\psi}$ is the closure of the map ${\f N}_{\varphi}^d\rightarrow{\rm L}^2(\QG)\,;\,x\mapsto\Lambda(\overline{xd^{1/2}})$. We recall that $J_{\psi}=\nu^{\,{\rm i}/4}J$, where $J$ denotes the modular conjugation for $\varphi$.
\end{noh}

\begin{noh}
The right regular representation of the quantum group $\QG$ is the multiplicative unitary $V\in\B({\rm L}^2(\QG)\tens{\rm L}^2(\QG))$ defined by
\[
V(\Lambda_{\psi}(x)\tens\Lambda_{\psi}(y))=(\Lambda_{\psi}\tens\Lambda_{\psi})(\Delta(x)(1\tens y)),\quad \text{for all } x,\, y\in{\f N}_{\psi}.\qedhere
\]
\end{noh}

\begin{defin}
The quantum group $\widehat{\QG}$ dual of $\QG$ is defined by the Hopf-von Neumann algebra $({\rm L}^{\infty}(\widehat{\QG}),\widehat{\Delta})$, where:
\begin{itemize}
\item ${\rm L}^{\infty}(\widehat{\QG})$ is the strong closure of the algebra $\{(\id\tens\omega)(V)\,;\, \omega\in\B({\rm L}^2(\QG)\}$;
\item the coproduct $\widehat{\Delta}:{\rm L}^{\infty}(\widehat{\QG})\rightarrow{\rm L}^{\infty}(\widehat{\QG})\tens{\rm L}^{\infty}(\widehat{\QG})$ is defined by $\widehat{\Delta}(x):=V^*(1\tens x)V$ for all $x\in{\rm L}^{\infty}(\widehat{\QG})$.
\end{itemize} 
The quantum group $\widehat{\QG}$ admits left and right Haar weights \cite{KV2} and we can take the Hilbert space ${\rm L}^2(\QG)$ for \GNS space. We denote by $\widehat{J}$ the modular conjugation of the left Haar weight on $\widehat{\QG}$.
\end{defin}

		\subsection{Hopf C*-algebras associated with a quantum group}
	
\setcounter{thm}{0}

\numberwithin{thm}{subsection}
\numberwithin{prop}{subsection}
\numberwithin{lem}{subsection}
\numberwithin{cor}{subsection}
\numberwithin{propdef}{subsection}
\numberwithin{nb}{subsection}
\numberwithin{nbs}{subsection}
\numberwithin{rk}{subsection}
\numberwithin{rks}{subsection}
\numberwithin{defin}{subsection}
\numberwithin{ex}{subsection}
\numberwithin{exs}{subsection}
\numberwithin{noh}{subsection}
\numberwithin{conv}{subsection}	

We associate \cite{BS2,KV2} with the quantum group $\QG$ two Hopf C*-algebras $(S,\delta)$ and $(\widehat{S},\widehat{\delta})$ defined by:
\begin{itemize}
\item $S$ (resp.\ $\widehat{S}$) is the norm closure of the algebra $\{(\omega\tens\id)(V)\,;\,\omega\in\B({\rm L}^2(\QG))_*\}$ (resp.\ $\{(\id\tens\omega)(V)\,;\,\omega\in\B({\rm L}^2(\QG))_*\}$);
\item the coproduct $\delta:S\rightarrow\M(S\tens S)$ (resp.\ $\widehat\delta:\widehat{S}\rightarrow\M(\widehat{S}\tens \widehat{S})$) is given by: 
\[
\delta(x):=V(x\tens 1)V^*,\quad \text{for all } x\in S \quad \text{{\rm(}resp.\ }\widehat{\delta}(x):=V^*(1\tens x)V,\quad \text{for all } x\in\widehat{S}\text{{\rm)}}.
\]
\end{itemize}
We call $(S,\delta)$ (resp.\ $(\widehat{S},\widehat{\delta})$) the Hopf C*-algebra (resp.\ dual Hopf C*-algebra) associated with $\QG$. We can also denote by ${\rm C}_0(\QG):=S$ the Hopf C*-algebra of $\QG$. Note that ${\rm C}_0(\widehat{\QG})=\widehat{S}$. 

\begin{nbs}
\begin{itemize}
\item Consider the unitary operator $U:=\widehat{J}J\in\B({\rm L}^2(\QG))$. Since $U=\nu^{\,{\rm i}/4}J\widehat{J}$, we have $U^*=\nu^{-{\rm i}/4}U$. In particular, ${\rm Ad}(U)={\rm Ad}(U^*)$ on $\B({\rm L}^2(\QG))$.
\item We have the following non-degenerate faithful representation of the C*-algebra $S$ (resp.\ $\widehat{S}$):
\begin{align*}
L & :S\rightarrow\B({\rm L}^2(\QG))\; ;\; y\mapsto y;\quad R:S\rightarrow\B({\rm L}^2(\QG))\;;\;y\mapsto UL(y)U^*\\
\text{{\rm(}resp.\ }\rho & :\widehat{S}\rightarrow\B({\rm L}^2(\QG))\; ; \; x\mapsto x;\quad \lambda:\widehat{S}\rightarrow\B({\rm L}^2(\QG))\;;\;x\mapsto U\rho(x)U^*\text{{\rm)}}.\qedhere
\end{align*}
\end{itemize}
\end{nbs}
		
It follows from 2.15 \cite{KV2} that 
$
W=\Sigma(U\tens 1)V(U^*\tens 1)\Sigma
$
and $[W_{12},\,V_{23}]=0$. The right regular representation of $\widehat{\QG}$ is the multiplicative unitary $\widetilde{V}:=\Sigma(1\tens U)V(1\tens U^*)\Sigma$.

\begin{nb}\label{notC(V)}
Let $\s H$ be a Hilbert space and $X\in\B(\s H\tens\s H)$. We denote by ${\cal C}(X)$\index[symbol]{ca@${\cal C}(-)$} the norm closure of the subspace $\{(\id\tens\omega)(\Sigma X)\,;\, \omega\in\B(\s H)_*\}$ of $\B(\s H)$. If $X$ is a multiplicative unitary, then $\{(\id\tens\omega)(\Sigma X)\,;\, \omega\in\B(\s H)_*\}$ is a subalgebra of $\B(\s H)$ \cite{BS2}.
\end{nb}

\begin{defin}
\cite{BS2} The locally compact quantum group $\QG$ is said to be regular if $\K({\rm L}^2(\QG))={\cal C}(V)$.
\end{defin}

Note that the quantum group $\QG$ is regular if, and only if, $\K({\rm L}^2(\QG))={\cal C}(W)$.

		\subsection{Continuous actions of locally compact quantum groups}
		
We use the notations introduced in the previous paragraph. Let $A$ be a C*-algebra.

\begin{defin}
\begin{enumerate}
\item An action of the quantum group $\QG$ on $A$ is a non-degenerate *-homomorphism $\delta_A:A\rightarrow\M(A\tens S)$ satisfying $(\delta_A\tens\id_S)\delta_A=(\id_A\tens\delta)\delta_A$.
\item An action $\delta_A$ of $\QG$ on $A$ is said to be strongly (resp.\ weakly) continuous if 
\[
[\delta_A(A)(1_A\tens S)]=A\tens S \quad (\text{resp.\ } A=[(\id_A\tens\omega)\delta_A(A)\,;\,\omega\in\B({\rm L}^2(\QG))_*]).
\]
\item A $\QG$-C*-algebra is a pair $(A,\delta_A)$, where $A$ is a C*-algebra and $\delta_A:A\rightarrow\M(A\tens S)$ is a strongly continuous action of $\QG$ on $A$.\qedhere
\end{enumerate}
\end{defin}

If $\QG$ is regular, any weakly continuous action of $\QG$ is necessarily continuous in the strong sense, cf.\ 5.8 \cite{BSV}.

\begin{nbs}
Let $\delta_A:A\rightarrow\M(A\tens S)$ (resp.\ $\delta_A:A\rightarrow\M(A\tens\widehat{S})$) be a strongly continuous action of $\QG$ (resp.\ $\widehat{\QG}$) on the C*-algebra $A$. Denote by $\pi_L$ (resp.\ $\widehat{\pi}_{\lambda}$) the representation of $A$ on the Hilbert $A$-module $A\tens{\rm L}^2(\QG)$ defined by $\pi_L:=(\id_A\tens L)\delta_A$ (resp.\ $\widehat{\pi}_{\lambda}:=(\id_A\tens\lambda)\delta_A$).
\end{nbs}

\begin{defin}(cf.\ 7.1 \cite{BS2})
Let $(A,\delta_A)$ be a $\QG$-C*-algebra (resp.\ $\widehat{\QG}$-C*-algebra). We call (reduced) crossed product of $A$ by $\QG$ (resp.\ $\widehat{\QG}$), the C*-subalgebra $A\rtimes\QG$ (resp.\ $A\rtimes\widehat{\QG}$) of $\Lin(A\tens{\rm L}^2(\QG))$ generated by the products $\pi_L(a)(1_A\tens\rho(x))$ (resp.\ $\widehat{\pi}_{\lambda}(a)(1_A\tens L(x))$) for $a\in A$ and $x\in\widehat S$ (resp.\ $x\in S$).
\end{defin}

The crossed product $A\rtimes\QG$ (resp.\ $A\rtimes\widehat{\QG}$) is endowed with a strongly continuous action of $\widehat{\QG}$ (resp.\ $\QG$), cf.\ 7.3 \cite{BS2}. If $\QG$ is regular, then the Takesaki-Takai duality extends to this setting, cf.\ 7.5 \cite{BS2}.

\begin{defin}
Let $A$ and $B$ be two C*-algebras. Let $\delta_A:A\rightarrow\M(A\tens S)$ and $\delta_B:B\rightarrow\M(B\tens S)$ be two actions of $\QG$ on $A$ and $B$ respectively. A non-degenerate *-homomorphism $f:A\rightarrow\M(B)$ is said to be $\QG$-equivariant if $(f\tens\id_S)\delta_A=\delta_B\circ f$. We denote by ${\sf Alg}_{\QG}$ the category whose objects are the $\QG$-C*-algebras and the morphisms are the $\QG$-equivariant non-degenerate *-homomorphisms.
\end{defin}

		\subsection{Equivariant Hilbert C*-modules and bimodules}

\paragraph{Preliminaries.} In this paragraph, we briefly recall some classical notations and elemen\-tary facts concerning Hilbert C*-modules. Let $A$ be a C*-algebra and $\s E$ a Hilbert $A$-module.

\begin{nbs}\index[symbol]{ia@$\iota_A,\,\iota_{\s E},\,\iota_{\s E^*},\,\iota_{\K(\s E)}$, canonical morphisms}
Let us consider the following maps:
\begin{itemize}
\item $\iota_A:A\rightarrow\K(\s E\oplus A)$, the *-homomorphism given by $\iota_A(a)(\xi\oplus b)=0\oplus ab$ for all $a,b\in A$ and $\xi\in\s E$;
\item $\iota_{\s E}:\s E\rightarrow\K(\s E\oplus A)$, the bounded linear map given by $\iota_{\s E}(\xi)(\eta\oplus a)=\xi a\oplus 0$ for all $a\in A$ and $\xi,\eta\in\s E$;
\item $\iota_{\s E^*}:\s E^*\rightarrow\K(\s E\oplus A)$, the bounded linear map given by $\iota_{\s E^*}(\xi^*)(\eta\oplus a)=0\oplus\xi^*\eta$ for all $\xi,\eta\in\s E$ and $a\in A$;
\item $\iota_{\K(\s E)}:\K(\s E)\rightarrow\K(\s E\oplus A)$, the *-homomorphism given by $\iota_{\K(\s E)}(k)(\eta\oplus a)=k\eta\oplus 0$ for all $k\in\K(\s E)$, $\eta\in\s E$ and $a\in A$.\qedhere
\end{itemize}
\end{nbs}

The result below follows from straightforward computations.

\begin{prop}\label{prop32}
We have the following statements:
\begin{enumerate}
\item $\iota_{\s E}(\xi a)=\iota_{\s E}(\xi)\iota_A(a)$ and $\iota_A(a)\iota_{\s E^*}(\xi^*)=\iota_{\s E^*}(a\xi^*)$ for all $\xi\in\s E$ and $a\in A$;
\item $\iota_{\s E^*}(\xi^*)=\iota_{\s E}(\xi)^*$ and $\iota_{\K(\s E)}(\theta_{\xi,\eta})=\iota_{\s E}(\xi)\iota_{\s E}(\eta)^*$ for all $\xi,\eta\in\s E$;
\item $\iota_{\s E}(\xi)^*\iota_{\s E}(\eta)=\iota_A(\langle\xi,\,\eta\rangle)$ for all $\xi,\eta\in\s E$;
\item $\K(\s E\oplus A)$ is the C*-algebra generated by the set $\iota_A(A)\cup\iota_{\s E}(\s E)$.\qedhere
\end{enumerate}
\end{prop}

\begin{rks}\label{rk3}
\begin{enumerate}
\item For $a\in A$, $\xi\in\s E$ and $k\in\K(\s E)$, the operators $\iota_A(a)$, $\iota_{\s E}(\xi)$, $\iota_{\s E^*}(\xi^*)$ and $\iota_{\K(\s E)}(k)$ can be represented by 2-by-2 matrices acting on the Hilbert $A$-module $\s E\oplus A$ as follows:
\[
\iota_A(a)=\begin{pmatrix}0 & 0\\0 & a\end{pmatrix}\!;\quad \iota_{\s E}(\xi)=\begin{pmatrix}0 & \xi\\0 & 0\end{pmatrix}\!;\quad \iota_{\s E^*}(\xi^*)=\begin{pmatrix}0 & 0\\\xi^* & 0\end{pmatrix}\!;\quad \iota_{\K(\s E)}(k)=\begin{pmatrix}k & 0\\0 & 0\end{pmatrix}\!.
\]
Moreover, any operator $x\in\K(\s E\oplus A)$ can be written in a unique way as follows:
\[
x=\begin{pmatrix}k & \xi\\ \eta^* & a\end{pmatrix}\!, \quad \text{with } k\in\K(\s E),\ \xi,\eta\in\s E \; \text{and}\; a\in A.
\]
\item Note that $\iota_A$ and $\iota_{\K(\s E)}$ extend uniquely to strictly/*-strongly continuous unital *-homomor\-phisms $\iota_{A}:\M(A)\rightarrow\Lin(\s E\oplus A)$ and $\iota_{\K(\s E)}:\Lin(\s E)\rightarrow\Lin(\s E\oplus A)$. Besides, we have 
$
\iota_{A}(m)(\xi\oplus a)=0\oplus ma
$
and
$
\iota_{\K(\s E)}(T)(\xi\oplus a)=T\xi\oplus 0
$
for all $m\in\M(A)$, $T\in\Lin(\s E)$, $\xi\in\s E$ and $a\in A$. 
\item $\iota_{\s E^*}$ admits an extension to a bounded linear map $\iota_{\s E^*}:\Lin(\s E,A)\rightarrow\Lin(\s E\oplus A)$ in a straightforward way. Similarly, up to the identification $\s E=\K(A,\s E)$, we can also extend $\iota_{\s E}$ to a bounded linear map $\iota_{\s E}:\Lin(A,\s E)\rightarrow\Lin(\s E\oplus A)$.
\item As in 1, we can represent the operators $\iota_{A}(m)$, $\iota_{\K(\s E)}(T)$, $\iota_{\s E^*}(S)$ and $\iota_{\s E}(S^*)$, for $m\in\M(A)$, $T\in\Lin(\s E)$ and $S\in\Lin(A,\s E)$, by 2-by-2 matrices. Moreover, any operator $x\in\Lin(\s E\oplus A)$ can be written in a unique way as follows:
\[
x=\begin{pmatrix}T & S'\\ S^* & m\end{pmatrix}\!, \quad \text{with } T\in\Lin(\s E),\ S,S'\in\Lin(A,\s E) \; \text{and} \; m\in\M(A).\qedhere
\]
\end{enumerate}
\end{rks}

By using the matrix notations described above, we derive easily the following useful technical lemma:

\begin{lem}\label{ehmlem1}
Let $x\in\mathcal{L}(\s{E}\oplus A)$ {\rm(}resp.\ $x\in\K(\s E\oplus A))$. We have:
\begin{enumerate}
 \item $x\in\iota_{\s E}(\mathcal{L}(A,\s{E}))$ {\rm(}resp.\ $\iota_{\s E}(\s E))$ if, and only if, $x\iota_\s{E}(\xi)=0$ for all $\xi\in \s{E}$ and $\iota_A(a)x=0$ for all $a\in A$; in that case, we have $\iota_{A}(m)x=0$ for all $m\in\M(A)$;
 \item $x\in\iota_{\K(\s{E})}(\mathcal{L}(\s{E}))$ {\rm(}resp.\ $\iota_{\K(\s E)}(\K(\s E)))$ if, and only if, $x\iota_A(a)=0$ and $\iota_A(a)x=0$ for all $a\in A$; in that case, we have $x\iota_{A}(m)=\iota_{A}(m)x=0$ for all $m\in\M(A)$.\qedhere
\end{enumerate}
\end{lem}

Let us recall the notion of relative multiplier module, cf.\ 2.1 \cite{BS1}.

\begin{defin}
Let $A$ and $B$ be two C*-algebras and let $\s E$ be a Hilbert C*-module over $A$. Up to the identification $\s E\tens B=\K(A\tens B,\s E\tens B)$, we define $\widetilde{\M}(\s E\tens B)$\index[symbol]{mb@$\widetilde{\M}(\s E\tens B)$, relative multiplier module} to be the following subspace of $\Lin(A\tens B,\s E\tens B)$:
\[\{T\in\Lin(A\tens B,\s E\tens B) \; ; \; \forall x\in B, \; (1_{\s E}\tens x)T\in\s E\tens B\;\;\mathrm{and} \;\; T(1_A\tens x)\in\s E\tens B\}.\]
Note that $\widetilde{\M}(\s E\tens B)$ is a Hilbert C*-module over $\widetilde{\M}(A\tens B)$, whose $\widetilde{\M}(A\tens B)$-valued inner product is given by: 
\[
\langle\xi,\eta\rangle:=\xi^*\circ\eta,\quad \text{for all }\xi,\eta\in\widetilde{\M}(\s E\tens B)\subset\Lin(A\tens B,\s E\tens B).
\]
Note also that we have $\K(\widetilde{\M}(\s E\tens B))\subset\widetilde{\M}(\K(\s E)\tens B)$.
\end{defin}

\begin{propdef}\label{not4}
Let $B\subset\B(\s K)$ be a C*-algebra of operators on a Hilbert space $\s K$. For all $T\in\Lin(A\tens B,\s E\tens B)$ and $\omega\in\B(\s K)_*$, there exists a unique $(\id_{\s E}\tens\omega)(T)\in\Lin(A,\s E)$ such that
\[
\iota_{\s E}(\id_{\s E}\tens\omega)(T)=(\id_{\K(\s E\oplus A)}\tens\omega)(\iota_{\s E\tens B}(T))\in\Lin(\s E\oplus A),
\]
where $\iota_{\s E\tens B}:\Lin(A\tens B,\s E\tens B)\rightarrow\Lin((\s E\tens B)\oplus(A\tens B))=\M(\K(\s E\oplus A)\tens B)$. If $B\subset\B(\s K)$ is non-degenerate and $T\in\widetilde{\M}(\s E\tens B)$, then we have $(\id_{\s E}\tens\omega)(T)\in\s E$.
\end{propdef}

\begin{proof}
This is a direct consequence of \ref{ehmlem1} 1 and the fact that $\iota_{\s E\tens B}(T)\in\widetilde{\M}(\K(\s E\oplus A)\tens B)$ if $T\in\widetilde{\M}(\s E\tens B)$.
\end{proof}

\paragraph{Notion of equivariant Hilbert C*-module.}

In this paragraph, we recall the notion of equivariant Hilbert C*-module through the three equivalent pictures developed in \S 2 \cite{BS1}. Let us fix a $\QG$-C*-algebra $(A,\delta_A)$ and a Hilbert $A$-module $\s E$.

\begin{defin}
An action of the locally compact quantum group $\QG$ on $\s E$ is a linear map $\delta_{\s E}:\s E\rightarrow\widetilde{\M}(\s E\tens S)$ such that:
\begin{enumerate}
\item $\delta_{\s E}(\xi)\delta_A(a)=\delta_{\s E}(\xi a)$ and $\delta_A(\langle\xi,\,\eta\rangle)=\langle\delta_{\s E}(\xi),\,\delta_{\s E}(\eta)\rangle$, for all $a\in A$ and $\xi,\,\eta\in\s E$;
\item $[\delta_{\s E}(\s E)(A\tens S)]=\s E\tens S$;
\item the linear maps $\delta_{\s E}\tens\id_S$ and $\id_{\s E}\tens\delta$ extend to linear maps from $\Lin(A\tens S,\s E\tens S)$ to $\Lin(A\tens S\tens S,\s E\tens S\tens S)$ and we have $(\delta_{\s E}\tens\id_S)\delta_{\s E}=(\id_{\s E}\tens\delta)\delta_{\s E}$.
\end{enumerate}
An action $\delta_{\s E}$ of $\QG$ on $\s E$ is said to be continuous if we have 
$
[(1_{\s E}\tens S)\delta_{\s E}(\s E)]=\s E\tens S.
$
A $\QG$-equivariant Hilbert $A$-module is a Hilbert $A$-module $\s E$ endowed with a continuous action $\delta_{\s E}:\s E\rightarrow\widetilde{\M}(\s E\tens S)$ of $\QG$ on $\s E$.
\end{defin}

\begin{noh}
The datum of a continuous action of $\QG$ on $\s E$ is equivalent to that of a continuous action $\delta_{\K(\s E\oplus A)}:\K(\s E\oplus A)\rightarrow\M(\K(\s E\oplus A)\tens S)$ of $\QG$ on the linking C*-algebra $\K(\s E\oplus A)$ such that the *-homomorphism $\iota_A:A\rightarrow\K(\s E\oplus A)$ is $\QG$-equivariant, cf.\ 2.7 \cite{BS1}.
\end{noh}

\begin{noh}\label{unitary}
If $\delta_{\s E}$ is an action of $\QG$ on $\s E$, we have the unitary operator $\s V\in\Lin(\s E\tens_{\delta_A}(A\tens S),\s E\tens S)$ defined by $\s V(\xi\tens_{\delta_A}x):=\delta_{\s E}(\xi)x$ for all $\xi\in\s E$ and $x\in A\tens S$. This unitary satisfies the relation 
\[
(\s V\tens_{\GC}\id_S)(\s V\tens_{\delta_A\tens\id_S}1)=\s V\tens_{\id_A\tens\,\delta}1 \quad \text{in} \quad \Lin(\s E\tens_{\delta_A^2}(A\tens S\tens S),\s E\tens S\tens S), 
\]
where $\delta_A^2:=(\delta_A\tens\id_S)\delta_A=(\id_A\tens\delta)\delta_A$, cf.\ 2.3 and 2.4 (a) \cite{BS1} for the details. Conversely, if there exists a unitary operator $\s V\in\Lin(\s E\tens_{\delta_A}(A\tens S),\s E\tens S)$ satisfying the above relation and the fact that $\s VT_{\xi}\in\widetilde{\M}(\s E\tens S)$ for all $\xi\in\s E$, where $T_{\xi}\in\Lin(A\tens S,\s E\tens_{\delta_A}(A\tens S))$ is defined by $T_{\xi}(x):=\xi\tens_{\delta_A}x$ for all $x\in A\tens S$, then the map $\delta_{\s E}:\s E\rightarrow\widetilde{\M}(\s E\tens S)\,;\,\xi\mapsto\s VT_{\xi}$ is an action of $\QG$ on $\s E$, cf.\ 2.4 (b) \cite{BS1}.
\end{noh}

\begin{noh}
An action of $\QG$ on $\s E$ determines an action $\delta_{\K(\s E)}:\K(\s E)\rightarrow\widetilde{\M}(\K(\s E)\tens S)$ of $\QG$ on $\K(\s E)$ defined by $\delta_{\K(\s E)}(k)=\s V(k\tens_{\delta_A}1)\s V^*$ for all $k\in\K(\s E)$, where $\s V$ is the unitary operator associated to the action, cf.\ 2.8 \cite{BS1}. Moreover, if $\s E$ is a $\QG$-equivariant Hilbert module, then $\K(\s E)$ turns into a $\QG$-C*-algebra.
\end{noh}


\section{Measured quantum groupoids}

\setcounter{thm}{0}

\numberwithin{thm}{section}
\numberwithin{prop}{section}
\numberwithin{lem}{section}
\numberwithin{cor}{section}
\numberwithin{propdef}{section}
\numberwithin{nb}{section}
\numberwithin{nbs}{section}
\numberwithin{rk}{section}
\numberwithin{rks}{section}
\numberwithin{defin}{section}
\numberwithin{ex}{section}
\numberwithin{exs}{section}
\numberwithin{noh}{section}
\numberwithin{conv}{section}

For reminders concerning the relative tensor product of Hilbert spaces and the fiber product of von Neumann algebras, we refer the reader to the appendix of this article (cf.\ \ref{tensorproduct}).

\begin{defin}(cf.\ 3.7 \cite{E08}, 4.1 \cite{Le})
We call a measured quantum groupoid an octuple ${\cal G}=(N, M, \alpha, \beta, \Gamma, T, T', \nu)$, where:
\begin{itemize}
	\item $M$ and $N$ are von Neumann algebras;
	\item $\Gamma:M\rightarrow\fprod{M}{\beta}{\alpha}{M}$ is a faithful normal unital *-homomorphism, called the coproduct;
	\item $\alpha:N\rightarrow M$ and $\beta:N^{\rm o}\rightarrow M$ are faithful normal unital *-homormorphisms, called the range and source maps of $\cal G$;
	\item $T:M_+\rightarrow\alpha(N)_+^{\rm ext}$ and $T':M_+\rightarrow\beta(N^{\rm o})_+^{\rm ext}$ are \nsf operator-valued weights;
	\item $\nu$ is a \nsf weight on $N$;
\end{itemize}
such that the following conditions are satisfied:
\begin{enumerate}
	\item $[\alpha(n'),\,\beta(n^{\rm o})]=0$, for all $n,n'\in N$;
	\item $\Gamma(\alpha(n))=\reltens{\alpha(n)}{\beta}{\alpha}{1}$ and $\Gamma(\beta(n^{\rm o}))=\reltens{1}{\beta}{\alpha}{\beta(n^{\rm o})}$, for all $n\in N$;
	\item $\Gamma$ is coassociative, {\it i.e.\ }$(\fprod{\Gamma}{\beta}{\alpha}{\id})\Gamma=(\fprod{\id}{\beta}{\alpha}{\Gamma})\Gamma$;
	\item the \nsf weights $\varphi$ and $\psi$ on $M$ given by $\varphi=\nu\circ\alpha^{-1}\circ T$ and $\psi=\nu\circ\beta^{-1}\circ T'$ satisfy:
		\begin{itemize}
			\item $\forall x\in\f M_T^+, \, T(x)=(\fprod{\id}{\beta}{\alpha}{\varphi})\Gamma(x), \quad \forall x\in\f M_{T'}^+,\,T'(x)=(\fprod{\psi}{\beta}{\alpha}{\id})\Gamma(x)$,
			\item $\sigma^{\varphi}_t$ and $\sigma^{\psi}_s$ commute for all $s,t\in\GR$.\qedhere
		\end{itemize}
\end{enumerate}
\end{defin}

Let ${\cal G}=(N, M, \alpha, \beta, \Gamma, T, T', \nu)$ be a measured quantum groupoid. We denote by $(\s H,\pi,\Lambda)$ the \GNS construction for $(M,\varphi)$ where $\varphi:=\nu\circ\alpha^{-1}\circ T$. Let $(\sigma_t)_{t\in\GR}$, $\nabla$ and $J$ be respectively the modular automorphism group, the modular operator and the modular conjugation for $\varphi$. In the following, we identify $M$ with its image by $\pi$ in $\B(\s H)$. 
\begin{itemize}
	\item We have a coinvolutive *-antiautomomorphism $R_{\cal G}:M\rightarrow M$\index[symbol]{r@$R_{\cal G}$, unitary coinverse} such that $R_{\cal G}^2=\id_M$ (cf.\ 3.8 \cite{E08}).
\end{itemize}
From now on, we will assume that $T'=R_{\cal G}\circ T\circ R_{\cal G}$ and then also $\psi=\varphi\circ R_{\cal G}$.
\begin{itemize}
	\item There exist self-adjoint positive non-singular operators $\lambda$ and $d$ respectively affiliated to ${\cal Z}(M)$\index[symbol]{z@${\cal Z}(-)$, center} and $M$ such that
	 $
	 (D\psi:D\varphi)_t=\lambda^{{\rm i}t^2/2}d^{\,{\rm i}t}
	 $
for all $t\in\GR$. The operators $\lambda$ and $d$ are respectively called the scaling operator and the modular operator of $\cal G$. 
	\item The \GNS construction for $(M,\psi)$ is given by $(\s H,\pi_{\psi},\Lambda_{\psi})$, where: $\Lambda_{\psi}$ is the closure of the operator which sends any element $x\in M$ such that $x d^{1/2}$ is closable and its closure $\overline{x d^{1/2}}\in\f N_{\varphi}$ to $\Lambda_{\varphi}(\overline{x d^{1/2}})$; $\pi_{\psi}:M\rightarrow\B(\s H)$ is given by the formula $\pi_{\psi}(a)\Lambda_{\psi}(x)=\Lambda_{\psi}(ax)$.
	\item The modular conjugation $J_{\psi}$ for $\psi$ is given by $J_{\psi}=\lambda^{{\rm i}/4}J$.
	\item We will denote by
	$
	\s W_{\cal G}:\reltens{\s H}{\beta}{\alpha}{\s H}\rightarrow\reltens{\s H}{\alpha}{\widehat{\beta}}{\s H}
	$
	the pseudo-multiplicative unitary of $\cal G$\index[symbol]{wa@$\s W_{\cal G}$, pseudo-multiplicative unitary} (cf.\ 3.3.4 \cite{Val}, 3.6 \cite{E08}). 
\end{itemize}

\begin{propdef}(cf.\ 3.10 \cite{E08}) We define the (Pontryagin) dual of $\cal G$ to be the measured quantum groupoid $\widehat{\cal G}:=(N,\widehat{M},\alpha,\widehat{\beta},\widehat{\Gamma},\widehat{T},\widehat{R}\circ\widehat{T}\circ\widehat{R},\nu)$, where:
\begin{itemize}
	\item $\widehat{M}$ is the von Neumann algebra generated by 
	$
	\{(\omega\star\id)(\s W_{\cal G})\,;\,\omega\in\B(\s H)_*\}\subset\B(\s H);
	$
	\item $\widehat\beta:N^{\rm o}\rightarrow\widehat{M}$ is given by $\widehat\beta(n^{\rm o}):=J\alpha(n^*)J$ for all $n\in N$;
	\item $\widehat{\Gamma}:\widehat{M}\rightarrow\fprod{\widehat{M}}{\widehat\beta}{\alpha}{\widehat{M}}$ is given for all $x\in\widehat{M}$ by
	$
	\widehat{\Gamma}(x):=\sigma_{\alpha\widehat\beta}\s W_{\cal G}(\reltens{x}{\beta}{\alpha}{1})\s W_{\cal G}^*\sigma_{\widehat\beta\alpha};
	$
	\item there exists a unique \nsf weight $\widehat\varphi$ on $\widehat{M}$ whose \GNS construction is $(\s H,\id,\Lambda_{\widehat\varphi})$, where $\Lambda_{\widehat\varphi}$ is the closure of the operator 
	$
	(\omega\star\id)(\s W_{\cal G})\mapsto a_{\varphi}(\omega)
	$
defined for normal linear forms $\omega$ in a dense subspace of 
$
{\s I}_{\varphi}=\{\omega\in\B(\s H)_*\,;\,\exists k\in\GR_+,\,\forall x\in\f N_{\varphi},\, |\omega(x^*)|^2\leqslant k\varphi(x^*x)\}
$
and $a_{\varphi}(\omega)\in\s H$ satisfies
	$
	\omega(x^*)=\langle\Lambda_{\varphi}(x),\,a_{\varphi}(\omega)\rangle
	$
	for all $x\in\f N_{\varphi}$;
	\item $\widehat{T}$ is the unique \nsf operator-valued weight from $\widehat{M}$ to $\alpha(N)$ such that $\widehat\varphi=\nu\circ\alpha^{-1}\circ\widehat{T}$ and $\widehat{T}'=R_{\widehat{\cal G}}\circ\widehat{T}\circ R_{\widehat{\cal G}}$, where $R_{\widehat{\cal G}}:\widehat{M}\rightarrow\widehat{M}$ is given by $R_{\widehat{\cal G}}(x):=Jx^*J$ for all $x\in\widehat{M}$.
\end{itemize}
The pseudo-multiplicative unitary $\s W_{\widehat{\cal G}}$ of $\widehat{\cal G}$ is given by 
$
\s W_{\widehat{\cal G}}=\sigma_{\beta\alpha}\s W_{\cal G}^*\sigma_{\widehat\beta\alpha}.
$
\end{propdef}

We will denote by $\widehat{J}$ the modular conjugation for $\widehat\varphi$. Note that the scaling operator of $\widehat{\cal G}$ is $\lambda^{-1}$. In particular, we have $\lambda^{{\rm i}t}\in{\cal Z}(M)\cap{\cal Z}(\widehat{M})$ for all $t\in\GR$. 
\begin{itemize}
\item Let $\widehat\alpha(n):=J\beta(n^{\rm o})^*J=\widehat{J}\,\widehat\beta(n^{\rm o})^*\widehat{J}$ for $n\in N$. We recall the following relations (cf.\ 3.11 (v) \cite{E08}):
	$
	M\cap\widehat{M}=\alpha(N)$, $M\cap\widehat{M}'=\beta(N^{\rm o})$,
	$M'\cap\widehat{M}=\widehat\beta(N^{\rm o})$ and $M'\cap\widehat{M}'=\widehat\alpha(N)$.
\item Let $U:=\widehat{J}J\in\B(\s H)$\index[symbol]{u@$U$}. Then, $U^*=\lambda^{-i/4}U$ and $U^2=\lambda^{i/4}$ (cf.\ 3.11 (iv) \cite{E08}). In particular, $U$ is unitary. We have
$
\widehat\alpha(n)=U\alpha(n)U^*
$
and
$
\widehat\beta(n^{\rm o})=U\beta(n^{\rm o})U^*
$
for all $n\in N$. Since $\lambda^{-{\rm i}/4}\in{\cal Z}(M)$, we also have 
$
\widehat\alpha(n)=U^*\alpha(n)U
$
and
$
\widehat\beta(n^{\rm o})=U^*\beta(n^{\rm o})U
$
for all $n\in N$.
\end{itemize}

\begin{propdef}(cf.\ 3.12 \cite{E08})
\begin{itemize}
\item The octuple $(N^{\rm o},M,\beta,\alpha,\varsigma_{\beta\alpha}\circ\Gamma,R_{\cal G}\circ T\circ R_{\cal G},T,\nu^{\rm o})$ is a measured quantum groupoid denoted by $\cal G^{\rm o}$ and called the opposite of $\cal G$. The pseudo-multiplicative unitary of $\cal G^{\rm o}$ is given by $\s W_{\cal G^{\rm o}}=(_{\beta}\reltens{\widehat J}{\alpha}{\widehat\alpha}{\widehat J}_{\widehat\beta})\s W_{\cal G}(_{\beta}\reltens{\widehat J}{\alpha}{\alpha}{\widehat J}_{\beta})$.
\item  Let $C_M:M\rightarrow M'$ be the canonical *-antihomomorphism given by $C_M(x):=Jx^*J$ for all $x\in M$. Let us define:
	\[
	\Gamma^{\rm c}:=(\fprod{C_M}{\beta}{\alpha}{C_M})\circ\Gamma\circ C_M^{-1};\quad R_{\cal G}^{\rm c}:=C_M\circ R_{\cal G}\circ C_M^{-1};\quad T^{\rm c}=C_M\circ T\circ C_M^{-1}.
	\]
	Then, the octuple $(N^{\rm o},M',\widehat\beta,\widehat\alpha,\Gamma^{\rm c},T^{\rm c},R_{\cal G}^{\rm c}T^{\rm c}R_{\cal G}^{\rm c},\nu^{\rm o})$ is a measured quantum groupoid denoted by ${\cal G}^{\rm c}$ and called the commutant of $\cal G$. The pseudo-multiplicative unitary $\s W_{{\cal G}^{\rm c}}$ of ${\cal G}^{\rm c}$ is given by
	$
	\s W_{{\cal G}^{\rm c}}=(_{\widehat{\beta}}\reltens{J}{\alpha}{\alpha}{J}_{\widehat\beta})\s W_{\cal G}({}_{\beta}\reltens{J}{\widehat\alpha}{\alpha}{J}_{\widehat\beta}).
	$\qedhere
\end{itemize}
\end{propdef}

\begin{nbs}\label{pseudomult}
For a given measured quantum groupoid $\cal G$, we will need the following pseudo-multiplicative unitaries:\index[symbol]{va@$\widehat{\s V}$, $\s V$, $\widetilde{\s V}$}
 \[
 \widehat{\s V}  := {\s W}_{\cal G}; \quad   {\s V}  := {\s W}_{\widehat{({\cal G}^{\rm o})}} =   {\s W}_{(\widehat{\cal G})^{\rm c}}; \quad 
  \widetilde{\s V}  := {\s W}_{({\cal G}^{\rm o})^{\rm c}}.\qedhere
 \]
\end{nbs}

\begin{conv}\label{rkconv} 
Henceforth, we will refer to $(\widehat{\cal G})^{\rm c}$ instead of $\widehat{\cal G}$ as the dual of ${\cal G}$ since this groupoid is better suited for right actions of $\cal G$. We have
\[
(\widehat{\cal  G})^{\rm c} = (N^{\rm o}, \widehat M',  \beta, \widehat\alpha, \widehat\Gamma^{\rm c}, \widehat T^{\rm c}, \widehat T^{\rm c}{'}, \nu^{\rm o}),
\]
where the coproduct and the operator-valued weights are given by:
\begin{itemize}
	\item $\widehat\Gamma^{\rm c}(x) = ({\s W}_{(\widehat{\cal  G} )^{\rm c}})^* (\reltens{1}{\beta}{\alpha}{x}) {\s W}_{(\widehat{\cal  G} )^{\rm c}}$, for all $x\in \widehat M'$;
	\item $\widehat T^{\rm c}  = C_{\widehat M} \circ \widehat T \circ C_{\widehat M}^{-1}$, where 
$C_{\widehat M}  : \widehat M  \rightarrow \widehat M'$ ; $x \mapsto \widehat J x^* \widehat J$;
	\item $\widehat T^{\rm c}{'}  = R_{(\widehat{\cal  G} )^{\rm c}} \circ  \widehat T^{\rm c} \circ R_{(\widehat{\cal  G} )^{\rm c}}$.
\end{itemize}
Note also that the commutant weight  $\widehat{\varphi}^{\rm c} := \nu^{\rm o} \circ \beta^{-1} \circ \widehat T^{\rm c}$ derived from the weight $\widehat\varphi$ is left invariant for the coproduct $\widehat\Gamma^{\rm c}$.
In the following, we will simply denote by $\widehat{\cal G}$ the dual groupoid of $\cal G$ (since no ambiguity will arise with the Pontryagin dual). Note that the bidual groupoid is $({\cal G}^{\rm o})^{\rm c}=({\cal G}^{\rm c})^{\rm o}$. 
\end{conv} 

		\subsection{Case where the basis is finite-dimensional}\label{MQGfinitebasis}
	
\setcounter{thm}{0}

\numberwithin{thm}{subsection}
\numberwithin{prop}{subsection}
\numberwithin{lem}{subsection}
\numberwithin{cor}{subsection}
\numberwithin{propdef}{subsection}
\numberwithin{nb}{subsection}
\numberwithin{nbs}{subsection}
\numberwithin{rk}{subsection}
\numberwithin{rks}{subsection}
\numberwithin{defin}{subsection}
\numberwithin{ex}{subsection}
\numberwithin{exs}{subsection}
\numberwithin{noh}{subsection}
\numberwithin{conv}{subsection}
	
In \cite{DC}, De Commer provides an equivalent definition of a measured quantum groupoid on a finite basis. This definition is far more tractable since it allows us to circumvent the use of relative tensor products and fiber products.

\medskip
 
In the following, we fix a finite-dimensional C*-algebra
$
N:=\bigoplus_{1\leqslant l\leqslant k}{\rm M}_{n_l}(\GC)
$
endowed with the non-normalized Markov trace $\epsilon:=\bigoplus_{1\leqslant l\leqslant k}n_l\cdot{\rm Tr}_l$, where ${\rm Tr}_l$ denotes the non-normalized trace on ${\rm M}_{n_l}(\GC)$.\index[symbol]{ta@${\rm Tr}_l$, non-normalized Markov trace on ${\rm M}_{n_l}(\GC)$}

\medskip

We refer to \S \ref{tensorproduct} of the appendix for the definitions of $v_{\beta\alpha}$ and $q_{\beta\alpha}$.
Let us a fix a measured quantum groupoid ${\cal G}=(N,M,\alpha,\beta,\Gamma,T,T',\epsilon)$. We have a unital normal *-isomorphism
$
\fprod{M}{\beta}{\alpha}{M} \rightarrow q_{\beta\alpha}(M\tens M)q_{\beta\alpha}\,;\,
x \mapsto v_{\beta\alpha}^*xv_{\beta\alpha}
$
(cf.\ \ref{propcoiso}). Let $\Delta:M\rightarrow M\tens M$ be the (non necessarily unital) faithful normal *-homomor\-phism given by
$
\Delta(x)=v_{\beta\alpha}^*\Gamma(x)v_{\beta\alpha}
$
for all $x\in M$. We have $\Delta(1)=q_{\beta\alpha}$. This has led De Commer to the following equivalent definition of a measured quantum groupoid on a finite basis.

\begin{defin}\label{defMQG}(cf.\ 11.1.2 \cite{DC})
A measured quantum groupoid on the finite-dimen\-sional basis 
$
N
$
is an octuple ${\cal G}=(N,M,\alpha,\beta,\Delta,T,T',\epsilon)$, where:
\begin{itemize}
	\item $M$ is a von Neumann algebra, $\alpha:N\rightarrow M$ and $\beta:N^{\rm o}\rightarrow M$ are unital faithful normal *-homomorphisms;
	\item $\Delta:M\rightarrow M\tens M$ is a faithful normal *-homomorphism;
	\item $T:M_+\rightarrow\alpha(N)_+^{\ext}$ and $T':M_+\rightarrow\beta(N^{\rm o})_+^{\ext}$ are \nsf operator-valued weights;
\end{itemize}
such that the following conditions are satisfied:
\begin{enumerate}
	\item $[\alpha(n'),\,\beta(n^{\rm o})]=0$, for all $n,n'\in N$;
	\item $\Delta(1)=q_{\beta\alpha}$;
	\item $(\Delta\tens\id)\Delta=(\id\tens\Delta)\Delta$;
	\item $\Delta(\alpha(n))=\Delta(1)(\alpha(n)\tens 1)$ and $\Delta(\beta(n^{\rm o}))=\Delta(1)(1\tens\beta(n^{\rm o}))$, for all $n\in N$;
	\item the \nsf weights $\varphi$ and $\psi$ on $M$ given by $\varphi:=\epsilon\circ\alpha^{-1}\circ T$ and $\psi:=\epsilon\circ\beta^{-1}\circ T'$ satisfy:
		\[
		T(x)=(\id\tens\varphi)\Delta(x)\quad\text{for all } x\in\f M_T^+,\quad T'(x)=(\psi\tens\id)\Delta(x),\quad\text{for all } x\in\f M_{T'}^+;
		\]
	\item $\sigma_t^T\circ\beta=\beta$ and $\sigma_t^{T'}\circ\alpha=\alpha$, for all $t\in\GR$.\qedhere
\end{enumerate}
\end{defin}		

Let us fix a measured quantum groupoid ${\cal G}=(N,M,\alpha,\beta,\Delta,T,T',\epsilon)$.

\begin{nbs}
Let us consider the injective bounded linear map 
\[
\iota_{\,\widehat\alpha\alpha}^{\beta}:\B(\reltens{\s H}{\widehat\alpha}{\beta}{\s H},\reltens{\s H}{\beta}{\alpha}{\s H})  \rightarrow\B(\s H\tens\s H)\quad ;\quad
X  \mapsto v_{\beta\alpha}^*Xv_{\widehat\alpha\beta}.
\]
Similarly, we also define $\iota_{\!\beta\widehat\beta}^{\alpha}$ and $\iota_{\widehat\beta\beta}^{\widehat\alpha}$. Let\index[symbol]{vb@$V$, $W$, $\widetilde{V}$, multiplicative partial isometries}
\[
V:=\iota_{\,\widehat\alpha\alpha}^{\beta}(\s V),\quad W:=\iota_{\!\beta\widehat\beta}^{\alpha}(\widehat{\s V})\quad \text{and} \quad
\widetilde{V}:=\iota_{\widehat\beta\beta}^{\widehat\alpha}(\widetilde{\s V}),
\]
where $\s V=\s W_{\widehat{\cal G}}$, $\widehat{\s V}=\s  W_{\cal G}$ and $\widetilde{\s V}=\s W_{({\cal G}^{\rm o})^{\rm c}}$ (cf.\ \ref{pseudomult}).
\end{nbs}

In what follows, we recall the main properties satisfied by $V$, $W$ and $\widetilde{V}$. The proof of the results below are derived from the properties satisfied by the pseudo-multiplicative unitaries $\s V$, $\widehat{\s V}$ and $\widetilde{\s V}$ (cf.\ \cite{E08}, \S 11 \cite{DC} and \S 2 \cite{BC}).

\begin{prop}\label{inifinproj}(cf.\ 3.11 (iii), 3.12 (v), (vi) \cite{E08}, 2.2 \cite{BC}) 
The operators $V,W$ and $\widetilde{V}$ are multiplicative partial isometries acting on $\s H\tens\s H$ such that:
\begin{enumerate}
\item $W = \Sigma (U \otimes 1) V (U^* \otimes 1)\Sigma$,\quad $\widetilde V = \Sigma (1 \otimes U) V (1 \otimes U^*)\Sigma=(U\tens U)W(U^*\tens U^*)$;
\item $V^* = (J \otimes \widehat J) V (J \otimes \widehat J)$,\quad  $W^* = (\widehat J \otimes J) W (\widehat J \otimes J)$;
\item the initial and final projections are given by
\[
V^* V =q_{\widehat\alpha\beta}=\widetilde V \widetilde V^*,\quad W^*W  =q_{\beta\alpha}= V V^*,\quad W W^* = q_{\alpha\widehat\beta} \quad \text{and} \quad \widetilde V^*\widetilde V  = q_{\widehat\beta \widehat\alpha}.\qedhere
\]
\end{enumerate}
\end{prop}

\begin{prop}(cf.\ 3.8, 3.12 \cite{E08})
\begin{enumerate}
\item The von Neumann algebra $M$ {\rm(}resp.\ $\widehat{M}${\rm)} is the weak closure of
$\{(\id\tens\omega)(W)\,;\,\omega\in\B(\s H)_*\}$ {\rm(}resp.\ $\{(\omega\tens\id)(W)\,;\,\omega\in\B(\s H)_*\}${\rm)}.
\item We have
$
W\in M\tens\widehat{M}
$, 
$
V\in\widehat{M}'\tens M,
$
and
$
\widetilde{V}\in M'\tens\widehat{M}'.
$
In particular, we have the commutation relations $[W_{12},\, V_{23}] =  0$ and $[V_{12},\, \widetilde V_{23}] =  0$.
\item The coproduct $\Delta:M\rightarrow M\tens M$ of $\cal G$ (resp.\ $\widehat{\Delta}:\widehat{M}'\rightarrow\widehat{M}'\tens \widehat{M}'$ of $\widehat{\cal G}$) satisfies 
\begin{align*}
\Delta(x) &= V (x \otimes 1) V^* = W^* (1 \otimes x) W,\quad \text{for all } x\in M \\
\text{{\rm(}resp.\ } \widehat{\Delta}(x) &=V^* (1\otimes x) V = \widetilde V (x \otimes 1) \widetilde V^*,\quad \text{for all } x\in\widehat{M}'\text{{\rm)}}.\qedhere
\end{align*}
\end{enumerate}
\end{prop}

\begin{prop}\label{prop34}(cf.\ 3.2.\ (i), 3.6.\ (ii) \cite{E08} and 11.1.2 \cite{DC})
For all $n\in N$, we have:
\begin{enumerate}
	\item $[V, \alpha(n) \otimes 1] =0$,\quad $[V, \widehat \beta(n^{\rm o}) \otimes 1] = 0$,\quad   $[V, 1 \otimes \widehat\alpha(n)] = 0$,\quad $[V,  1 \otimes \widehat \beta(n^{\rm o})] = 0$;
	\item $V(1 \otimes \alpha(n)) =  (\widehat\alpha(n) \otimes 1) V$,\quad $V(\beta(n^{\rm o}) \otimes 1) =  (1 \otimes \beta(n^{\rm o}))V$;
	\item $[W, \widehat \beta(n^{\rm o}) \otimes 1] = 0$,\quad $[W, \widehat\alpha(n) \otimes 1] = 0$,\quad $[W, 1 \otimes \beta(n^{\rm o})]= 0$,\quad $[W, 1 \otimes \widehat\alpha(n)] = 0$;
	\item $W(1 \otimes \widehat \beta(n^{\rm o})) =  (\beta(n^{\rm o}) \otimes 1) W$,\quad $W(\alpha(n) \otimes 1) =  (1 \otimes \alpha(n))W$;
	\item $[\widetilde V, \alpha(n) \otimes 1] = 0$,\quad $[\widetilde V, \beta(n^{\rm o}) \otimes 1] = 0$,\quad $[\widetilde V, 1 \otimes \alpha(n)] = 0$,\quad $[\widetilde V, 1 \otimes \widehat \beta(n^{\rm o})] =  0$;
	\item $\widetilde V(1 \otimes \beta(n^{\rm o})) =  (\widehat \beta(n^{\rm o}) \otimes 1) \widetilde V$,\quad $\widetilde V(\widehat\alpha(n) \otimes 1) =  (1 \otimes \widehat\alpha(n))\widetilde V$.\qedhere
\end{enumerate}
\end{prop}

\begin{prop}\label{NonComRel}(cf.\ 11.1.4 \cite{DC})
For all $n\in N$, we have:
\begin{enumerate}
	\item $W(\beta(n^{\rm o})\tens 1)=W(1\tens\alpha(n))$,\quad $(1\tens\widehat{\beta}(n^{\rm o}))W=(\alpha(n)\tens 1)W$;
	\item $V(1\tens\beta(n^{\rm o}))=V(\widehat{\alpha}(n)\tens 1)$,\quad $(1\tens\alpha(n))V=(\beta(n^{\rm o})\tens 1)V$;
	\item $\widetilde{V}(\widehat{\beta}(n^{\rm o})\tens 1)=\widetilde{V}(1\tens\widehat{\alpha}(n))$,\quad $(1\tens\beta(n^{\rm o}))\widetilde{V}=(\widehat{\alpha}(n)\tens 1)\widetilde{V}$.\qedhere
\end{enumerate}
\end{prop}

		\subsection{Weak Hopf C*-algebras associated with a measured quantum groupoid}\label{WHC*A}
		
We recall the definitions and the main results concerning the weak Hopf $\Cstar$-algebras associated with a measured quantum groupoid on a finite basis, cf.\ \S 11.2 \cite{DC} (with different notations and conventions, cf.\ \S 2.3 \cite{BC}). Let us fix a measured quantum groupoid ${\cal G}=(N,M,\alpha,\beta,\Delta,T,T',\epsilon)$ on the finite-dimensional basis $N=\bigoplus_{1\leqslant l\leqslant k}{\rm M}_{n_l}(\GC)$.

\begin{nbs}
With the notations of \S\ref{MQGfinitebasis}, we denote by $S$ (resp.\ $\widehat{S}$) the norm closure of the subalgebra\index[symbol]{sb@$S$, $\widehat{S}$, weak Hopf C*-algebras}
\[
\{(\omega\tens\id)(V)\,;\,\omega \in \B(\s H)_\ast\} \quad \text{{\rm(}resp.}\,\{(\id\tens\omega)(V)\, ;\, \omega \in \B(\s H)_\ast\}\text{{\rm)}}.
\]
According to \S 11.2 \cite{DC}, we have the following statements:
\begin{itemize}
\item the Banach space $S$ (resp.\ $\widehat{S}$) is a non-degenerate C*-subalgebra of $\B(\s H)$, weakly dense in $M$ (resp.\ $\widehat{M}'$);
\item the C*-algebra $S$ (resp.\ $\widehat{S}$) is endowed with the faithful non-degenerate *-representa\-tions:\index[symbol]{l@$L$, $R$, $\rho$, $\lambda$, canonical representations of $S$ and $\widehat{S}$} 
\begin{align*}
L&:S\rightarrow\B(\s H)\,;\,x\mapsto x;\quad R:S\rightarrow\B(\s H)\,;\,x\mapsto UL(x)U^* \\
\text{{\rm(}resp.\ }\rho&:\widehat{S}\rightarrow\B(\s H)\,;\,x\mapsto x;\quad \lambda:\widehat{S}\rightarrow\B(\s H)\,;\,x\mapsto U\rho(x)U^*\text{{\rm)}};
\end{align*}
\item $\alpha(N)\subset\M(S)$, $\beta(N^{\rm o})\subset\M(S)$, $\beta(N^{\rm o})\subset\M(\widehat{S})$ and $\widehat{\alpha}(N)\subset\M(\widehat{S})$; 
\item $V\in\M( \widehat S \otimes S)$, $W \in \M( S \otimes \lambda(\widehat S))$ and $\widetilde V \in \M(R(S) \otimes \widehat S)$;
\item $\Delta$ {\rm (}resp.\ $\widehat{\Delta}${\rm )} restricts to a strictly continuous *-homomorphism $\delta:S\rightarrow\M(S\tens S)$ (resp.\ $\widehat{\delta}:\widehat{S}\rightarrow\M(\widehat{S}\tens\widehat{S})$), which uniquely extends to a strictly continuous *-homomorphism $\delta:\M(S)\rightarrow\M(S\tens S)$ (resp.\ $\widehat{\delta}:\M(\widehat{S})\rightarrow\M(\widehat{S}\tens\widehat{S})$) satisfying $\delta(1_S)=q_{\beta\alpha}$ (resp.\ $\widehat{\delta}(1_{\widehat{S}})=q_{\widehat\alpha\beta}$);
\item $\delta$ (resp.\ $\widehat{\delta}$) is coassociative and satisfies  
$
[\delta(S)(1_S \otimes S)] = \delta(1_S) (S \otimes S)=[\delta(S)(S \otimes 1_S)]
$
(resp.\ $[\widehat\delta(\widehat S)(1_{\widehat S} \otimes \widehat S)]  = \widehat\delta(1_{\widehat S}) (\widehat S \otimes \widehat S) = [\widehat\delta(\widehat S)(\widehat S \otimes 1_{\widehat S})]$);
\item the unital faithful *-homomorphisms $\alpha:N\rightarrow\M(S)$ and $\beta:N^{\rm o}\rightarrow\M(S)$ satisfy
	\[
	\delta(\alpha(n))=\delta(1_S)(\alpha(n)\tens 1_S) \; \text{ and } \; \delta(\beta(n^{\rm o}))=\delta(1_S)(1_S\tens\beta(n^{\rm o})),\; \text{ for all } n\in N;
	\]
\item the unital faithful *-homomorphisms $\beta:N^{\rm o}\rightarrow\M(\widehat{S})$ and $\widehat{\alpha}:N\rightarrow\M(\widehat{S})$ satisfy
	\[
	\widehat{\delta}(\beta(n^{\rm o}))=\widehat{\delta}(1_{\widehat{S}})(\beta(n^{\rm o})\tens 1_{\widehat{S}})\; \text{ and } \; \widehat{\delta}(\widehat{\alpha}(n))=\widehat{\delta}(1_{\widehat{S}})(1_{\widehat{S}}\tens\widehat{\alpha}(n)),\; \text{ for all } n\in N.\qedhere
	\]
\end{itemize}
\end{nbs}

\begin{rks}\label{rk18}
{\setlength{\baselineskip}{1.1\baselineskip}
The C*-algebras $S$ and $\widehat{S}$ are separable. Indeed, the Hilbert space $\s H$ is separable and we have $S=[(\omega_{\xi,\eta}\tens\id)(V)\,;\,\xi,\eta\in\s H]$ and $\widehat{S}=[(\id\tens\omega_{\xi,\eta})(V)\,;\,\xi,\eta\in\s H]$ (cf.\ \S \ref{integration}). In particular, the C*-algebras $S$ and $\widehat{S}$ are $\sigma$-unital.\qedhere
\par}
\end{rks}

\begin{defin}
With the above notations, we call the pair $(S,\delta)$ {\rm(}resp.\ $(\widehat{S},\widehat{\delta})${\rm)} the weak Hopf $\Cstar$-algebra {\rm(}resp.\ dual weak Hopf $\Cstar$-algebra{\rm)} associated with the measured quantum groupoid ${\cal G}$.
\end{defin}

\begin{rk}
With the notations of the above definition, the pair $(\widehat{S},\widehat{\delta})$ is the weak Hopf $\Cstar$-algebra of $\widehat{\cal G}$ while its dual weak Hopf $\Cstar$-algebra is the pair $(R(S),\delta_R)$, where $R(S)=USU^*$ and the coproduct $\delta_R$ is given by $\delta_R(y):=\widetilde{V}^*(1\tens y)\widetilde{V}$ for all $y\in R(S)$.
\end{rk}

\begin{defin}\label{DefRegMQG}
(cf.\ 4.7 \cite{E05}, 2.37 \cite{BC}) The measured quantum groupoid $\cal G$ is said to be regular if $\K_{\beta}={\cal C}(V)$ (cf.\ \ref{notC(V)} and \ref{not18} for the notations).
\end{defin}

Note that $\cal G$ is regular if, and only if, $\K_{\alpha}={\cal C}(W)$. Moreover, the dual measured quantum groupoid $\widehat{\cal G}$ is regular if, and only if, $\cal G$ is regular (cf.\ 2.8 \cite{BC}). For more information concerning the notion of regularity (and the notion of semi-regularity) for measured quantum groupoids on a finite basis, we refer the reader to \cite{BC} and \cite{C2}.

		\subsection{Measured quantum groupoid associated with a monoidal equivalence}\label{sectionColinking}

We will recall the construction of the measured quantum groupoid associated with a monoidal equivalence between two locally compact quantum groups provided by De Commer \cite{DC,DC3}. First of all, we will need to recall the definitions and the crucial results of De Commer \cite{DC,DC3}.

\begin{defin}
Let $\QG$ be a locally compact quantum group. A right (resp.\ left) Galois action of $\QG$ on a von Neumann algebra $N$ is an ergodic integrable right (resp.\ left) action $\alpha_N:N\rightarrow N\tens {\rm L}^{\infty}(\QG)$ $($resp.\ $\gamma_N:N\rightarrow {\rm L}^{\infty}(\QG)\tens N)$ such that the crossed product $N\rtimes_{\alpha_N}\QG$ $($resp.\ $\QG\;_{\gamma_N}\!\!\ltimes N)$ is a type I factor. Then, the pair $(N,\alpha_N)$ (resp. $(N,\gamma_N)$) is called a right $($resp.\ left$)$ Galois object for $\QG$.
\end{defin}

Let $\QG$ be a locally compact quantum group and let us fix a right Galois object $(N,\alpha_N)$ for $\QG$. In his thesis, De Commer was able to build a locally compact quantum group $\QH$ equipped with a left Galois action $\gamma_N$ on $N$ commuting with $\alpha_N$, {\it i.e.\ }$(\id\tens\alpha_N)\gamma_N=(\gamma_N\tens\id)\alpha_N$. This construction is called the \textit{reflection technique} and $\QH$ is called the \textit{reflected locally compact quantum group across} $(N,\alpha_N)$.\hfill\break
In a canonical way, he was also able to associate a right Galois object $(O,\alpha_O)$ for $\QH$ and a left Galois action $\gamma_O:O\rightarrow {\rm L}^{\infty}(\QG)\tens O$ of $\QG$ on $O$ commuting with $\alpha_O$. Finally, De Commer has built a measured quantum groupoid 
\[
{\cal G}_{\QH,\QG}=(\GC^2,M,\alpha,\beta,\Delta,T,T',\epsilon)
\]
where: $M={\rm L}^{\infty}(\QH)\oplus N\oplus O\oplus {\rm L}^{\infty}(\QG)$; $\Delta:M\rightarrow M\tens M$ is made up of the coactions and coproducts of the constituents of $M$; the operator-valued weights $T$ and $T'$ are given by the invariants weights; the non-normalized Markov trace $\epsilon$ on $\GC^2$ is simply given by $\epsilon(a,b)=a+b$ for all $(a,b)\in\GC^2$. Moreover, the source and target maps $\alpha$ and $\beta$ have range in ${\cal Z}(M)$ and generate a copy of $\GC^4$.\hfill\break
Conversely, if ${\cal G}=(\GC^2,M,\alpha,\beta,\Delta,T,T',\epsilon)$ is a measured quantum groupoid whose source and target maps have range in ${\cal Z}(M)$ and generate a copy of $\GC^4$, then ${\cal G}$ is of the form ${\cal G}_{\QH,\QG}$ in a unique way, where $\QH$ and $\QG$ are locally compact quantum groups canonically associated with $\cal G$.

\medskip

In what follows, we fix a measured quantum groupoid ${\cal G}=(\GC^2,M,\alpha,\beta,\Delta,T,T',\epsilon)$ whose source and target maps have range in ${\cal Z}(M)$ and generate a copy of $\GC^4$. It is worth noticing that for such a groupoid we have:

\begin{lem}\label{lem26}(cf.\ 2.21 \cite{BC})
$\widehat\alpha=\beta$ and $\widehat{\beta}=\alpha$.
\end{lem}

Following the notations introduced in \cite{DC}, we recall the precise description of the left and right regular representations $W$ and $V$ of $\cal G$ introduced in the previous section. We identify $M$ with its image by $\pi$ in $\B(\s H)$, where $(\s H,\pi,\Lambda)$ is the \GNS construction for $M$ endowed with the \nsf weight $\varphi=\epsilon\circ\alpha^{-1}\circ T$. We also consider the \nsf weight $\psi=\epsilon\circ\beta^{-1}\circ T'$. Denote by $(\varepsilon_1,\varepsilon_2)$ the standard basis of the vector space $\GC^2$.\index[symbol]{eaa@$(\varepsilon_1,\varepsilon_2)$, standard basis of $\GC^2$}  

\begin{nbs}\label{not10} Let us introduce some useful notations and make some remarks concerning them.
\begin{itemize}
	\item For $i,j=1,2$, we define the following nonzero central self-adjoint projection of $M$:\index[symbol]{pa@$p_{ij}$}
	\[
	p_{ij}:=\alpha(\varepsilon_i)\beta(\varepsilon_j).
	\]
	It follows from $\beta(\varepsilon_1)+\beta(\varepsilon_2)=1_M$ and $\alpha(\varepsilon_1)+\alpha(\varepsilon_2)=1_M$ that
	\[
	\alpha(\varepsilon_i)=p_{i1}+p_{i2} \quad \text{and} \quad \beta(\varepsilon_j)=p_{1j}+p_{2j},\quad \text{for all }i,j=1,2.
	\]
	\item We have 
	\begin{center}	
	$\Delta(1)=\alpha(\varepsilon_1)\tens \beta(\varepsilon_1)+\alpha(\varepsilon_2)\tens \beta(\varepsilon_2)$ \; and \; $\widehat{\Delta}(1)=\beta(\varepsilon_1)\tens \beta(\varepsilon_1)+\beta(\varepsilon_2)\tens \beta(\varepsilon_2)$
	\end{center} 
	since $\widehat{\alpha}=\beta$.
	\item Let $M_{ij}:=p_{ij}M$, for $i,j=1,2$. Then, $M_{ij}$ is a nonzero von Neumann subalgebra of $M$.\index[symbol]{mc@$M_{ij}$}
	\item Let $\s H_{ij}:=p_{ij}\s H$, for $i,j=1,2$. Then, $\s H_{ij}$ is a nonzero Hilbert subspace of $\s H$ for all $i,j=1,2$.\index[symbol]{h@$\s H_{ij}$}
	\item Let $\varphi_{ij}:=\varphi$\,$\restriction_{(M_{ij})_+}$ and $\psi_{ij}:=\psi$\,$\restriction_{(M_{ij})_+}$, for $i,j=1,2$. Then, $\varphi_{ij}$ and $\psi_{ij}$ are \nsf weights on $M_{ij}$.\index[symbol]{pb@$\varphi_{ij}$, $\psi_{ij}$}
	\item For all $i,j,k=1,2$, we denote by $\Delta_{ij}^k:M_{ij}\rightarrow M_{ik}\tens M_{kj}$ the unital normal *-homomorphism given by\index[symbol]{da@$\Delta_{ij}^k$} 
	\[
	\Delta_{ij}^k(x):=(p_{ik}\tens p_{kj})\Delta(x),\quad \text{for all } x\in M_{ij}.
	\]
	\item We have $Jp_{kl}=p_{kl}J$, $\widehat{J}p_{kl}=p_{lk}\widehat{J}$ and $Up_{kl}=p_{lk}U$ for all $k,l=1,2$. We define the anti-unitaries $J_{kl}:\s H_{kl}\rightarrow\s H_{kl}$, $\widehat{J}_{kl}:\s H_{kl}\rightarrow\s H_{lk}$ and the unitary $U_{kl}:\s H_{kl}\rightarrow\s H_{lk}$ by setting $J_{kl}=p_{kl}Jp_{kl}$, $\widehat{J}_{kl}=p_{lk}\widehat{J}p_{kl}$ and $U_{kl}=p_{lk}Up_{kl}=\widehat{J}_{kl}J_{kl}$.
	\item For $i,j,k,l=1,2$, let $\Sigma_{ij\tens kl}:=\Sigma_{\s H_{ij}\tens\s H_{kl}}:\s H_{ij}\tens\s H_{kl}\rightarrow\s H_{kl}\tens\s H_{ij}$.\qedhere
\end{itemize}
\end{nbs}

We readily obtain:
\[
M=\bigoplus_{i,j=1,2}M_{ij};\quad \s H=\bigoplus_{i,j=1,2}\s H_{ij};\quad \Delta(p_{ij})=p_{i1}\tens p_{1j}+p_{i2}\tens p_{2j},\text{ for all }i,j=1,2.
\]
Note that in terms of the parts $\Delta_{ij}^k$ of $\Delta$, the coassociativity condition reads as follows:
\[
(\Delta_{ik}^l\tens\id_{M_{kj}})\Delta_{ij}^k=(\id_{M_{il}}\tens\Delta_{lj}^k)\Delta_{ij}^l,\quad \text{for all } i,j,k,l=1,2.
\]
The \GNS representation for $(M_{ij},\varphi_{ij})$ is obtained by restriction of the \GNS representation of $(M,\varphi)$ to $M_{ij}$. In particular, the \GNS space $\s H_{\varphi_{ij}}$ is identified with $\s H_{ij}$.

\begin{prop}\label{prop35} For all $i,j,k,l=1,2$, we have:
\begin{align*}
(p_{ij} \otimes 1) V (p_{kl} \otimes 1) & = \delta_k^i (p_{ij} \otimes p_{jl}) V (p_{il} \otimes p_{jl});\\
(1 \otimes p_{ij}) W (1 \otimes p_{kl}) & = \delta_j^l (p_{ik} \otimes p_{ij}) W (p_{ik} \otimes p_{kj});\\
(1\tens p_{ji})\widetilde{V}(1\tens p_{lk}) & =\delta_j^l (p_{ki}\tens p_{ji})\widetilde{V}(p_{ki}\tens p_{jk}).\qedhere
\end{align*}
\end{prop}

\begin{nbs} The operators $V$, $W$ and $\widetilde{V}$ each splits up into eight unitaries\index[symbol]{vb@$V_{jl}^i$, $W_{ik}^j$, $\widetilde{V}_{ki}^j$} 
\[
V_{jl}^i : \s H_{il} \otimes \s H _{jl} \rightarrow \s H_{ij} \otimes \s H_{jl},\; 
W_{ik}^j : \s H_{ik} \otimes \s H_{kj} \rightarrow \s H_{ik} \otimes \s H_{ij} \text{ and }
\widetilde{V}_{ki}^j : \s H_{ki}\tens\s H_{jk}\rightarrow\s H_{ki}\tens\s H_{ji}
\] 
for $i,j,k,l=1,2$, given by $V_{jl}^i=(p_{ij} \otimes p_{jl}) V (p_{il} \otimes p_{jl})$, $ W_{ik}^j=(p_{ik} \otimes p_{ij}) W (p_{ik} \otimes p_{kj})$ and $\widetilde{V}_{ki}^j=(p_{ki}\tens p_{ji})\widetilde{V}(p_{ki}\tens p_{jk})$.
\end{nbs}

Let $i,j,k,l,l'=1,2$. These unitaries are related to each other by the following relations (cf.\ \ref{inifinproj}):
\[
W_{ik}^j=\Sigma_{ij\tens ik}(U_{ji}\tens 1)V_{ik}^j(U_{jk}^*\tens 1)\Sigma_{ik\tens kj};\;
\widetilde{V}_{ki}^j=\Sigma_{ji\tens ki}(1\tens U_{ik})V_{ik}^j(\tens U_{ik}^*)\Sigma_{ki\tens jk};
\]
\[
\widetilde{V}_{ki}^j=(U_{ik}\tens U_{ij})W_{ik}^j(U_{ik}^*\tens U_{kj}^*).
\]
Furthermore, we also have:
\[
(V_{jl}^i)^*=(J_{il}\tens\widehat{J}_{lj})V_{lj}^i(J_{ij}\tens\widehat{J}_{jl}) \quad \text{and} \quad 
(W_{ik}^l)^*=(\widehat{J}_{ki}\tens J_{kj})W_{ki}^j(\widehat{J}_{ik}\tens J_{ij}).
\]
Moreover, these unitaries satisfy the following pentagonal equations:
\begin{equation*}\label{penteq}
\begin{split}
(V_{jk}^i)_{12} (V_{kl}^i)_{13} (V_{kl}^j)_{23} = (V_{kl}^j)_{23} (V_{jl}^i)_{12};&\quad
(W_{ij}^k)_{12} (W_{ij}^l)_{13} (W_{jk}^l)_{23} = (W_{ik}^l)_{23}(W_{ij}^k)_{12};\\
(\widetilde{V}_{ji}^k)_{12}(\widetilde{V}_{ji}^l)_{13}&(\widetilde{V}_{kj}^l)_{23}=(\widetilde{V}_{ki}^l)_{23}(\widetilde{V}_{ji}^k)_{12}.
\end{split}
\end{equation*}
We also have the following commutation relations:
\begin{equation*}\label{comrel}
(V_{kj}^l)_{23}(W_{ll'}^j)_{12} = (W_{ll'}^k)_{12}(V_{kj}^{l'})_{23};\quad 
(V_{ki}^l)_{12}(\widetilde{V}_{ki}^{j})_{23}=(\widetilde{V}_{ki}^{j})_{23}(V_{ki}^l)_{12}.
\end{equation*}
Furthermore, we have
\begin{equation*}
\Delta_{ij}^k(x)=(W_{ik}^j)^*(1\tens x)W_{ik}^j=V_{kj}^i(x\tens 1)(V_{kj}^i)^*,\quad\text{for all } x\in M_{ij}.
\end{equation*}
Note that for all $\omega\in\B(\s H)_*$ we have:
\begin{equation*}\label{eqdefpi}
\begin{split}
({\rm id}  \otimes p_{jl}\omega p_{jl}) (V_{jl}^i) = p_{ij} ({\rm id}  \otimes \omega)(V) p_{il};& \quad (p_{ik}\omega p_{ik} \otimes {\rm id} )(W_{ik}^j) = 
p_{ij} (\omega \otimes {\rm id} )(W) p_{kj};\\
(p_{ki}\omega p_{ki} \tens \id)(\widetilde{V}_{ki}^j)&=p_{ji}(\id\tens\omega)(\widetilde{V})p_{jk}.
\end{split}
\end{equation*}

\begin{prop}\label{prop36}
Let $i , j=1,2$ such that $i\neq j$. With the notations of \ref{not10}, we have:
\begin{enumerate}
	\item $\QG_i:=(M_{ii} , \Delta_{ii}^i , \varphi_{ii}, \psi_{ii})$ is a locally compact quantum group whose left {\rm(}resp.\ right\,{\rm)} regular representation is $W_{ii}^i$ {\rm(}resp.\ $V_{ii}^i)$;
	\item $(M_{ij} , \Delta_{ij}^j)$ is a right Galois object for $\QG_j$ whose canonical implementation is $V_{jj}^i$;
	\item $(M_{ij} , \Delta_{ij}^i)$ is a left Galois object for $\QG_i$ whose canonical implementation is $W_{ii}^j$;
	\item the actions $\Delta_{ij}^j$ and $\Delta_{ij}^i$ on $M_{ij}$ commute;
	\item the Galois isometry associated with the right Galois object $(M_{ij},\Delta_{ij}^j)$ for $\QG_j$ (cf.\ 6.4.1, 6.4.2 \cite{DC}) is the unitary $\Sigma_{ij\tens jj}(W_{ij}^j)^*\Sigma_{ij\tens ij}$.\qedhere
\end{enumerate}
\end{prop}

\begin{defin}
A measured quantum groupoid $(\GC^2,M,\alpha,\beta,\Delta,T,T',\epsilon)$ such that the sour\-ce and 
target maps have range in ${\cal Z}(M)$ and generate a copy of $\GC^4$ will be denoted by ${\cal G}
_{\QG_1,\QG_2}$, where $\QG_i=(M_{ii},\Delta_{ii}^i,\varphi_{ii},\psi_{ii})$ (cf.\ \ref{prop36}) and will be called a colinking measured quantum groupoid.
\end{defin}

\begin{defin}
Let $\QG$ and $\QH$ be two locally compact quantum groups. We say that $\QG$ and $\QH$ are monoidally equivalent if there exists a colinking measured quantum groupoid ${\cal G}_{\QG_1,\QG_2}$ between two locally compact quantum groups $\QG_1$ and $\QG_2$ such that $\QH$ $($resp.\ $\QG)$ is isomorphic to $\QG_1$ $($resp.\ $\QG_2)$.\index[symbol]{g@${\cal G}_{\QG_1,\QG_2}$, colinking measured quantum groupoid}
\end{defin}

Let us recall the following necessary and sufficient condition for a colinking measured quantum groupoid to be regular.

\begin{thm}\label{theo7}(cf.\ 2.45 \cite{BC})
Let ${\cal G}_{\QG_1,\QG_2}$ be a colinking measured quantum groupoid associated with two monoidally equivalent locally compact quantum groups $\QG_1$ and $\QG_2$. The groupoid ${\cal G}_{\QG_1,\QG_2}$ is  regular if, and only if, the quantum groups $\QG_1$ and $\QG_2$ are regular.
\end{thm}

Let $(S,\delta)$ be the weak Hopf $\Cstar$-algebra associated with $\cal G$. Note that 
\[
p_{ij}=\alpha(\varepsilon_i)\beta(\varepsilon_j)\in{\cal Z}(\M(S)),\quad\text{for all } i,j=1,2.
\]

\begin{nbs}\label{not19}
Let us recall the notations below (cf.\ 2.26 \cite{BC}).
\begin{itemize}
	\item Let $S_{ij}:=p_{ij}S$, for $i,j=1,2$. Then, $S_{ij}$ is a $\Cstar$-algebra (actually a closed two-sided ideal) of $S$ weakly dense in $M_{ij}$.\index[symbol]{sca@$S_{ij}$}
	\item For $i,j=1,2$, we have the following faithful non-degenerate *-representations of the C*-algebra $S_{ij}$:
	\[
	L_{ij}:S_{ij}\rightarrow\B(\s H_{ij}),\; x\mapsto \restr{L(x)}{\s H_{ij}};\quad
	R_{ij}:S_{ij}\rightarrow\B(\s H_{ji}),\; x\mapsto U_{ij}L_{ij}(x)U_{ji}.
	\]
	\item For $i,j,k=1,2$, let 
	$
	\iota_{ij}^k:\M(S_{ik}\tens S_{kj})\rightarrow\M(S\tens S)
	$
	\index[symbol]{ib@$\iota_{ij}^k$}be the unique strictly continuous extension of the inclusion map $S_{ik}\tens S_{kj}\subset S\tens S$ satisfying $\iota_{ij}^k(1_{S_{ik}\tens S_{kj}})=p_{ik}\tens p_{kj}$. 
	\item Let $\delta_{ij}^k:S_{ij}\rightarrow\M(S_{ik}\tens S_{kj})$ be the unique *-homomorphism such that\index[symbol]{db@$\delta_{ij}^k$} 
\[
\iota_{ij}^k\circ\delta_{ij}^k(x)=(p_{ik}\tens p_{kj})\delta(x),\quad\text{for all }x\in S_{ij}.\qedhere
\]
\end{itemize}
\end{nbs}

With these notations, we have:

\begin{prop}\label{prop5bis}(cf.\ 7.4.13, 7.4.14 \cite{DC}, 2.27 \cite{BC})
Let $i,j,k,l=1,2$.
\begin{enumerate}
	\item $(\delta_{ik}^l\tens\id_{S_{kj}})\delta_{ij}^k=(\id_{S_{il}}\tens\delta_{lj}^k)				\delta_{ij}^l$.
	\item $\delta_{ij}^k(x)=(W_{ik}^j)^*(1_{\s H_{ik}}\tens x)W_{ik}^j=V_{kj}^i(x\tens 1_{\s H_{kj}})(V_{kj}^i)^*$, for all $x\in S_{ij}$.
	\item $[\delta_{ij}^k(S_{ij})(1_{S_{ik}}\tens S_{kj})]=S_{ik}\tens S_{kj}=[\delta_{ij}				^k(S_{ij})(S_{ik}\tens 1_{S_{kj}})]$. In particular, we have 
	\[
	S_{kj}=[(\id_{S_{ik}}\tens\omega)\delta_{ij}^k(x)\,;\,x\in S_{ij},\,\omega\in\B(\s H_{kj})_*].
	\]
	\item The pair $(S_{jj},\delta_{jj}^j)$ is the Hopf C*-algebra associated with $\QG_j$.\qedhere
\end{enumerate}
\end{prop}

\section{C*-algebras acted upon by measured quantum groupoids}

\setcounter{thm}{0}

\numberwithin{thm}{subsection}
\numberwithin{prop}{subsection}
\numberwithin{lem}{subsection}
\numberwithin{cor}{subsection}
\numberwithin{propdef}{subsection}
\numberwithin{nb}{subsection}
\numberwithin{nbs}{subsection}
\numberwithin{rk}{subsection}
\numberwithin{rks}{subsection}
\numberwithin{defin}{subsection}
\numberwithin{ex}{subsection}
\numberwithin{exs}{subsection}
\numberwithin{noh}{subsection}
\numberwithin{conv}{subsection}

	\subsection{Continuous actions, crossed product and biduality}\label{subsectionBiduality}

In this section, we fix a measured quantum groupoid ${\cal G}=(N,M,\alpha,\beta,\Delta,T,T',\epsilon)$ on the finite-dimensional basis $N=\bigoplus_{1\leqslant l\leqslant k}\,{\rm M}_{n_l}(\GC)$ and we use all the notations introduced in \S\S\ \ref{MQGfinitebasis}, \ref{WHC*A}. In the following, we recall the definitions, notations and results of \S\S\ 3.1, 3.2.1, 3.2.2 and 3.3.1 \cite{BC} (see also \cite{C} chapter 4) and those of \S 5.3 \cite{C2}.

		\subsubsection{Notion of actions of measured quantum groupoids on a finite basis}\label{CoAction}

\begin{lem}\label{lem19bis}
Let $A$ and $B$ be two $\Cstar$-algebras, $f:A\rightarrow\M(B)$ a *-homomorphism and $e\in\M(B)$. The following statements are equivalent:
\begin{enumerate}[label=(\roman*)]
	\item there exists an approximate unit $(u_{\lambda})_{\lambda}$ of $A$ such that $f(u_{\lambda})\rightarrow e$ with respect to the strict topology;
	\item $f$ extends to a strictly continuous *-homomorphism $f:\M(A)\rightarrow\M(B)$, necessarily unique, such that $f(1_A)=e$;
	\item $[f(A)B]=eB$.
\end{enumerate}
In that case, $e$ is a self-adjoint projection, for all approximate unit $(v_{\mu})_{\mu}$ of $A$ we have $f(v_{\mu})\rightarrow e$ with respect to the strict topology and $[Bf(A)]=Be$.
\end{lem}

\begin{defin}\label{defactmqg}
An action of $\cal G$ on a $\Cstar$-algebra $A$ is a pair $(\beta_A,\delta_A)$ consisting of a non-degenerate *-homomorphism $\beta_A:N^{\rm o}\rightarrow\M(A)$ and a faithful *-homomorphism $\delta_A:A\rightarrow\M(A\tens S)$ such that:
\begin{enumerate}
\item $\delta_A$ extends to a strictly continuous *-homomorphism $\delta_A:\M(A)\rightarrow\M(A\tens S)$ such that $\delta_A(1_A)=q_{\beta_A\alpha}$ (cf.\ \ref{ProjectionCAlg});
\item $(\delta_A\tens\id_S)\delta_A=(\id_A\tens\delta)\delta_A$;
$\phantom{q_{\beta_A\alpha}\ref{rk15}}$ 
\item $\delta_A(\beta_A(n^{\rm o}))=q_{\beta_A\alpha}(1_A\tens\beta(n^{\rm o}))$, for all $n\in N$.
\end{enumerate}
We say that the action $(\beta_A,\delta_A)$ is strongly (resp.\ weakly) continuous if we have
\[
[\delta_A(A)(1_A\tens S)]=q_{\beta_A\alpha}(A\tens S) \quad \text{(resp.\ }A=[(\id_A\tens\omega)\delta_A(a)\,;\, a\in A,\, \omega\in\B(\s H)_*]\text{)}.
\]
A $\cal{G}$-C*-algebra is a triple $(A,\beta_A,\delta_A)$ consisting of a C*-algebra $A$ endowed with a strongly continuous action $(\beta_A,\delta_A)$.
\end{defin}

\begin{rks}\label{rk15}
\begin{enumerate}
\item By \ref{lem19bis}, the condition 1 is equivalent to requiring that for some (and then any) approximate unit $(u_{\lambda})$ of $A$, we have $\delta_A(u_{\lambda})\rightarrow q_{\beta_A\alpha}$ with respect to the strict topology of $\M(A\tens S)$. It is also equivalent to $[\delta_A(A)(A\tens S)]=q_{\beta_A\alpha}(A\tens S)$.
	\item Condition 1 implies that the *-homomorphisms $\delta_A\tens\id_S$ and $\id_A\tens\delta$ extend uniquely to strictly continuous *-homomorphisms from $\M(A\tens S)$ to $\M(A\tens S\tens S)$ such that 
$
(\delta_A\tens\id_S)(1_{A\tens S})=q_{\beta_A\alpha,12}
$
and 
$
(\id_A\tens\delta)(1_{A\tens S})=q_{\beta\alpha,23}.
$ In particular, condition 2 does make sense and we denote by $\delta_A^2:=(\delta_A\tens\id_S)\delta_A:A\rightarrow\M(A\tens S\tens S)$\index[symbol]{dca@$\delta_A^2$, iterated coaction map} the iterated coaction map.
\item An element $m\in\M(A)$ is said to be $\delta_A$-invariant if $\delta_A(m)=q_{\beta_A\alpha}(m\tens 1_S)$. In that case, we have $[\beta_A(n^{\rm o}),\, m]=0$ for all $n\in N$ (cf.\ 6.1.23 \cite{C2}). In particular, we have $[q_{\beta_A\alpha},\,m\tens 1_S]=0$.\qedhere
\end{enumerate}
\end{rks}

\begin{exs}\label{ex2} Let us give two basic examples.
\begin{itemize}
\item $(S,\beta,\delta)$ is a $\cal{G}$-$\Cstar$-algebra.
\item Let $\beta_{N^{\rm o}}:=\id_{N^{\rm o}}$. Let $\delta_{N^{\rm o}}:N^{\rm o}\rightarrow\M(N^{\rm o}\tens S)$ be the faithful unital *-homomorphism given by $\delta_{N^{\rm o}}(n^{\rm o}):=q_{\beta_{N^{\rm o}}\alpha}(1_{N^{\rm o}}\tens\beta(n^{\rm o}))$ for all $n\in N$. Then, the pair $(\beta_{N^{\rm o}},\delta_{N^{\rm o}})$ is an action of $\cal{G}$ on $N^{\rm o}$ called the trivial action.\qedhere
\end{itemize}
\end{exs}

\begin{prop}\label{prop39}
Let $(\delta_A,\beta_A)$ be an action of $\cal G$ on $A$. We have the following statements:
\begin{enumerate}
\item the iterated coaction map $\delta_A^2$ extends uniquely to a strictly continuous *-homomorphism $\delta_A^2:\M(A)\rightarrow\M(A\tens S\tens S)$ such that $\delta_A^2(1_A)=q_{\beta_A\alpha,12}q_{\beta\alpha,23}$; moreover, we have $\delta_A^2(m)=(\delta_A\tens\id_S)\delta_A(m)=(\id_A\tens\delta)\delta_A(m)$ for all $m\in\M(A)$;
\item for all $n\in N$, we have $\delta_A(\beta_A(n^{\rm o}))=(1_A\tens\beta(n^{\rm o}))q_{\beta_A\alpha}$;
\item if $(\beta_A,\delta_A)$ is strongly continuous, then we have $[(1_A\tens S)\delta_A(A)]=(A\tens S)q_{\beta_A\alpha}$.\qedhere
\end{enumerate}
\end{prop}

Let us provide a more explicit definition of what an action of the dual measured quantum groupoid $\widehat{\cal G}$ on a $\Cstar$-algebra $B$ is.

\begin{defin}
An action of $\widehat{\cal G}$ on a $\Cstar$-algebra $B$ is a pair $(\alpha_B,\delta_B)$ consisting of a non-degenerate *-homomorphism $\alpha_B:N\rightarrow\M(B)$ and a faithful *-homomorphism $\delta_B:B\rightarrow\M(B\tens\widehat{S})$ such that:
\begin{enumerate}
\item $\delta_B$ extends to a strictly continuous *-homomorphism $\delta_B:\M(B)\rightarrow\M(B\tens\widehat{S})$ such that $\delta_B(1_B)=q_{\alpha_B\beta}$ (cf.\ \ref{ProjectionCAlg});
\item $(\delta_B\tens\id_{\widehat{S}})\delta_B=(\id_B\tens\widehat\delta)\delta_B$;
\item $\delta_B(\alpha_B(n))=q_{\alpha_B\beta}(1_B\tens\widehat\alpha(n))$, for all $n\in N$.
\end{enumerate}
We say that the action $(\alpha_B,\delta_B)$ is strongly (resp.\ weakly) continuous if we have  
\[
[\delta_B(B)(1_B\tens\widehat S)]=q_{\alpha_B\beta}(B\tens\widehat S) \quad \text{(resp.\ } B=[(\id_B\tens\omega)\delta_B(b)\,;\, b\in B,\, \omega\in\B(\s H)_*]\text{)}.
\]
If $(\delta_B,\alpha_B)$ is a strongly continuous action of $\widehat{\cal G}$ on $B$, we say that the triple $(B,\alpha_B,\delta_B)$ is a  $\widehat{\cal G}$-$\Cstar$-algebra. 
\end{defin}

\begin{thm}
If the quantum groupoid $\cal G$ is regular, then any weakly continuous action of $\cal G$ is strongly continuous.
\end{thm}

\begin{defin}\label{defEquiHom}
For $i=1,2$, let $A_i$ be a $\Cstar$-algebra. For $i=1,2$, let $(\beta_{A_i},\delta_{A_i})$ be an action of $\cal{G}$ on $A_i$. A non-degenerate *-homomorphism 
$
f:A_1\rightarrow\M(A_2)
$ 
is said to be $\cal{G}$-equivariant if $(f\tens\id_S)\delta_{A_1}=\delta_{A_2}\circ f$ and $ f\circ\beta_{A_1}=\beta_{A_2}$.
\end{defin}

\begin{rk}\label{rkEqMorph} $\vphantom{\id_{\widehat{S}}}$With the notations and hypotheses of \ref{defEquiHom}, if $f$ satisfies the relation $(f\tens\id_S)\delta_{A_1}=\delta_{A_2}\circ f$,$\vphantom{\widehat{\cal G}}$ then $f$ satisfies necessarily the relation $f\circ\beta_{A_1}=\beta_{A_2}$, {\it i.e.}\ $f$ is $\cal G$-equivariant. 
\end{rk}

\begin{nb}
We denote by ${\sf Alg}_{\mathcal{G}}$\index[symbol]{ac@${\sf Alg}_{\mathcal{G}}$, category of $\cal G$-C*-algebras} the category whose objects are the $\cal{G}$-$\Cstar$-algebras and whose set of arrows between $\cal{G}$-$\Cstar$-algebras is the set of $\cal G$-equivariant non-degenerate *-homomorphisms.
\end{nb}

We will also need to consider possibly degenerate $\cal G$-equivariant *-homomor\-phism.

\begin{defin}\label{defEquiHombis}
For $i=1,2$, let $A_i$ be a $\Cstar$-algebra endowed with an action $(\beta_{A_i},\delta_{A_i})$ of $\cal G$ on $A_i$ such that $\delta_{A_i}(A_i)\subset\widetilde{\M}(A_i\tens S)$ ({\it e.g.\ }if $(\beta_{A_i},\delta_{A_i})$ is strongly continuous). A *-homomorphism $f:A_1\rightarrow\M(A_2)$ is said to be $\cal G$-equivariant $\vphantom{\widetilde{\M}}$if for all $a\in A_1$ and $n\in N$ we have $\delta_{A_2}(f(a))=(f\tens\id_S)\delta_{A_1}(a)$ (cf.\ \S \ref{sectionNotations}) and $f(\beta_{A_1}(n^{\rm o})a)=\beta_{A_2}(n^{\rm o})f(a)$.$\vphantom{\widetilde{\M}}$
\end{defin}

In a similar way, we define the notion of $\widehat{\cal G}$-equivariant *-homomorphism and the category of $\widehat{\cal G}$-C*-algebras ${\sf Alg}_{\widehat{\cal G}}$.

\medbreak

Let us introduce some definitions.

\begin{defin}\label{defLinkAlg}
A linking $\cal G$-C*-algebra (resp.\ linking $\widehat{\cal G}$-C*-algebra) is a quintuple $(J,\beta_J,\delta_J,e_1,e_2)$ (resp.\ $(J,\alpha_J,\delta_J,e_1,e_2)$) consisting of a C*-algebra $J$ endowed with a continuous action $(\beta_J,\delta_J)$ of $\cal G$ (resp.\ $(\alpha_J,\delta_J)$ of $\widehat{\cal G}$) and nonzero self-adjoint projections $e_1,e_2\in\M(J)$ satisfying the following conditions:
\begin{enumerate}
\item $e_1+e_2=1_J$;
\item $[Je_jJ]=J$, for all $j=1,2$;
\item $\delta_J(e_j)=q_{\beta_J\alpha}(e_j\tens 1_S)$ (resp.\ $\delta_J(e_j)=q_{\alpha_J\beta}(e_j\tens 1_{\widehat{S}})$), for all $j=1,2$ ({\it i.e.\ }$e_j$ is $\delta_J$-invariant).\qedhere
\end{enumerate} 
\end{defin}

		\subsubsection{Crossed product and dual action}\label{crossedproductCstar}

Let us fix a strongly continuous action $(\beta_A,\delta_A)$ of $\cal G$ on a C*-algebra $A$. 

\begin{nbs}\label{notCrossedProduct}
The *-representation\index[symbol]{pc@$\pi_L$, $\pi_{A,L}$} 
\begin{center}
$
\pi_L:=(\id_A\tens L)\circ\delta_A:A\rightarrow\Lin(A\tens\s H)
$
\end{center}
of $A$ on the Hilbert $A$-module $A\tens\s H$ extends uniquely to a strictly/*-strongly continuous faithful *-representation $\pi_L:\M(A)\rightarrow\Lin(A\tens\s H)$ such that $\pi_L(1_A)=q_{\beta_A\alpha}$. Moreover, we have
$
\pi_L(m)=\pi_L(m)q_{\beta_A\alpha}=q_{\beta_A\alpha}\pi_L(m)
$ 
for all $m\in\M(A)$. Consider the Hilbert $A$-module
\[
{\cal E}_{A,L}:=q_{\beta_A\alpha}(A\tens\s H).
\]
\index[symbol]{eb@${\cal E}_{A,L}$}By restricting $\pi_L$, we obtain a strictly/*-strongly continuous faithful unital *-representation
\[
\pi:\M(A)\rightarrow\Lin({\cal E}_{A,L})\; ;\; m\mapsto \restr{\pi_L(m)}{{\cal E}_{A,L}}\!\!.
\]
We have $[1_A\tens T,\,q_{\beta_A\alpha}]=0$ for all $T\in\M(\widehat{S})$. We then obtain a strictly/*-strongly continuous unital *-representation
\[
\widehat{\theta}:\M(\widehat{S})\rightarrow\Lin({\cal E}_{A,L})\; ;\; T\mapsto \restr{(1_A\tens T)}{{\cal E}_{A,L}}\!\!.
\]
\index[symbol]{pd@$\pi$, $\widehat{\theta}$}Note that if $\beta_A$ is faithful, then so is $\widehat{\theta}$.
\end{nbs}

For later reference, let us state the following lemma.

\begin{lem}\label{lem31}
If $m\in\M(A)$ is $\delta_A$-invariant, then $[\pi(m),\,\widehat{\theta}(x)]=0$ for all $x\in\widehat{S}$.
\end{lem}

\begin{proof}
{\setlength{\baselineskip}{1.1\baselineskip}
Let us fix a $\delta_A$-invariant element $m$ of $\M(A)$ and $x\in\widehat{S}$. We have to show that $[\pi_L(m),\,1_A\tens\rho(x)]=0$. This follows from the formula $\pi_L(m)=q_{\beta_A\alpha}(m\tens 1_{\s H})$ and the commutation relation $[q_{\beta_A\alpha},\,1_A\tens\rho(x)]=0$ ensuing from $\widehat{S}\subset\widehat{M}'\subset\alpha(N)'$.\qedhere
\par}
\end{proof}

\begin{propdef}\label{defCrossedProduct}
The norm closed subspace of $\Lin({\cal E}_{A,L})$ spanned by the products of the form $\pi(a)\widehat{\theta}(x)$ for $a\in A$ and $x\in\widehat{S}$ is a C*-subalgebra called the (reduced) crossed product of $A$ by the strongly continuous action $(\beta_A,\delta_A)$ of $\cal G$ and denoted by $A\rtimes{\cal G}$.\index[symbol]{ad@$A\rtimes{\cal G}$, crossed product}
\end{propdef}

In particular, the morphism $\pi$ (resp.\ $\widehat{\theta}$) defines a faithful unital *-homomorphism (resp.\ unital *-homomorphism) $\pi:\M(A)\rightarrow\M(A\rtimes{\cal G})$ (resp.\  $\widehat{\theta}:\widehat{S}\rightarrow\M(A\rtimes{\cal G})$). We will sometimes lay emphasis on $A$ and denote the canonical morphism $\pi$ (resp.\ $\widehat{\theta}$) by $\pi_A:A\rightarrow\M(A\rtimes{\cal G})$ (resp.\ $\widehat{\theta}_A:\widehat{S}\rightarrow\M(A\rtimes{\cal G})$). Similarly, we will occasionally denote the *-representation $\pi_L:A\rightarrow\Lin(A\tens\s H)$ by $\pi_{A,L}$.

\medbreak

Since $[\widetilde{V},\,\alpha(n)\tens 1]=0$, we have $[\widetilde{V}_{23},q_{\beta_A\alpha,12}]=0$. The operator $\widetilde{V}_{23}\in\Lin(A\tens\s H\tens\s H)$ restricts to a partial isometry 
\[
X:=\restr{\widetilde{V}_{23}}{{\cal E}_{A,L}\tens\s H}\,\in\Lin({\cal E}_{A,L}\tens\s H),
\]
whose initial and final projections are $X^*X=\restr{q_{\widehat{\beta}\alpha,23}}{{\cal E}_{A,L}\tens\s H}$ and $XX^*=\restr{q_{\widehat{\alpha}\beta,23}}{{\cal E}_{A,L}\tens\s H}$.

\begin{propdef}\label{defDualAct}
Let
\[
\delta_{A\rtimes{\cal G}}:A\rtimes{\cal G}\rightarrow\Lin({\cal E}_{A,L}\tens\s H) \quad \text{and} \quad \alpha_{A\rtimes{\cal G}}:N\rightarrow\M(A\rtimes{\cal G})
\] 
be the linear maps defined by:
\begin{itemize}
\item $\delta_{A\rtimes{\cal G}}(b):=X(b\tens 1)X^*$, for all
$b\in A\rtimes{\cal G}$;
\item $\alpha_{A\rtimes{\cal G}}(n):=\widehat{\theta}(\widehat{\alpha}(n))=(1_A\tens\widehat{\alpha}(n))\!\!\restriction_{{\cal E}_{A,L}}$, for all $n\in N$.
\end{itemize}
Then, $\delta_{A\rtimes{\cal G}}$ is a faithful *-homomorphism and $\alpha_{A\rtimes{\cal G}}$ is a non-degenerate *-homomorphism. Moreover, we have the following statements:
\begin{enumerate}
\item $\delta_{A\rtimes{\cal G}}(\pi(a)\widehat{\theta}(x))=(\pi(a)\tens 1_{\widehat{S}})
(\widehat{\theta}\tens\id_{\widehat{S}})\widehat{\delta}(x)$, for all $a\in A$ and $x\in\widehat{S}$; in particular, $\delta_{A\rtimes{\cal G}}$ takes its values in $\M((A\rtimes{\cal G})\tens\widehat{S})$;
\item $\alpha_{A\rtimes{\cal G}}(n)\pi(a)\widehat{\theta}(x)=\pi(a)\widehat{\theta}(\widehat{\alpha}
(n)x)$ and $\pi(a)\widehat{\theta}(x)\alpha_{A\rtimes{\cal G}}(n)=\pi(a)\widehat{\theta}(x
\widehat{\alpha}(n))$ for all $n\in N$, $a\in A$ and $x\in\widehat{S}$.\qedhere
\end{enumerate}
\end{propdef}

\begin{propdef}
With the notations of \ref{defDualAct}, the pair $(\alpha_{A\rtimes{\cal G}},\delta_{A\rtimes{\cal G}})$ is a strongly continuous action of $\widehat{\cal G}$ on ${A\rtimes{\cal G}}$ called the dual action of $(\beta_A,\delta_A)$.
\end{propdef}

In the following result, we investigate the functoriality of the crossed product construction.

\begin{prop}\label{FunctorCrossedProd}
For $i=1,2$, let $A_i$ be a $\cal G$-C*-algebra. Let $f:A_1\rightarrow\M(A_2)$ be a $\cal G$-equivariant *-homomorphism (cf.\ \ref{defEquiHombis}). 
\begin{enumerate}
\item There exists a unique *-homomorphism $f_*:A_1\rtimes{\cal G}\rightarrow\M(A_2\rtimes{\cal G})$ such that for all $a\in A_1$ and $x\in\widehat{S}$, $f(\pi_{A_1}(a)\widehat{\theta}_{A_1}(x))=\pi_{A_2}(f(a))\widehat{\theta}_{A_2}(x)$. Moreover, $f_*$ is $\widehat{\cal G}$-equivariant. Note that if $f(A_1)\subset A_2$, then $f_*(A_1\rtimes{\cal G})\subset A_2\rtimes{\cal G}$.
\item The correspondence $-\rtimes{\cal G}:{\sf Alg}_{\cal G}\rightarrow{\sf Alg}_{\widehat{\cal G}}$ is functorial.\qedhere
\end{enumerate}
\end{prop}

\begin{proof}
1. Let $i=1,2$. Let us denote $\pi_{A_i,L}:{A_i}\rightarrow\Lin({A_i}\tens\s H)\,;\,a\mapsto(\id_{A_i}\tens L)\delta_{A_i}(a)$ and $B_i:=A_i\rtimes{\cal G}$. In this proof, we make the identifications (cf.\ \S \ref{sectionNotations})
\begin{center}
$
\widetilde{\M}(A_i\tens\K)\subset\M({A_i}\tens\K)=\Lin({A_i}\tens\s H).
$ 
\end{center}
We also identify
\begin{center}
$
\Lin({\cal E}_{{A_i},L})=\{T\in\Lin({A_i}\tens\s H)\,;\, Tq_{\beta_{A_i}\alpha}=T=q_{\beta_{A_i}\alpha}T\}.
$ 
\end{center}
We then have $B_i=[\pi_{{A_i},L}(a)(1_{A_i}\tens\rho(x))\,;\,a\in {A_i},\, x\in\widehat{S}]$. It follows from $L(S)\K=\K$ and $\delta_{A_i}({A_i})\subset\widetilde{\M}({A_i}\tens S)$ that $B_i\subset\widetilde{\M}({A_i}\tens\K)$. Let $f\tens\id_{\K}:\widetilde{\M}({A_1}\tens\K)\rightarrow\Lin(A_2\tens\s H)$ (cf.\ \ref{sectionNotations}). By a straightforward computation, we have
\begin{center}
$
(f\tens\id_{\K})(\pi_{{A_1},L}(a)(1_{A_1}\tens\rho(x)))=\pi_{A_2,L}(f(a))(1_{A_2}\tens\rho(x)),\quad \text{for all } a\in A_1 \text{ and } x\in\widehat{S}.
$
\end{center}
Hence, $(f\tens\id_{\K})(B_1)\subset\M(B_2)$. Let $f_*:=\restr{(f\tens\id_{\K})}{B_1}:B_1\rightarrow\M(B_2)$. We have proved that the *-homomorphism $f_*$ satisfies $f(\pi_{A_1}(a)\widehat{\theta}_{A_1}(x))=\pi_{A_2}(f(a))\widehat{\theta}_{A_2}(x)$ for all $a\in A_1$ and $x\in\widehat{S}$. In particular, for all $a\in A_1$ and $u\in\widehat{S}\tens\widehat{S}$ we have 
\[
(f_*\tens\id_{\widehat{S}})((\pi_{A_1}(a)\tens 1_{\widehat{S}})(\widehat{\theta}_{A_1}\tens\id_{\widehat{S}})(u))=(\pi_{A_2}(f(a))\tens 1_{\widehat{S}})(\widehat{\theta}_{A_2}\tens\id_{\widehat{S}})(u).
\]
Let $a\in A_1$ and $x,x'\in\widehat{S}$. By a straightforward computation, it follows from \ref{defDualAct} 1, the previous formula and the relation $\widehat{\delta}(x)(1_{\widehat{S}}\tens x')\in\widehat{S}\tens\widehat{S}$ that 
\[
(f_*\tens\id_{\widehat{S}})(\delta_{B_1}(\pi_{A_1}(a)\widehat{\theta}_{A_1}(x))(1_{B_1}\tens x'))=
\delta_{B_2}(f_*(\pi_{A_1}(a)\widehat{\theta}_{A_1}(x)))(1_{B_2}\tens x').
\]
Hence, $(f_*\tens\id_{\widehat{S}})\delta_{B_1}(\pi_{A_1}(a)\widehat{\theta}_{A_1}(x))=\delta_{B_2}(f_*(\pi_{A_1}(a)\widehat{\theta}_{A_1}(x)))$ for all $a\in A_1$ and $x\in\widehat{S}$. Moreover, it is easy to see that $f_*(\alpha_{A_1}(n)b)=\alpha_{A_2}(n)f_*(b)$ for all $b\in B_1$. Hence, $f_*$ is $\widehat{\cal G}$-equivariant.\newline
2. If $f$ is non-degenerate, then so is $f_*$. Indeed, we have $A_2=f(A_1)A_2$. Let $a\in A_1$, $b\in A_2$ and $x\in\widehat{S}$. By \ref{defCrossedProduct}, $\pi_{A_2}(f(a)b)\widehat{\theta}_{A_2}(x)=\pi_{A_2}(f(a))\pi_{A_2}(b)\widehat{\theta}_{A_2}(x)$ is the norm limit of finite sums of the form $\sum_i\pi_{A_2}(f(a))\widehat{\theta}_{A_2}(x_ix_i')\pi_{A_2}(b_i)=\sum_i f_*(\pi_{A_1}(a_i)\widehat{\theta}_{A_1}(x_i))\widehat{\theta}_{A_2}(x_i')\pi_{A_2}(b_i)$ with $x_i,x_i'\in\widehat{S}$ and $b_i\in A_2$. Hence, $\pi_{A_2}(f(a)b)\widehat{\theta}_{A_2}(x)\in [f_*(B_1)B_2]$. Hence, $f_*$ is non-degenerate. The functoriality of the correspondence $-\rtimes{\cal G}:{\sf Alg}_{\cal G}\rightarrow{\sf Alg}_{\widehat{\cal G}}$ easily follows.
\end{proof}

In a similar way, we define the crossed product of a C*-algebra $B$ by a strongly continuous action $(\alpha_B,\delta_B)$ of the dual measured quantum groupoid $\widehat{\cal G}$. 

\begin{nbs}
The *-representation\index[symbol]{pe@$\widehat{\pi}_{\lambda}$} 
\begin{center}
$
\widehat\pi_\lambda:=(\id_B\tens\lambda)\circ\delta_B:B\rightarrow\Lin(B\tens\s H)
$
\end{center}
of $B$ on the Hilbert $B$-module $B\tens\s H$ extends uniquely to a strictly/*-strongly continuous faithful *-representation $\widehat\pi_\lambda:\M(B)\rightarrow\Lin(B\tens\s H)$ such that $\widehat\pi_\lambda(1_B)=q_{\alpha_B\widehat\beta}$. Moreover, we have
$
\widehat\pi_\lambda(m)=\widehat\pi_\lambda(m)q_{\alpha_B\widehat\beta}=q_{\alpha_B\widehat\beta}\widehat\pi_\lambda(m)
$, 
for all $m\in\M(B)$. Consider the Hilbert $B$-module
\[
{\cal E}_{B,\lambda}:=q_{\alpha_B\widehat\beta}(B\tens\s H).
\]
\index[symbol]{ec@${\cal E}_{B,\lambda}$}By restricting $\widehat\pi_\lambda$, we obtain a strictly/*-strongly continuous faithful unital *-representation
\[
\widehat\pi:\M(B)\rightarrow\Lin({\cal E}_{B,\lambda})\; ;\; m\mapsto \restr{\widehat\pi_\lambda(m)}{{\cal E}_{B,\lambda}}\!.
\]
We have $[1_B\tens T,\,q_{\alpha_B\widehat\beta}]=0$ for all $T\in\M(S)$. We then obtain a strictly/*-strongly continuous unital *-representation
\[
\theta:\M(S)\rightarrow\Lin({\cal E}_{B,\lambda})\; ;\; T\mapsto \restr{(1_B\tens T)}{{\cal E}_{B,\lambda}}\!.
\]
\index[symbol]{pf@$\widehat{\pi}$, $\theta$}Note that if $\alpha_B$ is faithful, then so is $\theta$.
\end{nbs}

\begin{propdef}
The norm closed subspace of $\Lin({\cal E}_{B,\lambda})$ spanned by the products of the form $\widehat{\pi}(b)\theta(x)$ for $b\in B$ and $x\in S$ is a C*-subalgebra called the (reduced) crossed product of $B$ by the strongly continuous action $(\alpha_B,\delta_B)$ of $\widehat{\cal G}$ and denoted by $B\rtimes\widehat{\cal G}$.\index[symbol]{bc@$B\rtimes\widehat{\cal G}$, crossed product}
\end{propdef}

In particular, the morphism $\widehat\pi$ (resp.\ ${\theta}$) defines a faithful unital *-homomorphism (resp.\ unital *-homomorphism) $\widehat\pi:\M(B)\rightarrow\M(B\rtimes\widehat{\cal G})$ (resp.\  ${\theta}:S\rightarrow\M(B\rtimes\widehat{\cal G})$). We will sometimes want to lay emphasis on $B$ by denoting the canonical morphism $\widehat\pi$ (resp.\ $\theta$) by $\widehat{\pi}_B:B\rightarrow\M(B\rtimes\widehat{\cal G})$ (resp.\ $\theta_B:S\rightarrow\M(B\rtimes\widehat{\cal G})$).

\medbreak

Since $[V,\,\beta(n^{\rm o})\tens 1]=0$, we have $[V_{23},q_{\alpha_B\beta,12}]=0$. The operator $V_{23}\in\Lin(B\tens\s H\tens\s H)$ restricts to a partial isometry 
\[
Y:=\restr{V_{23}}{{\cal E}_{B,\lambda}\tens\s H}\,\in\Lin({\cal E}_{B,\lambda}\tens\s H),
\]
whose initial and final projections are $Y^*Y=\restr{q_{\widehat{\alpha}\beta,23}}{{\cal E}_{B,\lambda}\tens\s H}$ and $YY^*=\restr{q_{\beta\alpha,23}}{{\cal E}_{B,\lambda}\tens\s H}$.

\begin{propdef}\label{defDualActBis}
Let 
\[
\delta_{B\rtimes\widehat{\cal G}}:B\rtimes\widehat{\cal G}\rightarrow\Lin({\cal E}_{B,\lambda}\tens\s H) \quad \text{and} \quad \beta_{B\rtimes\widehat{\cal G}}:N^{\rm o}\rightarrow\Lin({\cal E}_{B,\lambda})
\]
be the linear maps defined by:
\begin{itemize}
\item $\delta_{B\rtimes\widehat{\cal G}}(c):=Y(c\tens 1_{\s H})Y^*$, for all $c\in B\rtimes\widehat{\cal G}$;
\item $\beta_{ B\rtimes\widehat{\cal G}}(n^{\rm o}):=\theta(\beta(n^{\rm o}))=\restr{(1_B\tens\beta(n^{\rm o}))}{{\cal E}_{B,\lambda}\tens\s H}$, for all $n\in N$.
\end{itemize}
Then, $\delta_{B\rtimes\widehat{\cal G}}$ is a faithful *-homomorphism and $\beta_{ B\rtimes\widehat{\cal G}}$ is a non-degenerate *-homomorphism. Moreover, we have the following statements:
\begin{enumerate}
\item $\delta_{B\rtimes\widehat{\cal G}}(\widehat\pi(b)\theta(x))=(\widehat\pi(b)\tens 1_{S})
(\theta\tens\id_S)\delta(x)$, for all $b\in B$ and $x\in S$; in particular, $\delta_{B\rtimes\widehat{\cal G}}$ takes its values in $\M((B\rtimes\widehat{\cal G})\tens S)$;
\item $\beta_{B\rtimes\widehat{\cal G}}(n^{\rm o})\widehat\pi(b)\theta(x)=\widehat\pi(b)\theta(\beta(n^{\rm o})x)$ and $\widehat\pi(b)\theta(x)\beta_{B\rtimes\widehat{\cal G}}(n^{\rm o})=\widehat\pi(b)\theta(x\beta(n^{\rm o}))$ for all $n\in N$, $b\in B$ and $x\in S$.\qedhere
\end{enumerate}
\end{propdef}

\begin{propdef}
With the notations of \ref{defDualActBis}, the pair $(\beta_{B\rtimes\widehat{\cal G}},\delta_{B\rtimes\widehat{\cal G}})$ is a strongly continuous action of ${\cal G}$ on ${B\rtimes\widehat{\cal G}}$ called the dual action of $(\alpha_B,\delta_B)$.
\end{propdef}

Let us state the functoriality of the crossed product construction.

\begin{prop}\label{FunctorCrossedProdBis}
For $i=1,2$, let $B_i$ be a $\widehat{\cal G}$-C*-algebra. Let $f:B_1\rightarrow\M(B_2)$ be $\widehat{\cal G}$-equivariant *-homomorphism (cf.\ \ref{defEquiHombis}). 
\begin{enumerate}
\item There exists a unique *-homomorphism $f_*:B_1\rtimes\widehat{\cal G}\rightarrow\M(B_2\rtimes\widehat{\cal G})$ such that for all $b\in B_1$ and $x\in\widehat{S}$, $f(\widehat{\pi}_{B_1}(b)\theta_{B_1}(x))=\widehat{\pi}_{B_2}(f(a))\theta_{B_2}(x)$. Moreover, $f_*$ is ${\cal G}$-equivariant.
\item The correspondence $-\rtimes\widehat{\cal G}:{\sf Alg}_{\widehat{\cal G}}\rightarrow{\sf Alg}_{\cal G}$ is functorial.\qedhere
\end{enumerate}
\end{prop}

In the proposition below, we investigate the image of a linking $\cal G$-C*-algebra (resp.\ linking $\widehat{\cal G}$-C*-algebra) by the functor $-\rtimes{\cal G}:{\sf Alg}_{\cal G}\rightarrow{\sf Alg}_{\widehat{\cal G}}$ (resp.\ $-\rtimes\widehat{\cal G}:{\sf Alg}_{\widehat{\cal G}}\rightarrow{\sf Alg}_{\cal G}$).

\begin{lem}\label{lem24}
Let $(A,\beta_A,\delta_A)$ be a $\cal G$-C*-algebra {\rm(}resp.\ $\widehat{\cal G}$-C*-algebra{\rm)}. For all $m\in\M(A)$, $\pi(m)\in\M(A\rtimes{\cal G})$ {\rm(}resp.\ $\widehat{\pi}(m)\in\M(A\rtimes\widehat{\cal G})${\rm)} is $\delta_{A\rtimes{\cal G}}$-invariant {\rm(}resp.\ $\delta_{A\rtimes\widehat{\cal G}}$-invariant{\rm)}.
\end{lem}

\begin{proof}
We have $(\widehat{\theta}\tens\id_{\widehat{S}})\widehat{\delta}(1_{\widehat{S}})=(\widehat{\theta}\tens\id_{\widehat{S}})(q_{\widehat{\alpha}\beta})=q_{\alpha_{A\rtimes{\cal G}}\beta}$. By strict continuity of $\widehat{\theta}$ and $\delta_{A\rtimes{\cal G}}$, it follows from \ref{defDualAct} 1 that $\delta_{A\rtimes{\cal G}}(\pi(a))=(\pi(a)\tens 1_{\widehat{S}})q_{\alpha_{A\rtimes{\cal G}}\beta}$ for all $a\in A$. Hence, $\delta_{A\rtimes{\cal G}}(\pi(m))=q_{\alpha_{A\rtimes{\cal G}}\beta}(\pi(m)\tens 1_{\widehat{S}})$ by strict continuity of $\pi$ and $\delta_{A\rtimes{\cal G}}$, {\it i.e.\ }$\pi(m)$ is $\delta_{A\rtimes{\cal G}}$-invariant. 
\end{proof}

\begin{prop}\label{prop40}
If the quintuple $(J,\beta_J,\delta_J,e_1,e_2)$ {\rm(}resp.\ $(K,\alpha_K,\delta_K,f_1,f_2)${\rm)} is a linking $\cal G$-C*-algebra {\rm(}resp.\ linking $\widehat{\cal G}$-C*-algebra{\rm)}, then the quintuple $(J\rtimes{\cal G},\alpha_{J\rtimes{\cal G}},\delta_{J\rtimes{\cal G}},\pi(e_1),\pi(e_2))$ {\rm(}resp.\ $(K\rtimes\widehat{\cal G},\beta_{K\rtimes\widehat{\cal G}},\delta_{K\rtimes\widehat{\cal G}},\widehat{\pi}(f_1),\widehat{\pi}(f_2))${\rm)} is a linking $\widehat{\cal G}$-C*-algebra {\rm(}resp.\ linking $\cal G$-C*-algebra{\rm)}.
\end{prop}

\begin{proof}
Let $(J,\beta_J,\delta_J,e_1,e_2)$ be a linking $\cal G$-C*-algebra. Let $K:=J\rtimes{\cal G}$ and $(\alpha_K,\delta_K)$ the dual action of $(\beta_J,\delta_J)$. $\vphantom{\widehat{\theta}}$Since the canonical morphism $\pi:J\rightarrow\M(K)$ is non-degenerate, we have $\pi(e_1)+\pi(e_2)=1_K$. $\vphantom{\widehat{\theta}}$Let $j=1,2$. Since $\pi(e_j)\in\M(K)$, we have $[K\pi(e_j)K]\subset K$. Any element of $K$ is the norm limit of finite sums of the form $\sum_{\lambda}\widehat{\theta}(x_{\lambda})\pi(a_{\lambda})\widehat{\theta}(x_{\lambda}')$ with $x_{\lambda},x_{\lambda}'\in\widehat{S}$ and $a_{\lambda}\in J$. Since $J=[Je_jJ]$, any element of $K$ is the norm limit of finite sums of the form $\sum_{\lambda}\widehat{\theta}(x_{\lambda})\pi(a_{\lambda})\pi(e_j)\pi(b_{\lambda})\widehat{\theta}(x_{\lambda}')$ with $x_{\lambda},x_{\lambda}'\in\widehat{S}$ and $a_{\lambda},b_{\lambda}\in J$. Hence, $K\subset [K\pi(e_j)K]$. Hence, $K=[K\pi(e_j)K]$.$\vphantom{\widehat{\theta}}$ Thus, the quintuple $(J\rtimes{\cal G},\alpha_{J\rtimes{\cal G}},\delta_{J\rtimes{\cal G}},\pi(e_1),\pi(e_2))$ is a linking $\widehat{\cal G}$-C*-algebra (cf.\ \ref{lem24}).
\end{proof}

		\subsubsection{Takesaki-Takai duality}\label{sectionTT}

Let $(\beta_A,\delta_A)$ (resp.\ $(\alpha_B,\delta_B)$) be a strongly continuous action of the groupoid ${\cal G}$ (resp.\ $\widehat{\cal G}$) on a C*-algebra $A$ (resp.\ $B$).

\begin{nbs}\label{not16}
The *-representation\index[symbol]{pga@$\pi_R$}\index[symbol]{pgb@$\pi_{\rho}$} of $A$ (resp.\ $B$) on the Hilbert $A$-module $A\tens\s H$ (resp.\ the Hilbert $B$-module $B\tens\s H$)
\[
\pi_R:=(\id_A\tens R)\circ\delta_A:A\rightarrow\Lin(A\tens\s H) \quad
\text{{\rm(}resp.\ }\pi_{\rho}:=(\id_B\tens\rho)\delta_B:B\rightarrow\Lin(B\tens\s H)\text{{\rm)}}
\]
 extends uniquely to a strictly/*-strongly continuous faithful *-representation $\pi_R:\M(A)\rightarrow\Lin(A\tens\s H)$ (resp.\ $\pi_{\rho}:\M(B)\rightarrow\Lin(B\tens\s H)$) satisfying $\pi_R(m)=(\id_A\tens R)\delta_A(m)$ for all $m\in\M(A)$ and $\pi_R(1_A)=q_{\beta_A\widehat{\alpha}}$ (resp.\ $\pi_{\rho}(m)=(\id_B\tens\rho)\delta_B(m)$ for all $m\in\M(B)$ and $\pi_{\rho}(1_B)=q_{\alpha_B\beta}$). Consider the Hilbert $A$-module (resp.\ the Hilbert $B$-module)
\[
\er:=q_{\beta_A\widehat{\alpha}}(A\tens\s H) \quad \text{{\rm(}resp.\ }{\cal E}_{B,\rho}:=q_{\alpha_B\beta}(B\tens\s H)\text{{\rm)}}.
\]
\index[symbol]{eda@$\er$}\index[symbol]{edb@${\cal E}_{B,\rho}$}We recall that the Banach space\index[symbol]{dcb@$D$/$E$, bidual $\cal G$/$\widehat{\cal G}$-C*-algebra}
\begin{align*}
D &:=[\pi_R(a)(1_A\tens\lambda(x)L(y))\,;\, a\in A,\, x\in\widehat{S},\, y\in S] \\
\text{{\rm(}resp.\ }E &:=[\pi_{\rho}(b)(1_B\tens R(y)\lambda(x))\,;\, b\in B,\, y\in S,\, x\in\widehat{S}]\text{{\rm)}}
\end{align*}
is a C*-subalgebra of $\Lin(A\tens\s H)$ (resp.\ $\Lin(B\tens\s H)$) such that $u q_{\beta_A\widehat{\alpha}}=u=uq_{\beta_A\widehat{\alpha}}$ for all $u\in D$ (resp.\ $v q_{\alpha_B\beta}=v=q_{\alpha_B\beta}v$ for all $v\in E$). Moreover, we have $D(A\tens\s H)=\er$ (resp.\ $E(B\tens\s H)={\cal E}_{B,\rho}$). We also recall that there exists a unique strictly/*-strongly continuous faithful *-representation 
$
j_D:\M(D)\rightarrow\Lin(A\tens\s H) 
$ 
\index[symbol]{ja@$j_D$, $j_E$} (resp.\ $j_E:\M(E)\rightarrow\Lin(B\tens\s H)$) extending the inclusion map $D\subset\Lin(A\tens\s H)$ (resp.\ $E\subset\Lin(B\tens\s H)$) such that $j_D(1_D)=q_{\beta_A\widehat{\alpha}}$ (resp.\ $j_E(1_E)=q_{\alpha_B\beta}$).
\end{nbs}

\begin{prop}\label{IsoTT}
There exists a unique *-isomorphism $\phi:(A\rtimes{\cal G})\rtimes\widehat{\cal G}\rightarrow D$ {\rm(}resp.\ $\psi:(B\rtimes\widehat{\cal G})\rtimes{\cal G}\rightarrow E${\rm)} such that $\phi(\widehat{\pi}(\pi(a)\widehat{\theta}(x))\theta(y))=\pi_R(a)(1_A\tens\lambda(x)L(y))$ for all $a\in A$, $x\in\widehat{S}$ and $y\in S$ {\rm(}resp.\ $\psi(\pi(\widehat{\pi}(b)\theta(y))\widehat{\theta}(x))=\pi_{\rho}(b)(1_B\tens R(y)\rho(x))$ for all $b\in B$, $y\in S$ and $x\in\widehat{S}${\rm)}.
\end{prop}

\begin{nbs}\label{not17}
We denote $\K:=\K(\s H)$\index[symbol]{ka@$\K$, C*-algebra of compact operators on the G.N.S.\ space $\s H:={\rm L}^2({\cal G})$} for short. Let $\delta_0:A\tens\K\rightarrow\M(A\tens\K\tens S)$ (resp.\ $\delta_0:B\tens\K\rightarrow\M(B\tens\K\tens \widehat{S})$) be the *-homomorphism defined by 
\[
\delta_0(a\tens k)=\delta_A(a)_{13}(1_A\tens k\tens 1_S) \quad
\text{(resp.\ }\delta_0(b\tens k)=\delta_B(b)_{13}(1_B\tens k\tens 1_{\widehat{S}})
\]
for all $a\in A$ (resp.\ $b\in B$) and $k\in\K$. The morphism $\delta_0$ extends uniquely to a strictly continuous *-homomorphism still denoted by $\delta_0:\M(A\tens\K)\rightarrow\M(A\tens\K\tens S)$ (resp.\ $\delta_0:\M(B\tens\K)\rightarrow\M(B\tens\K\tens \widehat{S})$) such that $\delta_0(1_{A\tens\K})=q_{\beta_A\alpha,13}$ (resp.\ $\delta_0(1_{B\tens\K})=q_{\alpha_B\beta,13}$). Let us denote by ${\cal V}\in\Lin(\s H\tens S)$ (resp.\ $\widetilde{\cal V}\in\Lin(\s H\tens \widehat{S})$)\index[symbol]{vb@${\cal V}$, $\widetilde{\cal V}$} the unique partial isometry such that $(\id_{\K}\tens L)({\cal V})=V$ (resp.\ $(\id_{\K}\tens\rho)(\widetilde{\cal V})=\widetilde{V}$).
\end{nbs}

\begin{thm}\label{BidualityTheo}
There exists a unique strongly continuous action $(\beta_D,\delta_D)$ {\rm(}resp.\ $(\alpha_E,\delta_E)${\rm)} of $\cal G$ {\rm(}resp.\ $\widehat{\cal G}${\rm)} on the C*-algebra $D:=[\pi_R(a)(1_A\tens\lambda(x)L(y))\,;\,a\in A,\,x\in\widehat{S},\, y\in S]$ {\rm(}resp.\ $E:=[\pi_{\rho}(b)(1_B\tens R(y)\lambda(x))\,;\, b\in B,\, y\in S,\, x\in\widehat{S}]${\rm)} defined by the relations: 
\begin{align*}
(j_D\tens\id_S)\delta_D(u)&={\cal V}_{23}\delta_0(u){\cal V}_{23}^*,\;\; u\in D ; \;\;
j_D(\beta_D(n^{\rm o}))=q_{\beta_A\widehat{\alpha}}(1_A\tens\beta(n^{\rm o})),\;\; n\in N.\\
\text{{\rm(}resp.\ }(j_E\tens\id_{\widehat{S}})\delta_E(v)&=\widetilde{\cal V}_{23}\delta_0(v)\widetilde{\cal V}_{23}^*,\quad v\in E;\;\; j_E(\alpha_E(n))=q_{\alpha_B\beta}(1_B\tens\widehat{\alpha}(n)),\;\; n\in N\text{{\rm)}}
\end{align*}
Moreover, the canonical *-isomorphism $\phi:(A\rtimes{\cal G})\rtimes\widehat{\cal G}\rightarrow D$ {\rm(}resp.\ $\psi:(B\rtimes\widehat{\cal G})\rtimes{\cal G}\rightarrow E${\rm)} {\rm(}cf.\ \ref{IsoTT}{\rm)} is $\cal G$-equivariant {\rm(}resp.\ $\widehat{\cal G}$-equivariant{\rm)}. If the groupoid $\cal G$ is regular, then we have $D=q_{\beta_A\widehat{\alpha}}(A\tens\K)q_{\beta_A\widehat{\alpha}}$ {\rm(}resp.\ $E=q_{\alpha_B\beta}(B\tens\K)q_{\alpha_B\beta}${\rm)}.
\end{thm}

The $\cal G$ (resp.\ $\widehat{\cal G}$)-C*-algebra $D$ (resp. $E$) defined above will be referred to as the bidual $\cal G$ (resp.\ $\widehat{\cal G}$)-C*-algebra of $A$ (resp.\ $B$). We investigate below the case of a linking $\cal G$ (resp.\ $\widehat{\cal G}$)-C*-algebra (cf.\ \ref{prop40}).

\begin{rk}\label{rk19}
Let $(J,\beta_J,\delta_J,e_1,e_2)$ (resp.\ $(K,\alpha_K,\delta_K,f_1,f_2)$) be linking $\cal G$ (resp.\ $\widehat{\cal G}$)-C*-algebra. We have a bidual linking $\cal G$ (resp.\ $\widehat{\cal G}$)-C*-algebra $(D,\beta_D,\delta_D,\pi_R(e_1),\pi_R(e_2))$ (resp.\ $(E,\alpha_E,\delta_E,\pi_{\rho}(f_1),\pi_{\rho}(f_2)$) and $\phi:(J\rtimes{\cal G})\rtimes\widehat{\cal G}\rightarrow D$ (resp.\ $\psi:(K\rtimes\widehat{\cal G})\rtimes{\cal G}\rightarrow E$) is an isomorphism of linking $\cal G$ (resp.\ $\widehat{\cal G}$)-C*-algebras.
\end{rk}

	\subsection{Case of a colinking measured quantum groupoid}\label{colinkingMQG}

In this section, we fix a colinking measured quantum groupoid ${\cal G}:={\cal G}_{\QG_1,\QG_2}$ associated with two monoidally equivalent locally compact quantum groups $\QG_1$ and $\QG_2$. We follow all the notations recalled in \S \ref{sectionColinking} concerning the objects associated with $\cal G$.

	\subsubsection{C*-algebra acted upon by a colinking measured quantum groupoid}\label{sectionC*AlgColink}

In the following, we recall the notations and the main results of \S 3.2.3 \cite{BC} concerning the equivalent description of the $\cal G$-C*-algebras in terms of $\QG_1$-C*-algebras and $\QG_2$-C*-algebras. Let us fix a $\cal G$-C*-algebra $(A,\beta_A,\delta_A)$.

\begin{nbs}\label{not12}
\begin{itemize}
\item The morphism $\beta_A:\GC^2\rightarrow\M(A)$ is central. Let $q_j:=\beta_A(\varepsilon_j)$\index[symbol]{qa@$q_j$, $q_{A,j}$} for $j=1,2$. Then, $q_j$ is a central self-adjoint projection of $\M(A)$ and $q_1+q_2=1_A$. Let $A_j:=q_jA$ for $j=1,2$. For $j=1,2$, $A_j$ is a $\Cstar$-subalgebra (actually a closed two-sided ideal) of $A$ and we have $A=A_1\oplus A_2$.
\item For $j,k=1,2$, let 
$
\pi_j^k:\M(A_k\tens S_{kj})\rightarrow\M(A\tens S)
$ 
\index[symbol]{ph@$\pi_j^k$, $\pi_{A,j}^k$}be the unique strictly continuous extension of the inclusion $A_k\tens S_{kj}\subset A\tens S$ such that $\pi_{j}^k(1_{A_k\tens S_{kj}})=q_k\tens p_{kj}$.
\end{itemize}
In case of ambiguity, we will denote $\pi_{A,j}^k$ and $q_{A,j}$ instead of $\pi_j^k$ and $q_j$.
\end{nbs}

\begin{prop}\label{actprop}
For all $j,k=1,2$, there exists a unique faithful non-degenerate *-homo\-morphism 
\[\delta_{A_j}^k:A_j\rightarrow\M(A_k\tens S_{kj})\]
\index[symbol]{dd@$\delta_{A_j}^k$}such that for all $x\in A_j$, we have
\[
\pi_j^k\circ\delta_{A_j}^k(x)=(q_k\tens p_{kj})\delta_A(x)=(q_k\tens 1_S)\delta_A(x)=(1_A\tens \alpha(\varepsilon_k))\delta_A(x)=(1_A\tens p_{kj})\delta_A(x).
\] 
Moreover, we have:
\begin{enumerate}
\item $\displaystyle{\delta_A(a)=\sum_{k,j=1,2}\pi_j^k\circ\delta_{A_j}^k(q_ja)}$, for all $a\in A$;
\item $(\delta_{A_k}^l\tens\id_{S_{kj}})\delta_{A_j}^k=(\id_{A_l}\tens\delta_{lj}^k)\delta_{A_j}^l$, for all $j,k,l=1,2$;
\item $[\delta_{A_j}^k(A_j)(1_{A_k}\tens S_{kj})]=A_k\tens S_{kj}$, for all $j,k=1,2$; in particular, we have

{
\centering
$
A_k=[(\id_{A_k}\tens\omega)\delta_{A_j}^k(a)\,;\,a\in A_j,\,\omega\in\B(\s H_{kj})_*];
$
\par
}
\item $\delta_{A_j}^j:A_j\rightarrow\M(A_j\tens S_{jj})$ is a strongly continuous action of $\QG_j$ on $A_j$.\qedhere
\end{enumerate}
\end{prop}

From this concrete description of $\cal G$-$\Cstar$-algebras we can also give a convenient description of the $\cal G$-equivariant *-homomorphisms. With the above notations, we have the result below.

\begin{prop}\label{lem7bis}
Let $A$ and $B$ be two $\cal G$-$\Cstar$-algebras. For $k=1,2$, let $\iota_k:\M(B_k)\rightarrow\M(B)$ be the unique strictly continuous extension of the inclusion map $B_k\subset B$ such that $\iota_k(1_{B_k})=q_{B,k}$.
\begin{enumerate}
\item Let $f:A\rightarrow\M(B)$ be a non-degenerate $\cal G$-equivariant *-homomorphism. Then, for all $j=1,2$, there exists a unique non-degenerate *-homomorphism $f_j:A_j\rightarrow\M(B_j)$ such that for $k=1,2$ we have
\begin{equation}\label{eqmorpheq}
(f_k\tens\id_{S_{kj}})\circ\delta_{A_j}^k=\delta_{B_j}^k\circ f_j.
\end{equation}
Moreover, we have $f(a)=\iota_1\circ f_1(aq_{A,1})+\iota_2\circ f_2(aq_{A,2})$ for all $a\in A$.
\item Conversely, for $j=1,2$ let $f_j:A_j\rightarrow\M(B_j)$ be a non-degenerate *-homomorphism such that {\rm(}\ref{eqmorpheq}\,{\rm)} holds for all $j,k=1,2$. Then, the map $f:A\rightarrow\M(B)$, defined for all $a\in A$ by 
\[
f(a):=\iota_1\circ f_1(aq_{A,1})+\iota_2\circ f_2(aq_{A,2}),
\]
is a non-degenerate $\cal G$-equivariant *-homomorphism.\qedhere
\end{enumerate} 
\end{prop}

\subsubsection{Induction of equivariant C*-algebras}

The above results show that for $j=1,2$ we have a functor 
\begin{center}
${\sf Alg}_{\cal G}\rightarrow{\sf Alg}_{\QG_j}\,;\,(A,\beta_A,\delta_A)\mapsto(A_j,\delta_{A_j}^j).$
\end{center}
In \S 4 \cite{BC}, it has been proved that if $\cal G$ is regular (cf.\ \ref{theo7}), then $(A,\delta_A,\beta_A)\rightarrow(A_1,\delta_{A_1}^1)$ is an equivalence of categories. Moreover, the authors build explicitly the inverse functor $(A_1,\delta_{A_1})\rightarrow(A,\beta_A,\delta_A)$. More precisely, if $\QG_1$ and $\QG_2$ are regular, then to any $\QG_1$-$\Cstar$-algebra $(A_1,\delta_{A_1})$ they associate a $\QG_2$-$\Cstar$-algebra $(A_2,\delta_{A_2})$ in a canonical way. The $\Cstar$-algebra $A:=A_1\oplus A_2$ can be equipped with a strongly continuous action $(\beta_A,\delta_A)$ of the groupoid $\cal G$. This allowed them to build the inverse functor $(A_1,\delta_{A_1})\rightarrow(A,\beta_A,\delta_A)$. The equivalence of categories $(A_1,\delta_{A_1})\rightarrow(A_2,\delta_{A_2})$ generalizes the correspondence of actions for monoidally equivalent \textit{compact} quantum groups of De Rijdt and Vander Vennet \cite{RV}. We bring to the reader's attention that an induction procedure has been developed by De Commer in the von Neumann algebraic setting (cf.\ \S 8 \cite{DC}).

\medskip

In the following, we recall the notations and the main results of \S 4 \cite{BC}. We assume the quantum groups $\QG_1$ and $\QG_2$ to be regular.

\begin{nbs}
Let $\delta_{A_1}:A_1\rightarrow\M(A_1\tens S_{11})$ be a continuous action of $\QG_1$ on a $\Cstar$-algebra $A_1$. Let us denote:\index[symbol]{de@$\delta_{A_1}^{(2)}$}
\[
\delta_{A_1}^1:=\delta_{A_1},\quad \delta_{A_1}^{(2)}:=(\id_{A_1}\tens\delta_{11}^2)\delta_{A_1}:A_1\rightarrow\M(A_1\tens S_{12}\tens S_{21}).
\]
Then, $\delta_{A_1}^{(2)}$ is a faithful non-degenerate *-homomorphism. In the following, we will identify $S_{21}$ with a C*-subalgebra of $\B(\s H_{21})$. Let\index[symbol]{ic@$\ind(A_1)$, induced C*-algebra}
\[
{\rm Ind}_{\QG_1}^{\QG_2}(A_1):=[(\id_{A_1\tens S_{12}}\tens\omega)\delta_{A_1}^{(2)}(a)\,;\,a\in A_1,\,\omega\in\B(\s H_{21})_*]\subset\M(A_1\tens S_{12}).\qedhere
\]
\end{nbs}

\begin{propdef}\label{propind4}
The Banach space $A_2:={\rm Ind}_{\QG_1}^{\QG_2}(A_1)\subset\M(A_1\tens S_{12})$ is a C*-algebra.
\begin{itemize}
	\item We have the relations $[A_2(1_{A_1}\tens S_{12})]=A_1\tens S_{12}=[(1_{A_1}\tens S_{12})A_2]$. In particular, the inclusion $A_2\subset\M(A_1\tens S_{12})$ defines a faithful non-degenerate *-homomorphism and $\M(A_2)\subset\M(A_1\tens S_{12})$.
	\item Let $\delta_{A_2}:=\restr{(\id_{A_1}\tens\delta_{12}^2)}{A_2}$. Then, we have the inclusion $\delta_{A_2}(A_2)\subset\M(A_2\tens S_{22})$ and the *-homomorphism $\delta_{A_2}:A_2\rightarrow\M(A_2\tens S_{22})$ is a continuous action of $\QG_2$ on $A_2$.
\end{itemize}
The pair $\ind(A_1,\delta_{A_1}):=(A_2,\delta_{A_2})$ is called the induced $\QG_2$-C*-algebra of $(A_1,\delta_{A_1})$.
\end{propdef}

\begin{propdef}\label{defHomInd}
Let $A_1$ and $B_1$ be $\QG_1$-C*-algebras. Let $A_2$ {\rm(}resp.\ $B_2${\rm)} be the induced $\QG_2$-C*-algebra of $A_1$ {\rm(}resp.\ $B_1${\rm)}. Let $\phi_1:A_1\rightarrow\M(B_1)$ be a (possibly degenerate) $\QG_1$-equivariant *-homomorphism. There exists a unique *-homomorphism $\phi_2:A_2\rightarrow\M(B_2)$ such that 
\[
\phi_2((\id_{A_1\tens S_{12}}\tens\omega)\delta_{A_1}^{(2)}(a))=(\id_{B_1\tens S_{12}}\tens\omega)\delta_{B_1}^{(2)}(\phi_1(a)),\;\; \text{for all } a\in A_1 \text{ and } \omega\in\B(\s H_{21})_*.
\]
Moreover, $\phi_2$ is $\QG_2$-equivariant. The map $\ind\phi_1:=\phi_2$ is called the induced $\QG_2$-equivariant *-homomorphism of $\phi_1$. In particular, the correspondence ${\rm Ind}_{\QG_1}^{\QG_2}:{\sf Alg}_{\QG_1}\rightarrow {\sf Alg}_{\QG_2}$ is functorial.
\end{propdef}

\begin{proof}
{\setlength{\baselineskip}{1.3\baselineskip}
Let $\phi_1\tens\id_{S_{12}}:\widetilde{\M}(A_1\tens S_{12})\rightarrow\M(B_1\tens S_{12})$ be the *-homomorphism extending $\phi_1\tens\id_{S_{12}}$ (cf.\ \S \ref{sectionNotations}. We have $A_2\subset\widetilde{\M}(A_1\tens S_{12})$ (cf.\ \ref{propind4}). Let $\phi_2:=\restr{(\phi_1\tens\id_{S_{12}})}{A_2}$. Let us prove that $\delta_{A_1}^{(2)}(A_1)\subset\widetilde{\M}_{S_{12}\tens S_{21}}(A_1\tens S_{12}\tens S_{21})$. Let $a\in A_1$, $s\in S_{12}$ and $s'\in S_{21}$. It follows from the relation $[\delta_{11}^2(S_{11})(S_{12}\tens 1_{S_{21}})]=S_{12}\tens S_{21}$ that $\delta_{A_1}^{(2)}(a)(1_{A_1}\tens s\tens s')$ is the norm limit of finite sums of the form $\sum_i(\id_{A_1}\tens\delta_{11}^2)(\delta_{A_1}(a)(1_{A_1}\tens t_i))(1_{A_1}\tens s_i\tens 1_{S_{21}})$ with $t_i\in S_{11}$ and $s_i\in S_{12}$. It then follows from the inclusion $\delta_{A_1}(A_1)(1_{A_1}\tens S_{11})\subset A_1\tens S_{11}$ ($\delta_{A_1}$ is continuous) that $\delta_{A_1}^{(2)}(a)(1_{A_1}\tens s\tens s')$ is the norm limit of finite sums of the form $\sum_i(\id_{A_1}\tens\delta_{11}^2)(a_i\tens t_i)(1_{A_1}\tens s_i\tens 1_{S_{21}})=\sum_i a_i\tens\delta_{11}^2(t_i)(s_i\tens 1_{S_{21}})$ with $a_i\in A_1$, $t_i\in S_{11}$ and $s_i\in S_{12}$. Hence, $\delta_{A_1}^{(2)}(a)(1_{A_1}\tens s\tens s')\in A_1\tens S_{12}\tens S_{21}$.\newline
It is clear that $(\phi_1\tens\id_{S_{12}})(\id_{A_1\tens S_{12}}\tens\omega)(x)=(\id_{B_1\tens S_{12}}\tens\omega)(\phi_1\tens\id_{S_{12}}\tens\id_{S_{21}})(x)$ for all $x\in A_1\tens S_{12}\tens S_{21}$ and $\omega\in\B(\s H_{21})_*$. By composing by $\id_{B_1}\tens\phi$ for $\phi\in\B(\s H_{12})_*$ and by factorizing $\phi$ and $\omega$ on the same side by an element of $S_{12}$ and $S_{21}$ respectively, we remark that this formula still holds if $x\in\widetilde{\M}_{S_{12}\tens S_{21}}(A_1\tens S_{12}\tens S_{21})$. By a straightforward computation, we have 
$
(\phi_1\tens\id_{S_{12}\tens S_{21}})\delta_{A_1}^{(2)}(a)=\delta_{B_1}^{(2)}(\phi_1(a))
$
for all $a\in A_1$. We then conclude that $\phi_2((\id_{A_1\tens S_{12}}\tens\omega)\delta_{A_1}^{(2)}(a))=(\id_{B_1\tens S_{12}}\tens\omega)\delta_{B_1}^{(2)}(\phi_1(a))$ for all $a\in A_1$ and $\omega\in\B(\s H_{21})_*$. The proof of the equivariance of $\phi_2$ is similar to that of 4.3 c) \cite{BC}.\qedhere
\par}
\end{proof}

By exchanging the roles of the quantum groups $\QG_1$ and $\QG_2$, we obtain {\it mutatis mutandis} a functor ${\rm Ind}_{\QG_2}^{\QG_1}:{\sf Alg}_{\QG_2}\rightarrow {\sf Alg}_{\QG_1}$.

\begin{prop}\label{propind1}
Let $j,k=1,$ with $j\neq k$. Let $(A_j,\delta_{A_j})$ be a $\QG_j$-$\Cstar$-algebra. Let 
\[
A_k:={\rm Ind}_{\QG_j}^{\QG_k}(A_j)\subset\M(A_j\tens S_{jk}) \quad \text{and} \quad C:={\rm Ind}_{\QG_k}^{\QG_j}(A_k)\subset\M(A_k\tens S_{kj})
\]
endowed with the continuous actions $\delta_{A_k}:=\restr{(\id_{A_j}\tens\delta_{jk}^k)}{A_k}$ and $\delta_C:=\restr{(\id_{A_k}\tens\delta_{kj}^j)}{C}$ respectively. Then, we have:
\begin{enumerate}
	\item $C\subset\M(A_k\tens S_{kj})\subset\M(A_j\tens S_{jk}\tens S_{kj})$ and $C=\delta_{A_j}^{(k)}(A_j)$;
	\item $\pi_j:A_j\rightarrow C\,;\,a\mapsto \delta_{A_j}^{(k)}(a):=(\id_{A_j}\tens\delta_{jj}^k)\delta_{A_j}(a)$
	is a $\QG_j$-equivariant *-isomorphism;\index[symbol]{pi@$\pi_j$}
	\item $\delta_{A_j}^k:A_j\rightarrow\M(A_k\tens S_{kj})\,;\,a\mapsto\delta_{A_j}^{(k)}(a):=(\id_{A_j}\tens\delta_{jj}^k)\delta_{A_j}(a)$
	is a faithful non-degenerate *-homomorphism.\qedhere
\end{enumerate}
\end{prop}

The above result shows that the functors $\ind$ and $\iind$ are inverse of each other.

\begin{nbs}\label{not11}
Let $(B_1,\delta_{B_1})$ be a $\QG_1$-$\Cstar$-algebra. Let $(B_2,\delta_{B_2})$ be the induced $\QG_2$-$\Cstar$-algebra, that is to say $B_2={\rm Ind}_{\QG_1}^{\QG_2}(B_1)$ and $\delta_{B_2}=\restr{(\id_{B_1}\tens\delta_{12}^2)}{B_2}$. In virtue of \ref{propind1}, we have four *-homomorphisms:
\[
\delta_{B_j}^k:B_j\rightarrow\M(B_k\tens S_{kj}),\quad j,k=1,2.
\]
Let us give a precise description of them. We denote $\delta_{B_1}^1:=\delta_{B_1}$ and $\delta_{B_2}^2:=\delta_{B_2}$. The *-homomorphism $\delta_{B_1}^2:B_1\rightarrow\M(B_2\tens S_{21})$ is given by
\[
b\in B_1\mapsto \delta_{B_1}^2(b):=\delta_{B_1}^{(2)}(b)\in\M(B_2\tens S_{21}) \;\;\text{{\rm(}with }\delta_{B_1}^{(2)}(b):=(\id_{B_1}\tens\delta_{11}^2)\delta_{B_1}^1(b),\text{ for } b\in B_1\text{{\rm)}}
\]
whereas the *-homomorphism $\delta_{B_2}^1:B_2\rightarrow\M(B_1\tens S_{12})$ is defined by the relation
\[
(\pi_1\tens\id_{S_{12}})\delta_{B_2}^1(b)=\delta_{B_2}^{(1)}(b) \quad \text{for } b\in B_2,
\]
where
$
\delta_{B_2}^{(1)}:=(\id_{B_2}\tens \delta_{22}^1)\delta_{B_2}^2
$ 
and
$
\pi_1:B_1\rightarrow{\rm Ind}_{\QG_2}^{\QG_1}(B_2)\,;\,b \mapsto \delta_{B_1}^{(2)}(b)
$
(cf.\ \ref{propind1} 2).
\end{nbs}

\begin{prop}\label{prop4}
Let $(A,\beta_A,\delta_A)$ be a ${\cal G}$-$\Cstar$-algebra. Let $j,k=1,\,2$ with $j\neq k$. With the notations of \ref{actprop}, let
\[
(\widetilde{A}_j,\delta_{\widetilde{A}_j}):={\rm Ind}_{\QG_k}^{\QG_j}(A_k,\delta_{A_k}^k).
\]
If $x\in A_j$, then we have $\delta_{A_j}^k(x)\in\widetilde{A}_j\subset\M(A_k\tens S_{kj})$ and the map 
$
\widetilde{\pi}_j:A_j\rightarrow\widetilde{A}_j\,;\,x\mapsto\delta_{A_j}^k(x)
$
is a $\QG_j$-equivariant *-isomorphism.\index[symbol]{pja@$\widetilde{\pi}_j$}
\end{prop}

\begin{prop}\label{prop37}
Let $(B_1,\delta_{B_1})$ be a $\QG_1$-$\Cstar$-algebra. Let $B_2={\rm Ind}_{\QG_1}^{\QG_2}(B_1)$ be the induced $\QG_2$-$\Cstar$-algebra. Let $B:=B_1\oplus B_2$. For $j,\,k=1,2$ with $j\neq k$, let $\pi_j^k:\M(B_k\tens S_{kj})\rightarrow\M(B\tens S)$ be the strictly continuous *-homomorphism extending the canonical injection $B_k\tens S_{kj}\rightarrow B\tens S$ and $\delta_{B_j}^k:B_j\rightarrow\M(B_k\tens S_{kj})$ the *-homomorphisms defined in \ref{not11}.
Let $\beta_B:\GC^2\rightarrow\M(B)$ and $\delta_B:B\rightarrow\M(B\tens S)$ be the *-homomorphisms defined by:
\[
\beta_B(\lambda,\mu):=\begin{pmatrix}\lambda & 0\\ 0 & \mu\end{pmatrix}\!,\quad (\lambda,\mu)\in\GC^2;\quad
\delta_B(b):=\sum_{k,j=1,2}\pi_j^k\circ\delta_{B_j}^k(b_j),\quad b=(b_1,b_2)\in B.\]
Therefore, we have:
\begin{enumerate}
	\item $(\beta_B,\delta_B)$ is a strongly continuous action of ${\cal G}$ on $B$;
	\item the correspondence ${\sf Alg}_{\QG_1}\rightarrow{\sf Alg}_{\cal G}\,;\,(B_1,\delta_{B_1})\mapsto(B,\beta_B,\delta_B)$ is functorial;
	\item the functors ${\sf Alg}_{\QG_1}\rightarrow{\sf Alg}_{\cal G}$ and ${\sf Alg}_{\cal G}\rightarrow{\sf Alg}_{\QG_1}$ are inverse of each other.\qedhere
\end{enumerate}
\end{prop}

\section{Hilbert C*-modules acted upon by measured quantum groupoids}

	\subsection{Notion of actions of measured quantum groupoids on a finite basis on Hilbert C*-modules}\label{SectEqHilb}

\paragraph{The three pictures.} In this paragraph, we recall the notion of ${\cal G}$-equivariant Hilbert C*-module for a measured quantum groupoid ${\cal G}$ on a finite basis in the spirit of \cite{BS1} (cf.\ \S 6.1 \cite{C2}). We fix a measured quantum groupoid $\cal G$ on a finite-dimensional basis $N=\bigoplus_{1\leqslant l\leqslant k}{\rm M}_{n_l}(\GC)$ endowed with the non-normalized Markov trace $\epsilon=\bigoplus_{1\leqslant l\leqslant k}n_l\cdot{\rm Tr}_l$. We use all the notations introduced in \S \ref{MQGfinitebasis} and \S \ref{WHC*A} concerning the objects associated with $\cal G$. Let us fix a ${\cal G}$-$\Cstar$-algebra $A$.

\medbreak

Following \S 2 \cite{BS1}, an action of ${\cal G}$ on a Hilbert $A$-module $\s E$ is defined in \cite{C2} by three equivalent data:
\begin{itemize}
\item a pair $(\beta_{\s E},\delta_{\s E})$ consisting of a *-homomorphism $\beta_{\s E}:N^{\rm o}\rightarrow\Lin(\s E)$ and a linear map $\delta_{\s E}:\s E\rightarrow\widetilde{\M}(\s E\tens S)$ (cf.\ \ref{hilbmodequ});
\item a pair $(\beta_{\s E},\s V)$ consisting of a *-homomorphism $\beta_{\s E}:N^{\rm o}\rightarrow\Lin(\s E)$ and an isometry $\s V\in\Lin(\s E\tens_{\delta_A}(A\tens S),\s E\tens S)$ (cf.\ \ref{isometry});
\item an action $(\beta_J,\delta_J)$ of $\cal G$ on $J:=\K(\s E\oplus A)$ (cf.\ \ref{compcoact});
\end{itemize}
satisfying some conditions.

\medskip

We have the following unitary equivalences of Hilbert modules: 
\begin{align}
A\tens_{\delta_A}(A\tens S) &\rightarrow q_{\beta_A \alpha}(A\tens S)\,;\, a\tens_{\delta_A}x \mapsto \delta_A(a)x;\label{ehmeq1}\\
(A\tens S)\tens_{\delta_A\tens\, \id_S}(A\tens S\tens S)&\rightarrow q_{\beta_A \alpha,12}(A\tens S\tens S) \, ;\, x\tens_{\delta_A\tens\, \id_S}y \mapsto (\delta_A\tens \id_S)(x)y;\label{ehmeq2}\\
(A\tens S)\tens_{\id_A\tens\,\delta}(A\tens S\tens S) &\rightarrow q_{\beta\alpha,23}(A\tens S\tens S)\,;\,x\tens_{\id_A\tens\,\delta}y\mapsto(\id_A\tens\delta)(x)y.\label{ehmeq3}
\end{align}

In the following, we fix a Hilbert $A$-module $\s E$. We will apply the usual identifications $\M(A\tens S)=\Lin(A\tens S)$, $\K(\s E)\tens S=\K(\s E\tens S)$ and $\M(\K(\s E)\tens S)=\Lin(\s E\tens S)$.

\begin{defin}\label{hilbmodequ}
An action of ${\cal G}$ on the Hilbert $A$-module $\s E$ is a pair $(\beta_{\s{E}},\delta_{\s{E}})$, where $\beta_{\s{E}}:N^{\rm o}\rightarrow\mathcal{L}(\s{E})$ is a non-degenerate *-homomorphism and $\delta_{\s{E}}:\s{E}\rightarrow\widetilde{\mathcal{M}}(\s{E}\tens S)$ is a linear map such that:
\begin{enumerate}
 \item for all $a\in A$ and $\xi,\,\eta\in \s{E}$, we have 
\[
\delta_\s{E}(\xi a)=\delta_\s{E}(\xi)\delta_A(a) \quad \text{and} \quad \langle\delta_\s{E}(\xi),\,\delta_\s{E}(\eta)\rangle=\delta_A(\langle\xi,\,\eta\rangle);
\]
 \item $[\delta_{\s{E}}(\s{E})(A\tens S)]=q_{\beta_{\s E}\alpha}(\s{E}\tens S)$;
 \item for all $\xi\in \s{E}$ and $n\in N$, we have $\delta_{\s{E}}(\beta_{\s{E}}(n^{\rm o})\xi)=(1_{\s E}\tens\beta(n^{\rm o}))\delta_{\s{E}}(\xi)$;
 \item the linear maps $\delta_{\s{E}}\tens \id_S$ and $\id_{\s{E}}\tens\delta$ extend to linear maps from $\mathcal{L}(A\tens S,\s{E}\tens S)$ to $\mathcal{L}(A\tens S\tens S,\s{E}\tens S\tens S)$ and we have 
\[
(\delta_\s{E}\tens \id_S)\delta_\s{E}(\xi)=(\id_\s{E}\tens\delta)\delta_{\s{E}}(\xi)\in\mathcal{L}(A\tens S\tens S,\s{E}\tens S\tens S),\quad \text{for all } \xi\in\s{E}.\qedhere
\]
\end{enumerate}
\end{defin}

\begin{rks}\label{rk2}
\begin{enumerate}
	\item If the second formula of the condition 1 holds, then $\delta_{\s E}$ is isometric (cf.\ \cite{BS2}, \ref{rk4} 1).
	\item If the condition 1 holds, then the condition 2 is equivalent to:
	\[
	[\delta_{\s E}(\s E)(1_A\tens S)]=q_{\beta_{\s E}\alpha}(\s E\tens S).
	\]
	Indeed, if $(u_{\lambda})_{\lambda}$ is an approximate unit of $A$ we have
	\[
	\delta_{\s E}(\xi)=\lim_{\lambda}\,\delta_{\s E}(\xi u_{\lambda})=\lim_{\lambda}\,\delta_{\s E}(\xi)\delta_A(u_{\lambda})=\delta_{\s E}(\xi)q_{\beta_A\alpha},\quad \text{for all } \xi\in\s E.
	\]
	By strong continuity of the action $(\beta_A,\delta_A)$, the condition 1 of Definition \ref{hilbmodequ} and the equality $\s E A=\s E$, we then have $[\delta_{\s E}(\s E)(A\tens S)]=[\delta_{\s E}(\s E)(1_A\tens S)]$ and the equivalence follows.
	\item Note that we have $q_{\beta_{\s E}\alpha}\delta_{\s E}(\xi)=\delta_{\s E}(\xi)=\delta_{\s E}(\xi)q_{\beta_A\alpha}$ for all $\xi\in\s E$.
	\item We will prove (cf.\ \ref{rk9}) that if $\delta_{\s E}$ satisfies the conditions 1 and 2 of \ref{hilbmodequ}, then the extensions of $\delta_{\s E}\tens\id_S$ and $\id_{\s E}\tens\delta$ always exist and satisfy the formulas: 
	\begin{align*}
(\id_{\s E}\tens\delta)(T)(\id_A\tens\delta)(x)&=(\id_{\s E}\tens\delta)(Tx);\\
(\delta_{\s E}\tens\id_S)(T)(\delta_A\tens\id_S)(x)&=(\delta_{\s E}\tens\id_S)(Tx);
\end{align*}
for all $x\in A\tens S$ and $T\in\Lin(A\tens S,\s E\tens S)$.\qedhere
\end{enumerate}
\end{rks}

\begin{nb}\label{not5}
For $\xi\in \s{E}$, let us denote by $T_{\xi}\in\mathcal{L}(A\tens S,\s{E}\tens_{\delta_A}(A\tens S))$ the operator defined by\index[symbol]{tb@$T_{\xi}$}
\[
T_{\xi}(x):=\xi\tens_{\delta_A}x,\quad\text{for all } x\in A\tens S.
\]
We have $T_{\xi}^*(\eta\tens_{\delta_A}y)=\delta_A(\langle\xi,\,\eta\rangle)y$ for all $\eta\in\s E$ and $y\in A\tens S$. In particular, we have $T_{\xi}^*T_{\eta}=\delta_A(\langle\xi,\,\eta\rangle)$ for all $\xi,\eta\in\s E$.
\end{nb}

\begin{defin}\label{isometry}
Let $\s{V}\in\mathcal{L}(\s{E}\tens_{\delta_A}(A\tens S),\s{E}\tens S)$ be an isometry and $\beta_{\s{E}}:N^{\rm o}\rightarrow\mathcal{L}(\s{E})$ a non-degenerate *-homomorphism such that:
\newcounter{saveenum}
\begin{enumerate}
 \item $\s{V}\s{V}^*=q_{\beta_{\s E}\alpha}$;
 \item $\s{V}(\beta_{\s{E}}(n^{\rm o})\tens_{\delta_A}1)=(1_{\s E}\tens\beta(n^{\rm o}))\s{V}$, for all $n\in N$. 
\setcounter{saveenum}{\value{enumi}}
\end{enumerate}
Then, $\s{V}$ is said to be admissible if we further have:
\begin{enumerate}
\setcounter{enumi}{\value{saveenum}}
 \item $\s{V}T_{\xi}\in\widetilde{\mathcal{M}}(\s{E}\tens S)$, for all $\xi\in \s{E}$;
 \item $(\s{V}\tens_{\mathbb{C}}\id_S)(\s{V}\tens_{\delta_A\tens\, \id_S}1)=\s{V}\tens_{\id_A\tens\,\delta}1\in\mathcal{L}(\s{E}\tens_{\delta_A^2}(A\tens S\tens S),\s{E}\tens S\tens S)$.\qedhere
\end{enumerate}
\end{defin}

The fourth statement in the previous definition makes sense since we have used the canonical identifications thereafter. By combining the associativity of the internal tensor product with the unitary equivalences (\ref{ehmeq2}) and (\ref{ehmeq3}), we obtain the following unitary equivalences of Hilbert $A\tens S$-modules:
\begin{align}
(\s{E}\tens_{\delta_A}(A\tens S))\tens_{\delta_A\tens\,\id_S}(A\tens S\tens S) &\rightarrow \s{E}\tens_{\delta_A^2}(A\tens S\tens S)\label{eq4}\\
(\xi\tens_{\delta_A}x)\tens_{\delta_A\tens\,\id_S}y &\mapsto \xi\tens_{\delta_A^2}(\delta_A\tens \id_S)(x)y;\notag\\[.5em]
(\s{E}\tens_{\delta_A}(A\tens S))\tens_{\id_A\tens\,\delta}(A\tens S\tens S) &\rightarrow \s{E}\tens_{\delta_A^2}(A\tens S\tens S)\label{eq5}\\
(\xi\tens_{\delta_A}x)\tens_{\id_A\tens\,\delta}y &\mapsto \xi\tens_{\delta_A^2}(\id_A\tens\delta)(x)y.\notag
\intertext{We also have the following:}
(\s{E}\tens S)\tens_{\delta_A\tens\,\id_S}(A\tens S\tens S) &\rightarrow (\s{E}\tens_{\delta_A}(A\tens S))\tens S \label{eq6}\\
(\xi\tens s)\tens_{\delta_A\tens\,\id_S}(x\tens t)&\mapsto(\xi\tens_{\delta_A}x)\tens st;\notag\\[.5em]
(\s{E}\tens S)\tens_{\id_A\tens\,\delta}(A\tens S\tens S) & \rightarrow q_{\beta\alpha,23}(\s{E}\tens S\tens S)\subset\s{E}\tens S\tens S \label{eq7}\\
\xi\tens_{\id_A\tens\,\delta}y &\mapsto(\id_{\s{E}}\tens\delta)(\xi)y.\notag
\end{align}
In particular, 
$\s{V}\tens_{\delta_A\tens\,\id_S}1\in\mathcal{L}(\s{E}\tens_{\delta_A^2}(A\tens S\tens S),(\s{E}\tens S)\tens_{\delta_A\tens\,\id_S}(A\tens S\tens S))$ (\ref{eq4}) and
$\s{V}\tens_{\mathbb{C}}\id_S\in\mathcal{L}((\s{E}\tens S)\tens_{\delta_A\tens\,\id_S}(A\tens S\tens S),\s{E}\tens S\tens S)$ (\ref{eq6}).
\vspace{10pt}

The next result provides an equivalence of the definitions \ref{hilbmodequ} and \ref{isometry}. 

\begin{prop}\label{prop27}
 \begin{enumerate}[label=\alph*)]
  \item Let $\delta_{\s{E}}:\s{E}\rightarrow\widetilde{\mathcal{M}}(\s{E}\tens S)$ be a linear map and $\beta_{\s{E}}:N^{\rm o}\rightarrow\mathcal{L}(\s{E})$ a non-degenerate *-homomorphism which satisfy the conditions 1, 2, and 3 of Definition \ref{hilbmodequ}. Then, there exists a unique isometry $\s{V}\in\mathcal{L}(\s{E}\tens_{\delta_A}(A\tens S),\s{E}\tens S)$ such that $\delta_\s{E}(\xi)=\s{V}T_{\xi}$ for all $\xi\in\s{E}$. 
  Moreover, the pair $(\beta_{\s{E}},\s V)$ satisfies the conditions 1, 2, and 3 of Definition \ref{isometry}.
  \item Conversely, let $\s{V}\in\mathcal{L}(\s{E}\tens_{\delta_A}(A\tens S),\s{E}\tens S)$ be an isometry and $\beta_{\s{E}}:N^{\rm o}\rightarrow\mathcal{L}(\s{E})$ a non-degenerate *-homomorphism, which satisfy the conditions 1, 2, and 3 of Definition \ref{isometry}. Let us consider the map $\delta_{\s{E}}:\s{E}\rightarrow\Lin(A\tens S,\s{E}\tens S)$ given by $\delta_{\s{E}}(\xi):=\s{V}T_{\xi}$ for all $\xi\in\s{E}$. Then, the pair $(\beta_{\s{E}},\delta_{\s{E}})$ satisfies the conditions 1, 2 and 3 of Definition \ref{hilbmodequ}.
  \item Let us assume that the above statements hold. Then, the pair $(\beta_{\s{E}},\delta_{\s{E}})$ is an action of $\cal G$ on $\s E$ if, and only if, $\s{V}$ is admissible.\qedhere
 \end{enumerate}
\end{prop}

\begin{nb}\label{not7}
Let ${\cal E}$ and ${\cal F}$ be Hilbert C*-modules. Let $q\in\Lin({\cal E})$ be a self-adjoint projection and $T\in\Lin(q{\cal E},{\cal F})$. Let $\widetilde{T}:{\cal E}\rightarrow{\cal F}$ be the map defined by $\widetilde{T}\xi:=Tq\xi$, for all $\xi\in{\cal E}$. Therefore, $\widetilde T\in\Lin({\cal E},{\cal F})$ and $\widetilde{T}^*=qT^*$. By abuse of notation, we will still denote by $T$ the adjointable operator $\widetilde T$.
\end{nb}

\begin{rks}\label{rk9} As a consequence of Proposition \ref{prop27}, we have the statements below.
\begin{itemize}
\item  By applying \ref{not7} and the identifications (\ref{ehmeq3}, \ref{eq7}), we have obtained a linear map $\id_{\s E}\tens\delta:\Lin(A\tens S,\s E\tens S)\rightarrow \Lin(A\tens S\tens S,\s E\tens S\tens S)$ given by
\[
(\id_{\s E}\tens\delta)(T):=T\tens_{\id_A\tens\,\delta} 1,\quad \text{for all } T\in\Lin(A\tens S,\s E\tens S);
\]
\item If $\delta_{\s E}$ satisfies the conditions 1 and 2 of Definition \ref{hilbmodequ}, let $\s V$ be the isometry associated with $\delta_{\s E}$ (cf.\ \ref{prop27} a)). By applying \ref{not7} and the identifications (\ref{ehmeq2}, \ref{eq6}), the linear map $\delta_{\s E}\tens\id_S:\Lin(A\tens S,\s E\tens S)\rightarrow\Lin(A\tens S\tens S,\s E\tens S\tens S)$ is defined by
\[
(\delta_{\s E}\tens\id_S)(T):=(\s V \tens_{\GC} 1_S)(T\tens_{\delta_A\tens\,\id_S} 1),\quad \text{for all }  T\in\Lin(A\tens S,\s E\tens S).
\]
\end{itemize}
Note that the extensions $\id_{\s E}\tens\delta$ and $\delta_{\s E}\tens\id_S$ satisfy the following formulas:
\begin{equation}\label{eq26bis}
(\id_{\s E}\tens\delta)(T)(\id_A\tens\delta)(x)=(\id_{\s E}\tens\delta)(Tx); \,
(\delta_{\s E}\tens\id_S)(T)(\delta_A\tens\id_S)(x)=(\delta_{\s E}\tens\id_S)(Tx);
\end{equation}
for all $x\in A\tens S$ and $T\in\Lin(A\tens S,\s E\tens S)$.
\end{rks}

Let us denote by $J:=\K(\s E\oplus A)$ the linking $\Cstar$-algebra associated with the Hilbert $A$-module $\s E$. In the following, we apply the usual identifications $\M(J)=\Lin(\s E\oplus A)$ and $\M(J\tens S)=\Lin((\s E\tens S)\oplus(A\tens S))$.

\begin{defin}\label{compcoact}
An action $(\beta_J,\delta_J)$ of ${\cal G}$ on $J$ is said to be compatible with the action $(\beta_A,\delta_A)$ if: 
\begin{enumerate}
 \item $\delta_J:J\rightarrow\M(J\tens S)$ is compatible with $\delta_A$, {\it i.e.} 
$ 
 \iota_{A\tens S}\circ\delta_A=\delta_J\circ\iota_A;
$
 \item $\beta_J:N^{\rm o}\rightarrow\M(J)$ is compatible with $\beta_A$, {\it i.e.} 
$ 
 \iota_A(\beta_A(n^{\rm o})a)=\beta_J(n^{\rm o})\iota_A(a)
$,
 for all $n\in N$ and $a\in A$.\qedhere
\end{enumerate}
\end{defin}

\begin{prop}\label{propfib}
Let $(\beta_J,\delta_J)$ be a compatible action of ${\cal G}$ on $J$. There exists a unique non-degenerate *-homomorphism $\beta_{\s E}:N^{\rm o}\rightarrow\Lin(\s{E})$ such that 
\[
\beta_J(n^{\rm o})=\begin{pmatrix}\beta_{\s E}(n^{\rm o})&0\\0&\beta_A(n^{\rm o})\end{pmatrix}\!,\quad\text{for all } n\in N.
\]
Moreover, we have
$
q_{\beta_J\alpha}=\begin{pmatrix}q_{\beta_{\s E}\alpha}&0\\0& q_{\beta_A\alpha}\end{pmatrix}\!.
$
\end{prop}

\begin{prop}\label{prop1}
a) Let us assume that the $\Cstar$-algebra $J$ is endowed with a compatible action $(\beta_J,\delta_J)$ of ${\cal G}$ such that $\delta_J(J)\subset\widetilde{\M}(J\tens S)$. Then, we have the following statements:
\begin{itemize}
 \item there exists a unique linear map $\delta_{\s E}:\s E\rightarrow\widetilde{\M}(\s E\tens S)$ such that 
$\iota_{\s E\tens S}\circ\delta_{\s E}=\delta_J\circ\iota_{\s E}$;
furthermore, the pair $(\beta_{\s E},\delta_{\s{E}})$ is an action of ${\cal G}$ on $\s E$, where $\beta_{\s E}:N^{\rm o}\rightarrow\Lin(\s E)$ is the *-homomorphism defined in \ref{propfib};
 \item there exists a unique faithful *-homomorphism $\delta_{\mathcal{K}(\s{E})}:\mathcal{K}(\s{E})\rightarrow\widetilde{\mathcal{M}}(\mathcal{K}(\s{E})\tens S)$ such that 
 $\iota_{\K(\s{E}\tens S)}\circ\delta_{\mathcal{K}(\s{E})}=\delta_J\circ\iota_{\mathcal{K}(\s{E})}$; moreover, the pair $(\beta_{\s E},\delta_{\mathcal{K}(\s{E})})$ is an action of ${\cal G}$ on $\mathcal{K}(\s{E})$.
\end{itemize}
b) Conversely, let $(\beta_{\s E},\delta_{\s{E}})$ be an action of ${\cal G}$ on the Hilbert $A$-module $\s E$. Then, there exists a faithful *-homomorphism $\delta_J:J\rightarrow\widetilde{\mathcal{M}}(J\tens S)$ such that 
$
\iota_{\s E\tens S}\circ\delta_\s{E}=\delta_J\circ\iota_\s{E}.
$
Moreover, we define a unique action $(\beta_J,\delta_J)$ of ${\cal G}$ on $J$ compatible with $(\beta_A,\delta_A)$ by setting
\[
\beta_J(n^{\rm o})=\begin{pmatrix}\beta_{\s{E}}(n^{\rm o})&0\\0&\beta_A(n^{\rm o})\end{pmatrix}\!, \quad \text{for all } n\in N. \qedhere
\]
\end{prop}

If $\s E_1$ and $\s E_2$ are Hilbert $A$-modules acted upon by $\cal G$, then so is their direct sum $\s E_1\oplus\s E_2$ in a canonical way.

\begin{propdef}\label{propdef10}
For $i=1,2$, let $\s E_i$ be a Hilbert $A$-module acted upon by $\cal G$. Let $\s E:=\s E_1\oplus\s E_2$. For $i=1,2$, let $j_{\s E_i}:\Lin(A\tens S,\s E_i\tens S)\rightarrow\Lin(A\tens S,\s E\tens S)$ be the linear extension of the canonical injection $\s E_i\tens S\rightarrow\s E\tens S$. Let $\beta_{\s E}:N^{\rm o}\rightarrow\Lin(\s E)$ and $\delta_{\s E}:\s E\rightarrow\Lin(A\tens S,\s E\tens S)$ be the maps defined by:
\[
\beta_{\s E}(n^{\rm o}):=\begin{pmatrix}\beta_{\s E_1}(n^{\rm o}) & 0 \\ 0 & \beta_{\s E_2}(n^{\rm o})\end{pmatrix}\!,\quad n\in N;\quad 
\delta_{\s E}(\xi):=\sum_{i=1,2}j_{\s E_i}\circ\delta_{\s E_i}(\xi_i),\quad \xi=(\xi_1,\xi_2)\in\s E.
\]
Then, the pair $(\beta_{\s E},\delta_{\s E})$ is an action of $\cal G$ on $\s E$.
\end{propdef}

\begin{rks}\label{rk16}
Let $(\beta_{\s E},\delta_{\s E})$ be an action of $\cal G$ on the Hilbert $A$-module $\s E$. 
\begin{enumerate}
\item The map $\delta_{\K(\s E)}:\K(\s E)\rightarrow\widetilde{\M}(\K(\s E)\tens S)$ defined in \ref{prop1} a) is the unique *-homomor\-phism satisfying the relation
$
\delta_{\K(\s E)}(\theta_{\xi,\eta})=\delta_{\s E}(\xi)\circ\delta_{\s E}(\eta)^*
$
for all $\xi,\eta\in\s E$. 
\item For all $F\in\Lin(\s E)$ and $\zeta\in\s E$, $\delta_{\s E}(F\zeta)=\delta_{\K(\s E)}(F)\delta_{\s E}(\zeta)$. Indeed, for all $\xi,\eta,\zeta\in\s E$ we have
$
\delta_{\s E}(\theta_{\xi,\eta}\zeta)=\delta_{\s E}(\xi)\delta_A(\langle\eta,\,\zeta\rangle)=\delta_{\s E}(\xi)\langle\delta_{\s E}(\eta),\, \delta_{\s E}(\zeta)\rangle=\delta_{\K(\s E)}(\theta_{\xi,\eta})\delta_{\s E}(\zeta). 
$
Hence, $\delta_{\s E}(k\zeta)=\delta_{\K(\s E)}(k)\delta_{\s E}(\zeta)$ for all $k\in\K(\s E)$ and $\zeta\in\s E$. The claim is then proved by strict continuity of $\delta_{\K(\s E)}$.
\item For all $k\in\K(\s E)$, $\delta_{\K(\s E)}(k)=\s V(k\tens_{\delta_A}1)\s V^*$, where $\s V\in\Lin(\s E\tens_{\delta_A}(A\tens S),\s E\tens S)$ is the isometry associated with the action $(\beta_{\s E},\delta_{\s E})$ (cf.\ \ref{isometry}).\qedhere
\end{enumerate}
\end{rks}

\begin{propdef}\label{propdef7}
Let $(\beta_{\s E},\delta_{\s E})$ be an action of $\cal G$ on the Hilbert $A$-module $\s E$. Let $\s V\in\Lin(\s E\tens_{\delta_A}(A\tens S),\s E\tens S)$ be the isometry associated with $(\beta_{\s E},\delta_{\s E})$ (cf.\ \ref{prop27} a)). Let us endow the C*-algebras $J$ and $\K(\s E)$ with the actions defined in \ref{prop1}. Let $F\in\Lin(\s E)$. The following statements are equivalent:
\begin{enumerate}[label=(\roman*)]
\item $\delta_{\s E}(F\xi)=(F\tens 1_S)\delta_{\s E}(\xi)$, for all $\xi\in\s E$;
\item $F$ is $\delta_{\K(\s E)}$-invariant;
\item $\s V(F\tens_{\delta_A}1)\s V^*=q_{\beta_{\s E}\alpha}(F\tens 1_S)$;
\item $\iota_{\K(\s E)}(F)$ is $\delta_J$-invariant.
\end{enumerate}
In that case, $F$ is said to be ($\delta_{\s E}$-)invariant.
\end{propdef}

\begin{proof}
(ii) $\Rightarrow$ (i): For all $\xi\in\s E$, we have (cf.\ \ref{rk15} 3, \ref{rk2} 3)
\begin{center}
$\delta_{\s E}(F\xi)=\delta_{\K(\s E)}(F)\delta_{\s E}(\xi)=(F\tens 1_S)q_{\beta_{\s E}\alpha}\delta_{\s E}(\xi)=(F\tens 1_S)\delta_{\s E}(\xi)$.
\end{center}
(i) $\Rightarrow$ (ii): For all $\xi,\eta\in\s E$, we have (cf.\ \ref{rk16} 1)
\begin{center}
$\delta_{\K(\s E)}(F\theta_{\xi,\eta})=\delta_{\K(\s E)}(\theta_{F\xi,\eta})=\delta_{\s E}(F\xi)\delta_{\s E}(\eta)^*=(F\tens 1_S)\delta_{\s E}(\xi)\delta_{\s E}(\eta)^*=(F\tens 1_S)\delta_{\K(\s E)}(\theta_{\xi,\eta})$.
\end{center}
Hence, $\delta_{\K(\s E)}(Fk)=(F\tens 1_S)\delta_{\K(\s E)}(k)$ for all $k\in\K(\s E)$. Hence, $\delta_{\K(\s E)}(F)=(F\tens 1_S)q_{\beta_{\s E}\alpha}$.\newline
(ii) $\Leftrightarrow$ (iii): cf.\ \ref{rk16} 3.\newline
(iii) $\Leftrightarrow$ (iv): This is a direct consequence of the relation $\delta_J\circ\iota_{\K(\s E)}=\iota_{\K(\s E\tens S)}\circ\delta_{\K(\s E)}$.
\end{proof}

Let us recall the notion of equivariant unitary equivalence between Hilbert C*-modules over possibly different C*-algebras acted upon by $\cal G$.

\begin{defin}\label{def1}
Let $A$ and $B$ be two $\cal G$-$\Cstar$-algebras and $\phi:A\rightarrow B$ a $\cal G$-equivariant *-isomorphism. Let $\s E$ and $\s F$ be two Hilbert modules over respectively $A$ and $B$ acted upon by $\cal G$. A $\phi$-compatible unitary operator $\Phi:\s E\rightarrow\s F$ is said to be $\cal G$-equivariant if we have
\[
\delta_{\s F}(\Phi\xi)=(\Phi\tens\id_S)\delta_{\s E}(\xi), \quad \text{for all } \xi\in\s E. \qedhere
\]
We recall that the linear map $\Phi\tens\id_S:\Lin(A\tens S,\s E\tens S)\rightarrow\Lin(B\tens S,\s F\tens S)$ (cf.\ \ref{not3}) is the extension of the $\phi\tens\id_S$-compatible unitary operator $\Phi\tens\id_S:\s E\tens S\rightarrow\s F\tens S$ (cf.\ \ref{propdef6}). Note that we have $\Phi\circ\beta_{\s E}(n^{\rm o})=\beta_{\s F}(n^{\rm o})\circ\Phi$ for all $n\in N$ (cf.\ 6.1.13 \cite{C2}).
\end{defin}

\begin{defin}
Two Hilbert C*-modules $\s E$ and $\s F$ acted upon by $\cal G$ are said to be $\cal G$-equivariantly unitarily equivalent if there exists a $\cal G$-equivariant unitary operator from $\s E$ onto $\s F$.
\end{defin}

It is clear that the $\cal G$-equivariant unitary equivalence defines an equivalence relation on the class consisting of the Hilbert $\Cstar$-modules acted upon by $\cal G$. For equivalent definitions of the $\cal G$-equivariant unitary equivalence in the two other pictures, we refer to \S 6.1 \cite{C2}. 

\begin{rk}
Let $B$ be a $\widehat{\cal G}$-C*-algebra. An action of the dual measured quantum groupoid $\widehat{\cal G}$ on a Hilbert $B$-module $\s F$ is defined by three equivalent data:
\begin{itemize}
\item a pair $(\alpha_{\s F},\delta_{\s F})$ consisting of a *-homomorphism $\alpha_{\s F}:N\rightarrow\Lin(\s F)$ and a linear map $\delta_{\s F}:\s F\rightarrow\widetilde{\M}(\s F\tens\widehat{S})$;
\item a pair $(\alpha_{\s E},\s V)$ consisting of a *-homomorphism $\alpha_{\s F}:N\rightarrow\Lin(\s F)$ and an isometry $\s V\in\Lin(\s F\tens_{\delta_B}(B\tens\widehat{S}),\s F\tens\widehat{S})$ ;
\item an action $(\alpha_K,\delta_K)$ of $\widehat{\cal G}$ on $K:=\K(\s F\oplus B)$;
\end{itemize}
satisfying some conditions. The details are left to the reader's attention.
\end{rk}

\paragraph{Equivariant Hilbert modules and bimodules.} In this paragraph, we recall the notion of continuity for actions of the quantum groupoid $\cal G$ on Hilbert C*-modules and the notion of equivariant representation of a $\cal G$-C*-algebra on a Hilbert C*-module acted upon by $\cal G$ (cf.\  \S 7 \cite{C2}). Let $A$ be a $\cal G$-C*-algebra.

\begin{defin}\label{defEqHilbMod}
An action $(\beta_{\s E},\delta_{\s E})$ of $\cal G$ on a Hilbert $A$-module $\s E$ is said to be continuous if we have $[(1_{\s E}\tens S)\delta_{\s E}(\s E)]=(\s E\tens S)q_{\beta_A\alpha}$. A $\cal G$-equivariant Hilbert $A$-module is a Hilbert $A$-module $\s E$ endowed with a continuous action of $\cal G$.
\end{defin}

\begin{prop}\label{prop31}
Let $\s E$ be a $\cal G$-equivariant Hilbert $A$-module. Let $B:=\K(\s E)$. We have the following statements:
\begin{enumerate}
\item the action $(\beta_B,\delta_B)$ of $\cal G$ on $B$ defined in \ref{prop1} is strongly continuous;
\item we define a continuous action of $\cal G$ on the Hilbert $B$-module $\s E^*$ by setting:
\begin{itemize}
\item $\beta_{\s E^*}(n^{\rm o})T:=\beta_A(n^{\rm o})\circ T$, for all $n\in N$ and $T\in\s E^*$,
\item $\delta_{\s E^*}(T)x:=\delta_{\s E}(T^*)^*\circ x$, for all $T\in\s E^*$ and $x\in B\tens S$;
\end{itemize}
where we have applied the usual identifications $B\tens S=\K(\s E\tens S)$ and $\s E=\K(A,\s E)$.\qedhere
\end{enumerate}
\end{prop}

\begin{prop}\label{propcont}
Let $\s E$ be a Hilbert $A$-module endowed with an action $(\beta_{\s E},\delta_{\s E})$ of $\cal G$ on $\s E$. Let $J:=\K(\s E\oplus A)$ be the associated linking C*-algebra. Let $(\beta_J,\delta_J)$ be the action defined in \ref{prop1}. Then, the action $(\beta_{\s E},\delta_{\s E})$ is continuous if, and only if, the action $(\beta_J,\delta_J)$ is strongly continuous.
\end{prop}

\begin{nbs}\label{rkLinkAlg} There is a one-to-one correspondence between $\cal G$-equivariant Hilbert C*-modules (cf.\ \ref{defEqHilbMod}) and linking $\cal G$-C*-algebras (cf.\ \ref{defLinkAlg}).
\begin{itemize}
\item Let $(J,\beta_J,\delta_J,e_1,e_2)$ be a linking $\cal G$-C*-algebra. By restriction of the action $(\beta_J,\delta_J)$, the corner $e_2Je_2$ (resp.\ $e_1Je_2$) turns into a $\cal G$-C*-algebra (resp.\ $\cal G$-equivariant Hilbert C*-module over $e_2Je_2$). We also have the identification of $\cal G$-C*-algebras $\K(e_1Je_2)=e_1Je_1$.
\item Conversely, if $(\s E,\beta_{\s E},\delta_{\s E})$ is a $\cal G$-equivariant Hilbert $A$-module, then the C*-algebra $J:=\K(\s E\oplus A)$ endowed with the continuous action $(\beta_J,\delta_J)$ (cf.\ \ref{prop1}, \ref{propcont}) and the projections $e_1:=\iota_{\s E}(1_{\s E})$ and $e_2:=\iota_{A}(1_A)$ is a linking $\cal G$-C*-algebra.\qedhere
\end{itemize}
\end{nbs}

\begin{thm}\label{corActReg}
Let $\s E$ be a Hilbert $A$-module. If the quantum groupoid $\cal G$ is regular, then any action of ${\cal G}$ on $\s E$ is continuous.
\end{thm}

\begin{nb}\label{rk8}
Let $A$ and $B$ be two C*-algebras and $\s E$ a Hilbert $B$-module. If $\gamma:A\rightarrow\Lin(\s E)$ is a *-homomorphism, then we extend $\gamma\tens\id_S$ to a *-homomorphism $\gamma\tens\id_S:\widetilde\M(A\tens S)\rightarrow\Lin(\s E\tens S)$ up to the identification $\M(\K(\s E)\tens S)=\Lin(\s E\tens S)$ (cf.\ \S \ref{sectionNotations}).
\end{nb}

\begin{defin}\label{defbimod}
Let $A$ and $B$ be two ${\cal G}$-$\Cstar$-algebras, $\s E$ a Hilbert $B$-module, $(\beta_{\s E},\delta_{\s E})$ an action of $\cal G$ on $\s E$ and $\gamma:A\rightarrow\Lin(\s E)$ a *-representation. We say that $\gamma$ is ${\cal G}$-equivariant if we have:
\begin{enumerate}
\item $\delta_{\s E}(\gamma(a)\xi)=(\gamma\tens\id_S)(\delta_A(a))\circ\delta_{\s E}(\xi)$, for all $a\in A$ and $\xi\in\s E$;
\item $\beta_{\s E}(n^{\rm o})\circ\gamma(a)=\gamma(\beta_A(n^{\rm o})a)$, for all $n\in N$ and $a\in A$.
\end{enumerate}
A $\cal G$-equivariant Hilbert $A$-$B$-bimodule is a countably generated $\cal G$-equivariant Hilbert $B$-module endowed with a $\cal G$-equivariant *-representation of $A$.
\end{defin}

\begin{rks}\label{rk11}
\begin{enumerate}
\item Provided that the second condition in the above definition is verified, the first condition is equivalent to: 
\begin{equation}\label{eq21}
\s V(\gamma(a)\tens_{\delta_B}1)\s V^*=(\gamma\tens\id_S)\delta_A(a) \quad \text{for all }  a\in A,
\end{equation}
where $\s V\in\Lin(\s E\tens_{\delta_B}(B\tens S),\s E\tens S)$ denotes the isometry defined in \ref{prop27} a).
\item We recall that the action $\delta_{\K(\s E)}$ of $\cal G$ on $\K(\s E)$ is defined by $\delta_{\K(\s E)}(k):=\s V(k\tens_{\delta_B}1)\s V^*$ for all $k\in\K(\s E)$. Hence, (\ref{eq21}) can be restated as follows: $\delta_{\K(\s E)}(\gamma(a))=(\gamma\tens\id_S)\delta_A(a)$ for all $a\in A$. In particular, if $\gamma$ is non-degenerate, then Definition \ref{defbimod} simply means that the *-homomorphism $\gamma:A\rightarrow\M(\K(\s E))$ is $\cal G$-equivariant (cf.\ \ref{defEquiHom}).
\item If $\gamma:A\rightarrow\Lin(\s E)$ is a non-degenerate *-representation such that
\[
\delta_{\s E}(\gamma(a)\xi)=(\gamma\tens\id_S)(\delta_A(a))\circ\delta_{\s E}(\xi)\quad \text{for all }  a\in A  \text{ and }  \xi\in\s E,
\]
then we have $\beta_{\s E}(n^{\rm o})\circ\gamma(a)=\gamma(\beta_A(n^{\rm o})a)$ for all $n\in N$ and $a\in A$.\qedhere
\end{enumerate}
\end{rks}

For further usage, let us introduce a writing convention concerning bimodule structures.

\begin{conv}\label{conv2}
Let $A$ and $B$ be two $\cal G$ (resp.\ $\widehat{\cal G}$)-C*-algebras. When dealing with a $A$-$B$-bimodule structure and in order to avoid any confusion between similar objects associated with $A$ and $B$, we will sometimes specify those associated with $A$ (resp.\ $B$) by using the lower index ``g'' (resp.\ ``d'')\footnote{The letter ``g'' (resp.\ ``d'') stands for ``{\it gauche}'' (resp.\ ``{\it droite}''), the french word for ``left'' (resp.\ ``right''). This choice was not motivated by chauvinism but only by the fact that the letters ``l'' and ``r'' would have been confusing.}. For example, we will denote by $D_{\rm g}$ and $D_{\rm d}$ (resp.\ $E_{\rm g}$ and $E_{\rm d}$) the bidual $\cal G$ (resp.\ $\widehat{\cal G}$)-C*-algebra of $A$ and $B$ respectively (cf.\ \ref{not16}).
\end{conv}

We recall below the tensor product construction.

\begin{prop}\label{prop18} Let $C$ {\rm(}resp.\ $B${\rm)} be a $\cal G$-$\Cstar$-algebra. Let $\s E_1$ {\rm(}resp.\ $\s E_2${\rm)} be a Hilbert module over $C$ {\rm(}resp.\ $B${\rm)} endowed with an action $(\beta_{\s E_1},\delta_{\s E_1})$ {\rm(}resp.\ $(\beta_{\s E_2},\delta_{\s E_2})${\rm)} of $\cal G$. Let $\gamma_2:C\rightarrow\Lin(\s E_2)$ be a $\cal G$-equivariant *-representation. Consider the Hilbert $B$-module $\s E:=\s E_1\tens_{\gamma_2}\s E_2$. Denote by
\[
\Delta(\xi_1,\xi_2):=(\delta_{\s E_1}(\xi_1)\tens_{\widetilde{\gamma}_2\tens\id_S}1)\circ\delta_{\s E_2}(\xi_2), \; \text{ for } \; \xi_1\in\s E_1 \; \text{ and } \; \xi_2\in\s E_2.
\]
We have $\Delta(\xi_1,\xi_2)\in\widetilde{\M}(\s E\tens S)$ for all $\xi_1\in\s E_1$ and $\xi_2\in\s E_2$. Let $\beta_{\s E}:N^{\rm o}\rightarrow\Lin(\s E)$ be the *-homomorphism defined by
\[
\beta_{\s E}(n^{\rm o}):=\beta_{\s E_1}(n^{\rm o})\tens_{\gamma_2}1, \; \text{ for all } \; n\in N.
\]
There exists a unique map $\delta_{\s E}:\s E\rightarrow\widetilde{\M}(\s E\tens S)$ defined by the formula
$
\delta_{\s E}(\xi_1\tens_{\gamma_2}\xi_2):=\Delta(\xi_1,\xi_2)
$
for $\xi_1\in\s E_1$ and $\xi_2\in\s E_2$ such that the pair $(\beta_{\s E},\delta_{\s E})$ is an action of $\cal G$ on $\s E$. 
\end{prop}

The operator $\delta_{\s E_1}(\xi_1)$ is considered here as an element of $\Lin(\widetilde{C}\tens S,\s E_1\tens S)\supset\widetilde{\M}(\s E_1\tens S)$. In particular, we have $\delta_{\s E_1}(\xi_1)\tens_{\widetilde{\gamma}_2\tens\id_S}1\in\Lin(\s E_2\tens S,\s E\tens S)$ up to the identifications
\begin{align}
(\widetilde{C}\tens S)\tens_{\widetilde{\gamma}_2\tens\id_S}(\s E_2\tens S)&\rightarrow\s E_2\tens S,\; x\tens_{\widetilde{\gamma}_2\tens\id_S}\eta\mapsto(\widetilde{\gamma}_2\tens\id_S)(x)\eta \quad\text{and}\\[0.5em]
(\s E_1\tens S)\tens_{\widetilde{\gamma}_2\tens\id_S}(\s E_2\tens S)&\rightarrow\s E\tens S,\; (\xi_1\tens s)\tens_{\widetilde{\gamma}_2\tens\id_S}(\xi_2\tens t)\mapsto (\xi_1\tens_{\gamma_2}\xi_2)\tens st.\label{eqId}
\end{align}

\begin{rk}\label{rk17}
We recall the definition of the isometry $\s V\in\Lin(\s E\tens_{\delta_B}(B\tens S),\s E\tens S)$ associated with the action $(\beta_{\s E},\delta_{\s E})$ (cf.\ \ref{isometry}). We refer to the proof of 7.9 \cite{C2} for more details. For $i=1,2$, let $\s V_i$ be the isometry associated with the actions $(\beta_{\s E_i},\delta_{\s E_i})$. Let $\widetilde{\s V}_2\in\Lin(\s E\tens_{\delta_B}(B\tens S),\s E_1\tens_{(\gamma_2\tens\id_S)\delta_C}(\s E_2\tens S))$ be the unitary defined for all $\xi_1\in\s E_1$, $\xi_2\in\s E_2$ and $x\in B\tens S$ by
$
\widetilde{\s V}_2((\xi_1\tens_{\gamma_2}\xi_2)\tens_{\delta_B}x):=\xi_1\tens_{(\gamma_2\tens\id_S)\delta_C}\s V_2(\xi_2\tens_{\delta_B}x).
$
Up to the identifications
\begin{align}
(\s E_1\tens_{\delta_C}(C\tens S))\tens_{\gamma_2\tens\id_S}(\s E_2\tens S)&\rightarrow\s E_1\tens_{(\gamma_2\tens\id_S)\delta_C}(\s E_2\tens S)\label{eq29}\\
(\xi_1\tens_{\delta_C}x)\tens_{\gamma_2\tens\id_S}\eta &\mapsto\xi_1\tens_{(\gamma_2\tens\id_S)\delta_C}(\gamma_2\tens\id_S)(x)\eta\quad \text{and}\notag\\[.3em] 
(\s E_1\tens S)\tens_{\gamma_2\tens\id_S}(\s E_2\tens S)&\rightarrow\s E\tens S\label{eq30}\\
(\xi_1\tens s)\tens_{\gamma_2\tens\id_S}(\xi_2\tens t)&\mapsto (\xi_1\tens_{\gamma_2}\xi_2)\tens st\notag
\end{align}
we have $\s V=(\s V_1\tens_{\gamma_2\tens\id_S}1)\widetilde{\s V}_2$.
\end{rk}

The following result is straightforward.

\begin{prop}\label{propleftaction}
We use all the notations and hypotheses of \ref{prop18}. If $A$ is a $\cal G$-C*-algebra and $\gamma_1:A\rightarrow\Lin(\s E_1)$ is a $\cal G$-equivariant *-representation, then $\gamma:A\rightarrow\Lin(\s E)$ the *-representation defined by $\gamma(a):=\gamma_1(a)\tens_{\gamma_2}1$ for all $a\in A$ is $\cal G$-equivariant. If $(\s E_1,\gamma_1)$ is a $\cal G$-equivariant $A$-$C$-bimodule and $\s E_2$ is a $\cal G$-equivariant $C$-$B$-bimodule, then the pair $(\s E,\gamma)$ is a $\cal G$-equivariant $A$-$B$-bimodule.
\end{prop}

The $\cal G$-equivariance of the internal tensor product associativity map is straightforward and left to the reader's discretion.

\begin{lem}
We use all the notations and hypotheses of \ref{prop18}. If $F\in\Lin(\s E_1)$ is invariant, then so is $F\tens_{\gamma_2}1\in\Lin(\s E)$.
\end{lem}

\begin{proof}
For all $\xi_1\in\s E_1$ and $\xi_2\in\s E_2$, $\Delta(F\xi_1,\xi_2)=((F\tens 1_S)\tens_{\widetilde{\gamma}_2\tens\id_S}1)\Delta(\xi_1,\xi_2)$. However, $(F\tens 1_S)\tens_{\widetilde{\gamma}_2\tens\id_S}1$ is identified to $(F\tens_{\gamma_2}1)\tens 1_S$ through the identification (\ref{eqId}). Hence, $\delta_{\s E}(F\xi_1\tens_{\gamma_2}\xi_2)=((F\tens_{\gamma_2}1)\tens 1_S)\delta_{\s E}(\xi_1\tens_{\gamma_2}\xi_2)$ for all $\xi_1\in\s E_1$ and $\s E_2\in\s E_2$. Hence, $F\tens_{\gamma_2}1\in\Lin(\s E)$ is invariant (cf.\ \ref{propdef7}).
\end{proof}

\paragraph{Biduality and equivariant Morita equivalence.} In this paragraph, we recall the notion of equivariant Morita equivalence between $\cal G$-C*-algebras. We also recall the canonical $\cal G$-equivariant Morita equivalence between a $\cal G$-C*-algebra (resp.\ $\widehat{\cal G}$-C*-algebra) $A$ (resp.\ $B$) and the double crossed product $(A\rtimes{\cal G})\rtimes\widehat{\cal G}$ (resp.\ $(B\rtimes\widehat{\cal G})\rtimes{\cal G}$).

\begin{defin}(cf.\ \S 6 \cite{Rie74})
Let $A$ and $B$ be two C*-algebras. An imprimitivity $A$-$B$-bimodule is an $A$-$B$-bimodule $\s E$, which is a full left Hilbert $A$-module for an $A$-valued inner product ${}_A\langle\cdot,\,\cdot\rangle$ and a full right Hilbert $B$-module for a $B$-valued inner product $\langle\cdot,\,\cdot\rangle_B$ such that ${}_A\langle\xi,\,\eta\rangle\zeta=\xi\langle\eta,\,\zeta\rangle_B$ for all $\xi,\,\eta,\,\zeta\in\s E$.
\end{defin}

\begin{rks}
Let $A$ and $B$ be two C*-algebras and $\s E$ an imprimitivity $A$-$B$-bimodule. We recall that the norms defined by the inner products ${}_A\langle\cdot,\,\cdot\rangle$ on ${}_A\s E$ and $\langle\cdot,\,\cdot\rangle_B$ on $\s E_B$ coincide. We also recall that the left (resp.\ right) action of $A$ (resp.\ $B$) on $\s E$ defines a non-degenerate *-homomorphism $\gamma:A\rightarrow\Lin(\s E_B)$ (resp.\ $\rho:B\rightarrow\Lin({}_A\s E)$).
\end{rks}

\begin{defin}
Let $A$ and $B$ be two $\cal G$-C*-algebras. A $\cal G$-equivariant imprimitivity $A$-$B$-bimodule is an imprimitivity $A$-$B$-bimodule $\s E$ endowed with a continuous action of $\cal G$ on $\s E_B$ such that the left action $\gamma:A\rightarrow\Lin(\s E_B)$ is $\cal G$-equivariant. In that case, we say that $A$ and $B$ are $\cal G$-equivariantly Morita equivalent.
\end{defin}

If the quantum groupoid $\cal G$ is regular, then the $\cal G$-equivariant Morita equivalence is a reflexive, symmetric and transitive relation on the class of $\cal G$-C*-algebras.

\medbreak

In the following result, we assume the quantum groupoid $\cal G$ to be regular.
 
\begin{thm}\label{theo8} Let $(A,\beta_A,\delta_A)$ {\rm(}resp.\ $(B,\alpha_B,\delta_B)${\rm)} be a $\cal G$-C*-algebra {\rm(}resp.\ $\widehat{\cal G}$-C*-algebra{\rm)}. In the statements below, we use all the notations of \S \ref{sectionTT}.
\begin{enumerate}
\item There exists a unique continuous action $(\beta_{\er},\delta_{\er})$ {\rm(}resp.\ $(\alpha_{{\cal E}_{B,\rho}},\delta_{{\cal E}_{B,\rho}})${\rm)} of $\cal G$ {\rm(}resp.\ $\widehat{\cal G}${\rm)} on the Hilbert $A$-module $\er$ {\rm(}resp.\ the Hilbert $B$-module ${\cal E}_{B,\rho}${\rm)} given for all $a\in A$ {\rm(}resp.\ $b\in B${\rm)}, $\zeta\in\s H$ and $n\in N$ by the formulas:$\vphantom{\widehat{\cal G}}$
\begin{align*}
\delta_{\er}(q_{\beta_A\widehat{\alpha}}(a\tens\zeta))&={\cal V}_{23}\delta_A(a)_{13}(1_A\tens\zeta\tens 1_S);\quad \beta_{\er}(n^{\rm o}):=\restr{(1_A\tens\beta(n^{\rm o}))}{\er}; \\
\text{{\rm(}resp.\ }\delta_{{\cal E}_{B,\rho}}(q_{\alpha_B\beta}(b\tens\zeta))&=\widetilde{\cal V}_{23}\delta_B(b)_{13}(1_B\tens\zeta\tens 1_{\widehat{S}});\quad 
\alpha_{{\cal E}_{B,\rho}}(n):=\restr{(1_B\tens\widehat{\alpha}(n))}{{\cal E}_{B,\rho}}\!\!\text{{\rm)}};
\end{align*}
\item Endowed with the *-representation $D\rightarrow\Lin({\cal E}_{A,R})\, ;\, u\mapsto\restr{u}{\er}$ {\rm(}resp.\ $E\rightarrow\Lin({\cal E}_{B,\rho})\, ;\, v\mapsto\restr{v}{{\cal E}_{B,\rho}}${\rm)}, the $\cal G$-equivariant Hilbert $A$-module $\er$ {\rm(}resp.\ the $\widehat{\cal G}$-equivariant Hilbert $B$-module ${\cal E}_{B,\rho}${\rm)}$\vphantom{\restr{\psi(v)}{\er}}$ is a $\cal G$-equivariant Hilbert $D$-$A$-bimodule {\rm(}resp.\ $\widehat{\cal G}$-equivariant Hilbert $E$-$B$-bimodule{\rm)}.
\item The $\cal G$-C*-algebras {\rm(}resp.\ $\widehat{\cal G}$-C*-algebras{\rm)} $A$ and $D$ {\rm(}resp.\ $B$ and $E${\rm)} are Morita equivalent via the $\cal G$-equivariant {\rm(}resp.\ $\widehat{\cal G}$-equivariant{\rm)} imprimitivity $A$-$D$-bimodule $\er$ {\rm(}resp.\ $B$-$E$-bimodule ${\cal E}_{B,\rho}${\rm)}.\qedhere
\end{enumerate}
\end{thm} 
 
	\subsection{Crossed product, dual action and biduality}\label{sectionHilbModCrPrd}

\subsubsection{Crossed product}

In this paragraph, we define and investigate the crossed product of a Hilbert module acted upon by a measured quantum groupoid on a finite-dimensional basis. Let us specify some notations.

\medbreak

Let $(A,\beta_A,\delta_A)$ be a $\cal G$-$\Cstar$-algebra. Denote by $B:=A\rtimes\cal G$ the crossed product endowed with the dual action $(\alpha_B,\delta_B)$. Let $\pi:A\rightarrow\M(B)$ and $\widehat{\theta}:\widehat{S}\rightarrow\M(B)$ be the canonical morphisms (cf.\ \ref{crossedproductCstar}). Let $\s E$ be a Hilbert $A$-module and $(\beta_{\s E},\delta_{\s E})$ an action of $\cal G$ on $\s E$.

\begin{defin}
We call crossed product of $\s E$ by the action $(\beta_{\s E},\delta_{\s E})$ the Hilbert $B$-module $\s E\tens_{\pi} B$ denoted by $\s E\rtimes\cal G$.
\end{defin}

\begin{nb}
For $\xi\in\s E$, we denote by $\Pi(\xi)\in\Lin(B,\s E\rtimes{\cal G})$ the adjointable operator defined by $\Pi(\xi)b:=\xi\tens_{\pi}b$ for all $b\in B$. We have $\Pi(\xi)^*(\eta\tens_{\pi}b)=\pi(\langle\xi,\,\eta\rangle)b$ for all $\eta\in\s E$ and $b\in B$. We then have a linear map $\Pi:\s E\rightarrow\Lin(B,\s E\rtimes{\cal G})$ (also denoted by $\Pi_{\s E}$ for emphasis).
\end{nb}

\begin{prop}\label{prop41}
We have:
\begin{enumerate}
\item $\Pi$ is non-degenerate, {\it i.e.\ }$[\Pi(\s E)B]=\s E\rtimes{\cal G}$;
\item $\Pi(\xi a)=\Pi(\xi)\pi(a)$, for all $\xi\in\s E$ and $a\in A$;
\item $\Pi(\xi)^*\Pi(\eta)=\pi(\langle\xi,\,\eta\rangle)$, for all $\xi,\eta\in\s E$;
\item $\Pi(\xi)\widehat{\theta}(x)\in\s E\rtimes{\cal G}$ for all $\xi\in\s E$ and $x\in\widehat{S}$ and 
$
\s E\rtimes{\cal G}=[\Pi(\xi)\widehat{\theta}(x)\,;\, \xi\in\s E,\, x\in\widehat{S}].
$\qedhere
\end{enumerate}
\end{prop}

\begin{proof}
Statements 1, 2 and 3 are direct consequences of the definitions. For all $\xi\in\s E$, $a\in A$ and $x\in\widehat{S}$, we have $\Pi(\xi a)\widehat{\theta}(x)=\Pi(\xi)(\pi(a)\widehat{\theta}(x))\in\s E\rtimes{\cal G}$. Hence, $\Pi(\xi)\widehat{\theta}(x)\in\s E\rtimes{\cal G}$ for all $\xi\in\s E$ and $x\in\widehat{S}$ since $\s E A=\s E$. The formula $\s E\rtimes{\cal G}=[\Pi(\xi)\widehat{\theta}(x)\,;\, \xi\in\s E,\, x\in\widehat{S}]$ follows from the relations $[\Pi(\s E)B]=\s E\rtimes{\cal G}$ and $B=[\pi(a)\widehat{\theta}(x)\,;\,a\in A,\, x\in\widehat{S}]$.
\end{proof}

\begin{prop}\label{defdualaction}
Let $\alpha_{\s E\rtimes\cal G}: N\rightarrow\Lin(\s E\rtimes\cal G)$ and $\delta_{\s E\rtimes\cal G}:\s E\rtimes{\cal G}\rightarrow\Lin(B\tens\widehat{S},(\s E\rtimes{\cal G})\tens\widehat{S})$  be the linear maps defined by:
\[
\alpha_{\s E\rtimes\cal G}(n):=1_{\s E}\tens_{\pi}\alpha_B(n),\quad n\in N;\quad
\delta_{\s E\rtimes\cal G}(\xi\tens_{\pi}b):=(\Pi(\xi)\tens 1_{\widehat{S}})\delta_{B}(b),\quad \xi\in\s E,\; b\in B.
\] 
Then, the pair $(\alpha_{\s E\rtimes\cal G},\delta_{\s E\rtimes\cal G})$ is a continuous action of $\widehat{\cal G}$ on the crossed product $\s E\rtimes{\cal G}$.
\end{prop}

\begin{proof}
Since $\delta_{B}(B)\in\widetilde{\M}(B\tens\widehat{S})$, it is clear that $(\Pi(\xi)\tens 1_{\widehat{S}})\delta_{B}(b)\in\widetilde{\M}((\s E\rtimes{\cal G})\tens\widehat{S})$ for all $\xi\in\s E$ and $b\in B$. We have $\delta_{B}(\pi(a)b)=(\pi(a)\tens 1_{\widehat{S}})\delta_{B}(b)$ for all $a\in A$ and $b\in B$ (cf.\ \ref{defDualAct} 1). Therefore, we have a well-defined linear map
\[
\s E\odot_{B} B\rightarrow\widetilde{\M}(B\tens\widehat{S})\subset\Lin( B\tens\widehat{S},(\s E\rtimes{\cal G})\tens\widehat{S})\,;\, \xi\tens_{\pi} b\mapsto (\Pi(\xi)\tens 1_{\widehat{S}})\delta_{B}(b).
\]
Let $\xi,\eta\in\s E$. For all $b,c\in B$ and $x,y\in\widehat{S}$, we have 
\begin{align*}
\langle(\Pi(\xi)\tens 1_{\widehat{S}})(b\tens x),\, (\Pi(\eta)\tens 1_{\widehat{S}})(c\tens y)\rangle &=\langle(\xi\tens_{\pi} b)\tens x,\, (\eta\tens_{\pi} c)\tens y\rangle \\
&=b^*\pi(\langle\xi,\,\eta\rangle)c\tens x^*y\\
&=(b\tens x)^*(\pi(\langle\xi,\,\eta\rangle)\tens 1_{\widehat{S}})(c\tens y).
\end{align*}
Hence, $(\Pi(\xi)\tens 1_{\widehat{S}})^*(\Pi(\eta)\tens 1_{\widehat{S}})=\pi(\langle\xi,\,\eta\rangle)\tens 1_{\widehat{S}}$.
Therefore, for all $b,c\in B$ we have
\begin{align*}
\langle(\Pi(\xi)\tens 1_{\widehat{S}})\delta_{B}(b),\,(\Pi(\eta)\tens 1_{\widehat{S}})\delta_{B}(c)\rangle &=\delta_B(b)^*(\pi(\langle\xi,\,\eta\rangle)\tens 1_{\widehat{S}})\delta_B(c)\\
&=\delta_{B}(b^*\pi(\langle\xi,\,\eta\rangle)c) \quad \text{(cf.\ \ref{defDualAct} 1)}\\
&=\delta_{B}(\langle\xi\tens_{\pi} b,\,\eta\tens_{\pi}c\rangle).
\end{align*}
Hence, there exists a unique bounded linear map $\delta_{\s E\rtimes{\cal G}}:\s E\rtimes{\cal G}\rightarrow\widetilde{\M}((\s E\rtimes{\cal G})\tens\widehat{S})$ such that $\delta_{\s E\rtimes{\cal G}}(\xi\tens_{B} b)=(\Pi(\xi)\tens 1_{\widehat{S}})\delta_{B}(b)$ for all $\xi\in\s E$ and $b\in B$. Moreover, we have also proved that $\langle\delta_{\s E\rtimes{\cal G}}(\xi),\,\delta_{\s E\rtimes{\cal G}}(\eta)\rangle=\delta_B(\langle\xi,\,\eta\rangle)$ for all $\xi,\eta\in\s E\rtimes{\cal G}$. It is clear that $\delta_{\s E\rtimes{\cal G}}(\xi)\delta_B(b)=\delta_{\s E\rtimes{\cal G}}(\xi b)$ for all $\xi\in\s E\rtimes{\cal G}$ and $b\in B$.\newline
Let us fix $n\in N$. We recall that $\alpha_B(n):=\widehat{\theta}(\widehat{\alpha}(n))$. It follows from the inclusion $\widehat{\alpha}(N)\subset M'$ that $[1_A\tens\rho(\widehat{\alpha}(n)),\,\pi_L(a)]=0$ for all $a\in A$. Hence, $[\alpha_B(n),\,\pi(a)]=0$ for all $a\in A$. Thus, the map $1_{\s E}\tens_{\pi}\alpha_B(n)\in\Lin(\s E\rtimes{\cal G})$ is well defined. It is clear that $\alpha_{\s E\rtimes{\cal G}}:N\rightarrow\Lin(\s E\rtimes{\cal G})$ is a non-degenerate *-homomorphism.

\medbreak

We have $[1_{\s E\rtimes{\cal G}}\tens\widehat{\alpha}(n),\, \Pi(\xi)\tens 1_{\widehat{S}}]=0$ and $(1_B\tens\widehat{\alpha}(n))\delta_B(b)=\delta_B(\alpha_B(n)b)$ for all $n\in N$, $\xi\in\s E$ and $b\in B$. It then follows that 
$\delta_{\s E\rtimes{\cal G}}(\alpha_{\s E\rtimes{\cal G}}(n)\xi)=(1_{\s E\rtimes{\cal G}}\tens\widehat{\alpha}(n))\delta_{\s E\rtimes{\cal G}}(\xi)$ for all $\xi\in\s E\rtimes{\cal G}$ and $n\in N$.\newline
By continuity of the dual action $(\alpha_B,\delta_B)$, we have
\[
[\delta_{\s E\rtimes{\cal G}}(\s E\rtimes{\cal G})(B\tens\widehat{S})]=[(\Pi(\xi)\tens 1_{\widehat{S}})q_{\alpha_B\beta}(b\tens x) \, ; \, \xi\in\s E,\, b\in B,\, x\in\widehat{S}].
\]
Let $n,n'\in N$, $b\in B$, $x\in\widehat{S}$ and $\xi\in\s E$. We have 
\begin{align*}
(\Pi(\xi)\tens 1_{\widehat{S}})(\alpha_B(n')\tens\beta(n^{\rm o}))(b\tens x)&=(\xi\tens_{\pi}\alpha_B(n')b)\tens \beta(n^{\rm o})x\\
&=(\alpha_{\s E\rtimes{\cal G}}(n')\tens\beta(n^{\rm o}))((\xi\tens_{\pi}b)\tens x).
\end{align*}
Hence, $(\Pi(\xi)\tens 1_{\widehat{S}})q_{\alpha_B\beta}(b\tens x)=q_{\alpha_{\s E\rtimes{\cal G}}\beta}((\xi\tens_{\pi}b)\tens x)$. Therefore, we have
\[
[\delta_{\s E\rtimes{\cal G}}(\s E\rtimes{\cal G})(B\tens\widehat{S})]=q_{\alpha_{\s E\rtimes{\cal G}}\beta}((\s E\rtimes{\cal G})\tens\widehat{S}).
\]
The maps $\delta_{\s E\rtimes{\cal G}}\tens\id_{\widehat S}$ and $\id_{\s E\rtimes{\cal G}}\tens\widehat\delta$ extend to linear maps from $\Lin(B\tens\widehat{S},(\s E\rtimes{\cal G})\tens\widehat{S})$ to $\Lin(B\tens\widehat{S}\tens\widehat{S},(\s E\rtimes{\cal G})\tens\widehat{S}\tens\widehat{S})$ (cf.\ \ref{rk9}). For all $\xi\in\s E$ and $b\in B$, we have
\begin{align*}
(\id_{\s E\rtimes{\cal G}}\tens\widehat{\delta})\delta_{\s E\rtimes{\cal G}}(\xi\tens_{\pi}b)
&=(\id_{\s E\rtimes{\cal G}}\tens\widehat{\delta})(\Pi(\xi)\tens 1_{\widehat{S}})\delta_B(b)\\
&=(\Pi(\xi)\tens 1_{\widehat{S}}\tens 1_{\widehat{S}})(\id_B\tens\widehat{\delta})\delta_B(b)\\
&=((\Pi(\xi)\tens 1_{\widehat{S}})\delta_B\tens\id_{\widehat{S}})\delta_B(b)\\
&=(\delta_{\s E\rtimes{\cal G}}\circ\Pi(\xi)\tens\id_{\widehat{S}})\delta_B(x)\\
&=(\delta_{\s E\rtimes{\cal G}}\tens\id_{\widehat{S}})\delta_{\s E\rtimes{\cal G}}(\xi\tens_{\pi}b).
\end{align*}
Hence, $(\delta_{\s E\rtimes{\cal G}}\tens\id_{\widehat{S}})\delta_{\s E\rtimes{\cal G}}=(\id_{\s E\rtimes{\cal G}}\tens\widehat{\delta})\delta_{\s E\rtimes{\cal G}}$. It follows from the above that the  pair $(\alpha_{\s E\rtimes{\cal G}},\delta_{\s E\rtimes{\cal G}})$ is an action of $\widehat{{\cal G}}$ on the Hilbert $B$-module $\s E\rtimes{\cal G}$. 
By continuity of $(\alpha_B,\delta_B)$, we have $[(1_{\s E\rtimes{\cal G}}\tens\widehat{S})\delta_{\s E\rtimes{\cal G}}({\s E\rtimes{\cal G}})]=((\s E\rtimes{\cal G})\tens\widehat{S})q_{\alpha_B\beta}$ and the triple $(\s E\rtimes{\cal G},\alpha_{\s E\rtimes{\cal G}},\delta_{\s E\rtimes{\cal G}})$ is actually a $\widehat{{\cal G}}$-equivariant Hilbert $B$-module.
\end{proof}

\begin{defin}
The action $(\alpha_{\s E\rtimes\cal G},\delta_{\s E\rtimes\cal G})$ of the measured quantum groupoid $\widehat{\cal G}$ on the crossed product $\s E\rtimes\cal G$ is called the dual action of $(\beta_{\s E},\delta_{\s E})$.
\end{defin}

\begin{lem}\label{lem25}
For all $F\in\Lin(\s E)$, the operator $F\tens_{\pi}1_B\in\Lin(\s E\rtimes{\cal G})$ is invariant.
\end{lem}

\begin{proof}
This is an immediate consequence of the definition of the action of $\widehat{\cal G}$ on the crossed product $\s E\rtimes{\cal G}$ and the fact that $\Pi(F\xi)=(F\tens_{\pi}1_B)\Pi(\xi)$ for all $\xi\in\s E$.
\end{proof}

\begin{prop}\label{prop23}
Let $A_1$ and $A_2$ be two $\cal G$-$\Cstar$-algebras, $\s E_1$ and $\s E_2$ Hilbert C*-modules over respectively $A_1$ and $A_2$ acted upon by $\cal G$. Let $\phi:A_1\rightarrow A_2$ be a $\cal G$-equivariant *-isomorphism and $\Phi:\s E_1\rightarrow\s E_2$ a $\cal G$-equivariant unitary equivalence of Hilbert modules over the isomorphism $\phi$. There exists a unique $\widehat{\cal G}$-equivariant unitary equivalence of Hilbert modules $\Phi_*:\s E_1\rtimes{\cal G}\rightarrow\s E_2\rtimes{\cal G}$ over the $\widehat{\cal G}$-equivariant *-isomorphism $\phi_*:A_1\rtimes{\cal G}\rightarrow A_2\rtimes{\cal G}$ such that
\[
\Phi_*(\xi\tens_{\pi_{A_1}} b)=\Phi\xi \tens_{\pi_{A_2}} \phi_*(b), \quad \text{for all } b\in A_1\rtimes{\cal G} \text{ and } \xi\in\s E_1.\qedhere
\]
\end{prop}

\begin{proof}
We have $\Phi(\xi a)=\Phi(\xi)\phi(a)$ and $\phi_*(\pi_{A_1}(a))=\pi_{A_2}(\phi(a))$ for all $\xi\in\s E_1$ and $a\in A_1$. Hence, $\Phi(\xi a)\tens_{\pi_{A_2}}\phi_*(b)=\Phi\xi\tens_{\pi_{A_2}}\phi_*(\pi_{A_1}(a)b)$ for all $a\in A_1$ and $\xi\in\s E_1$. Therefore, we have a linear map
\[
\Phi_*:\s E_1\odot_{\pi_{A_1}}(A_1\rtimes{\cal G})\rightarrow\s E_2\rtimes{\cal G}\,;\, \xi\tens_{\pi_{A_1}}b\mapsto \Phi\xi\tens_{\pi_{A_2}}\phi_*(b).
\]
For all $\xi,\eta\in\s E_1$, we have $\pi_{A_2}(\langle \Phi\xi,\, \Phi\eta\rangle)=\pi_{A_2}(\phi(\langle\xi,\,\eta\rangle))=\phi_*(\pi_{A_1}(\langle\xi,\,\eta\rangle))$. Hence, 
\[
\langle \Phi\xi\tens_{\pi_{A_1}}\phi_*(b),\, \Phi\eta\tens_{\pi_{A_2}}\phi_*(c)\rangle=\phi_*(\langle\xi\tens_{\pi_{A_1}}b,\,\eta\tens_{\pi_{A_1}}c\rangle) 
\]
for all $\xi,\eta\in\s E_1$ and $b,c\in A_1\rtimes{\cal G}$. Therefore, $\Phi_*$ extends to a unitary equivalence $\Phi_*:\s E_1\rtimes{\cal G}\rightarrow\s E_2\rtimes{\cal G}$ over $\phi_*$. Since $\phi_*$ is $\widehat{\cal G}$-equivariant and $\Pi_{\s E_2}(\Phi\xi)\circ\phi_*=\Phi_*\circ \Pi_{\s E_1}(\xi)$ for all $\xi\in\s E_1$, we have $\delta_{\s E_1\rtimes{\cal G}}(\Phi_*(\xi\tens_{\pi_{A_1}}b))=(\Phi_*\tens\id_{\widehat{S}})\delta_{\s E_1\rtimes{\cal G}}(\xi\tens_{\pi_{A_1}}b)$ for all $\xi\in\s E_1$ and $b\in A_1\rtimes{\cal G}$. Hence, $\delta_{\s E_1\rtimes{\cal G}}\circ \Phi_*=(\Phi_*\tens\id_{\widehat{S}})\delta_{\s E_1\rtimes{\cal G}}$. Hence, $\Phi_*$ is equivariant. 
\end{proof}

Let $(J,\beta_J,\delta_J,e_1,e_2)$ be a linking $\cal G$-$\Cstar$-algebra (cf.\ \ref{defLinkAlg}). Let us denote $A:=e_2 J e_2$ and $\s E:=e_1 J e_2$ with their structure of $\cal G$-C*-algebra and $\cal G$-equivariant Hilbert $A$-module (cf.\ \ref{rkLinkAlg}). We consider the crossed products $A\rtimes\cal G$ {\rm(}resp. $ K:=J\rtimes\cal G${\rm)} endowed with the dual action $(\alpha_{A\rtimes\cal G},\delta_{A\rtimes\cal G})$ {\rm(}resp. $(\alpha_ K,\delta_ K)${\rm)} and the canonical morphisms $\pi_A:A\rightarrow\M(A\rtimes{\cal G})$ and $\widehat{\theta}_A:\widehat{S}\rightarrow\M(A\rtimes{\cal G})$ {\rm(}resp. $\pi_J:J\rightarrow\M(K)$ and $\widehat{\theta}_J:\widehat{S}\rightarrow\M(K)${\rm)}.

\medbreak

We know that the quintuple $( K,\alpha_ K,\delta_ K,\pi_J(e_1),\pi_J(e_2))$ is a linking $\widehat{\cal G}$-$\Cstar$-algebra (cf.\ \ref{prop40}). Let $B:=\pi_J(e_2) K \pi_J(e_2)$ and $\s F:=\pi_J(e_1) K \pi_J(e_2)$ respectively endowed with their structure of $\widehat{\cal G}$-C*-algebra and $\widehat{\cal G}$-equivariant Hilbert $B$-module (cf.\ \ref{rkLinkAlg}). We show that we have a $\widehat{\cal G}$-equivariant unitary equivalence between $\s E\rtimes\cal G$ and $\s F$. More precisely, we have the proposition below.

\begin{prop}\label{prop4bis}
With the above notations and hypotheses, there exists a unique $\widehat{\cal G}$-equivariant *-isomorphism $\chi:A\rtimes{\cal G}\rightarrow B$ such that 
$\chi(\pi_A(a)\widehat{\theta}_A(x))=\pi_J(a)\widehat{\theta}_J(x)$,
for all $a\in A$ and $x\in\widehat{S}$. Moreover, the map
$
X:\s E \rtimes {\cal G}\rightarrow \s F\,;\, \xi\tens_{\pi_A}u\mapsto \pi_J(\xi)\chi(u)
$
is a $\chi$-compatible unitary operator.
\end{prop}

\begin{proof}
Since $[e_2 J e_2]=J$, the inclusion map $A\tens\s\K\subset J\tens\K$ extends uniquely to a *-strong continuous *-homomorphism $\tau_A:\Lin(A\tens\s H)\rightarrow\Lin(J\tens\s H)$
such that $\tau_A(1_{A\tens\K})=e_2\tens 1_{\K}$ up to the identifications $\M(A\tens\K)=\Lin(A\tens\s H)$ and $\M(J\tens\K)=\Lin(J\tens\s H)$. Now we recall that we have the identifications
\begin{align*}
\Lin({\cal E}_{A,L})&=\{T\in\Lin(A\tens\s H)\,;\, Tq_{\beta_A\alpha}=T=q_{\beta_A\alpha}T\} \;\; \text{and} \\ 
\Lin({\cal E}_{J,L})&=\{T\in\Lin(J\tens\s H)\,;\, Tq_{\beta_J\alpha}=T=q_{\beta_J\alpha}T\}.
\end{align*}
We also recall that for $n\in N$, we have $\beta_A(n^{\rm o}):=\restr{\beta_J(n^{\rm o})}{A}$ (with the identification $\M(A)=\Lin(A)$) since $[\beta_J(n^{\rm o}),\,e_2]=0$. As a result, $\tau_A$ induces by restriction to $\Lin({\cal E}_{A,L})$ a *-strong *-homomorphism still denoted by $\tau_A:\Lin({\cal E}_{A,L})\rightarrow\Lin({\cal E}_{J,L})$.
We have the following formulas:
\[
\tau_A(\widehat{\theta}_A(x))=\widehat{\theta}_J(x), \quad x\in\widehat{S} ; \quad \tau_A(\pi_A(a))=\pi_J(a), \quad a\in A.
\]
Hence, $\chi:=\restr{\tau_A}{A\rtimes{\cal G}}:A\rtimes{\cal G}\rightarrow  K$ is the unique *-homomorphism such that 
\[
\chi(\pi_A(a)\widehat{\theta}_A(x))=\pi_J(a)\widehat{\theta}_J(x),\quad \text{for all }a\in A \text{ and }x\in\widehat{S}.
\]
Note that since $\tau_A$ is faithful so is $\chi$. It follows from $K=[\pi_J(A)\widehat{\theta}_J(\widehat{S})]$ and the fact that $[\pi_J(e_2),\,\widehat{\theta}_J(x)]=0$ for all $x\in\widehat{S}$ that the image of $\chi$ is $B:=\pi_J(e_2)K\pi_J(e_2)$. Let us prove that $\chi$ is $\widehat{\cal G}$-equivariant. 
We recall that $\delta_{A\rtimes\cal G}(\pi_A(a)\widehat{\theta}_A(x))=(\pi_A(a)\tens 1_{\widehat{S}})(\widehat{\theta}_A\tens\id_{\widehat{S}})\widehat{\delta}(x)$ for all $a\in A$ and $x\in\widehat{S}$ (cf.\ \ref{defDualAct} 1). It then follows from $\chi\circ\pi_A=\pi_J$ and $\chi\circ\widehat{\theta}_A=\widehat{\theta}_J$ that for all $a\in A$ and $x\in\widehat{S}$ we have 
\[
(\chi\tens\id_{\widehat{S}})\delta_{A\rtimes\cal G}(\pi_A(a)\widehat{\theta}_A(x))=\delta_K(\pi_J(a)\widehat{\theta}_J(x))=\delta_K(\chi(\pi_A(a)\widehat{\theta}_A(x))).
\]
Since $\pi_J(xa)=\pi_J(x)\chi(\pi_A(a))$ for $x\in J$ and $a\in A$, we have $\pi_J(xa)\chi(b)=\pi_J(x)\chi(\pi_A(a)b)$ for $x\in J$, $a\in A$ and $b\in A\rtimes{\cal G}$. Therefore, we have a well-defined linear map
\[
X:\s E\odot_{\pi_A}(A\rtimes{\cal G})\rightarrow K \; ; \; \xi\tens_{\pi_A} u \mapsto \pi_J(\xi)\chi(u).
\]
For $\xi,\eta\in\s E$, we have $\pi_J(\xi)^*\pi_J(\eta)=\pi_J(\langle\xi,\,\eta\rangle)=\chi(\pi_A(\langle\xi,\,\eta\rangle))$. Therefore, for all $\xi,\eta\in\s E$ and $u,v\in A\rtimes{\cal G}$ we have $(\pi_J(\xi)\chi(u))^*\pi_J(\eta)\chi(v)=\chi(\langle\xi\tens_{\pi_A} u,\, \eta\tens_{\pi_A} v\rangle)$. As a result, $X$ extends uniquely to a bounded linear map
$
X:\s E\rtimes{\cal G}\rightarrow K
$
such that $X(\langle \zeta_1,\, \zeta_2\rangle)=\chi(\langle\zeta_1,\,\zeta_2\rangle)$ for all $\zeta_1,\zeta_2\in\s E\rtimes{\cal G}$. It is clear that $X(\zeta u)=X(\zeta)\chi(u)$ for all $\zeta\in\s E\rtimes{\cal G}$ and $u\in A\rtimes{\cal G}$.

\medbreak

For all $\xi\in\s E$, we have $(X\tens\id_{\widehat{S}})\circ (T_{\xi}\tens\id_{\widehat{S}})=(\pi_J(\xi)\tens\id_{\widehat{S}})\circ (\chi\tens\id_{\widehat{S}})$. Hence
\[(X\tens\id_{\widehat{S}})\delta_{\s E\rtimes{\cal G}}(\xi\tens_{\pi_A}u)=(\pi_J(\xi)\tens 1_{\widehat{S}})(\chi\tens\id_{\widehat{S}})\delta_{A\rtimes{\cal G}}(u)=\delta_K(\pi_J(\xi)\chi(u))
\]
for all $\xi\in\s E$ and $u\in A\rtimes{\cal G}$, which proves that $X$ is $\widehat{\cal G}$-equivariant. Finally, $X$ induces a $\widehat{\cal G}$-equivariant unitary equivalence of Hilbert modules from $\s E\rtimes{\cal G}$ onto $\s F:=\pi_J(e_1)K\pi_J(e_2)$ over the isomorphism of $\widehat{\cal G}$-$\Cstar$-algebras $\chi:A\rtimes{\cal G}\rightarrow \pi_J(e_2)K\pi_J(e_2)$.
\end{proof}

The continuous action $(\beta_J,\delta_J)$ (resp.\ $(\alpha_K,\delta_K)$) also endows the $\Cstar$-algebra $e_1 J e_1$ (resp.\ $\pi_J(e_1)K\pi_J(e_1)$) identified with $\K(\s E)$ (resp.\ $\K(\s F)$) with a continuous action $(\beta_{\K(\s E)},\delta_{\K(\s E)})$ (resp.\ $(\alpha_{\K(\s F)},\delta_{\K(\s F)})$ of $\cal G$ (resp.\ $\widehat{\cal G}$).

\begin{prop}\label{prop5}
The map
$
\K(\s E)\rtimes{\cal G}\rightarrow\K(\s F) \, ; \, \pi_{\K(\s E)}(k)\widehat{\theta}_{\K(\s E)}(x)\mapsto \pi_J(k)\widehat{\theta}_J(x)
$
is a $\widehat{\cal G}$-equivariant *-isomorphism. 
\end{prop}

\begin{proof}
The proof is the same as that of the above proposition (by exchanging the projections $e_1$ and $e_2$).
\end{proof}

\begin{cor}\label{cor3}
Let $A$ be a $\cal G$-$\Cstar$-algebra and $\s E$ a $\cal G$-equivariant Hilbert $A$-module. We have a canonical $\widehat{\cal G}$-equivariant *-isomorphism
$
\K(\s E)\rtimes{\cal G}\simeq\K(\s E\rtimes{\cal G}).
$
Moreover, if $F\in\Lin(\s E)$ then the operator $F\tens_{\pi_A}1\in\Lin(\s E\rtimes{\cal G})$ is identified with $\pi_{\K(\s E)}(F)$ through the identification $\Lin(\s E\rtimes{\cal G})\simeq\M(\K(\s E)\rtimes{\cal G})$.
\end{cor}

\begin{proof}
It suffices to apply \ref{prop4bis} and \ref{prop5} to $J:=\K(\s E\oplus A)$ equipped with its structure of linking $\cal G$-$\Cstar$-algebra (cf.\ \ref{rkLinkAlg}).
\end{proof}

\begin{cor}\label{cor5}
Let $A$ be a $\cal G$-$\Cstar$-algebra and $\s E$ a $\cal G$-equivariant Hilbert $A$-module. Let $\widehat{\theta}:\widehat{S}\rightarrow\M(\K(\s E)\rtimes{\cal G})$ and $\Pi:\s E\rightarrow\Lin(A\rtimes{\cal G},\s E\rtimes{\cal G})$ be the canonical morphisms. With the identification $\M(\K(\s E)\rtimes{\cal G})=\Lin(\s E\rtimes{\cal G})$ (cf.\ \ref{cor3}), we have $\widehat{\theta}(x)\Pi(\xi)\in\s E\rtimes{\cal G}$ for all $x\in\widehat{S}$ and $\xi\in\s E$. Moreover, we have $\s E\rtimes{\cal G}=[\widehat{\theta}(x)\Pi(\xi)\,;\, \xi\in\s E,\, x\in\widehat{S}]$.
\end{cor}

\begin{proof}
{\setlength{\baselineskip}{1.2\baselineskip}
Let us equip $J:=\K(\s E\oplus A)$ with its structure of linking $\cal G$-$\Cstar$-algebra (cf.\ \ref{rkLinkAlg}). Let $\xi\in\s E$, $b\in A\rtimes{\cal G}$ and $k\in\K(\s E)$, then $\Pi(\xi)b$ (resp.\ $\pi_{\K(\s E)}(k)$) is identified to $\pi_J(\iota_{\s E}(\xi))\chi(b)$ (resp.\ $\pi_J(\iota_{\K(\s E)}(k))$) through the identification of \ref{prop4bis} (resp.\ \ref{prop5}). Hence, $\pi_{\K(\s E)}(k)\Pi(\xi)b$ is identified to $\pi_J(\iota_{\s E}(k\xi))\chi(b)$. Thus, we have $\pi_{\K(\s E)}(k)\Pi(\xi)b=\Pi(k\xi)\chi(b)$. As a result, we have $\pi_{\K(\s E)}(k)\Pi(\xi)=\Pi(k\xi)$ for all $k\in\K(\s E)$ and $\xi\in\s E$.\newline
If $\xi\in\s E$, $x\in\widehat{S}$ and $k\in\K(\s E)$, we have $\widehat{\theta}_{\K(\s E)}(x)\Pi(k\xi)=\widehat{\theta}_{\K(\s E)}(x)\pi_{\K(\s E)}(k)\Pi(\xi)\in\s E\rtimes{\cal G}$ since $\widehat{\theta}_{\K(\s E)}(x)\pi_{\K(\s E)}(k)\in\K(\s E)\rtimes{\cal G}=\K(\s E\rtimes{\cal G})$. Hence, $\widehat{\theta}_{\K(\s E)}(x)\Pi(\xi)\in\s E\rtimes{\cal G}$ for all $x\in\widehat{S}$ and $\xi\in\s E$ since $\K(\s E)\s E=\s E$. Let $\xi\in\s E$ and $x\in\widehat{S}$, then $\Pi(\xi)\widehat{\theta}_A(x)$ is identified to $\pi_J(\iota_{\s E}(\xi))\widehat{\theta}_J(x)$. Moreover, $\pi_J(\iota_{\s E}(\xi))\widehat{\theta}_J(x)$ is the norm limit of finite sums of the form $\sum_{i}\widehat{\theta}_J(x_i)\pi_J(u_i)$ with $x_i\in\widehat{S}$ and $u_i\in J$. However, we have $[\pi_J(e_j),\,\widehat{\theta}_J(y)]=0$ for all $y\in\widehat{S}$ and $j=1,2$. Hence, $\Pi(\xi)\widehat{\theta}_A(x)$ is the norm limit of finite sums of elements of the form $\sum_i\widehat{\theta}_J(x_i)\pi_J(e_1u_ie_2)$. Write $e_1u_ie_2=\iota_{\s E}(\xi_i)$ with $\xi_i\in\s E$, then $\Pi(\xi)\widehat{\theta}_A(x)$ is the norm limit of finite sums of the form $\sum_i\widehat{\theta}_{\K(\s E)}(x_i)\Pi(\xi_i)$.\qedhere
\par}
\end{proof}

\begin{propdef}\label{propdef8}
Let $B$ be a $\cal G$-$\Cstar$-algebra and $\s E$ a $\cal G$-equivariant Hilbert $B$-module. Let $\gamma:A\rightarrow\Lin(\s E)$ be a $\cal G$-equivariant *-representation. By applying \ref{FunctorCrossedProd} and \ref{cor3}, we have a canonical $\widehat{\cal G}$-equivariant *-representation $\gamma_*:A\rtimes{\cal G}\rightarrow\Lin(\s E\rtimes{\cal G})$. Moreover, if $\s E$ is a $\cal G$-equivariant Hilbert $A$-$B$-bimodule, then $\s E\rtimes{\cal G}$ is a $\widehat{\cal G}$-equivariant Hilbert $A\rtimes{\cal G}$-$B\rtimes{\cal G}$-bimodule.
\end{propdef}

\begin{proof}
{\setlength{\baselineskip}{1.1\baselineskip}
We only have to prove that if $\s E$ is countably generated as a Hilbert $B$-module, then $\s E\rtimes{\cal G}$ is countably generated as a Hilbert $B\rtimes{\cal G}$-module. Let $\{\xi_i\,;\,i\in\GN\}$ be a generating set for the Hilbert $B$-module $\s E$. Let $\{x_i\,;\,i\in\GN\}$ be a total subset of $\widehat{S}$ (cf.\ \ref{rk18}). We claim that $\{\widehat{\theta}_{\K(\s E)}(x_i)\Pi(\xi_i)\,;\,i\in\GN\}\subset\s E\rtimes{\cal G}$ is a generating set for the Hilbert $B\rtimes{\cal G}$-module $\s E\rtimes{\cal G}$. Indeed, this follows easily from the relation $\s E\rtimes{\cal G}=[\widehat{\theta}_{\K(\s E)}(x)\Pi(\xi)\widehat{\theta}_B(x')\,;\,\xi\in\s E,\, x,x'\in\widehat{S}]$ (cf.\ \ref{prop41}, \ref{cor5} and the fact that any element of $\widehat{S}$ can be written as a product of two elements of $\widehat{S}$) and \ref{prop41} 2.\qedhere
\par}
\end{proof}

\begin{prop}\label{prop42}
Let $A$, $B$ and $C$ be three $\cal G$-$\Cstar$-algebras. Let $\s E_1$ and $\s E_2$ be Hilbert modules over $C$ and $B$ respectively. Let $(\beta_{\s E_1},\delta_{\s E_1})$ and $(\beta_{\s E_2},\delta_{\s E_2})$ be actions of $\cal G$ on $\s E_1$ and $\s E_2$ respectively. Let $\gamma_2:C\rightarrow\Lin(\s E_2)$ be a $\cal G$-equivariant *-representation. Let $\s E:=\s E_1\tens_{\gamma_2}\s E_2$ be the Hilbert $B$-module acted upon by $\cal G$ defined in \ref{prop18}. Let $\gamma_{2\ast}:C\rtimes{\cal G}\rightarrow\Lin(\s E_2\rtimes{\cal G})$ be the $\widehat{\cal G}$-equivariant *-representation defined in \ref{propdef8}.
\begin{enumerate}
\item There exists a unique $\cal G$-equivariant unitary $\Xi:(\s E_1\rtimes{\cal G})\tens_{\gamma_{2\ast}}(\s E_1\rtimes{\cal G})\rightarrow\s E\rtimes{\cal G}$ such that
\[
\Xi(\widehat{\theta}_{\K(\s E_1)}(x_1)\Pi_{\s E_1}(\xi_1)\tens_{\gamma_{2\ast}}\Pi_{\s E_2}(\xi_2)\widehat{\theta}_B(x_2))=\widehat{\theta}_{\K(\s E)}(x_1)\Pi_{\s E}(\xi_1\tens_{\gamma_2}\xi_2)\widehat{\theta}_B(x_2)
\]
for all $x_1,\,x_2\in\widehat{S}$, $\xi_1\in\s E_1$ and $\xi_2\in\s E_2$.
\item Let $\gamma_1:A\rightarrow\Lin(\s E_1)$ be a $\cal G$-equivariant *-representation. Denote by $\gamma:A\rightarrow\Lin(\s E)$ the $\cal G$-equivariant *-representation defined by $\gamma(a):=\gamma_1(a)\tens_{\gamma_2}1$ for all $a\in A$ (cf.\ \ref{propleftaction}) and $\gamma_{1\ast}:A\rtimes{\cal G}\rightarrow\Lin(\s E_2\rtimes{\cal G})$ the $\widehat{\cal G}$-equivariant *-representation defined in \ref{propdef8}. Let $\kappa:A\rightarrow\Lin((\s E_1\rtimes{\cal G})\tens_{\gamma_{2\ast}}(\s E_2\rtimes{\cal G}))$ be the $\cal G$-equivariant *-representation defined by $\kappa(a):=\gamma_{1\ast}(a)\tens_{\gamma_{2\ast}}1$ for all $a\in A$ (cf.\ \ref{propleftaction}). We then have $\Xi\circ\kappa(a)=\gamma_*(a)\circ\Xi$ for all $a\in A$.\qedhere
\end{enumerate}
\end{prop}

For the proof, we will need to use a concrete interpretation of the crossed product.

\begin{nbs}\label{not2} Let $A$ be a $\cal G$-C*-algebra. Consider the $\widehat{\cal G}$-C*-algebra $B:=A\rtimes\cal G$. Let $\s E$ be a Hilbert $A$-module acted upon by $\cal G$.
\begin{enumerate}
\item Let $\kappa:S\rightarrow\B(\s H)$ be a non-degenerate *-homomorphism. We have the following unitary equivalences of Hilbert $A$-modules:
\begin{align*}
(A\tens S)\tens_{\id_A\tens\kappa}(A\tens \s H) & \rightarrow A\tens \s H, \;
(a\tens s)\tens_{\id_A\tens\kappa}(b\tens\eta) \mapsto ab \tens \kappa(s)\eta;\\
(\s E\tens S)\tens_{\id_A\tens\kappa}(A\tens \s H) & \rightarrow \s E\tens \s H, \;\,
(\xi\tens s)\tens_{\id_A\tens\kappa}(a\tens\eta) \mapsto \xi a \tens \kappa(s)\eta.
\end{align*}
By using the above identifications, the map $\id_{\s E}\tens\kappa$ extends to a linear map
\[
\id_{\s E}\tens\kappa:\Lin(A\tens S,{\s E}\tens S)\rightarrow \Lin(A\tens \s H,{\s E}\tens \s H)\,;\, T\mapsto (\id_{\s E}\tens\kappa)(T):=T \tens_{\id_A\tens\kappa} 1.
\]
The extension is uniquely determined by the formula
\[
(\id_{\s E}\tens\kappa)(Tx)=(\id_{\s E}\tens\kappa)(T)(\id_A\tens\kappa)(x),\text{ for } T\in\Lin(A\tens S,{\s E}\tens S) \text{ and } x\in A\tens S.
\]
For all $T,\,S\in\Lin(A\tens S,{\s E}\tens S)$ we have
$
(\id_{\s E}\tens\kappa)(T)^*(\id_{\s E}\tens\kappa)(S)=(\id_A\tens\kappa)(T^*S)
$
with the identification $\Lin(A\tens S)=\M(A\tens S)$.
\item By using the above notation, we consider the linear map 
$
\Pi_L:\s E\rightarrow\Lin(A\tens\s H,\s E\tens\s H)
$ 
\index[symbol]{pjb@$\Pi_L$, $\Pi$}defined by $\Pi_L(\xi):=(\id_{\s E}\tens L)\delta_{\s E}(\xi)$ for all $\xi\in\s E$. Note that we have:
\begin{itemize}
\item $\Pi_L(\xi a)=\Pi_L(\xi)\pi_L(a)$, for all $a\in A$ and $\xi\in\s E$;
\item $\Pi_L(\xi)^*\Pi_L(\eta)=\pi_L(\langle\xi,\,\eta\rangle)$, for all $\xi,\eta\in \s E$;
\item $q_{\beta_{\s E}\alpha}\Pi_L(\xi)=\Pi_L(\xi)=\Pi_L(\xi)q_{\beta_A\alpha}$, for all $\xi\in\s E$.
\end{itemize}
\item There exists a unique isometric linear map $\Psi_{L,\rho}:\s E\rtimes{\cal G}\rightarrow\Lin(A\tens \s H,\s E\tens\s H)$ such that $\Psi_{L,\rho}(\Pi(\xi)\widehat{\theta}(x))=\Pi_L(\xi)(1_A\tens\rho(x))$ for all $\xi\in\s E$ and $x\in\widehat{S}$. Indeed, let us denote by $\psi_{L,\rho}:B\rightarrow\Lin(A\tens\s H)$ the unique faithful strictly/*-strongly continuous *-representation such that $\psi_{L,\rho}(\pi(a)\widehat{\theta}(x))=\pi_L(a)(1_A\tens\rho(x))$ for all $a\in A$ and $x\in\widehat{S}$ (cf.\ 4.2.3 \cite{C}). Since $\psi_{L,\rho}(\pi(a))=\pi_L(a)$ for all $a\in A$, there exists a unique isometric linear map $\Psi_{L,\rho}:\s E\rtimes{\cal G}\rightarrow\Lin(A\tens\s H,\s E\tens\s H)$ such that $\Psi_{L,\rho}(\xi\tens_{\pi}b)=\Pi_L(\xi)\psi_{L,\rho}(b)$ for all $\xi\in\s E$ and $b\in B$. It is clear that $\Psi_{L,\rho}(\Pi(\xi)\widehat{\theta}(x))=\Pi_L(\xi)(1_A\tens\rho(x))$ for all $\xi\in\s E$ and $x\in\widehat{S}$.\qedhere
\end{enumerate}
\end{nbs}

\begin{proof}[Proof of Proposition \ref{prop42}]
{\setlength{\baselineskip}{1.1\baselineskip}
1. According to \ref{not2} 3, we identify $\s E_1\rtimes{\cal G}$ (resp.\ $\s E_2\rtimes{\cal G}$) with a subspace of $\Lin(C\tens\s H,\s E_1\tens\s H)$ (resp.\ $\Lin(B\tens\s H,\s E_2\tens\s H)$). Since $\delta_C(C)(1_C\tens S)\subset C\tens S$, we even have $\s E_1\rtimes{\cal G}\subset\Lin(\widetilde{C}\tens\s H,\s E_1\tens\s H)$. Let $\widetilde{\gamma}_2\tens\id_{\K}:\Lin(\widetilde{C}\tens\s H)\rightarrow\Lin(\s E_2\tens\s H)$ (cf.\ \S \ref{sectionNotations}). With the identification $(\widetilde{C}\tens\s H)\tens_{\widetilde{\gamma}_2}\s E_2=\s E_2\tens\s H$, we have $(\widetilde{\gamma}_2\tens\id_{\K})(T)=T\tens_{\widetilde{\gamma}_2}1$ for all $T\in\Lin(\widetilde{C}\tens\s H)$. By using the identification $(\s E_1\tens\s H)\tens_{\widetilde{\gamma}_2}\s E_2=\s E\tens\s H$, we then obtain an isometric adjointable operator
\begin{align*}
\Xi:\Lin(\widetilde{C}\tens\s H,\s E_1\tens\s H)\tens_{\widetilde{\gamma}_2\tens\id_{\K}}\Lin(B\tens\s H,\s E_2\tens\s H)&\rightarrow\Lin(B\tens\s H,\s E\tens\s H)\\ 
T_1\tens_{\widetilde{\gamma}_2\tens\id_{\K}}T_2 &\mapsto(T_1\tens_{\widetilde{\gamma}_2}1)T_2.
\end{align*}
We prove that $\Xi(\widehat{\theta}_{\K(\s E_1)}(x_1)\Pi_{\s E_1}(\xi_1)\tens_{\gamma_{2\ast}}\Pi_{\s E_2}(\xi_2)\widehat{\theta}_B(x_2))=\widehat{\theta}_{\K(\s E)}(x_1)\Pi_{\s E}(\xi_1\tens_{\gamma_2}\xi_2)\widehat{\theta}_B(x_2)$ for all $x_1,\,x_2\in\widehat{S}$, $\xi_1\in\s E_1$ and $\xi_2\in\s E_2$ by a straightforward computation. Hence, $\Xi$ induces by restriction a unitary $\Xi\in\Lin((\s E_1\rtimes{\cal G})\tens_{\gamma_{2\ast}}(\s E_1\rtimes{\cal G}),\s E\rtimes{\cal G})$. The equivariance of $\Xi$ will follow from the definitions and the formulas
\begin{align*}
\delta_{\s E\rtimes{\cal G}}(\widehat{\theta}_{\K(\s E)}(x_1)\Pi_{\s E}(\xi)\widehat{\theta}_B(x_2))&=(\widehat{\theta}_{\K(\s E)}\tens\id_{\widehat{S}})\widehat{\delta}(x_1)(\Pi_{\s E}(\xi)\tens 1_{\widehat{S}})(\widehat{\theta}_B\tens\id_{\widehat{S}})\widehat{\delta}(x_2),\\
\delta_{\s E_1\rtimes{\cal G}}(\widehat{\theta}_{\K(\s E_1)}(x_1)\Pi_{\s E_1}(\xi_1))&=(\widehat{\theta}_{\K(\s E_1)}\tens\id_{\widehat{S}})\widehat{\delta}(x_1)(\Pi_{\s E_1}(\xi_1)\tens 1_{\widehat{S}}) \quad \text{and}\\
\delta_{\s E_2\rtimes{\cal G}}(\Pi_{\s E_2}(\xi_2)\widehat{\theta}_{B}(x_2))&=(\Pi_{\s E_2}(\xi_2)\tens 1_{\widehat{S}})(\widehat{\theta}_{B}\tens\id_{\widehat{S}})\widehat{\delta}(x_2)
\end{align*}
for all $\xi\in\s E$, $\xi_1\in\s E_1$, $\xi_2\in\s E_2$ and $x_1,\,x_2\in\widehat{S}$.\newline
2. This will follow from the formulas $\gamma_{1\ast}(a)\widehat{\theta}_{\K(\s E_1)}(x_1)\Pi_{\s E_1}(\xi_1)=\widehat{\theta}_{\K(\s E_1)}(x_1)\Pi_{\s E_1}(\gamma_1(a)\xi_1)$ and $\kappa(a)\widehat{\theta}_{\K(\s E)}(x_1)\Pi_{\s E}(\xi)\widehat{\theta}_B(x_2)=\widehat{\theta}_{\K(\s E)}(x_1)\Pi_{\s E}(\gamma(a)\xi)\widehat{\theta}_B(x_2)$ for all $\xi_1\in\s E_1$, $\xi\in\s E$ and $x_1,x_2\in\widehat{S}$ (cf.\ \ref{FunctorCrossedProd} 1).\qedhere
\par}
\end{proof}

In a similar way, we define the crossed product of a Hilbert C*-module by an action of the dual measured quantum groupoid $\widehat{\cal G}$. The details are left to the reader's attention.

\medbreak

Let $(B,\alpha_B,\delta_B)$ be a $\widehat{\cal G}$-$\Cstar$-algebra. Let us denote $C:=B\rtimes\widehat{\cal G}$ the crossed product endowed with the dual action $(\beta_C,\delta_C)$. Let $\widehat{\pi}:B\rightarrow\M(C)$ and $\theta:S\rightarrow\M(C)$ be the canonical morphisms (cf.\ \ref{crossedproductCstar}). Let $\s F$ be a Hilbert $B$-module and $(\alpha_{\s F},\delta_{\s F})$ an action of $\widehat{\cal G}$ on $\s F$.

\begin{defin}
We call crossed product of $\s F$ by the action $(\alpha_{\s F},\delta_{\s F})$ the Hilbert $C$-module $\s F\tens_{\widehat\pi} C$ denoted by $\s F\rtimes\widehat{\cal G}$.
\end{defin}

\begin{nb}
For $\xi\in\s F$, we denote by $\widehat{\Pi}(\xi)\in\Lin(B,\s F\rtimes\widehat{\cal G})$ the adjointable operator defined by $\widehat{\Pi}(\xi)c:=\xi\tens_{\widehat\pi}c$ for all $c\in C$. We have $\widehat{\Pi}(\xi)^*(\eta\tens_{\widehat\pi}c)=\widehat{\pi}(\langle\xi,\,\eta\rangle)c$ for all $\eta\in\s F$ and $c\in C$. We then have a linear map $\widehat{\Pi}:\s F\rightarrow\Lin(B,\s F\rtimes\widehat{\cal G})$.
\end{nb}

\begin{propdef}
Let us denote by $\delta_{\s F\rtimes\widehat{\cal G}}:\s F\rtimes\widehat{\cal G}\rightarrow\Lin(C\tens S,(\s F\rtimes\widehat{\cal G})\tens S)$ and $\beta_{\s F\rtimes\widehat{\cal G}}:N^{\rm o}\rightarrow\Lin(\s F\rtimes\widehat{\cal G})$ the linear maps defined by:
\[
\beta_{\s F\rtimes\widehat{\cal G}}(n^{\rm o}):=1_{\s F}\tens_{\widehat{\pi}}\beta_C(n^{\rm o}),\;\; n\in N;\quad
\delta_{\s F\rtimes\widehat{\cal G}}(\xi\tens_{\widehat{\pi}}c):=(\widehat{\Pi}(\xi)\tens 1_S)\delta_C(c),\;\; \xi\in\s F,\; c\in C.
\]
Then, the pair $(\beta_{\s F\rtimes\widehat{\cal G}},\delta_{\s F\rtimes\widehat{\cal G}})$ is a continuous action of $\cal G$ on the crossed product $\s F\rtimes\widehat{\cal G}$ called the dual action of $(\alpha_{\s F},\delta_{\s F})$.
\end{propdef}

Let $(K,\alpha_K,\delta_K,f_1,f_2)$ be a linking $\widehat{\cal G}$-$\Cstar$-algebra. Let us consider the $\widehat{\cal G}$-C*-algebra $B:=f_2Kf_2$ and the $\widehat{\cal G}$-equivariant Hilbert $B$-module $\s F:=f_1 Kf_2$. 
We consider the $\cal G$-C*-algebras $B\rtimes\widehat{\cal G}$ (resp. $L:=K\rtimes\widehat{\cal G}$) endowed with the canonical morphisms $\widehat{\pi}_B:B\rightarrow\M(B\rtimes\widehat{\cal G})$ and $\theta_B:S\rightarrow\M(B\rtimes\widehat{\cal G})$ (resp. $\widehat{\pi}_L:L\rightarrow\M(K)$ and $\theta_L:S\rightarrow\M(K)$).
We know that $(L,\beta_L,\delta_L,\widehat{\pi}_L(f_1),\widehat{\pi}_L(f_2))$ is a linking $\cal G$-$\Cstar$-algebra. Let us consider the $\cal G$-C*-algebra $C:=\widehat{\pi}_K(f_2) L \widehat{\pi}_K(f_2)$ and $\cal G$-equivariant Hilbert $C$-module $\s G:=\widehat{\pi}_K(f_1) L \widehat{\pi}_K(f_2)$.

\begin{prop}\label{prop6}
With the above notations and hypotheses, we have the following statements.
\begin{enumerate}
\item There exists a unique $\cal G$-equivariant *-isomorphism $\psi:B\rtimes\widehat{\cal G}\rightarrow C$ such that for all $b\in B$ and $y\in S$ we have 
$\psi(\widehat{\pi}_B(b)\theta_B(y))=\widehat{\pi}_K(b)\theta_K(y)$. Moreover, we have $\cal G$-equivariant unitary equivalence 
$
\Psi:\s F \rtimes \widehat{\cal G}\rightarrow \s G\,;\, \eta\tens_{\widehat{\pi}_B}u\mapsto \widehat{\pi}_K(\eta)\psi(u)
$
over $\psi:B\rtimes\widehat{\cal G}\rightarrow C$.
\item The map
$
\K(\s E)\rtimes{\cal G}\rightarrow\K(\s F) \, ; \, \widehat{\pi}_{\K(\s F)}(k)\theta_{\K(\s F)}(y) \mapsto \widehat{\pi}_K(k)\widehat{\theta}_K(y)
$
is a $\cal G$-equivariant *-isomorphism.\qedhere
\end{enumerate}  
\end{prop}

\begin{cor}\label{cor2}
{\setlength{\baselineskip}{1.1\baselineskip}
Let $B$ be a $\widehat{\cal G}$-$\Cstar$-algebra and $\s F$ a $\widehat{\cal G}$-equivariant Hilbert $B$-module. We have a canonical $\cal G$-equivariant *-isomorphism
$
\K(\s F)\rtimes\widehat{\cal G}\simeq\K(\s F\rtimes\widehat{\cal G}).
$
Moreover, if $F\in\Lin(\s F)$ then the operator $F\tens_{\widehat{\pi}_B}1\in\Lin(\s F\rtimes\widehat{\cal G})$ is identified with $\widehat{\pi}_{\K(\s F)}(F)$ through the identification $\Lin(\s F\rtimes\widehat{\cal G})\simeq\M(\K(\s F)\rtimes\widehat{\cal G})$.\qedhere
\par}
\end{cor}

\subsubsection{Takesaki-Takai duality}\label{sectionTT2}

In the following paragraph, we investigate the double crossed product. Let $A$ be a $\cal G$-C*-algebra and $\s E$ a $\cal G$-equivariant Hilbert $A$-module. Let $D$ be the bidual $\cal G$-C*-algebra (cf.\ \ref{not16}).

\begin{propdef}
Let $\Pi_R:\s E\rightarrow\Lin(A\tens\s H,\s E\tens\s H)$ be the linear map defined by $\Pi_R(\xi):=(\id_{\s E}\tens R)\delta_{\s E}(\xi)$ for all $\xi\in\s E$ (cf.\ \ref{not2}). Let
\[
\s D:=[\Pi_R(\xi)(1_A\tens\lambda(x)L(y))\,;\,\xi\in\s E,\,x\in\widehat{S},\,y\in S]\subset\Lin(A\tens\s H,\s E\tens\s H).
\]
For the natural right action of $D$ by composition of operators and the $D$-valued inner product given by $\langle\zeta_1,\,\zeta_2\rangle:=\zeta_1^*\circ\zeta_2$ for $\zeta_1,\zeta_2\in \s D$, we turn $\s D$ into a Hilbert $D$-module. Let 
\[
{\cal E}_{\s E,R}:=q_{\beta_{\s E}\widehat{\alpha}}(\s E\tens\s H)\subset\s E\tens\s H.
\] 
Then, ${\cal E}_{\s E,R}$ is a Hilbert sub-$A$-module of $\s E\tens\s H$ and $[\s D\er]={\cal E}_{\s E,R}$.
\end{propdef}

\begin{proof}
By combining the facts that $[\lambda(\widehat{S})L(S)]$ and $D$ are C*-algebras with the formula $\Pi_R(\xi)\pi_R(a)=\Pi_R(\xi a)$ for $\xi\in\s E$ and $a\in A$, we obtain the inclusion $\s D D\subset \s D$. Moreover, we have $\Pi_R(\xi)^*\Pi_R(\eta)=\pi_R(\langle\xi,\,\eta\rangle)$ for all $\xi,\eta\in\s E$. It then follows that $\s D^*\s D\subset D$. Since $q_{\beta_{\s E}\widehat{\alpha}}\Pi_R(\xi)=\Pi_R(\xi)=\Pi_R(\xi)q_{\beta_A\widehat{\alpha}}$ for all $\xi\in\s E$ (cf.\ \ref{rk2} 3), we have the inclusion $[\Pi_R(\s E)\er]\subset{\cal E}_{\s E,R}$. The converse inclusion follows from $[\delta_{\s E}(\s E)(1_A\tens S)]=q_{\beta_{\s E}\alpha}(\s E\tens S)$. Hence, $[\Pi_R(\s E)\er]={\cal E}_{\s E,R}$. The relation $[\s D\er]={\cal E}_{\s E,R}$ follows from ${\cal E}_{A,R}=[D{\cal E}_{A,R}]$ and $[\Pi_R(\s E)D]=\s D$.
\end{proof}

We will endow $\s D$ with a structure of $\cal G$-equivariant Hilbert $D$-module. Actually, the action $(\beta_{\s D},\delta_{\s D})$ defined in Theorem \ref{theo1} will be obtained by transport of structure through the identification $(\s E\rtimes{\cal G})\rtimes\widehat{\cal G}\simeq\s D$ of Theorem \ref{theo2}.

\medbreak

Let us denote by $\sigma:S\tens\K\rightarrow\K\tens S$ the flip *-homomorphism. As in \ref{rk9}, we define the linear extensions $\id_{\s E}\tens\sigma:\Lin(A\tens S\tens\K,\s E\tens S\tens\K)\rightarrow\Lin(A\tens \K\tens S,\s E\tens \K\tens S)$ and $\delta_{\s E}\tens\id_{\K}:\Lin(A\tens\K,\s E\tens\K)\rightarrow\Lin(A\tens S\tens \K,\s E\tens S\tens\K)$. We state below the main results of this paragraph.

\begin{thm}\label{theo1}
Let $\delta_{\s D}:\s D\rightarrow\Lin(D\tens S,\s D\tens S)$ and $\beta_{\s D}:N^{\rm o}\rightarrow\Lin(\s D)$ be the maps defined by the formulas:
\begin{align*}
\beta_{\s D}(n^{\rm o})&:=q_{\beta_{\s E}\widehat{\alpha}}(1_{\s E}\tens \beta(n^{\rm o})),\quad n\in N;\\
\delta_{\s D}(\zeta)&:=(1_{\s E}\tens {\cal V})(\id_{\s E}\tens\sigma)(\delta_{\s E}\tens\id_{\K})(\zeta)(1_A\tens {\cal V})^*,\quad \zeta\in\s D.
\end{align*}
Then, the pair $(\beta_{\s D},\delta_{\s D})$ is a continuous action of $\cal G$ on the Hilbert $D$-module $\s D$. Moreover, for all $\xi\in\s E$, $x\in\widehat{S}$ and $y\in S$ we have
\[
\delta_{\s D}( \Pi_R(\xi)(1_A\tens\lambda(x)L(y)))=
(\Pi_R(\xi)\tens 1_S)(1_A\tens\lambda(x)\tens 1_S)(1_A\tens(L\tens\id_S)\delta(y)).
\]
If $\cal G$ is regular, then we have $\s D=q_{\beta_{\s E}\widehat{\alpha}}(\s E\tens\K)q_{\beta_A\widehat{\alpha}}$.
\end{thm}

If $\cal G$ is regular, we have $\s D\subset\s E\tens\K$ up to the identification $\s E\subset\Lin(A,\s E)$.

\begin{thm}\label{theo2}
There exists a unique unitary equivalence 
$
\Phi:(\s E\rtimes{\cal G})\rtimes\widehat{\cal G}\rightarrow \s D
$
over the canonical *-isomorphism $\phi:(A\rtimes{\cal G})\rtimes\widehat{\cal G}\rightarrow D$ (cf.\ref{IsoTT}) such that
\[
\Phi(\widehat{\Pi}(\Pi(\xi)\widehat{\theta}(x))\theta(y))=\Pi_R(\xi)(1_A\tens\lambda(x)L(y)),\quad \text{for all }\xi\in\s E,\, x\in\widehat{S} \text{ and } y\in S.
\]
Moreover, $\Phi$ is $\cal G$-equivariant.
\end{thm}

\begin{proof}[Proofs of Theorems \ref{theo1} and \ref{theo2}]
At the risk of considering $\K(\s E\oplus A)$, we can assume that $\s E$ is a top right-hand corner in some linking $\cal G$-$\Cstar$-algebra $(J,\beta_J,\delta_J,e_1,e_2)$. By combining \ref{prop6} 1 and \ref{prop5}, we can identify $(\s E\rtimes{\cal G})\rtimes\widehat{\cal G}$ with the top right-hand corner of the linking $\cal G$-C*-algebra $((J\rtimes{\cal G})\rtimes\widehat{\cal G},\beta_{(J\rtimes{\cal G})\rtimes\widehat{\cal G}},\delta_{(J\rtimes{\cal G})\rtimes\widehat{\cal G}},\widehat{\pi}(\pi(e_1)),\widehat{\pi}(\pi(e_2)))$. Let us denote by $D_J\subset\Lin(J\tens\s H)$ the bidual $\cal G$-C*-algebra of $J$. By applying the biduality theorem (cf.\ \ref{BidualityTheo} and \ref{rk19}), we can identify $(\s E\rtimes{\cal G})\rtimes\widehat{\cal G}$ with the top right-hand corner of the linking $\cal G$-C*-algebra $(D_J,\beta_{D_J},\delta_{D_J},\pi_R(e_1),\pi_R(e_2))$. Since $\pi_R(e_j)=q_{\beta_J\widehat{\alpha}}(e_j\tens 1_{\K})$ for $j=1,2$ and $\widehat{\alpha}(N)=M'\cap\widehat{M}'$, we have $[\pi_R(e_j),\, 1_J\tens\lambda(x)L(y)]=0$ for all $x\in\widehat{S}$ and $y\in S$. Hence, we can identify $\pi_R(e_1)D_J\pi_R(e_2)$ with $\s D$. The action $(\beta_{\s D},\delta_{\s D})$ is then obtained by transport of structure.
\end{proof}

\begin{cor}\label{cor1}
Assume that $\cal G$ is regular. The formulas:
\begin{align*}
\delta_{{\cal E}_{\s E,R}}(q_{\beta_{\s E}\widehat{\alpha}}(\xi\tens\eta))&:={\cal V}_{23}\delta_{\s E}(\xi)_{13}(1_A\tens\eta\tens 1_S),\quad \xi\in\s E,\, \eta\in\s H; \\
\beta_{{\cal E}_{\s E,R}}(n^{\rm o})&:=\restr{(1_{\s E}\tens \beta(n^{\rm o}))}{{\cal E}_{\s E,R}}, \quad n\in N;
\end{align*}
define an action of $\cal G$ on the Hilbert $A$-module ${\cal E}_{\s E,R}$. Moreover, we have a canonical identification of $\cal G$-equivariant Hilbert $A$-modules
\[
((\s E\rtimes{\cal G})\rtimes\widehat{\cal G})\tens_{(A\tens{\cal G})\rtimes\widehat{\cal G}}{\cal E}_{A,R}={\cal E}_{\s E,R}
\]
up to the identification of $\cal G$-C*-algebras $(A\rtimes{\cal G})\rtimes\widehat{\cal G}=D$.\qedhere
\end{cor}

\begin{proof}
It is clear that the formula
$
\s D\tens_{D}{\cal E}_{A,R}\rightarrow{\cal E}_{\s E,R}\; ; \; \zeta\tens_{D}\xi\mapsto\zeta(\xi)
$
defines a unitary equivalence of Hilbert $A$-modules. Let $(\beta_{{\cal E}_{\s E,R}},\delta_{{\cal E}_{\s E,R}})$ be the action of $\cal G$ on ${\cal E}_{\s E,R}$ obtained from the action of $\cal G$ on $\s D\tens_{D}{\cal E}_{A,R}$ by transport of structure. By a straightforward computation, we prove that $(\beta_{{\cal E}_{\s E,R}},\delta_{{\cal E}_{\s E,R}})$ satisfies the formulas stated above.
By Theorem \ref{theo2}, we have a unitary equivalence of $\cal G$-equivariant Hilbert $D$-modules
\[
((\s E\rtimes{\cal G})\rtimes\widehat{\cal G})\tens_{(A\rtimes{\cal G})\rtimes\widehat{\cal G}}D=\s D.
\]
By taking the internal tensor product by ${\cal E}_{A,R}$ and using the associativity, we obtain
\begin{align*}
(((\s E\rtimes{\cal G})\rtimes\widehat{\cal G})\tens_{(A\rtimes{\cal G})\rtimes\widehat{\cal G}}D)\tens_{D}{\cal E}_{A,R}
&=((\s E\rtimes{\cal G})\rtimes\widehat{\cal G})\tens_{(A\rtimes{\cal G})\rtimes\widehat{\cal G}}(D\tens_{D}{\cal E}_{A,R})\\
&=((\s E\rtimes{\cal G})\rtimes\widehat{\cal G})\tens_{(A\tens{\cal G})\rtimes\widehat{\cal G}}{\cal E}_{A,R}.
\end{align*}
Hence, $((\s E\rtimes{\cal G})\rtimes\widehat{\cal G})\tens_{(A\tens{\cal G})\rtimes\widehat{\cal G}}{\cal E}_{A,R}={\cal E}_{\s E,R}$.
\end{proof}

\begin{lem}\label{lem14}
Assume that $\cal G$ is regular. For all $F\in\Lin(\s E)$, $(\zeta\mapsto\pi_R(F)\circ\zeta)\in\Lin(\s D)$ and $\restr{\pi_R(F)}{{\cal E}_{\s E,R}}\in\Lin({\cal E}_{\s E,R})$ are invariant.
\end{lem} 

In order to keep the notations simple, we will sometimes denote by $\pi_R(F)$ the operators defined above since no ambiguity will arise.

\begin{proof}
Let $T\in\Lin(\s D)$ be the operator defined by $T(\zeta):=\pi_R(F)\circ\zeta$ for all $\zeta\in\s D$. The operator $(F\tens_{\pi}1)\tens_{\widehat{\pi}}1\in\Lin((\s E\rtimes{\cal G})\rtimes\widehat{\cal G})$ is invariant (cf.\ \ref{lem25}). However, the operator $(F\tens_{\pi}1)\tens_{\widehat{\pi}}1$ is identified to $T\in\Lin(\s D)$ through the identification $(\s E\rtimes{\cal G})\rtimes\widehat{\cal G}=\s D$ (cf.\ \ref{theo2}). Hence, $T$ is invariant. The operator $T\tens_{D}1\in\Lin(\s D\tens_{D}{\cal E}_{A,R})$ is identified to $\restr{\pi_R(F)}{{\cal E}_{\s E,R}}\in\Lin({\cal E}_{\s E,R})$ through the identification $\s D\tens_{D}{\cal E}_{A,R}\rightarrow{\cal E}_{\s E,R}\; ; \; \zeta\tens_{D}\xi\mapsto\zeta(\xi)$. Hence, the operator $\restr{\pi_R(F)}{{\cal E}_{\s E,R}}\in\Lin({\cal E}_{\s E,R})$ is invariant.
\end{proof}

\begin{propdef}\label{propdef5}
Assume that $\cal G$ is regular. Let $B$ be a $\cal G$-$\Cstar$-algebra, $\s E$ a $\cal G$-equivariant Hilbert $B$-module and $\gamma:A\rightarrow\Lin(\s E)$ a $\cal G$-equivariant *-representation of $A$ on $\s E$. Then, for all $d\in D$ we have $(\gamma\tens\id_{\K})(d)q_{\beta_{\s E}\widehat\alpha}=(\gamma\tens\id_{\K})(d)=q_{\beta_{\s E}\widehat\alpha}(\gamma\tens\id_{\K})(d)$ in $\Lin(\s E\tens\s H)$.
Moreover, the map $d\in D\mapsto\restr{(\gamma\tens\id_{\K})(d)}{{\cal E}_{\s E,R}}$ is a $\cal G$-equivariant *-representation of $D$ on ${\cal E}_{\s E,R}$. If $\s E$ is a $\cal G$-equivariant Hilbert $A$-$B$-bimodule, then ${\cal E}_{\s E,R}$ is a $\cal G$-equivariant Hilbert $D$-$B$-bimodule.
\end{propdef}

\begin{proof}
We have $\gamma(\beta_A(n^{\rm o})a)=\beta_{\s E}(n^{\rm o})\gamma(a)$ for all $a\in A$ and $n\in N$. It then follows that $(\gamma\tens\id_{\K})(q_{\beta_A\widehat{\alpha}}x q_{\beta_A\widehat{\alpha}})=q_{\beta_{\s E}\widehat{\alpha}}(\gamma\tens\id_{\K})(x)q_{\beta_{\s E}\widehat{\alpha}}$ for all $x\in A\tens\K$. In particular, we have $(\gamma\tens\id_{\K})(d)q_{\beta_{\s E}\widehat\alpha}=(\gamma\tens\id_{\K})(d)=q_{\beta_{\s E}\widehat\alpha}(\gamma\tens\id_{\K})(d)$ for all $d\in D$. As a result, the *-representation $\gamma\tens\id_{\K}:A\tens\K\rightarrow\Lin(\s E\tens\s H)$ induces by restriction a *-representation $\gamma_0:D\rightarrow\Lin({\cal E}_{\s E,R})$. Let us prove that $\gamma_0$ is $\cal G$-equivariant. Let us fix $\xi\in\s E$, $\eta\in\s H$, $a\in A$ and $k\in\K$. We have $\delta_{\s E}(\gamma(a)\xi)=(\gamma\tens\id_S)(\delta_A(a))\circ\delta_{\s E}(\xi)$. By a straightforward computation, we have
$
\delta_{{\cal E}_{\s E,R}}(q_{\beta_{\s E}\widehat{\alpha}}(\gamma(a)\xi\tens k\eta))=(\gamma\tens\id_{\K}\tens\id_S)({\cal V}_{23}\delta_0(a\tens k))\delta_{\s E}(\xi)_{13}(1_A\tens\eta\tens 1_S).
$
For all $x\in A\tens\K$, we have $\delta_0(x q_{\beta_A\widehat{\alpha}})=\delta_0(x)q_{\widehat{\alpha}\beta,23}$. Hence, ${\cal V}_{23}\delta_0(x q_{\beta_A\widehat{\alpha}})=\delta_{A\tens\K}(x){\cal V}_{23}$ for all $x\in A\tens\K$. In particular, we have ${\cal V}_{23}\delta_0(d)=\delta_{D}(d){\cal V}_{23}$ for all $d\in D$. Hence, $\delta_{{\cal E}_{\s E,R}}(\gamma_0(d)\zeta)=(\gamma_0\tens\id_{D})(\delta_{D}(d))\circ\delta_{{\cal E}_{\s E,R}}(\zeta)$ for all $\zeta\in{\cal E}_{\s E,R}$ and $d\in D$. It is easily seen that $\gamma_0(\beta_{D}(n^{\rm o})d)=\beta_{{\cal E}_{\s E,R}}(n^{\rm o})\gamma_0(d)$ for all $n\in N$ and $d\in D$. If $\s E$ is countably generated as a Hilbert $B$-module, then so is $\s E\tens\s H$ since $\s H$ is separable. Hence, the submodule ${\cal E}_{\s E,R}$ of $\s E\tens\s H$ is countably generated.
\end{proof}

	\subsection{Case of a colinking measured quantum groupoid}\label{ActColinkHilb}

Let us fix a colinking measured quantum groupoid ${\cal G}:={\cal G}_{\QG_1,\QG_2}$ associated with two monoidally equivalent locally compact quantum groups $\QG_1$ and $\QG_2$. 

		\subsubsection{Hilbert C*-modules acted upon by a colinking measured quantum groupoid}\label{sectionHilbModColink}
		
In the following, we recall the description of Hilbert C*-modules acted upon by $\cal G$ in terms of Hilbert C*-modules acted upon by $\QG_1$ and $\QG_2$ (cf.\ \S 6.2 \cite{C2}). Let $A$ be a $\cal G$-$\Cstar$-algebra. We follow all the notations of \S \ref{sectionColinking} (resp. \ref{not12} and \ref{actprop}) concerning the objects associated with $\cal G$ (resp.\ $A$). Let us fix a Hilbert $A$-module $\s E$ endowed with an action $(\beta_{\s E},\delta_{\s E})$ of $\cal G$.

\begin{nbs}\label{not8}
We introduce some useful notations to describe the action $(\beta_{\s E},\delta_{\s E})$.
\begin{itemize}
\item Let $q_{\s E,j}:=\beta_{\s E}(\varepsilon_j)$ for $j=1,2$. Note that $q_{\s E,1}$ and $q_{\s E,2}$ are orthogonal self-adjoint projections of $\Lin(\s E)$ and $q_{\s E,1}+q_{\s E,2}=1_{\s E}$.\index[symbol]{qc@$q_{\s E,j}$}
\item Let $J:=\K(\s E\oplus A)$ be the linking $\Cstar$-algebra associated with $\s E$ endowed with the action $(\beta_J,\delta_J)$ of $\cal G$ (cf.\ \ref{prop1} b)). Since $\beta_J(\GC^2)\subset{\cal Z}({\cal M}(J))$ (cf.\ 3.2.3 \cite{BC}), we have $\beta_{\s E}(n)\xi=\xi\beta_A(n)$ in $\Lin(A,\s E)$ for all $n\in\GC^2$ and $\xi\in\s E$, {\it i.e.\ }$(\beta_{\s E}(n)\xi)a=\xi(\beta_A(n)a)$ for all $n\in\GC^2$, $\xi\in\s E$ and $a\in A$. Hence, 
\begin{equation}\label{eq1.11}
(q_{\s E,j}\xi)a=\xi(q_{A,j}a), \quad \text{for all } \xi\in\s E,\, a\in A,\, j=1,2.
\end{equation}
In particular, we have
\begin{equation*}
\langle q_{\s E,j}\xi,\, q_{\s E,j}\eta\rangle=q_{A,j}\langle\xi,\, \eta\rangle,\quad \text{for all } \xi,\eta\in\s E.
\end{equation*}
For $j=1,2$, we then define the following Hilbert $A_j$-module $\s E_j:=q_{\s E,j}\s E$. Note that $\s E=\s E_1\oplus\s E_2$. 
\item For $j,k=1,2$, let
$
\Pi_j^k:\s E_k\tens S_{kj}\rightarrow\s E\tens S.
$
\index[symbol]{pk@$\Pi_j^k$}be the inclusion map. It is clear that $\Pi_j^k$ is a $\pi_j^k$-compatible operator (cf.\ \ref{def2}). We can consider its canonical linear extension 
$
\Pi_j^k:\Lin(A_k\tens S_{kj},\s E_k\tens S_{kj})\rightarrow \Lin(A\tens S,\s E\tens S),
$
up to the canonical injective maps $\s E_k\tens S_{kj}\rightarrow\Lin(A_k\tens S_{kj},\s E_k\tens S_{kj})$ and $\s E\tens S\rightarrow\Lin(A\tens S,\s E\tens S)$, $\vphantom{\Pi_j^k}$defined by
$
\Pi_j^k(T)(x):=\Pi_j^k\circ T((q_{A,k}\tens p_{kj})x)
$ for all $T\in\Lin(A_k\tens S_{kj},\s E_k\tens S_{kj})$ and $x\in A\tens S$.\qedhere
\end{itemize}
\end{nbs}

\begin{lem}
With the above notations and hypotheses, we have a canonical unitary equivalence of Hilbert $A\tens S$-modules
$
\s E\tens_{\delta_A}(A\tens S)=\bigoplus_{j,k=1,2}\s E_j\tens_{\delta_{A_j}^k}\!\!(A_k\tens S_{kj}).
$
\end{lem}

\begin{propdef}\label{propdef3}
Let $\s V\in\Lin(\s E\tens_{\delta_A}(A\tens S),\s E\tens S)$ be the isometry associated with the action $(\beta_{\s E},\delta_{\s E})$ (cf.\ \ref{prop27} a). For all $j,k=1,2$, there exists a unique unitary\index[symbol]{ve@$\s V_j^k$}
\[
\s V_j^k\in\Lin(\s E_j\tens_{\delta_{A_j}^k}\!\!(A_k\tens S_{kj}),\s E_k\tens S_{kj})\\[-.5em]
\]
such that 
\[
\s V(\xi\tens_{\delta_A}x)=\sum_{j,k=1,2}\s V_j^k(q_{\s E,j}\xi\tens_{\delta_{A_j}^k}\!\!(q_{A,k}\tens p_{kj})x),\quad \text{for all } \xi\in\s E \text{ and } x\in A\tens S.\qedhere
\]
\end{propdef}

For $j,k,l=1,2$ we have the following set of unitary equivalences of Hilbert modules:
\begingroup
\allowdisplaybreaks
\begin{align}
A_j\tens_{\delta_{A_j}^k}\!\!(A_k\tens S_{kj}) & \rightarrow A_k\tens S_{kj} \\[-.5em]
a\tens_{\delta_{A_j}^k}x & \mapsto \delta_{A_j}^k(a)x; \notag\\[.5em]
(A_k\tens S_{kj})\tens_{\delta_{A_k}^l\tens\,\id_{S_{kj}}}\!\!(A_l\tens S_{lk}\tens S_{kj}) &  \rightarrow A_l\tens S_{lk}\tens S_{kj} \label{eq24}\\[-.5em]
x\tens_{\delta_{A_k}^l\tens\,\id_{S_{kj}}}y & \mapsto (\delta_{A_k}^l\tens\id_{S_{kj}})(x)y; \notag\\[.5em]
(A_l\tens S_{lj})\tens_{\id_{A_l}\tens\,\delta_{lj}^k}\!\!(A_l\tens S_{lk}\tens S_{kj}) & \rightarrow A_l\tens S_{lk}\tens S_{kj} \label{eq9} \\[-.5em]
x\tens_{\id_{A_l}\tens\,\delta_{lj}^k}y & \mapsto (\id_{A_l}\tens\delta_{lj}^k)(x)y; \notag\\[.5em]
(\s E_j\tens_{\delta_{A_j}^k}\!\!(A_k\tens S_{kj}))\tens_{\delta_{A_k}^l\tens\,\id_{S_{kj}}}\!\!(A_l\tens S_{lk}\tens S_{kj}) & \rightarrow \s E_j\tens_{(\delta_{A_k}^l\tens\,\id_{S_{kj}})\delta_{A_j}^k}\!\!(A_l\tens S_{lk}\tens S_{kj}) \label{eq1.14}\\[-.5em]
(\xi\tens_{\delta_{A_j}^k}x)\tens_{\delta_{A_k}^l\tens\,\id_{S_{kj}}}y & \mapsto \xi\tens_{(\delta_{A_k}^l\tens\,\id_{S_{kj}})\delta_{A_j}^k}(\delta_{A_k}^l\tens\id_{S_{kj}})(x)y;\notag \\[.5em]
(\s E_j\tens_{\delta_{A_j}^l}\!\!(A_l\tens S_{lj}))\tens_{\id_{A_l}\tens\,\delta_{lj}^k}\!\!(A_l\tens S_{lk}\tens S_{kj}) & \rightarrow \s E_j\tens_{(\id_{A_l}\tens\,\delta_{lj}^k)\delta_{A_j}^l}\!\!(A_l\tens S_{lk}\tens S_{kj}) \label{eq1.15} \\[-.5em]
(\xi\tens_{\delta_{A_j}^l}x)\tens_{\id_{A_l}\tens\,\delta_{lj}^k}y & \mapsto \xi\tens_{(\id_{A_l}\tens\,\delta_{lj}^k)\delta_{A_j}^l}(\id_{A_l}\tens\delta_{lj}^k)(x)y; \notag \\[.5em]
(\s E_k\tens S_{kj})\tens_{\delta_{A_k}^l\tens\,\id_{S_{kj}}}\!\!(A_l\tens S_{lk}\tens S_{kj}) &\rightarrow (\s E_k\tens_{\delta_{A_k}^l}\!\!(A_l\tens S_{lk}))\tens S_{kj} \label{eq1.16}\\[-.5em]
(\xi\tens s)\tens_{\delta_{A_k}^l\tens\,\id_{S_{kj}}}(x\tens t) & \mapsto (\xi\tens_{\delta_{A_k}^l} x)\tens st; \notag \\[.5em]
(\s E_l\tens S_{lj})\tens_{\id_{A_l}\tens\,\delta_{lj}^k}\!\!(A_l\tens S_{lk}\tens S_{kj}) & \rightarrow \s E_l\tens S_{lk}\tens S_{kj} \label{eq10} \\[-.5em]
\xi\tens_{\id_{A_l}\tens\,\delta_{lj}^k}y & \mapsto (\id_{\s E_l}\tens\delta_{lj}^k)(\xi)y\notag.
\end{align}	
\endgroup

\begin{prop}\label{prop9}
For all $j,k,l=1,2$, we have
\[
(\s V_k^l\tens_{\GC}\id_{S_{kj}})(\s V_j^k\tens_{\delta_{A_k}^l\tens\id_{S_{kj}}}1)=\s V_j^l\tens_{\id_{A_l}\tens\delta_{lj}^k}1.\qedhere
\]
\end{prop}

For $j,k,l=1,2$, $\s V_k^l\tens_{\GC}\id_{S_{kj}}\in\Lin((\s E_k\tens S_{kj})\tens_{\delta_{A_k}^l\tens\,\id_{S_{kj}}}(A_l\tens S_{lk}\tens S_{kj}),\s E_l\tens S_{lk}\tens S_{kj})$ (\ref{eq1.16}), $\s V_j^k\tens_{\delta_{A_k}^l\tens\,\id_{S_{kj}}}1\in\Lin(\s E_j\tens_{(\delta_{A_k}^l\tens\,\id_{S_{kj}})\delta_{A_j}^k}(A_l\tens S_{lk}\tens S_{kj}),(\s E_k\tens S_{kj})\tens_{\delta_{A_k}^l\tens\,\id_{S_{kj}}}(A_l\tens S_{lk}\tens S_{kj}))$ (\ref{eq1.14}) and $\s V_j^l\tens_{\id_{A_l}\tens\,\delta_{lj}^k}1\in\Lin(\s E_j\tens_{(\id_{A_l}\tens\,\delta_{lj}^k)\delta_{A_j}^l}(A_l\tens S_{lk}\tens S_{kj}),\s E_l\tens S_{lk}\tens S_{kj})$ (\ref{eq10}).
Moreover, the composition $(\s V_k^l\tens_{\GC}\id_{S_{kj}})(\s V_j^k\tens_{\delta_{A_k}^l\tens\id_{S_{kj}}}1)$ does make sense since $(\delta_{A_k}^l\tens\id_{S_{kj}})\delta_{A_j}^k=(\id_{A_l}\tens\delta_{lj}^k)\delta_{A_j}^l$.

\begin{propdef}\label{propdef4}
For $j,k=1,2$, let 
$
\delta_{\s E_j}^k:\s E_j\rightarrow\Lin(A_k\tens S_{kj},\s E_k\tens S_{kj})
$
be the linear map defined by\index[symbol]{df@$\delta_{\s E_j}^k$}
\[
\delta_{\s E_j}^k(\xi)x:=\s V_j^k(\xi\tens_{\delta_{A_j}^k}x), \quad \text{for all } \xi\in\s E_j \text{ and } x\in A_k\tens S_{kj}.
\]
For all $j,\, k,\, l=1,2$, we have:
\begin{enumerate}[label=(\roman*)]
\item $\displaystyle{\delta_{\s E}(\xi)=\sum_{k,j=1,2} \Pi_j^k\circ\delta_{\s E_j}^k(q_{\s E,j}\xi)}$, for all $\xi\in\s E$;
\item $\delta_{\s E_j}^k(\s E_j)\subset\widetilde{\M}(\s E_k\tens S_{kj})$;$\vphantom{\displaystyle{\sum_{k,j=1,2}}}$ 
\item $\delta_{\s E_j}^k(\xi a)=\delta_{\s E_j}^k(\xi)\delta_{A_j}^k(a)$ and $\langle\delta_{\s E_j}^k(\xi),\, \delta_{\s E_j}^k(\eta)\rangle = \delta_{A_j}^k(\langle\xi,\, \eta\rangle)$, for all $\xi,\,\eta\in\s E_j$ and $a\in A_j$;$\vphantom{\displaystyle{\sum_{k,j=1,2}}}$
\item $[\delta_{\s E_j}^k(\s E_j)(1_{A_k}\tens S_{kj})]=\s E_k\tens S_{kj}$; in particular, we have
\[
\s E_k=[(\id_{\s E_k}\tens\omega)\delta_{\s E_j}^k(\xi) \; ; \; \omega\in\B(\s H_{kj})_*,\, \xi\in\s E_j] \quad (\text{cf.}\ \ref{not4}).
\]
\item $\delta_{\s E_k}^l\tens\id_{S_{kj}}$ {\rm(}resp.\ $\id_{\s E_l}\tens\delta_{lj}^k${\rm)} extends to a linear map from $\mathcal{L}(A_k\tens S_{kj},\s E_k\tens S_{kj})$ {\rm(}resp. $\Lin(A_l\tens S_{lj},\s E_l\tens S_{lj})${\rm)} to $\mathcal{L}(A_l\tens S_{lk}\tens S_{kj},\s E_l\tens S_{lk}\tens S_{kj})$ and for all $\xi\in\s E_j$ we have 
\[
(\delta_{\s E_k}^l\tens\id_{S_{kj}})\delta_{\s E_j}^k(\xi)=(\id_{\s E_l}\tens\delta_{lj}^k)\delta_{\s E_j}^l(\xi) \in \Lin(A_l\tens S_{lk}\tens S_{kj},\s E_l\tens S_{lk}\tens S_{kj});
\]
\item if $\s E$ is a $\cal G$-equivariant Hilbert $A$-module, then we have $[(1_{\s E_k}\tens S_{kj})\delta_{\s E_j}^k(\s E_j)]=\s E_k\tens S_{kj}$.
\end{enumerate}
If $\s E$ is a $\cal G$-equivariant Hilbert $A$-module, then $(\s E_j,\delta_{\s E_j}^j)$ is a $\QG_j$-equivariant Hilbert $A_j$-module and $\s V_j^j$ is the associated unitary.
\end{propdef}

According to this concrete description of $\cal G$-equivariant Hilbert C*-modules, we have a description of the $\cal G$-equivariant unitary equivalences in terms of $\QG_j$-equivariant unitary equivalences for $j=1,2$.

\begin{lem}\label{lem5}
Let $A$ and $B$ be $\cal G$-$\Cstar$-algebras. Let $\s E$ and $\s F$ be Hilbert $\Cstar$-modules over $A$ and $B$ respectively acted upon by $\cal G$.
\begin{enumerate}
\item Let $\Phi:\s E\rightarrow\s F$ be a $\cal G$-equivariant unitary equivalence over a $\cal G$-equivariant *-isomorphism $\phi:A\rightarrow B$. For $j=1,2$, there exists a unique map $\Phi_j:\s E_j\rightarrow\s F_j$ satisfying the formula $\Phi(\xi)=\Phi_1(q_{\s E,1}\xi)+\Phi_2(q_{\s E,2}\xi)$ for all $\xi\in\s E$. Moreover, we have:
\begin{enumerate}[label=(\roman*)]
\item for $j=1,\,2$, the map $\Phi_j$ is a unitary equivalence over the *-isomorphism $\phi_j:A_j\rightarrow B_j$ (cf.\ \ref{lem7bis} 1);
\item for all $j,k=1,2$, we have
\begin{equation}\label{eq14}
(\Phi_k\tens\id_{S_{kj}})\circ\delta_{\s E_j}^k=\delta_{\s F_j}^k\circ\Phi_j.
\end{equation}
\end{enumerate}
In particular, $\Phi_j$ is a $\QG_j$-equivariant $\phi_j$-compatible unitary operator.
\item Conversely, for $j=1,\, 2$ let $\Phi_j:\s E_j\rightarrow\s F_j$ be a $\QG_j$-equivariant unitary equivalence over a $\QG_j$-equivariant *-isomorphism $\phi_j:A_j\rightarrow B_j$ such that (\ref{eqmorpheq}) and (\ref{eq14}) hold for all $j,k=1,2$. Then, the map $\Phi:\s E\rightarrow\s F$, defined  by $\Phi(\xi):=\Phi_1(q_{\s E,1}\xi)+\Phi_2(q_{\s E,2}\xi)$ for all $\xi\in\s E$, is a $\cal G$-equivariant unitary equivalence over the $\cal G$-equivariant *-isomorphism $\phi:A\rightarrow B$ (cf.\ \ref{lem7bis} 2).\qedhere
\end{enumerate}
\end{lem}

We also have a description of the $\cal G$-equivariant *-representations in terms of $\QG_j$-equivariant *-representations for $j=1,2$.

\begin{lem}\label{lem11}
Let $A$, $B$ be two $\cal G$-$\Cstar$-algebras and $\s E$ a $\cal G$-equivariant Hilbert $B$-module. We follow the notations of \ref{not12} and \ref{not8} concerning these objects.
\begin{enumerate}
\item Let $\gamma:A\rightarrow\Lin(\s E)$ be a $\cal G$-equivariant *-representation. We have 
\[
\gamma(a)q_{\s E,j}=\gamma(q_{A,j}a)=q_{\s E,j}\gamma(a), \; \text{ for all } \; a\in A \; \text{ and } \; j=1,2.
\]
There exist unique *-representations $\gamma_j:A_j\rightarrow\Lin(\s E_j)$ for $j=1,2$ such that for all $a\in A$, $\gamma(a)=\gamma_1(q_{A,1}a)+\gamma_2(q_{A,2}a)$. Furthermore, for $j,k=1,2$ we have
\begin{equation}\label{eq22}
\delta_{\s E_j}^k(\gamma_j(a)\xi)=(\gamma_k\tens\id_{S_{kj}})(\delta_{A_j}^k(a))\circ\delta_{\s E_j}^k(\xi),\quad \text{for all } a\in A_j \text{ and }  \xi\in\s E_j.
\end{equation}
In particular, the *-representation $\gamma_j:A_j\rightarrow\Lin(\s E_j)$ is $\QG_j$-equivariant.
\item Conversely, let $\gamma_j:A_j\rightarrow\Lin(\s E_j)$ be a $\QG_j$-equivariant *-representation for $j=1,2$. Let $\gamma:A\rightarrow\Lin(\s E)$ be the *-representation defined by $\gamma(a):=\gamma_1(q_{A,1}a)+\gamma_2(q_{A,2}a)$ for all $a\in A$. Assume further that (\ref{eq22}) holds for all $j,k=1,2$. Then, the *-representation $\gamma:A\rightarrow\Lin(\s E)$ is $\cal G$-equivariant.
\end{enumerate}
Moreover, the pair $(\s E,\gamma)$ is a $\cal G$-equivariant Hilbert $A$-$B$-bimodule if, and only if, the pair $(\s E_j,\gamma_j)$ is a $\QG_j$-equivariant Hilbert $A_j$-$B_j$-bimodule for $j=1,2$.
\end{lem}

\begin{proof}
{\setlength{\baselineskip}{1.1\baselineskip}
Since $\beta_A$ is central and $\gamma$ is $\cal G$-equivariant, we have $[\gamma(a),\,\beta_{\s E}(n)]=\gamma([a,\,\beta_A(n)])=0$ for all $n\in\GC^2$. Hence, 
$
\gamma(a)q_{\s E,j}=\gamma(q_{A,j}a)=q_{\s E,j}\gamma(a)
$
for all $a\in A$ and $j=1,2$.
We denote by $\gamma_j:A_j\rightarrow\Lin(\s E_j)$ the *-representation defined by $\gamma_j(a):=\restr{\gamma(a)}{\s E_j}$ for all $a\in A_j$. We have $\gamma(a)=\gamma_1(q_{A,1}a)+\gamma_2(q_{A,2}a)$ for all $a\in A$ and (\ref{eq22}) is a straightforward restatement of the fact that $\delta_{\s E}(\gamma(a)\xi)=(\gamma\tens\id_S)(\delta_A(a))\circ\delta_{\s E}(\xi)$. The converse and the last statement are obvious.\qedhere
\par}
\end{proof}

Note that (\ref{eq22}) can be restated in the following ways: 
\begin{align*}
\s V_j^k(\gamma_j(a)\tens_{\delta_{B_j}^k}1)(\s V_j^k)^*&=(\gamma_k\tens\id_{S_{kj}})\delta_{A_j}^k(a),\quad a\in A; \\
\delta_{\K(\s E_j)}^k(\gamma_j(a))&=(\gamma_k\tens\id_{S_{kj}})\delta_{A_j}^k(a),\quad a\in A.
\end{align*}

The following lemma is straightforward.

\begin{lem}\label{lem10}
Let $A$ and $B$ be two $\cal G$-$\Cstar$-algebras. Let $\s E$ and $\s F$ be two $\cal G$-equivariant Hilbert $B$-modules. 
\begin{enumerate}
\item Let $\pi:A\rightarrow\Lin(\s E)$ and $\gamma:A\rightarrow\Lin(\s F)$ be $\cal G$-equivariant *-representations. Let $\Phi\in\Lin(\s E,F)$ be a $\cal G$-equivariant unitary such that $\Phi\circ\pi(a)=\gamma(a)\circ\Phi$ for all $a\in A$. Then, for $j=1,2$ the $\QG_j$-equivariant unitary $\Phi_j\in\Lin(\s E_j,\s F_j)$ {\rm(}cf.\ \ref{lem5} 1{\rm)} satisfies for all $a\in A_j$ the relation $\Phi_j\circ\pi_j(a)=\gamma_j(a)\circ\Phi_j$ {\rm(}cf.\ \ref{lem11} 1{\rm)}.
\item Conversely, for $j=1,2$ let us fix $\QG_j$-equivariant *-representations $\pi_j:A_j\rightarrow\Lin(\s E_j)$ and $\gamma_j:A_j\rightarrow\Lin(\s F_j)$ and a $\QG_j$-equivariant unitary $\Phi_j\in\Lin(\s E_j,\s F_j)$ satisfying the relation $\Phi_j\circ\pi_j(a)=\gamma_j(a)\circ\Phi_j$ for all $a\in A_j$. Then, the $\cal G$-equivariant unitary $\Phi\in\Lin(\s E,\s F)$ {\rm(}cf.\ \ref{lem5} 2{\rm)} satisfies the relation $\Phi\circ\pi(a)=\gamma(a)\circ\Phi$ for all $a\in A$ {\rm(}cf.\ \ref{lem11} 2{\rm)}. \qedhere
\end{enumerate}
\end{lem}

\subsubsection{Induction of equivariant Hilbert C*-modules}

In the following, we recall the induction procedure for equivariant Hilbert C*-modules (cf.\ 4.3 \cite{BC}, 6.3 \cite{C2}). We assume that the quantum groups $\QG_1$ and $\QG_2$ are regular.
Fix a $\QG_1$-$\Cstar$-algebra $(A_1,\delta_{A_1})$ and a $\QG_1$-equivariant Hilbert $A_1$-module $(\s E_1,\delta_{\s E_1})$. We denote by $J_1:=\K(\s E_1\oplus A_1)$ the associated linking $\Cstar$-algebra endowed with the continuous action $\delta_{J_1}$ of $\QG_1$.

\begin{nbs} Let us fix some notations.
\begin{itemize}
\item Let $\id_{\s E_1}\tens\delta_{11}^2:\Lin(A_1\tens S_{11},\s E_1\tens S_{11})\rightarrow\Lin(A_1\tens S_{12}\tens S_{21},\s E_1\tens S_{12}\tens S_{21})$ be the unique linear extension of $\id_{\s E_1}\tens\delta_{11}^2:\s E_1\tens S_{11}\rightarrow\Lin(A_1\tens S_{12}\tens S_{21},\s E_1\tens S_{12}\tens S_{21})$ such that
$
(\id_{\s E_1}\tens\delta_{11}^2)(T)(\id_{A_1}\tens\delta_{11}^2)(x)=(\id_{\s E_1}\tens\delta_{11}^2)(Tx)
$
for all $x\in\M(A_1\tens S_{11})$ and $T\in\Lin(A_1\tens S_{11},\s E_1\tens S_{11})$.
\item Let $\delta_{\s E_1}^{(2)}:\s E_1\rightarrow\Lin(A_1\tens S_{12}\tens S_{21},\s E_1\tens S_{12}\tens S_{21})$ be the linear map defined by
$
\delta_{\s E_1}^{(2)}(\xi):=(\id_{\s E_1}\tens\delta_{11}^2)\delta_{\s E_1}(\xi)
$
for all $\xi\in\s E_1$.\index[symbol]{dg@$\delta_{\s E_1}^{(2)}$}
\item Consider the Banach subspace of $\Lin(A_1\tens S_{12},\s E_1\tens S_{12})$ defined by (cf.\ \ref{not4}):\index[symbol]{id@$\ind(\s E_1)$, induced Hilbert module} 
\[
\ind({\s E}_1):=\![(\id_{\s E_1\tens S_{12}}\tens\omega)\delta_{\s E_1}^{(2)}(\xi)\,;\,\xi\in\s E_1,\omega\in\B(\s H_{21})_*].\qedhere
\]
\end{itemize}
\end{nbs} 

\begin{prop}\label{prop3}
Let us denote by $A_2:=\ind(A_1)$ the induced $\QG_2$-$\Cstar$-algebra of $A_1$. Let $\s E_2:=\ind(\s E_1)$.
\begin{enumerate}
\item We have
$
[\s E_2(1_{A_1}\tens S_{12})]=\s E_1\tens S_{12}=[(1_{\s E_1}\tens S_{11})\s E_2].
$
In particular, $\s E_2\subset\widetilde{\M}(\s E_1\tens S_{12})$.
\item $\s E_2$ is a Hilbert $A_2$-module for the right action by composition and the $A_2$-valued inner product given by $\langle\xi,\, \eta\rangle:=\xi^*\circ\eta$ for $\xi,\eta\in\ind(\s E_1)$.\qedhere
\end{enumerate} 
\end{prop}

Let us denote by $(A_2,\delta_{A_2}):=\ind(A_1,\delta_{A_1})$ and $(J_2,\delta_{J_2}):=\ind(J_1,\delta_{J_1})$ the induced $\QG_2$-$\Cstar$-algebra of $(A_1,\delta_{A_1})$ and $(J_1,\delta_{J_1})$ respectively. We also denote by $\s E_2:=\ind(\s E_1)$ the induced Hilbert $A_2$-module as defined above.$\vphantom{\ind}$

\begin{nb}
Let 
\[
\id_{\s E_1}\tens\delta_{12}^2:\Lin(A_1\tens S_{12},\s E_1\tens S_{12})\rightarrow\Lin(A_1\tens S_{12}\tens S_{22},\s E_1\tens S_{12}\tens S_{22})
\] 
be the unique linear extension of $\id_{\s E_1}\tens\delta_{12}^2:\s E_1\tens S_{12}\rightarrow\Lin(A_1\tens S_{12}\tens S_{22},\s E_1\tens S_{12}\tens S_{22})$ such that
$
(\id_{\s E_1}\tens\delta_{12}^2)(T)(\id_{A_1}\tens\delta_{12}^2)(x)=(\id_{\s E_1}\tens\delta_{12}^2)(Tx)
$
for all $x\in\M(A_1\tens S_{12})$ and $T\in\Lin(A_1\tens S_{12}, \s E_1\tens S_{12})$.
\end{nb}

\begin{propdef}\label{prop12}
There exists a unique linear map 
\begin{center}
$
\delta_{\s E_2}:\s E_2\rightarrow\Lin(A_2\tens S_{22},\s E_2\tens S_{22})
$
\end{center} 
satisfying the relation $[\delta_{\s E_2}(\xi)a]b=(\id_{\s E_1}\tens\delta_{12}^2)(\xi)(ab)$ for all $\xi\in\s E_2$, $a\in A_2\tens S_{22}$ and $b\in A_1\tens S_{12}\tens S_{22}$. Moreover, the pair $(\s E_2,\delta_{\s E_2})$ is a $\QG_2$-equivariant Hilbert $A_2$-module.
\end{propdef}

In the proposition below, we state that the above induction procedure for equivariant Hilbert C*-modules is equivalent to that of \S 4.3 \cite{BC}.

\begin{nbs}
Let $e_{1,1}:=\iota_{\K(\s E_1)}(1_{\s E_1})\in\M(J_1)$ and $e_{2,1}:=\iota_{A_1}(1_{A_1})\in\M(J_1)$, where we identify $\M(J_1)=\Lin(\s E_1\oplus A_1)$. Let $(J_2,\delta_{J_2},e_{1,2},e_{2,2})$ be the induced linking $\QG_2$-$\Cstar$-algebra, with $e_{l,2}:=e_{l,1}\tens 1_{S_{12}}\in\M(J_2)$ for $l=1,2$ (cf.\ 4.14 \cite{BC}). Consider $e_{2,2}J_2e_{2,2}$ and $e_{1,2}J_2e_{2,2}$ endowed with their structure of $\QG_2$-$\Cstar$-algebra and $\QG_2$-equivariant Hilbert $e_{2,2}J_2e_{2,2}$-module (cf.\ \cite{BS1}). Recall that $\ind\iota_{A_1}:A_2\rightarrow J_2\,;\, x\mapsto(\iota_{A_1}\tens\id_{S_{12}})(x)$ induces a $\QG_2$-equivariant *-isomormorphism $A_2\rightarrow e_{2,2}J_2e_{2,2}$ (cf.\ 4.17, 4.18 \cite{BC}).
\end{nbs}

\begin{prop}\label{prop11}
We use the above notations.
\begin{enumerate}[label=(\roman*)]
\item There exists a unique bounded linear map $\ind\iota_{\s E_1}:\s E_2\rightarrow J_2$ such that
\[
\ind\iota_{\s E_1}((\id_{\s E_1\tens S_{12}}\tens\omega)\delta_{\s E_1}^{(2)}(\xi))
=(\id_{J_1\tens S_{12}}\tens\omega)\delta_{J_1}^{(2)}(\iota_{\s E_1}(\xi)),
\]
for all $\xi\in\s E_1$ and $\omega\in\B(\s H_{21})_*$. Moreover, we have $\ind\iota_{\s E_1}(\s E_2)=e_{1,2}J_2e_{2,2}$ and $\ind\iota_{\s E_1}$ induces a $\QG_2$-equivariant unitary equivalence
$
\s E_2 \rightarrow e_{1,2}J_2e_{2,2} \; ; \; \xi \mapsto \ind\iota_{\s E_1}(\xi)
$
over the $\QG_2$-equivariant *-isomorphism 
$
A_2 \rightarrow e_{2,2}J_2 e_{2,2} \; ; \; a \mapsto \ind\iota_{A_1}(a).
$
\item There exists a unique *-homomorphism $\tau:\K(\s E_2\oplus A_2)\rightarrow J_2$ such that
$
\tau\circ\iota_{\s E_2}=\ind\iota_{\s E_1}
$
and
$
\tau\circ\iota_{A_2}=\ind\iota_{A_1}.
$ 
Moreover, $\tau$ is an isomorphism of linking $\QG_2$-$\Cstar$-algebras.
\item If $T\in\ind(\K(\s E_1))\subset\Lin(\s E_1\tens S_{12})$ and $\eta\in\s E_2\subset\Lin(A_1\tens S_{12},\s E_1\tens S_{12})$, then we have $T\circ\eta\in\s E_2$. Moreover, for all $T\in\ind(\K(\s E_1))$, we have $[\eta\mapsto T\circ\eta]\in\K(\s E_2)$. More precisely, the map
$
\ind(\K(\s E_1))\rightarrow\K(\s E_2)\,;\, T \mapsto [\eta\mapsto T\circ\eta]
$
is a $\QG_2$-equivariant *-isomorphism.$\vphantom{\ind}$\qedhere
\end{enumerate}
\end{prop}

In the result below, we recall how to induce $\QG_1$-equivariant unitary equivalence.

\begin{propdef}\label{prop14}
Let us fix some notations. Consider: 
\begin{itemize}
\item two $\QG_1$-$\Cstar$-algebras $A_1$ and $B_1$; 
\item two $\QG_1$-equivariant Hilbert modules $\s E_1$ and $\s F_1$ over $A_1$ and $B_1$ respectively;
\item a $\QG_1$-equivariant unitary equivalence $\Phi_1:\s E_1\rightarrow\s F_1$ over a $\QG_1$-equivariant *-isomorphism $\phi_1:A_1\rightarrow B_1$.
\end{itemize} 
Denote by:
\begin{itemize}
\item $A_2:=\ind(A_1)$ and $B_2:=\ind(B_1)$ the induced $\QG_2$-$\Cstar$-algebras;
\item $\ind(\phi_1):A_2\rightarrow B_2$ the induced $\QG_2$-equivariant *-isomorphism;
\item $\s E_2:=\ind(\s E_1)$ and $\s F_2:=\ind(\s F_1)$ the induced $\QG_2$-equivariant Hilbert modules over $A_2$ and $B_2$ respectively;
\item $
\Phi_1\tens\id_{S_{12}}:\Lin(A_1\tens S_{12},\s E_1\tens S_{12})\rightarrow\Lin(B_1\tens S_{12},\s F_1\tens S_{12})
$
the unique linear map such that $(\Phi_1\tens\id_{S_{12}})(T)(\phi_1\tens\id_{S_{12}})(x)=(\Phi_1\tens\id_{S_{12}})(Tx)$ for all $\Lin(A_1\tens S_{12},\s E_1\tens S_{12})$ and $x\in A_1\tens S_{12}$ (cf.\ \ref{not3}).
\end{itemize}
Then, $(\Phi_1\tens\id_{S_{12}})(\s E_2)\subset\s F_2$ and the map 
$
\ind(\Phi_1):=\restr{(\Phi_1\tens\id_{S_{12}})}{\s E_2}:\s E_2\rightarrow\s F_2
$ 
is a $\QG_2$-equivariant unitary equivalence over $\ind(\phi_1):A_2\rightarrow B_2$. Moreover, for all $\xi\in\s E_1$ and $\omega\in\B(\s H_{21})_*$ we have
$
\ind(\Phi_1)((\id_{\s E_1 \tens S_{12}}\tens\omega)\delta_{\s E_1}^{(2)}(\xi))=(\id_{\s F_1 \tens S_{12}}\tens\omega)\delta_{\s F_1}^{(2)}(\Phi_1\xi).
$
\end{propdef}

We can also induce $\QG_1$-equivariant *-representations. Let us state a preliminary result.

\begin{lem}\label{lem27}
Let $A_1$ be a $\QG_1$-C*-algebra. If $A_1$ is $\sigma$-unital (resp.\ separable), then so is the induced $\QG_2$-C*-algebra $\ind(A_1)$.
\end{lem}

\begin{proof}
Let us assume that $A_1$ is $\sigma$-unital. Let $(u_n)_{n\geqslant 1}$ be a countable approximate unit of $A_1$. Let $\omega\in\B(\s H_{21})_*$ such that $\omega(1)=1$. Then, the sequence $((\id_{A_1\tens S_{12}}\tens\omega)\delta_{A_1}^{(2)}(u_n))_{n\geqslant 1}$ is an approximate unit of $\ind(A_1)$. Hence, $\ind(A_1)$ is $\sigma$-unital. Suppose now that $A_1$ is separable. Let $X$ (resp.\ $Y$) be a countable total subset of $A_1$ (resp.\ $\s H_{21}$).$\vphantom{\ind}$ Hence, the subset $\{(\id_{A_1\tens S_{12}}\tens\omega_{\xi,\eta})\delta_{A_1}^{(2)}(a)\,;\,a\in X,\, \xi,\eta\in Y\}$ of $\ind(A_1)$ is countable and spans a dense subspace of $\ind(A_1)$. Hence, the C*-algebra $\ind(A_1)$ is separable.
\end{proof}

\begin{propdef}\label{def5}
$\vphantom{\ind}$Let $A_1$ and $B_1$ be $\QG_1$-$\Cstar$-algebras and $\s E_1$ a $\QG_1$-equivariant Hilbert $A_1$-$B_1$-bimodule. Let $A_2:=\ind(A_1)$ and $B_2:=\ind(B_1)$ be the induced $\QG_2$-$\Cstar$-algebras. Let $\s E_2:=\ind(\s E_1)$ be the induced $\QG_2$-equivariant Hilbert $B_2$-module.$\vphantom{\ind}$ Let us consider a $\QG_1$-equivariant *-representation $\gamma_1:A_1\rightarrow\Lin(\s E_1)$. $\vphantom{\ind}$Up to the identifications $\Lin(\s E_1)=\M(\K(\s E_1))$ and $\ind\K(\s E_1)=\K(\s E_2)$ (cf.\ \ref{prop11} (iii)), we have a $\QG_2$-equivariant *-representation 
\begin{center}
$
\ind\gamma_1:A_2\rightarrow\Lin(\s E_2) \quad \text{(cf.\ \ref{defHomInd})}
$ 
\end{center}
called the induced $\QG_2$-equivariant *-representation of $\gamma_1$.$\vphantom{\ind}$ If $\s E_1$ is a $\QG_1$-equivariant Hilbert $A_1$-$B_1$-bimodule,$\vphantom{\ind}$ then $\s E_2$ is a $\QG_2$-equivariant Hilbert $A_2$-$B_2$-bimodule called the induced $\QG_2$-equivariant bimodule of $\s E_1$.$\vphantom{\ind}$
\end{propdef}

\begin{proof}
The fact that we have a well-defined induced $\QG_2$-equivariant *-representation $\ind\gamma_1:A_2\rightarrow\Lin(\s E_2)$ is just a restatement of \ref{defHomInd} and \ref{prop11} (iii). Let us assume that $\s E_1$ is countably generated as a Hilbert $B_1$-module, {\it i.e.\ }the C*-algebra $\K(\s E_1)$ is $\sigma$-unital. Hence, $\K(\s E_2)$ is $\sigma$-unital (cf.\ \ref{prop11} (iii) and \ref{lem27}), {\it i.e.\ }$\s E_2$ is a countably generated Hilbert $B_2$-module.$\vphantom{\ind}$
\end{proof}

By exchanging the roles of $\QG_1$ and $\QG_2$, we define as above an induction procedure for $\QG_2$-equivariant Hilbert modules.

\medskip 

In the following, we investigate the composition of $\ind$ and $\iind$. Let $j,k=1,2$. Let $A_j$ be a $\QG_j$-$\Cstar$-algebra and $\s E_j$ a $\QG_j$-equivariant Hilbert $A_j$-module. Denote by:
\begin{itemize}
\item $A_k:={\rm Ind}_{\QG_j}^{\QG_k}(A_j)$ and $\s E_k={\rm Ind}_{\QG_j}^{\QG_k}(\s E_j)\subset\Lin(A_j\tens S_{jk},\s E_j\tens S_{jk})$ the induced $\QG_k$-$\Cstar$-algebra and the induced $\QG_k$-equivariant Hilbert $A_k$-module;
\item $C={\rm Ind}_{\QG_k}^{\QG_j}(A_k)$ and $\s F:={\rm Ind}_{\QG_k}^{\QG_j}(\s E_k)\subset\Lin(A_k\tens S_{kj},\s E_k\tens S_{kj})$ the induced $\QG_j$-$\Cstar$-algebra  and the induced $\QG_j$-equivariant Hilbert $C$-module.
\end{itemize}

\begin{prop}\label{prop13}
With the above notations and hypotheses, we have the following statements:
\begin{enumerate}
\item there exists a unique map $\Pi_j:\s E_j \rightarrow\s F$ such that
\[
(\Pi_j(\xi)x)a=\delta_{\s E_j}^{(k)}(\xi)(xa), \quad \text{for all } \; \xi\in\s E_j,\, x\in A_k\tens S_{kj} \; \text{ and } \; a\in A_j\tens S_{jk}\tens S_{kj};
\]
moreover, $\Pi_j$ is a $\QG_j$-equivariant unitary equivalence over the $\QG_j$-equivariant *-isomorphism $\pi_j:A_j\rightarrow C\,;\,a\mapsto\delta_{A_j}^{(k)}(a)$;\index[symbol]{pl@$\Pi_j$}
\item  
$
\delta_{\s E_j}^k:\s E_j \rightarrow \widetilde{\M}(\s E_k\tens S_{kj})\,;\, \xi \mapsto \Pi_j(\xi)
$
is a well-defined linear map such that:
\begin{enumerate}[label=(\roman*)]
\item $\delta_{\s E_j}^k(\xi a)=\delta_{\s E_j}^k(\xi)\delta_{A_j}^k(a)$ and $\langle\delta_{\s E_j}^k(\xi),\, \delta_{\s E_j}^k(\eta)\rangle = \delta_{A_j}^{k}(\langle\xi,\, \eta\rangle)$ for all $\xi,\eta\in\s E_j$ and $a\in A_j$,
\item $[\delta_{\s E_j}^k(\s E_j)(1_{A_k}\tens S_{kj})]=\s E_2\tens S_{kj}=[(1_{\s E_k}\tens S_{kj})\delta_{\s E_j}^k(\s E_j)]$.\qedhere
\end{enumerate}
\end{enumerate}
\end{prop}

\begin{thm}\label{theo10}
Let $\QG_1$ and $\QG_2$ be two monoidally equivalent regular locally compact quantum groups. The map 
\begin{align*}
\ind &:(\s E_1,\delta_{\s E_1}) \mapsto (\s E_2:=\ind(\s E_1),\, \delta_{\s E_2}:\xi\in\s E_2\mapsto[x\in A_2\tens S_{22}\mapsto(\id_{\s E_1}\tens\delta_{12}^2)(\xi)x]),\\
\intertext{where $\s E_1$ is a Hilbert module over the $\QG_1$-$\Cstar$-algebra $A_1$ and $A_2=\ind(A_1)$ denotes the induced $\QG_2$-$\Cstar$-algebra, is a one-to-one correspondence up to unitary equivalence. The inverse map, up to unitary equivalence, is}
\iind &: (\s F_2,\delta_{\s F_2}) \mapsto (\s F_1:=\iind(\s F_2),\, \delta_{\s F_1}:\xi\in\s F_1\mapsto[x\in B_1\tens S_{11}\mapsto(\id_{\s F_2}\tens\delta_{21}^1)(\xi)x]),
\end{align*} 
where $\s F_2$ is a Hilbert module over the $\QG_2$-$\Cstar$-algebra $B_2$ and $B_1=\ind(B_2)$ denotes the induced $\QG_1$-$\Cstar$-algebra.
\end{thm}

Let $B_1$ be a $\QG_1$-$\Cstar$-algebra. Let us denote by $B_2:=\ind(B_1)$ the induced $\QG_2$-$\Cstar$-algebra. Let
$
\delta_{B_j}^k:B_j\rightarrow\M(B_k\tens S_{kj})
$
for $j,k=1,2$ be the *-homomorphisms defined in \ref{not11}.

\begin{nbs}
Let $\s E_1$ be a $\QG_1$-equivariant Hilbert $B_1$-module. Let us denote by $\s F_2=\ind(\s F_1)$ the induced $\QG_2$-equivariant Hilbert $B_2$-module. We have four linear maps
\[
\delta_{\s F_j}^k:\s F_j\rightarrow\Lin(B_k\tens S_{kj},\s F_k\tens S_{kj}), \quad \text{for } j,k=1,2,
\]
defined as follows:
\begin{itemize}
\item $\delta_{\s F_1}^1:=\delta_{\s F_1}$ and $\delta_{\s F_2}^2:=\delta_{\s F_2}$;
\item $\delta_{\s F_1}^2:\s F_1\rightarrow\Lin(B_2\tens S_{21},\s F_2\tens S_{21})$ is the unique linear map such that 
\[
(\delta_{\s F_1}^2(\xi)x)b=\delta_{\s F_1}^{(2)}(\xi)(xb)
\] 
for all $\xi\in\s F_1$, $x\in B_2\tens S_{21}$ and $b\in B_1\tens S_{12}\tens S_{22}$, where $\delta_{\s F_1}^{(2)}(\xi):=(\id_{\s E_1}\tens\delta_{11}^2)\delta_{\s F_1}(\xi)$ (cf.\ \ref{prop13});
\item $\delta_{\s F_2}^1:\s F_2\rightarrow\Lin(B_1\tens S_{12},\s F_1\tens S_{12})$ is the unique linear map such that for all $\xi\in\s F_2$, $x\in\iind(B_2)\tens S_{12}$ and $y\in B_2\tens S_{21}\tens S_{12}$, we have
\[
[(\Pi_1\tens\id_{S_{12}})(\delta_{\s F_2}^1(\xi))x]y=\delta_{\s F_2}^{(1)}(\xi)(xy),
\] 
where $\delta_{\s F_2}^{(1)}(\xi):=(\id_{\s F_1}\tens\delta_{22}^1)\delta_{\s F_2}(\xi)$ and $\Pi_1:\s F_1\rightarrow\iind(\s F_2)$ (cf.\ \ref{prop13} 1).\qedhere
\end{itemize}
\end{nbs}

\begin{lem}\label{lem4}
For all $j,k,l=1,2$, we have the following statements:
\begin{enumerate}
\item $\delta_{\s F_j}^k(\s F_j)\subset\widetilde{\M}(\s F_k\tens S_{kj})$;
\item $\delta_{\s F_j}^k(\xi b)=\delta_{\s F_j}^k(\xi)\delta_{B_j}^k(b)$ and $\langle\delta_{\s F_j}^k(\xi),\, \delta_{\s F_j}^k(\eta)\rangle = \delta_{B_j}^k(\langle\xi,\, \eta\rangle)$ for all $\xi,\eta\in\s F_j$ and $b\in B_j$;
\item $[\delta_{\s F_j}^k(\s F_j)(1_{B_k}\tens S_{kj})]=\s F_k\tens S_{kj}=[(1_{\s F_k}\tens S_{kj})\delta_{\s F_j}^k(\s F_j)]$;
\item $\delta_{\s F_k}^l\tens\id_{S_{kj}}$ {\rm(}resp.\ $\id_{\s F_l}\tens\delta_{lj}^k)$ extends uniquely to a linear map from $\Lin(B_k\tens S_{kj},\s E_k\tens S_{kj})$ to $\Lin(B_l\tens S_{lk}\tens S_{kj},\s E_l\tens S_{lk}\tens S_{kj})$ such that
\begin{align*}
(\delta_{\s F_k}^l\tens\id_{S_{kj}})(T)(\delta_{B_k}^l\tens\id_{S_{kj}})(x)&=(\delta_{\s F_k}^l\tens\id_{S_{kj}})(Tx)\\
\text{{\rm(}resp.\ }(\id_{\s F_l}\tens\delta_{lj}^k)(T)(\id_{B_l}\tens\delta_{lj}^k)(x)&=\id_{\s F_l}\tens\delta_{lj}^k)(Tx)\text{{\rm)}}
\end{align*}
for all $T\in\Lin(B_k\tens S_{kj},\s E_k\tens S_{kj})$ and $x\in B_k\tens S_{kj}$;
\item $(\delta_{\s F_k}^l\tens\id_{S_{kj}})\delta_{\s F_j}^k=(\id_{\s F_l}\tens\delta_{lj}^k)\delta_{\s F_j}^l$.\qedhere
\end{enumerate}
\end{lem}

Let us consider the $\Cstar$-algebra $B:=B_1\oplus B_2$ endowed with the continuous action $(\beta_B,\delta_B)$ (cf.\ \ref{prop37}).

\begin{prop}\label{prop16}
Let $\s F_1$ be a $\QG_1$-equivariant Hilbert $B_1$-module. Let $\s F_2:=\ind(\s F_1)$ be the induced $\QG_2$-equivariant Hilbert $B_2$-module. Consider the Hilbert $B$-module $\s F:=\s F_1\oplus\s F_2$. Denote by
$
\Pi_j^k:\Lin(B_k\tens S_{kj},\s F_k\tens S_{kj})\rightarrow\Lin(B\tens S,\s F\tens S) 
$ 
\index[symbol]{pk@$\Pi_j^k$}the linear extension of the canonical injection $\s F_k\tens S_{kj}\rightarrow\s F\tens S$.
Let us consider the linear maps $\delta_{\s F}:\s F\rightarrow\Lin(B\tens S,\s F\tens S)$ and $\beta_{\s F}:\GC^2\rightarrow\Lin(\s F)$ defined by:
\[
\delta_{\s F}(\xi):=\sum_{k,j=1,2}\Pi_j^k\circ\delta_{\s F_j}^k(\xi_j), \quad \xi=(\xi_1,\xi_2)\in\s F ;\quad 
\beta_{\s F}(\lambda,\mu):=\begin{pmatrix} \lambda & 0 \\ 0 & \mu\end{pmatrix}\!, \quad (\lambda,\mu)\in\GC^2.
\]
Then, the triple $(\s F,\beta_{\s F},\delta_{\s F})$ is a $\cal G$-equivariant Hilbert $B$-module.
\end{prop}

\begin{prop}\label{prop17}
Let $(\s E,\beta_{\s E},\delta_{\s E})$ be a $\cal G$-equivariant Hilbert $A$-module. In the following, we use the notations of \ref{propdef4}. Let $j,k=1,2$ with $j\neq k$. Let 
\[
\widetilde{A}_j:={\rm Ind}_{\QG_k}^{\QG_j}(A_k,\delta_{A_k}^k) \quad \text{and} \quad \widetilde{\s E}_j:={\rm Ind}_{\QG_k}^{\QG_j}(\s E_k,\delta_{\s E_k}^k).
\] 
If $\xi\in\s E_j$, then we have $\delta_{\s E_j}^k(\xi)\in\widetilde{\s E}_j\subset\widetilde{\M}(\s E_k\tens S_{kj})$ and the map $\widetilde{\Pi}_j:\s E_j \rightarrow\widetilde{\s E}_j \, ; \, \xi  \mapsto \delta_{\s E_j}^k(\xi)$ is a $\QG_j$-equivariant unitary equivalence over $\widetilde{\pi}_j:A_j\rightarrow\widetilde{A}_j$ (cf.\ \ref{prop4}).\index[symbol]{pm@$\widetilde{\Pi}_j$}
\end{prop}

\begin{thm}\label{theo9}
Let ${\cal G}_{\QG_1,\QG_2}$ be a colinking measured quantum groupoid between two regular monoidally equivalent locally compact quantum groups $\QG_1$ and $\QG_2$. 
Let $j=1,2$. The map $(\s E,\beta_{\s E},\delta_{\s E})\mapsto(\s E_j,\delta_{\s E_j}^j)$ is a one-to-one correspondence up to unitary equivalence (cf.\ \ref{propdef4} and \ref{lem5} 1). The inverse map, up to unitary equivalence, is $(\s F_j,\delta_{\s F_j})\mapsto(\s F,\beta_{\s F},\delta_{\s F})$ (cf.\ \ref{prop16}, \ref{prop14} and \ref{lem5} 2).
\end{thm}

\begin{prop}\label{def10}
We follow the hypotheses and notations of \ref{prop16}. Let $\gamma_1:A_1\rightarrow\Lin(\s E_1)$ be a $\QG_1$-equivariant *-representation of a $\QG_1$-C*-algebra $A_1$. Let $A_2$ be the induced $\QG_2$-C*-algebra of $A_1$ and $\gamma_2:A_2\rightarrow\Lin(\s F_2)$ be the induced $\QG_2$-equivariant *-representation of $\gamma_1$ (cf.\ \ref{def5}). Let us endow the $\Cstar$-algebra $A:=A_1\oplus A_2$ with the continuous action $(\beta_A,\delta_A)$ (cf.\ \ref{prop37}). Then, the map 
\[
\gamma:A\rightarrow\Lin(\s F)\,;\, (a_1,a_2)\mapsto\begin{pmatrix}\gamma_1(a_1) & 0\\ 0 & \gamma_2(a_2)\end{pmatrix}
\]
is a $\cal G$-equivariant *-representation. Moreover, if $\s F_1$ is a $\QG_1$-equivariant Hilbert Hilbert $A_1$-$B_1$-bimodule, then $\s F$ is a $\cal G$-equivariant Hilbert $A$-$B$-bimodule.
\end{prop}

\begin{proof}
This is a straightforward consequence of \ref{lem11} 2 and \ref{def5}.
\end{proof}

\subsubsection{Structure of the double crossed product}\label{sectionDoubleCrossedProduct}

In this paragraph, we assume the quantum groups $\QG_1$ and $\QG_2$ to be regular. We recall that the colinking measured quantum groupoid ${\cal G}:={\cal G}_{\QG_1,\QG_2}$ is regular. 

\medbreak

Let $A$ be a $\cal G$-C*-algebra and $\s E$ a $\cal G$-equivariant Hilbert $A$-$B$-module. In this paragraph, we restate the main results of \S 4.4 of \cite{BC} in order to describe the structure of the double crossed product $(\s E\rtimes{\cal G})\rtimes\widehat{\cal G}$. Let $D_{\rm g}$ (resp.\ $D_{\rm d}$) be the bidual $\cal G$-C*-algebra of $A$ (resp.\ $B$) (cf.\ \S \ref{sectionTT} and \ref{conv2}). Let $\s D$ be the bidual $\cal G$-equivariant Hilbert $D$-module of $\s E$ (cf.\ \S \ref{sectionTT2}). In the following, we use all the notations and results of \S\S \ref{sectionC*AlgColink} and \ref{sectionHilbModColink}. 

\medbreak

Let us recall the notations of 3.48 \cite{BC} concerning the structure of the bidual $\cal G$-C*-algebra $D_{\rm d}$ (and similarly for $D_{\rm g}$).
\smallbreak
$\bullet$ We have $D_{\rm d}=q_{\beta_B\widehat{\alpha}}(B\tens\K)q_{\beta_B\widehat{\alpha}}=\oplus_{j=1,2}\,B_j\tens\K(\s H_{1j}\oplus\s H_{2j})$. For all $j=1,2$, we will identify $D_{{\rm d},j}=B_j\tens\K(\s H_{1j}\oplus\s H_{2j})$. Let ${\cal B}_{ll',j,{\rm d}}:=B_j\tens\K(\s H_{l'j},\s H_{lj})$\index[symbol]{bd@${\cal B}_{ll',j,\bullet}$, ${\cal B}_{l,j,\bullet}$} for $l,l',j=1,2$. Let ${\cal B}_{l,j,{\rm d}}:={\cal B}_{ll,j,{\rm d}}=B_j\tens\K(\s H_{lj})$ for $l,j=1,2$. For $l,l',j=1,2$, ${\cal B}_{l,j,{\rm d}}$ and ${\cal B}_{l',j,{\rm d}}$ are C*-algebras and ${\cal B}_{ll',j,{\rm d}}$ turns into a Hilbert ${\cal B}_{l,j,{\rm d}}$-${\cal B}_{l',j,{\rm d}}$-bimodule.
\smallbreak
$\bullet$ For $l,l',j,k=1,2$, let $\delta_{{\cal B}_{ll',j,{\rm d}},0}^k:\B_{ll',j,{\rm d}}\rightarrow\Lin(B_k\tens\K(\s H_{l'j})\tens S_{kj},B_k\tens\K(\s H_{l'j},\s H_{lj})\tens S_{kj})$ be the linear map defined by
\[
\delta_{{\cal B}_{ll',j,{\rm d}},0}^k(b\tens T):=\delta_{B_j}^k(b)_{13}(1_{B_k}\tens T\tens 1_{S_{kj}}),\quad \text{for all } b\in B_j \text{ and } T\in\K(\s H_{l'j},\s H_{lj}).
\]
$\bullet$ For $l,l',j,k=1,2$, let $\delta_{{\cal B}_{ll',j,{\rm d}}}^k:{\cal B}_{ll',j,{\rm d}}\rightarrow\Lin({\cal B}_{l',k,{\rm d}}\tens S_{kj},{\cal B}_{ll',k,{\rm d}}\tens S_{kj})$ be the linear map defined by
\[
\delta_{{\cal B}_{ll',j,{\rm d}}}^k(x):=(V_{kj}^l)_{23}\delta_{{\cal B}_{ll',j,{\rm d}},0}^k(x)(V_{kj}^{l'})_{23}^*,\quad \text{for all } x\in{\cal B}_{ll',j,{\rm d}}.
\]
Note that $\delta_{{\cal B}_{l,j,{\rm d}}}^k:{\cal B}_{l,j,{\rm d}}\rightarrow\Lin({\cal B}_{l,k,{\rm d}}\tens S_{kj})$ is a *-homomorphism.\index[symbol]{dh@$\delta_{{\cal B}_{ll',j,\bullet},0}^k$, $\delta_{{\cal B}_{ll',j,\bullet}}^k$, $\delta_{{\cal B}_{l,j,\bullet}}^k$}

\medbreak

We identify $(\s E\rtimes{\cal G})\rtimes\widehat{\cal G}=\s D$ (cf.\ \ref{theo2}). We have $\s D=q_{\beta_{\s E}\widehat{\alpha}}(\s E\tens\K)q_{\beta_B\widehat{\alpha}}\subset\s E\tens\K$. in the following, we investigate the precise structure of $\s D$.
\smallbreak
$\bullet$ We have $q_{\s D,j}:=\beta_{\s D}(\varepsilon_j)=q_{\beta_{\s E}\widehat{\alpha}}(1_{\s E}\tens\beta(\varepsilon_j))=\beta_{\s E}(\varepsilon_j)\tens\beta(\varepsilon_j)=\sum_{l=1,2}q_{\s E,j}\tens p_{lj}$ for all $j=1,2$ (cf.\ \ref{lem26} and \ref{not10}).
\smallbreak
$\bullet$ For $j=1,2$, let us consider the Hilbert $D_{{\rm g},j}$-$D_{{\rm d},j}$-module $\s D_j:=q_{\s D,j}\s D$ (cf.\ \S \ref{sectionHilbModColink}). We have 
\[
\s D_j=\s E_j\tens\K(\s H_{1j}\oplus\s H_{2j})=\bigoplus_{l,l'=1,2}\,\s E_j\tens\K(\s H_{l'j},\s H_{lj}).
\]
For $l,l',j=1,2$, let us consider the Hilbert ${\cal B}_{l,j,{\rm g}}$-${\cal B}_{l',j,{\rm d}}$-bimodule ${\cal E}_{ll',j}:=\s E_j\tens\K(\s H_{l'j},\s H_{lj})$. For $l,j=1,2$, let ${\cal E}_{l,j}:={\cal E}_{ll,j}=\s E_j\tens\K(\s H_{lj})$\index[symbol]{edc@${\cal E}_{ll',j}$}.
\smallbreak
$\bullet$ For $j,k=1,2$, let us denote by $\Pi_j^k:\Lin(D_{{\rm d},k}\tens S_{kj},\s D_k\tens S_{kj})\rightarrow\Lin(D_{\rm d}\tens S,\s D\tens S)$ the linear extension of the inclusion map $\s D_k\tens S_{kj}\rightarrow \s D\tens S$. For $j,k=1,2$, let us denote by $\delta_{\s D_j}^k:\s D_j\rightarrow\Lin(D_{k,{\rm d}}\tens S_{kj},\s D_k\tens S_{kj})$ the linear map defined in \ref{propdef4}. For all $\zeta\in\s D$, we have  
\[
\delta_{\s D}(\zeta)=\sum_{j,k=1,2}\Pi_j^k\circ\delta_{\s D_j}^k(q_{\s D,j}\zeta).
\]
We recall that for $j=1,2$ the pair $(\s D_j,\delta_{\s D_j}^j)$ is a $\QG_j$-equivariant Hilbert $D_{{\rm g},j}$-$D_{{\rm d},j}$-bimodule.
\smallbreak
$\bullet$ For $l,l',j,k=1,2$, let $\delta_{{\cal E}_{ll',j},0}^k:{\cal E}_{ll',j}\rightarrow\Lin(B_k\tens\K(\s H_{l'j})\tens S_{kj},\s E_k\tens\K(\s H_{l'j},\s H_{lj})\tens S_{kj})$ be the linear map defined by:
\[
\delta_{{\cal E}_{ll',j},0}^k(\xi\tens T):=\delta_{\s E_j}^k(\xi)_{13}(1_{B_k}\tens T\tens 1_{S_{kj}}),\quad \text{for all } \xi\in\s E_j \text{ and } T\in\K(\s H_{l'j},\s H_{lj}).
\]
$\bullet$ Let $\delta_{{\cal E}_{ll',j}}^k:{\cal E}_{ll',j}\rightarrow\Lin({\cal B}_{l',k,{\rm d}}\tens S_{kj},{\cal E}_{ll',k}\tens S_{kj})$ be the linear defined by
\[
\delta_{{\cal E}_{ll',j}}^k(\zeta):=(V_{kj}^l)_{23}\delta_{{\cal E}_{ll',j},0}^k(\zeta)(V_{kj}^{l'})_{23}^*, \quad \text{for all } \zeta\in{\cal E}_{ll',j}.
\]
$\bullet$ Let $j,k,l,l'=1,2$. We denote by $\Pi_{ll',j}^k:\Lin({\cal B}_{l',k,{\rm d}}\tens S_{kj},{\cal E}_{ll',k})\rightarrow\Lin(D_{{\rm d},k}\tens S_{kj},\s D_k\tens S_{kj})$ the linear extension of the inclusion map ${\cal E}_{ll',k}\tens S_{kj}\rightarrow\s D_k\tens S_{kj}$. For $\zeta\in\s D_j$, let us denote by $\zeta_{ll'}$ the element of ${\cal E}_{ll',j}$ defined by $\zeta_{ll'}:=(q_{\s E,j}\tens p_{lj})\zeta(q_{B,j}\tens p_{l'j})$. For all $j,k=1,2$, we have 
\[
\delta_{\s D_j}^k(\zeta)=\sum_{l,l'=1,2}\Pi_{ll',j}^k\circ\delta_{{\cal E}_{ll',j}}^k(\zeta_{ll'}).
\]
For all $l,l',j=1,2$, the pair $({\cal E}_{ll',j},\delta_{{\cal E}_{ll',j}}^j)$ is a $\QG_j$-equivariant Hilbert ${\cal B}_{l,j,{\rm g}}$-${\cal B}_{l',j,{\rm d}}$-bimodule.

\section{Equivariant Kasparov theory}

In this chapter, we fix a regular measured quantum groupoid $\cal G$ on the finite-dimensional basis $N=\bigoplus_{1\leqslant l\leqslant k}{\rm M}_{n_l}(\GC)$ endowed with the non-normalized Markov trace. We use all the notations introduced in \S \ref{MQGfinitebasis} and \S \ref{WHC*A} concerning the objects associated with $\cal G$.

\subsection{Equivariant Kasparov groups}

Let us recall a definition.

\begin{defin}\label{defkaspbimod}(cf.\ \cite{Kas3}) Let $A$ and $B$ be $\Cstar$-algebras. A Kasparov $A$-$B$-bimodule is a triple $(\s E,\gamma,F)$ consisting of a countably generated Hilbert $B$-module $\s E$, a *-homomorphism $\gamma:A\rightarrow\Lin(\s E)$ and an operator $F\in\Lin(\s E)$ such that
\begin{equation}\label{eq20}
[\gamma(a),\, F]\in\K(\s E), \;\; \gamma(a)(F^2-1)\in\K(\s E) \;\; \text{and} \;\; \gamma(a)(F-F^*)\in\K(\s E) \;\; \text{for all } a\in A.
\end{equation}
If we have 
\begin{equation}
[\gamma(a),\, F]=\gamma(a)(F^2-1)=\gamma(a)(F-F^*)=0 \quad \text{for all }  a\in A,
\end{equation}
then the Kasparov $A$-$B$-bimodule $(\s E,\gamma,F)$ is said to be degenerate. 
\end{defin}

Let $A$ and $B$ be ${\cal G}$-$\Cstar$-algebras.

\begin{defin}\label{defEqkaspbimod} A $\cal G$-equivariant Kasparov $A$-$B$-bimodule is a triple $(\s E,\gamma,F)$, consisting of $\cal G$-equivariant $A$-$B$-bimodule $(\s E,\gamma)$ (cf.\ \ref{defbimod}) and an operator $F\in\Lin(\s E)$ such that:
\begin{enumerate}
\item the triple $(\s E,\gamma,F)$ is a Kasparov $A$-$B$-bimodule;
\item $[F,\beta_{\s E}(n^{\rm o})]=0$, for all $n\in N$;
\item $(\gamma\tens\id_S)(x)(\delta_{\K(\s E)}(F)-q_{\beta_{\s E}\alpha}(F\tens 1_S))\in\K(\s E\tens S)$, for all $x\in A\tens S$.
\end{enumerate}
The $\cal G$-equivariant Kasparov $A$-$B$-bimodule will sometimes be simply denoted by $(\s E,F)$ when the representation $\gamma$ is clear from the context.
\end{defin}

\begin{rks}\label{rk13} Let us make some comments concerning the previous definition.
\begin{enumerate}
\item Since $\beta_{\s E}(n^{\rm o})\gamma(a)=\gamma(\beta_A(n^{\rm o})a)$ for all $a\in A$ and $n\in N$, we have 
\[
(\gamma\tens\id_S)(xq_{\beta_A\alpha})=(\gamma\tens\id_S)(x)q_{\beta_{\s E}\alpha}=(\gamma\tens\id_S)(x)\s V\s V^* \; \text{ for all } \; x\in A\tens S,
\]
where $\s V\in\Lin(\s E\tens_{\delta_A}(A\tens S),\s E\tens S)$ is the isometry defined in \ref{prop27} a). The following statements are then equivalent to condition 3:
\begin{itemize}
\item $(\gamma\tens\id_S)(x)(\delta_{\K(\s E)}(F) - F\tens 1_S )\in\K(\s E\tens S)$, for all $x\in(A\tens S)q_{\beta_A\alpha}$;
\item $(\gamma\tens\id_S)(x)(\s V(F\tens_{\delta_B}1)\s V^*-q_{\beta_{\s E}\alpha}(F\tens 1_S))\in\K(\s E\tens S)$, for all $x\in A\tens S$;
\item $(\gamma\tens\id_S)(x)(\s V(F\tens_{\delta_B}1)\s V^*-F\tens 1_S)\in\K(\s E\tens S)$, for all $x\in (A\tens S)q_{\beta_A\alpha}$.
\end{itemize}
Note that it follows from condition 2 that $[F\tens 1_S,\, q_{\beta_{\s E}\alpha}]=0$.
\item As in Remarque 3.4 (2) \cite{BS1}, we prove that
\[
(\s V(F\tens_{\delta_B} 1)\s V^*-F\tens 1_S)(\gamma\tens\id_S)(x)\in\K(\s E\tens S),\quad \text{for all } x\in q_{\beta_A\alpha}(A\tens S).
\]
Since $(\gamma\tens\id_S)(q_{\beta_A\alpha} x)=q_{\beta_{\s E}\alpha}(\gamma\tens\id_S)(x)$, for all $x\in A\tens S$ and $[F\tens 1_S,\, q_{\beta_{\s E}\alpha}]=0$, this is also equivalent to:
\[
(\s V(F\tens_{\delta_B} 1)\s V^*-q_{\beta_{\s E}\alpha}(F\tens 1_S))(\gamma\tens\id_S)(x)\in\K(\s E\tens S),\quad \text{for all } x\in A\tens S.
\]
Note that the converse is also true, {\it i.e.}\ condition 3 is equivalent to these assertions.
\item If $F$ is invariant (cf.\ \ref{propdef7}), then conditions 2 and 3 of \ref{defEqkaspbimod} are satisfied. \qedhere
\end{enumerate}
\end{rks}

\begin{defin}\label{def3}
\begin{enumerate}
\item Two $\cal G$-equivariant Kasparov $A$-$B$-bimodules $(\s E_1,\gamma_1,F_1)$ and $(\s E_2,\gamma_2,F_2)$ are said to be unitarily equivalent if there exists a unitary $\Phi\in\Lin(\s E_1,\s E_2)$ such that:
\begin{enumerate}[label=(\roman*)]
\item $\Phi$ is $\cal G$-equivariant (cf.\ \ref{def1});
\item $\Phi\circ\gamma_1(a)=\gamma_2(a)\circ \Phi$, for all $a\in A$;
\item $F_2\circ \Phi=\Phi\circ F_1$;
\end{enumerate}
\item Let ${\sf E}_{\cal G}(A,B)$ be the set of unitary equivalence classes of $\cal G$-equivariant Kasparov $A$-$B$-bimodules.
\item A $\cal G$-equivariant Kasparov $A$-$B$-bimodules $(\s E,\gamma,F)$ is said to be degenerate, if $(\s E,\gamma,F)$ is a degenerate Kasparov $A$-$B$-bimodule such that 
\[
(\gamma\tens\id_S)(x)(\delta_{\K(\s E)}(F)-F\tens 1_S)=0, \quad \text{for all } x\in A\tens S.
\] 
We denote by ${\sf D}_{\cal G}(A,B)\subset {\sf E}_{\cal G}(A,B)$ the set of unitary equivalence classes of degenerate $\cal G$-equivariant Kasparov $A$-$B$-bimodules.
\item Let us consider the $\Cstar$-algebra $B[0,1]:=\rmc([0,1],B)$ of the $B$-valued continuous functions on $[0,1]$. We make the identification $B[0,1]=\rmc([0,1])\tens B$ in the following way:
\[
(f\tens b)(t):= f(t)b, \; \text{ for all } f\in\rmc([0,1]),\; b\in B \; \text{ and }\; t\in[0,1].
\]
Similarly, we make the identification $B[0,1]\tens S=(B\tens S)[0,1]$. In particular, we will identify $\M(B[0,1]\tens S)=\M(B\tens S)[0,1]$ and $\M(B[0,1])=\M(B)[0,1]$. Let $\delta_{B[0,1]}:B[0,1]\rightarrow\M(B[0,1]\tens S)$ and $\beta_{B[0,1]}:N^{\rm o}\rightarrow\M(B[0,1])$ be the maps defined by $\delta_{B[0,1]}(f)(t):=\delta_B(f(t))$ and $[\beta_{B[0,1]}(n^{\rm o})f](t)=\beta_B(n^{\rm o})f(t)$ for all $f\in B[0,1]$, $n\in N$ and $t\in[0,1]$.
Then, the pair $(\beta_{B[0,1]},\delta_{B[0,1]})$ is a continuous action of $\cal G$ on $B[0,1]$. For $t\in[0,1]$, let ${\rm e}_t:B[0,1]\rightarrow B$ be the evaluation at point $t$, {\it i.e.\ }the surjective *-homomorphism defined for all $f\in B[0,1]$ by ${\rm e}_t(f):=f(t)$. Note that ${\rm e}_t$ is $\cal G$-equivariant by definition of the action of $\cal G$ on $B[0,1]$.
\item Let $\s E$ be a $\cal G$-equivariant Hilbert $B$-module. Let us consider the Hilbert $B[0,1]$-module $\s E[0,1]:={\rm C}([0,1],\s E)$ of $\s E$-valued continuous functions on $[0,1]$. We make the identification $\s E[0,1]={\rm C}[0,1]\tens \s E$ as above. We equip the Hilbert $B$-module $\s E[0,1]$ with the action of $\cal G$ obtained by transport of structure through the identification $\K(\s E[0,1]\oplus B[0,1])=\K(\s E\oplus B)[0,1]$. We have $\beta_{\s E[0,1]}=\beta_{\K(\s E)[0,1]}$ up to the identification $\Lin(\s E[0,1])=\M(\K(\s E)[0,1])$. For all $x\in B[0,1]\tens S$ and $\xi\in\s E[0,1]\tens S$, we have $(\delta_{\s E[0,1]}(\xi)x)(t)=\delta_{\s E}(\xi(t))x(t)$ up to the identifications $B[0,1]\tens S=(B\tens S)[0,1]$ and $\s E[0,1]\tens S=(\s E\tens S)[0,1]$.\qedhere
\end{enumerate}
\end{defin}

\begin{prop}\label{prop20} Let $A_1$, $A_2$, $A$, $B_1$, $B_2$ and $B$ be $\cal G$-$\Cstar$-algebras.
\begin{enumerate}
\item Let $f:A_1\rightarrow A_2$ be a $\cal G$-equivariant *-homomorphism. Let $(\s E,\gamma,F)$ be a $\cal G$-equivariant Kasparov $A_2$-$B$-bimodule. Let $\gamma^*:=\gamma\circ f:A_1\rightarrow\Lin(\s E)$. Then the triple $(\s E,\gamma^*,F)$ is a $\cal G$-equivariant Kasparov $A_1$-$B$-bimodule. Moreover, we have the following well-defined map 
\[
f^*:{\sf E}_{\cal G}(A_2,B)\rightarrow{\sf E}_{\cal G}(A_1,B) \; ; \; (\s E,\gamma,F) \mapsto (\s E,\gamma^*,F).
\]
\item Let $g:B_1\rightarrow B_2$ be a $\cal G$-equivariant *-homomorphism. Let $(\s E,\gamma,F)$ be a $\cal G$-equivariant Kasparov $A$-$B_1$-bimodule. Let $\gamma_*:A\rightarrow\Lin(\s E\tens_g B_2)$ be the *-representation defined by $\gamma_*(a):=\gamma(a)\tens_g 1_{B_2}$ for all $a\in A$. Then the triple $(\s E\tens_g B_2,\gamma_*, F\tens_g 1_{B_2})$ is a $\cal G$-equivariant Kasparov $A$-$B_1$-bimodule. Moreover, we have the following well-defined map 
\[
g_*:{\sf E}_{\cal G}(A,B_1)\rightarrow{\sf E}_{\cal G}(A,B_2)\; ; \; (\s E,\gamma,F) \mapsto (\s E\tens_g B_2,\gamma_*,F\tens_g 1_{B_2}).\qedhere
\] 
\end{enumerate}
\end{prop}

\begin{proof}
Straightforward verifications.
\end{proof}

\begin{defin}\label{def6}
Let $(\s E_0,F_0),(\s E_1,F_1)\in{\sf E}_{\cal G}(A,B)$. An homotopy between $(\s E_0,F_0)$ and $(\s E_1,F_1)$ is an element $x\in{\sf E}_{\cal G}(A,B[0,1])$ such that ${\rm e}_{0\ast}(x)=(\s E_0,F_0)$ and ${\rm e}_{1\ast}(x)=(\s E_1,F_1)$. In that case, we say that $(\s E_0,F_0)$ and $(\s E_1,F_1)$ are homotopic. The homotopy relation is an equivalence relation on the class ${\sf E}_{\cal G}(A,B)$. We denote by ${\sf KK}_{\cal G}(A,B)$ the quotient set of ${\sf E}_{\cal G}(A,B)$ by the homotopy relation. We also denote by $[(\s E,F)]$ the class of $(\s E,F)\in{\sf E}_{\cal G}(A,B)$ in ${\sf KK}_{\cal G}(A,B)$.
\end{defin}

\begin{exs}\label{ex3}(cf.\ \cite{Kas3,Skand,BS1})
We can build homotopies of ${\sf E}_{\cal G}(A,B[0,1])$ in the following ways. 
\begin{enumerate}
\item An operator homotopy is a triple $(\s E,\gamma,(F_t)_{t\in[0,1]})$ consisting of a $\cal G$-equivariant Hilbert $A$-$B$-bimodule $(\s E,\gamma)$ and a family of adjointable operators $(F_t)_{t\in[0,1]}$ on $\s E$ such that:
\begin{itemize}
\item the triple $(\s E,\gamma,F_t)$ is a $\cal G$-equivariant Kasparov $A$-$B$-bimodule for all $t\in[0,1]$;
\item the map $(t\mapsto F_t)$ is norm continuous.
\end{itemize}
The family of operators $(F_t)_{t\in[0,1]}$ defines an element $F$ of $\Lin(\s E[0,1])$ (up to the identification $\Lin(\s E)[0,1]=\Lin(\s E[0,1])$, cf.\ \ref{def3} 4, 5) and the triple $(\s E[0,1],\gamma\tens 1,F)$ is a homotopy between $(\s E,\gamma,F_0)$ and $(\s E,\gamma,F_1)$. 
\item An important example of operator homotopy can be obtained in the following case. Let $(\s E,\gamma,F)$ be a $\cal G$-equivariant Kasparov $A$-$B$-bimodule. We call an operator $G\in\Lin(\s E)$ a compact perturbation of $F$ if for all $a\in A$ we have $\gamma(a)(F-G)\in\K(\s E)$ and $(F-G)\gamma(a)\in\K(\s E)$.
In that case, the triple $(\s E,\gamma,G)$ is a $\cal G$-equivariant Kasparov $A$-$B$-bimodule. Moreover, the triples $(\s E,\gamma,F)$ and $(\s E,\gamma,G)$ are operator homotopic via the obvious continuous path defined by $F_t:=(1-t)F+tG$ for $t\in[0,1]$.
\item Let $(\s E,(\gamma_t)_{t\in[0,1]},F)$ be a triple where $\s E$ is a $\cal G$-equivariant Hilbert $B$-module, $(\gamma_t)_{t\in[0,1]}$ is a family of $\cal G$-equivariant *-representations of $A$ on $\s E$ and $F\in\Lin(\s E)$ such that the triple $(\s E,\gamma_t,F)$ is a $\cal G$-equivariant Kasparov $A$-$B$-bimodule for all $t\in[0,1]$ and the map $(t\mapsto\gamma_t(a))$ is norm continuous for all $a\in A$. Up to the identification $\Lin(\s E)[0,1]=\Lin(\s E[0,1])$, the family $(\gamma_t)_{t\in[0,1]}$ defines a $\cal G$-equivariant *-representation $\gamma:A\rightarrow\Lin(\s E[0,1])$. Moreover, the triple $(\s E[0,1],\gamma,1\tens F)$ is a homotopy between $(\s E,\gamma_0,F)$ and $(\s E,\gamma_1,F)$.\qedhere
\end{enumerate}
\end{exs}

As for actions of quantum groups (cf.\ 3.3 \cite{BS1}), we have:

\begin{prop}
Endowed with the binary operation induced by the direct sum operation
$
([(\s E_1,F_1)],[(\s E_2,F_2)]) \mapsto [(\s E_1\oplus\s E_2,F_1\oplus F_2)],
$
the quotient set ${\sf KK}_{\cal G}(A,B)$ is an abelian group and the class of the identity element $0\in{\sf KK}_{\cal G}(A,B)$ is represented by any element of ${\sf D}_{\cal G}(A,B)$.
\end{prop}

It follows from \ref{def6} that the maps defined in \ref{prop20} factorize over the quotient maps so that we obtain homomorphisms of abelian groups $f^*:\kk_{\cal G}(A_2,B)\rightarrow\kk_{\cal G}(A_1,B)$ and $g_*:\kk_{\cal G}(A,B_1)\rightarrow\kk_{\cal G}(A,B_2)$.

\subsection{Kasparov's technical theorem}

\begin{nb} 
Let $A$ be a $\Cstar$-algebra. We denote by
\[
\der(A):=\{d\in\B(A)\,;\,\forall x,y\in A,\,d(xy)=d(x)y+xd(y)\}
\]
the Banach subspace of $\B(A)$ consisting of the continuous derivations of $A$. Any $d\in\der(A)$ extends uniquely to a strictly continuous linear map $d:\M(A)\rightarrow\M(A)$ defined by:
\[
d(m)a:=d(ma)-md(a) \; \text{ and } \; ad(m):=d(am)-d(a)m, \quad \text{for all } m\in\M(A),\, a\in A.
\]
Note that $d(1)=0$ and $d\in{\rm Der}(\M(A))$. For all $x\in A$, we have ${\rm Ad}(x)\in\der(A)$ (inner derivations of $A$).
\end{nb}

\begin{rks}\label{rk12}
\begin{enumerate}
\item If $m\in\M(A)$, we define an inner derivation ${\rm Ad}(m)\in\der(A)$ by setting ${\rm Ad}(m)(x):=[m,\,x]=mx-xm\in A$ for all $x\in A$.
\item Let $J$ be a closed two-sided ideal of $A$ and $d\in\der(A)$. We have $d(J)\subset J$. Indeed, let $(u_\lambda)$ be an approximate unit of $J$ and $x\in J$. We have $d(x)=\lim_\lambda d(xu_\lambda)$ with respect to the norm topology since $d$ is norm continuous. Moreover, for all $\lambda$ we have $d(xu_\lambda)=d(x)u_\lambda+xd(u_\lambda)\in J$ since $J$ is an ideal of $A$. Hence, $d(x)\in J$. In particular, $d$ induces an element of $\der(J)$ by restriction and we have a continuous linear map $\der(A)\rightarrow\der(J)$.\qedhere
\end{enumerate}
\end{rks}

\begin{lem}\label{lem6} (cf.\ 4.1 \cite{BS1}) Let $J$ be a $\Cstar$-algebra and $(\beta_J,\delta_J)$ 
an action of $\cal G$ on $J$. For all $h\in J$, $k\in S$, $z\in J\tens S$, for all compact $K\subset\der(J)$ and for all real number $\varepsilon>0$, there exists $u\in J$, such that $0\leqslant u\leqslant 1$ and $[u,\, \beta_J(n^{\rm o})]=0$ for all $n\in N$, which further satisfies:
\begin{enumerate}[label={\rm(\alph*)}]
\item $\|uh-h\|\leqslant\varepsilon$;
\item for all $d\in K$, $\|d(u)\|<\varepsilon$;
\item $\|(\delta_J(x)-q_{\beta_J\alpha}(x\tens 1))(1\tens k)\|<\varepsilon$;
\item $\|\delta_J(u)z-q_{\beta_J\alpha}z\|<\varepsilon$.\qedhere
\end{enumerate} 
\end{lem}

\begin{proof} Denote by $J(K)$ the C*-algebra consisting of the $J$-valued continuous functions on $K$. We adapt the proof of 4.1 \cite{BS1} by considering the affine map
\begin{align*}
\Phi:J &\rightarrow A:=J\oplus J(K)\oplus (J\tens S)\oplus (J\tens S)\\
x &\mapsto\Phi(x):= (xh-h,\,[d\mapsto d(x)],\,(\delta_J(x)-q_{\beta_J\alpha}(x\tens 1))(1\tens k),\,\delta_J(x)z-q_{\beta_J\alpha}z),
\end{align*}
which admits a unique strictly continuous extension, still denoted by $\Phi:\M(J)\rightarrow\M(A)$, such that $\Phi(1)=0$. By applying the Hahn-Banach theorem, we then conclude as in \cite{BS1} that $0$ is an adherent point of $\Phi(C)$, where $C$ is the nonempty closed convex subset $\{u\in J\,;\,0\leqslant u\leqslant 1,\, \forall n\in N,\, [u,\, \beta_J(n^{\rm o})]=0\}$ of $J$, which is just a restatement of the above lemma.
\end{proof}

\begin{defin}\label{def9} Let $A$ be a $\Cstar$-algebra endowed with an action $(\beta_A,\delta_A)$ of ${\cal G}$ such that $\delta_A(A)\subset\widetilde{\M}(A\tens S)$. A closed two-sided ideal $J$ of $A$ is said to be invariant by $(\beta_A,\delta_A)$ (or $(\beta_A,\delta_A)$-invariant) if
$\beta_A(N^{\rm o})\subset\M(A;J)$ and $\delta_A(J)\subset\M(\widetilde{A}\tens S; J\tens S)$. We denote by $\beta_J:N^{\rm o}\rightarrow\M(J)$ and $\delta_J:J\rightarrow\widetilde{\M}(J\tens S)$ the maps obtained by restricting $\beta_A$ and $\delta_A$. Then, we obtain an action of $\cal G$ on $J$.
\end{defin}

Before stating the generalization of Kasparov's technical theorem, we first state an easy lemma that will be used several times.

\begin{lem}\label{lem7}
Let $B$ be a $\Cstar$-algebra and $J$ a closed two-sided ideal of $B$. Let $x\in B$ and $b\in B_+$ such that $[x,\, b]\in J$. Then, we have $[x,\, b^{1/2}]\in J$. In particular, if $b$ commutes with $x$ so does $b^{1/2}$.
\end{lem}

\begin{proof}
Let us denote $A:=\{b\in B\,;\, [b,\,x]\in J\}$. It is easily seen that $A$ is a closed subalgebra of $B$. If $b\in B$ is a self-adjoint element commuting with $x$, then $b$ is an element of the $\Cstar$-subalgebra $A\cap A^*$ of $B$. In particular, if $b$ is a positive element of $B$ which also belongs to $A$, then so does $b^{1/2}$.
\end{proof}

\begin{thm}\label{Kaspthm} (cf. Th\'eor\`eme 4.3 \cite{BS1}) Let us consider:
\begin{itemize}
\item $J_1$ a $\sigma$-unital $\Cstar$-algebra endowed with an action $(\beta_{J_1},\delta_{J_1})$ of $\cal G$ such that $\delta_{J_1}(J_1)\subset\widetilde{\M}(J_1\tens S)$;
\item $J$ a $(\beta_{J_1},\delta_{J_1})$-invariant ideal of $J_1$;
\item $J_2$ a $\sigma$-unital subalgebra of $\M(J_1;J)$;
\item $\s F$ a separable subspace of $\der(J_1)$;
\item $J_2'$ a $\sigma$-unital subalgebra of $\M(J\tens S)$ such that $\delta_{J_1}(x)y\in J\tens S$ for all $x\in J_1$ and $y\in J_2'$.
\end{itemize}
There exists $M\in\M(J_1;J)$ such that $0\leqslant M\leqslant 1$ and $[M,\,\beta_{J_1}(n^{\rm o})]=0$ for all $n\in N$, which satisfies the following statements:
\begin{itemize}
\item $(1-M)J_2\subset J$;
\item for all $d\in\s F$, $d(M)\in J$;
\item $\delta_{J_1}(M)-q_{\beta_{J_1}\alpha}(M\tens 1_S)\in\widetilde{\M}(J\tens S)$;
\item $(q_{\beta_{J_1}\alpha}-\delta_{J_1}(M))J_2'\subset J\tens S$.\qedhere
\end{itemize}
\end{thm}

\begin{proof}
In essence, the proof is that of Th\'eor\`eme 4.3 \cite{BS1}. We denote $q:= q_{\beta_{J_1}\alpha}$ for short. Let $h_1\in J_1$, $h_2\in J_2$, $h_2'\in J_2'$ and $k\in S$ be strictly positive elements (cf.\ \ref{rk18}). Let $K$ be a compact of $\s F$ such that $\s F=[K]$.
By \ref{lem6}, there exists an increasing sequence $(u_l)_{l\in\GN}$ of elements of $J_1$ with $u_0=0$, which satisfies for all integer $l\geqslant 1$ the following statements:
\begin{enumerate}[label=(\roman*)]
\item $0\leqslant u_l\leqslant 1$; $[u_l,\,\beta_{J_1}(n^{\rm o})]=0$, for all $n\in N$;
\item $\|u_lh_1-h_1\|\leqslant 2^{-l}$;
\item for all $d\in K$, $\|d(u_l)\|\leqslant 2^{-l}$;
\item $\|(\delta_{J_1}(u_l)-q(u_l\tens 1))(1\tens k)\|\leqslant 2^{-l}$.
\end{enumerate}
Let us recall that any derivation of $J_1$ induces a derivation of $J$ by restriction (cf.\ \ref{rk12} 2). It follows from \ref{lem6} that there exists a sequence $(v_l)_{l\in\GN^*}$ of elements of $J$ such that $0\leqslant v_l\leqslant 1$ and $[v_l,\,\beta_J(n^{\rm o})]=0$ for all $n\in N$ and $l\in\GN^*$, which satisfies for all integer $l\geqslant 1$ the following statements:
\begin{enumerate}[label=(\alph*)]
\item $\|v_l(u_l-u_{l-1})^{1/2}h_2-(u_l-u_{l-1})^{1/2}h_2\|\leqslant 2^{-l}$;
\item for all $d\in K$, $\|d(v_l)\|\leqslant 2^{-l}$ and $\|[(u_l-u_{l-1})^{1/2},v_l]\|\leqslant 2^{-l}$;
\item $\|(\delta_{J_1}(v_l)-q(v_l\tens 1))(1\tens k)\|\leqslant 2^{-l}$;
\item $\|(\delta_{J_1}(v_l)-q)\delta_{J_1}(u_l-u_{l-1})^{1/2}h_2'\|\leqslant 2^{-l}$.
\end{enumerate}
More precisely, for each fixed integer $l\geqslant 1$, we have applied Lemma \ref{lem6} with $h:=(u_l-u_{l-1})^{1/2}h_2\in J$, $z:=\delta_{J_1}(u_l-u_{l-1})^{1/2}h_2'\in J\tens S$, $\varepsilon:=2^{-l}$ and the compact subset of $\der(J)$ consists of the derivation ${\rm Ad}((u_l-u_{l-1})^{1/2})$ and the image of the compact subset $K\subset\der(J_1)$ by the continuous map $\der(J_1)\rightarrow\der(J)$.
For all integer $l\geqslant 1$, let us denote:
\begin{itemize}
\item $\displaystyle M_l:=\sum_{i=1}^l (u_i-u_{i-1})^{1/2}v_i(u_i-u_{i-1})^{1/2}$;
\item $\displaystyle M_l':=\sum_{i=1}^l v_i(u_i-u_{i-1})$;
\item $\displaystyle N_l:=\sum_{i=1}^l (u_i-u_{i-1})^{1/2}(1-v_i)(u_i-u_{i-1})^{1/2}$.
\end{itemize}
Let us notice that:
\begin{enumerate}[label=\Alph*)]
\item for all $l\geqslant 1$, we have $M_l\in J$, $M_l'\in J$ and $N_l\in J_1$;
\item for all $l\in\GN^*$, we have
\[
 M_l-M_l'=\sum_{i=1}^l[(u_i-u_{i-1})^{1/2},\, v_i](u_i-u_{i-1})^{1/2};
\] 
for all $l\in\GN^*$ and $1\leqslant i\leqslant l$, we have $\|[(u_i-u_{i-1})^{1/2},\, v_i](u_i-u_{i-1})^{1/2}\|\leqslant 2^{-(i-1)}$ by (b); hence, the sequence $(M_l-M_l')_{l\geqslant 1}$ is norm convergent;
\item $M_l+N_l=\displaystyle \sum_{i=1}^l (u_i-u_{i-1})=u_l\underset{l\rightarrow\infty}{\longrightarrow} 1$ in $\M(J_1)$ with respect to the strict topology.
\end{enumerate}
Let us prove that $(M_l)_{l\geqslant 1}$ converges strictly in $\M(J_1)$. Since $h_1$ is strictly positive, it suffices to prove that $(M_lh_1)_{l\geqslant 1}$ and $(h_1M_l)_{l\geqslant 1}$ are norm  convergent in $J_1$. For all integer $l\geqslant 1$, we have 
\[
M_l'h_1=\sum_{i=1}^lv_i(u_ih_1-h_1)-\sum_{i=1}^l v_i(u_{i-1}h_1-h_1),
\] 
with $\|v_i(u_ih_1-h_1)\|\leqslant 2^{-i}$ and $\|v_i(u_{i-1}h_1-h_1)\|\leqslant 2^{-(i-1)}$ for all $1\leqslant i\leqslant l$. Hence, $(M_l'h_1)_{l\geqslant 1}$ is norm convergent. The norm convergence of $(M_lh_1)_{l\geqslant 1}$ follows from B). Since $h_1M_l=(M_lh_1)^*$, the norm convergence of $(h_1M_l)_{l\geqslant 1}$ is proved.
Let $M\in\M(J_1)$ (resp.\ $M'\in\M(J_1)$) be the strict limit of $(M_l)_{l\geqslant 1}$ (resp.\ $(M_l')_{l\geqslant 1}$). We have $M,M'\in\M(J_1;J)$ by A) and $M-M'\in J$ by B). Since $M_l$ (resp.\ $N_l$) is positive for all integer $l\geqslant 1$, so is $M$ (resp.\ $1-M$). Hence, $0\leqslant M\leqslant 1$. Let $n\in N$. Since $[\beta_{J_1}(n^{\rm o}),\, u_l]=0$ and $[\beta_{J_1}(n^{\rm o}),\, v_l]=0$ for all integer $l\geqslant 1$, we have $[\beta_{J_1}(n^{\rm o}),\, M_l]=0$ (cf.\ \ref{lem7}). Hence, $[M,\, \beta_{J_1}(n^{\rm o})]=0$ for all $n\in N$. In particular, we have $[M\tens 1_S,\, q]=0$.\hfill\break
For all $d\in K$, the sequence $(d(M_l'))_{l\geqslant 1}$ is norm convergent in $\M(J_1)$. Indeed, we have 
\begin{align*}
d(M_l')&=\sum_{i=1}^l d(v_i)(u_i-u_{i-1})+\sum_{i=1}^l v_i(d(u_i)-d(u_{i-1}))\\
&=\sum_{i=1}^l d(v_i)(u_i-u_{i-1})+\sum_{i=1}^l(v_i-v_{i-1})d(u_i)
\end{align*}
(recall that $u_0=0$), which is norm convergent by (iii) and (b). Since $d$ is strictly continuous, the norm limit of $(d(M_l'))_{l\geqslant 1}$ is $d(M')$. It follows from A) and \ref{rk12} 2 that $d(M_l')\in J$ for all $l\geqslant 1$. Hence, $d(M')\in J$.
Since $M-M'\in J$ and $d(M')\in J$, it then follows that $d(M)\in J$ for all $d\in K$. Hence, $d(M)\in J$ for all $d\in\s F$.
Let us prove that 
\begin{equation}\label{conv}
(\delta_{J_1}(M_l')-q(M_l'\tens 1_S))(1_{J_1}\tens k)\underset{l\rightarrow\infty}{\longrightarrow}(\delta_{J_1}(M')-q(M'\tens 1_S))(1_{J_1}\tens k)
\end{equation}
with respect to the norm topology. It suffices to see that $((\delta_{J_1}(M_l')-q(M_l'\tens 1_S))(1_{J_1}\tens k))_{l\geqslant 1}$ is norm convergent since this sequence is already convergent with respect to the strict topology towards $(\delta_{J_1}(M')-q(M'\tens 1_S))(1_{J_1}\tens k)$.
For all integer $l\geqslant 1$, we have
\begin{multline*}
\delta_{J_1}(M_l')-q(M_l'\tens 1_S)=\sum_{i=1}^l \delta_{J_1}(v_i)(\delta_{J_1}(u_i)-\delta_{J_1}(u_{i-1}))-q\sum_{i=1}^l v_i(u_i-u_{i-1})\tens 1_S\\
=\sum_{i=1}^l(\delta_{J_1}(v_i)\delta_{J_1}(u_i)-q(v_iu_i\tens 1_S))-\sum_{i=1}^l(\delta_{J_1}(v_i)\delta_{J_1}(u_{i-1})-q(v_iu_{i-1}\tens 1_S)).
\end{multline*}
We have $q(v_iu_i\tens 1_S)=q(v_i\tens 1_S)q(u_i\tens 1_S)$ for all integer $i\geqslant 1$. Hence, 
\[
\delta_{J_1}(v_i)\delta_{J_1}(u_i)-q(v_iu_i\tens 1_S)\!=\!\delta_{J_1}(v_i)(\delta_{J_1}(u_i)-q(u_i\tens 1_S))+(\delta_{J_1}(v_i)-q(v_i\tens 1_S))q(u_i\tens 1_S).
\]
Hence, $\sum_{i}\delta_{J_i}(v_i)(\delta_{J_1}(u_i)-q(u_i\tens 1_S))(1_{J_1}\tens k)$ and $\sum_{i}(\delta_{J_1}(v_i)-q(v_i\tens 1_S))q(u_i\tens 1)(1_{J_1}\tens k)$ are convergent by application of (iv) and (c) (and the fact that $\|u_l\|\leqslant 1$ and $\|v_l\|\leqslant 1$ for all integer $l\geqslant 1$); hence, so is $\sum_i(\delta_{J_1}(v_i)\delta_{J_1}(u_i)-q(v_iu_i\tens 1))(1_S\tens k)$. We prove that the series $\sum_i(\delta_{J_1}(v_i)\delta_{J_1}(u_{i-1})-q(v_iu_{i-1}\tens 1))(1_S\tens k)$ is convergent in a similar way and (\ref{conv}) is proved.\hfill\break
Since $k$ is strictly positive, the sequence $((\delta_{J_1}(M_l')-q(M_l'\tens 1_S))(1_{J_1}\tens s))_{l\geqslant 1}$ is norm convergent towards $(\delta_{J_1}(M')-q(M'\tens 1_S))(1_{J_1}\tens s)$ for all $s\in S$. However, since for all integer $l\geqslant 1$ and $s\in S$ we have $(\delta_{J_1}(M_n')-q(M_n'\tens 1))(1_{J_1}\tens s)\in J\tens S$ ($M_l'\in J$ and $J$ is invariant), it follows that $(\delta_{J_1}(M')-q(M'\tens 1_S))(1_{J_1}\tens s)\in J\tens S$ for all $s\in S$. Hence, $\delta_{J_1}(M')-q(M'\tens 1)\in\widetilde{\M}(J\tens S)$ since $M'$ is self-adjoint and $[M\tens 1_S,\, q]=0$. Moreover, we have $M=(M-M')+M'$ and $M-M'\in J$. Hence,
\[
\delta_{J_1}(M)-q(M\tens 1_S)\!=\!\delta_{J_1}(M-M')-q((M-M')\tens 1_S)+\delta_{J_1}(M')-q(M'\tens 1_S)\!\in\!\widetilde{\M}(J\tens S).
\]
By C), the sequence $(N_l)_{l\geqslant 1}$ converges strictly towards $1-M$. It follows from (a), the fact that $\|(1-v_i)(u_i-u_{i-1})^{1/2}h_2\|\leqslant 2^{-i}$ for all integer $i\geqslant 1$, (i) and the previous statement that the sequence $(N_lh_2)_{l\geqslant 1}$ converges in norm towards $(1-M)h_2$. However, we have $h_2\in J_2$ and $N_l\in J_1$ for all integer $l\geqslant 1$. Hence, $N_lh_2\in J$ for all integer $l\geqslant 1$. We then have $(1-M)h_2\in J$. Hence, $(1-M)J_2\subset J$ since $h_2$ is strictly positive.\hfill\break
By combining (d) with the fact that $\|\delta_{J_1}(1-v_i)\delta_{J_1}(u_i-u_{i-1})^{1/2}h_2'\|\leqslant 2^{-i}$ for all integer $i\geqslant 1$, we prove in a similar way that the sequence $(\delta_{J_1}(N_l)h_2')_{l\geqslant 1}$ converges in norm towards $\delta_{J_1}(1-M)h_2'$ and we prove that $\delta_{J_1}(1-M)J_2'\subset J\tens S$.
\end{proof}

\subsection{Kasparov's product}

In this paragraph, we define the Kasparov product in the equivariant framework for actions of measured quantum groupoids on a finite basis.

\medbreak

Let $C$ and $B$ be two $\cal G$-$\Cstar$-algebras. Let $\s E_1$ and $\s E_2$ be $\cal G$-equivariant Hilbert $\Cstar$-modules over $C$ and $B$ respectively. Let $\gamma_2:C\rightarrow\Lin(\s E_2)$ be a $\cal G$-equivariant *-representation. Let us also consider the $\cal G$-equivariant Hilbert $B$-module $\s E:=\s E_1\tens_{\gamma_2}\s E_2$ (cf. \ref{prop18}). For $\xi\in\s E_1$, we denote by $T_{\xi}\in\Lin(\s E_2,\s E)$ the operator defined by $T_{\xi}(\eta):=\xi\tens_{\gamma_2}\eta$ for all $\eta\in\s E_2$.

\medbreak

Let us recall the notion of connection.

\begin{defin} (cf.\ Definition A.1 \cite{CS} and Definition 8 \cite{Skand}) Let $F_2\in\Lin(\s E_2)$. We say that $F\in\Lin(\s E)$ is an $F_2$-connection for $\s E_1$ if for all $\xi\in\s E_1$, we have $T_{\xi}F_2-FT_{\xi}\in\K(\s E_2,\s E)$ and $F_2T_{\xi}^*-T_{\xi}^*F\in\K(\s E,\s E_2)$.
\end{defin}

In the lemmas below, we assume that the Hilbert $A$-module $\s E_1$ is countably generated.

\begin{lem}\label{lem8} (cf.\ Proposition A.2 a) \cite{CS}) Let $F_2\in\Lin(\s E_2)$ such that $[F_2,\,\gamma_2(a)]\in\K(\s E_2)$ for all $a\in A$. Then there exist $F_2$-connections $F$ for $\s E_1$ such that $[F,\,\beta_{\s E}(n^{\rm o})]=0$ for all $n\in N$.
\end{lem}

\begin{proof}
By Kasparov's stabilization theorem (cf.\ Theorem 2 \cite{Kas1}), we can assume that $\s E_1$ is a submodule of ${\s H}_{\widetilde C}=\s H\tens_{\GC}\widetilde C$ and $\s E_1=P({\s H}_{\widetilde C})$, where $P\in\Lin(\s H_{\widetilde C})$ is a projection. Let
\[
F:=(P\tens_{\gamma_2} 1_{\s E_2})(1_{{\s H}_{\tilde C}}\tens_{\GC} F_2)(P\tens_{\gamma_2} 1_{\s E_2})
\]
be the Grassmann connection (cf.\ A.2 a) \cite{CS}). But, since $\beta_{\s E}(n^{\rm o})=\beta_{\s E_1}(n^{\rm o})\tens_{\gamma}1$ we have $[1_{{\s H}_{\widetilde C}}\tens_{\GC} F_2,\, \beta_{\s E}(n^{\rm o})]=0$ for all $n\in N$. Moreover, if $T\in\Lin(\s E_1)$, we have $PT\xi=T\xi=TP\xi$ for all $\xi\in\s E_1$. Hence, $[P\tens_{\gamma} 1_{\s E_2},\, \beta_{\s E}(n^{\rm o})]=0$ for all $n\in N$ and the result is proved.
\end{proof}

\begin{lem}\label{lem9} (cf.\ 5.1 \cite{BS1}) Let $F_2\in\Lin(\s E_2)$ such that $(\s E_2,\gamma_2,F_2)\in{\sf E}_{\cal G}(C,B)$. Let $F\in\Lin(\s E)$ be a $F_2$-connection for $\s E_1$ such that $[F,\beta_{\s E}(n^{\rm o})]=0$ for all $n\in N$ (cf.\ \ref{lem8}). Then, we have $(\s E,\gamma, F)\in{\sf E}_{\cal G}(\K(\s E_1),B)$, where the left action $\gamma:\K(\s E_1)\rightarrow\Lin(\s E)$ of $\K(\s E_1)$ on $\s E$ is defined by $\gamma(k):= k\tens_{\gamma_2}1$ for all $k\in\K(\s E_1)$. 
\end{lem}

In the following proof, we use all the notations of \ref{rk17}.

\begin{proof}
{\setlength{\baselineskip}{1.15\baselineskip}
The pair $(\s E,\gamma)$ is a $\cal G$-equivariant Hilbert $\K(\s E_1)$-$B$-bimodule (cf.\ \ref{prop18} and \ref{propleftaction} where $A:=\K(\s E_1)$ and $\gamma_1$ is the inclusion map $\K(\s E_1)\subset\Lin(\s E_1)$). By Proposition 9 (h) \cite{Skand}, it then remains to prove that $(\s V(F\tens_{\delta_B}1)\s V^*-q_{\beta_{\s E}\alpha}(F\tens 1_S))(\gamma\tens\id_S)(x)\in\K(\s E\tens S)$ for all $x\in\K(\s E_1)\tens S$. It suffices to prove that 
$
(\s V(F\tens_{\delta_B}1)\s V^*-q_{\beta_{\s E}\alpha}(F\tens 1_S))T_{\xi}\in\K(\s E_2\tens S,\s E\tens S)
$  
for all $\xi\in\s E_1\tens S$, where $T_{\xi}\in\Lin(\s E_2\tens S,\s E\tens S)$ is the operator defined for all $\eta\in\s E_2\tens S$ by $T_{\xi}(\eta):=\xi\tens_{\gamma_2\tens\id_S}\eta$ up to (\ref{eq30}).
Let $\xi\in\s E_1\tens S$ and $\xi':= q_{\beta_{\s E_1}\alpha}\xi$, we have $q_{\beta_{\s E}\alpha}T_{\xi}=T_{\xi'}$. Since $[F,\beta_{\s E}(n^{\rm o})]=0$ for all $n\in N$, we have $q_{\beta_{\s E}\alpha}(F\tens 1_S)T_{\xi}=(F\tens 1_S)T_{\xi'}$. Moreover, since $\s V^*=\s V^*q_{\beta_{\s E}\alpha}$, we have $\s V^*T_{\xi}=\s V^*T_{\xi'}$. Hence, $\s V(F\tens_{\delta_B}1)\s V^*T_{\xi}=\s V(F\tens_{\delta_B}1)\s V^*T_{\xi'}$. Thus, we have
$
(\s V(F\tens_{\delta_B}1)\s V^*-q_{\beta_{\s E}\alpha}(F\tens 1_S))T_{\xi}=(\s V(F\tens_{\delta_B}1)\s V^*-F\tens 1_S)T_{\xi'}.
$
Therefore, we have to prove that $(\s V(F\tens_{\delta_B}1)\s V^*-F\tens 1_S)T_{\xi}\in\K(\s E_2\tens S,\s E\tens S)$ for all $\xi\in q_{\beta_{\s E_1}\alpha}(\s E_1\tens S)$. Since $\{\delta_{\s E_1}(\xi_1)x\,;\,\xi_1\in\s E_1,\,x\in C\tens S\}$ is a total subset of $q_{\beta_{\s E_1}\alpha}(\s E_1\tens S)$, it suffices to consider the case where $\xi=\delta_{\s E_1}(\xi_1)x$ with $\xi_1\in\s E_1$ and $x\in q_{\beta_C\alpha}(C\tens S)$ (cf.\ \ref{hilbmodequ} 2 and \ref{rk2} 3). Let $y:=(\gamma_2\tens\id_S)(x)\in\Lin(\s E_2\tens S)$. We have $\xi=\s V_1T_{\xi_1}(x)$. Hence, we have $T_{\xi}=(\s V_1\tens_{\gamma_2\tens\id_S}1)T_{\xi_1}y$. Hence, $\s V_{\vphantom{2}}^*T_{\xi}=\widetilde{\s V}_2^*T_{\xi_1}y$. By a direct computation, we have $\widetilde{\s V}_2^*(\xi_1\tens_{(\gamma_2\tens\id_S)\delta_C}\s V_2\eta)=(T_{\xi_1}\tens_{\delta_B}1)\eta$ for all $\eta\in\s E_2\tens_{\delta_B}(B\tens S)$. Since $\s V_2^{\vphantom{*}}\s V_2^*=q_{\beta_{\s E_2}\alpha}$ and $\s V_2^*\s V_2^{\vphantom{*}}=1$, we have $\widetilde{\s V}_2^*T_{\xi_1}\eta=(T_{\xi_1}\tens_{\delta_B}1)\s V_2^*\eta$ for all $\eta\in q_{\beta_{\s E_2}\alpha}(\s E_2\tens S)$. In particular, we have $\widetilde{\s V}_2^*T_{\xi_1}y=(T_{\xi_1}\tens_{\delta_B}1)\s V_2^*y$ (indeed, we have $q_{\beta_{\s E_2}\alpha}y=y$ since $x\in q_{\beta_C\alpha}(C\tens S)$). Thus, we have $\s V^*T_{\xi}=(T_{\xi_1}\tens_{\delta_B}1)\s V_2^*y$. In particular, we have
$
(F\tens_{\delta_B}1)\s V^*T_{\xi}=(FT_{\xi_1}\tens_{\delta_B}1)\s V_2^*y.
$
For all $\xi_2\in\s E_2$, $\zeta\in\s E$ and $s\in S$, we have 
\begin{center}
$\theta_{\zeta,\xi_2}\tens_{\delta_B}1=T_{\zeta}^{\vphantom{*}}T_{\xi_2}^*$ and $(T_{\xi_2}^*\tens_{\delta_B}1)\s V_2^*(1_{\s E_2}\tens s)=((1_{\s E_2}\tens s^*)\s V_2T_{\xi_2})_{\vphantom{2}}^*\in\K(\s E_2\tens S,B\tens S)$.
\end{center} 
Hence, $(k\tens_{\delta_B}1)\s V_2^*(1_{\s E_2}\tens s)\in\K(\s E_2\tens S,\s E\tens_{\delta_B}(B\tens S))$ for all $k\in\K(\s E_2,\s E)$ and $s\in S$. In particular, since $F$ is a $F_2$-connection for $\s E_1$ and $y\in\Lin(\s E_2)\tens S$, we have
\begin{center}
$
((FT_{\xi_1}-T_{\xi_1}F_2)\tens_{\delta_B}1)\s V_2^*y\in\K(\s E_2\tens S,\s E\tens_{\delta_B}(B\tens S)).
$
\end{center}
However, $(\s V_2(F_2\tens_{\delta_B}1)\s V_2^*-F_2\tens 1_S)y\in\K(\s E_2\tens S)$ (cf.\ \ref{rk13} 2), $[F_2\tens 1_S,\, y]\in\K(\s E_2\tens S)$ (since $y\in\gamma_2(C)\tens S$) and $\s V(T_{\xi_1}\tens_{\delta_B}1)\s V_2^*y=T_{\xi}$ (since $\s V\s V^*=q_{\beta_{\s E}\alpha}$ and $q_{\beta_{\s E}\alpha}T_{\xi}=T_{\xi}$). This completes the proof.\qedhere
\par}
\end{proof}

\begin{defin}\label{def7} (cf.\ D\'efinition 5.2 \cite{BS2}) Let $A$, $C$ and $B$ be three $\cal G$-$\Cstar$-algebras. Let $(\s E_1,\gamma_1,F_1)\in{\sf E}_{\cal G}(A,C)$ and $(\s E_2,\gamma_2,F_2)\in{\sf E}_{\cal G}(C,B)$. Let $\s E:=\s E_1\tens_{\gamma_2}\s E_2$ be the $\cal G$-equivariant Hilbert $A$-$B$-bimodule defined in \ref{prop18} and \ref{propleftaction}, where $\gamma:A\rightarrow\Lin(\s E)\,;\, a\mapsto\gamma_1(a)\tens_{\gamma_2}1$ denotes the left action of $A$ on $\s E$. We denote by $F_1\#_{\cal G}F_2$\index[symbol]{scb@$\#_{\cal G}$} the set of operators $F\in\Lin(\s E)$ satisfying the following conditions:
\begin{enumerate}[label=(\alph*)]
\item $(\s E,F)\in{\sf E}_{\cal G}(A,B)$;
\item $F$ is a $F_2$-connexion;
\item for all $a\in A$, the image of $\gamma(a)[F_1\tens_{\gamma_2}1,\,F]\gamma(a^*)$ in $\Lin(\s E)/\K(\s E)$ is positive.\qedhere
\end{enumerate}
\end{defin}

With the notations and hypothesis of the above definition, we have the following result.

\begin{thm}\label{Kasprod} (cf.\ Th\'eor\`eme 5.3 \cite{BS2}) We assume that the C*-algebra $A$ is separable. Then, the set $F_1\#_{\cal G}F_2$ is nonempty and the class of $(\s E,F)$ in $\kk_{\cal G}(A,B)$ is independent of $F\in F_1\#_{\cal G}F_2$ and only depends on the class of $(\s E_1,F_1)$ in $\kk_{\cal G}(A,C)$ and that of $(\s E_2,F_2)$ in $\kk_{\cal G}(C,B)$.
\end{thm}

\begin{proof} The proof is basically identical to that of the non-equivariant case (cf.\ \cite{Kas3,Skand}) or the equivariant case for actions of quantum groups (cf.\ \cite{BS1}). Let us prove that $F_1\#_{\cal G}F_2$ is nonempty. Let us denote $q:= q_{\beta_{\s E}\alpha}$ for short. Let $\s V\in\Lin(\s E\tens_{\delta_B}(B\tens S),\s E\tens S)$ be the isometry associated with the action of $\cal G$ on $\s E$. Let $T$ be a $F_2$-connection for $\s E_1$ such that $[T,\,\beta_{\s E}(n^{\rm o})]=0$ for all $n\in N$ (cf.\ \ref{lem8}). Let us fix a strictly positive element $k\in S$ (cf.\ \ref{rk18}). Let us define:
\begin{itemize}
\item $J_1:=\K(\s E_1)\tens_{\gamma_2}1 + \K(\s E)\subset\Lin(\s E)=\M(\K(\s E))$, $J:=\K(\s E)$;
\item $J_2:=\Cstar(\{T - T^*, \, 1 - T^2, \, [T,\,F_1\tens_{\gamma_2}1]\}\cup\{[T,\gamma(a)]\,;\, a\in A\})\subset\Lin(\s E)$;
\item $\s F:=[\{{\rm Ad}(F_1\tens_{\gamma_2} 1), \, {\rm Ad}(T)\}\cup\{{\rm Ad}(\gamma(a)) \,;\, a\in A\}]\subset\der(J_1)$;
\item $J_2':=\Cstar( \{ q(1\tens k)( \s V (T \tens_{\delta_B} 1) \s V^* - q(T\tens 1_S) ),\, ( \s V (T \tens_{\delta_B} 1) \s V^* - q(T\tens 1_S ))(1\tens k)q \})$ $\subset\Lin(\s E\tens S)$.
\end{itemize}
Then, we have:
\begin{itemize}
\item $J$ is an invariant closed two-sided ideal of $J_1$;
\item $J_2$ is a $\Cstar$-subalgebra of $\M(J_1;J)$; by assumption $A$ is separable, then so is $J_2$; hence, $J_2$ is $\sigma$-unital;
\item $\s F$ is a separable (since $A$ separable);
\item $J_2'$ is a $\sigma$-unital $\Cstar$-subalgebra of $\M(J\tens S)$ (separable).
\end{itemize}
Let $x\in\K(\s E_1)\tens_{\gamma_2} 1$. We have $\delta_{J_1}(x)=\s V(x\tens_{\delta_B} 1)\s V^*$. Since $\s V^*\s V=1$, we have
\[
[\delta_{J_1}(x),\, \s V(T\tens_{\delta_B} 1)\s V^*]=\s V([x,T]\tens_{\delta_B} 1)\s V^*
\]
and
\[
\s V([x,T]\tens_{\delta_B} 1)\s V^*=\delta_J([x,T])\in\widetilde{\M}(\K(\s E)\tens S) \quad \text{{\rm(}cf.\ Proposition 9 (e) \cite{Skand}{\rm)}}.
\]
Hence,
\begin{equation}\label{eq31}
[\delta_{J_1}(x),\, \s V(T\tens_{\delta_B} 1)\s V^*]=\s V([x,\, T]\tens_{\delta_B} 1)\s V^*\in\widetilde{\M}(\K(\s E)\tens S). 
\end{equation}
Moreover, by Lemma \ref{lem9} we have
$
\delta_{J_1}(x)(1_{\s E}\tens k)(\s V(T\tens_{\delta_B}1)\s V^*-q(T\tens 1_S))\in\K(\s E\tens S)
$
and
$
(\s V(T\tens_{\delta_B}1)\s V^*-q(T\tens 1_S))\delta_{J_1}(x)(1_{\s E}\tens k)\in\K(\s E\tens S).
$
Hence, 
\[
[\delta_{J_1}(x)(1_{\s E}\tens k),\,\s V(T\tens_{\delta_B}1)\s V^*]=[\delta_{J_1}(x)(1_{\s E}\tens k),\,q(T\tens 1_S)] \quad \text{mod.\ }\K(\s E\tens S).
\]
By combining the fact that $\delta_{J_1}(x)(1_{\s E}\tens k)\in q(J_1\tens S)$ with the fact that $[T,\, y]\in\K(\s E\tens S)$ for all $y\in J_1$ (cf.\ 9 (h) \cite{Skand}), we obtain $[\delta_{J_1}(x)(1\tens k),\,q(T\tens 1_S)]\in\K(\s E\tens S)$. Hence,
\begin{equation}\label{eq32}
[\delta_{J_1}(x)(1_{\s E}\tens k),\,\s V(T\tens_{\delta_B} 1)\s V^*] \in \K(\s E\tens S). 
\end{equation}
We also have
\begin{align*}
\delta_{J_1}(x)[\s V(T\tens_{\delta_B} 1)\s V^* & -q(T\tens 1_S),\, 1_{\s E}\tens k]\\
&=\delta_{J_1}(x)(\s V(T\tens_{\delta_B} 1)\s V^*-q(T\tens 1))(1_{\s E}\tens k) \\
&\qquad-\delta_{J_1}(x)(1_{\s E}\tens k)(\s V(T\tens_{\delta_B} 1)\s V^*-q(T\tens 1_S))\\
&=\delta_{J_1}(x)\s V(T\tens_{\delta_B} 1)\s V^*(1_{\s E}\tens k) - \delta_{J_1}(x)(T\tens 1_S)(1_{\s E}\tens k) \\
&\qquad - \delta_{J_1}(x)(1_{\s E}\tens k)\s V(T\tens_{\delta_B} 1)\s V^* + \delta_{J_1}(x)(1_{\s E}\tens k)q(T\tens 1_S)\\
&=\delta_{J_1}(x)\s V(T\tens_{\delta_B} 1)\s V^*(1_{\s E}\tens k) - \delta_{J_1}(x)(1_{\s E}\tens k)\s V(T\tens_{\delta_B} 1)\s V^* \\
&\qquad -\delta_{J_1}(x)(1_{\s E}\tens k)(T\tens 1_S)(1-q).
\end{align*}
Hence,
\begin{multline*}
\delta_{J_1}(x)[\s V(T\tens_{\delta_B} 1)\s V^* - q(T\tens 1_S),\, (1_{\s E}\tens k)q]
=\delta_{J_1}(x)[\s V(T\tens_{\delta_B} 1)\s V^* - q(T\tens 1_S),\, 1_{\s E}\tens k]q \\
=\delta_{J_1}(x)\s V(T\tens_{\delta_B} 1)\s V^*(1_{\s E}\tens k)q - \delta_{J_1}(x)(1_{\s E}\tens k)\s V(T\tens_{\delta_B} 1)\s V^*.
\end{multline*}
By applying (\ref{eq32}), we have
\begin{multline*}
\delta_{J_1}(x)[\s V(T\tens_{\delta_B} 1)\s V^* -q(T\tens 1_S),\, (1_{\s E}\tens k)q]\\
=[\delta_{J_1}(x),\,\s V(T\tens_{\delta_B}1)\s V^*](1_{\s E}\tens k)q \quad \text{mod.\ }\K(\s E\tens S).
\end{multline*}
By applying (\ref{eq31}), we finally obtain
\begin{equation*}\label{eq16}
\delta_{J_1}(x)[\s V(T\tens_{\delta_B} 1)\s V^*-q(T\tens 1_S),\, (1_{\s E}\tens k)q] \in \K(\s E\tens S). \tag{1}
\end{equation*}
By combining the previous relation with
\begin{equation*}\label{eq17}
\delta_{J_1}(x)(1_{\s E}\tens k)( \s V(T\tens_{\delta_B}1)\s V^* - q(T\tens 1_S))\in\K(\s E\tens S)\quad\text{{\rm(}cf.\ \ref{lem9}{\rm)}}, \tag{2}
\end{equation*}
we have
\begin{equation*}
\delta_{J_1}(x)( \s V(T\tens_{\delta_B}1)\s V^* - q(T\tens 1_S))(1_{\s E}\tens k)q\in\K(\s E\tens S). 
\end{equation*}
Since the above holds true when replacing $T$ by $T^*$ (cf.\ \ref{rk13} 2 and 9 (b) \cite{Skand}), it then follows that $\delta_{J_1}(x)y\in J\tens S$ for all $x\in J_1$ and $y\in J_2'$. We can apply Theorem \ref{Kaspthm}. Let us consider $M$ as in the theorem. Let
\[
F:=M^{1/2}(F_1\tens_{\gamma_2} 1) - (1-M)^{1/2}T.
\]
For all $n\in N$, we have $[F,\, \beta_{\s E}(n^{\rm o})]=0$. Indeed, since $(\s E_1,F_1)\in{\sf E}_{\cal G}(A,C)$ we have $[F_1,\,\beta_{\s E_1}(n^{\rm  o})]=0$. Hence, $[F_1\tens_{\gamma_2}1,\,\beta_{\s E}(n^{\rm  o})]=0$. We also have $[T,\,\beta_{\s E}(n^{\rm o})]=0$ by assumption. By Lemma \ref{lem7} ($[M,\,\beta_{\s E}(n^{\rm o})]=0$ and $M$ is positive), we also have $[M^{1/2},\,\beta_{\s E}(n^{\rm o})]=0$. Similarly, we have $[(1-M)^{1/2},\,\beta_{\s E}(n^{\rm o})]=0$. According to the non-equivariant case, it only remains to prove that
\[
x(\s V(F\tens_{\delta_B}1)\s V^* - q(F\tens 1_S))\in\K(\s E\tens S),\quad \text{for all } x\in(\gamma\tens\id_S)(A\tens S).
\] 
Let us fix $x\in(\gamma\tens\id_S)(A\tens S)$. We have $\s V(F\tens_{\delta_B}1)\s V^*=\delta_J(F)$. By combining the formula $\delta_J(F_1\tens_{\gamma_2}1)=\delta_{\K(\s E_1)}(F_1)\tens_{\gamma_2\tens\id_S} 1$ with the fact that the pair $(\s E_1,F_1)$ is a $\cal G$-equivariant Kasparov $A$-$C$-bimodule, we obtain
\begin{multline*}
x(\delta_J(F_1\tens_{\gamma_2}1)-q((F_1\tens_{\gamma_2}1)\tens 1_S))=x(\delta_{\K(\s E_1)}(F_1)-q_{\beta_{\s E_1}\alpha}(F_1\tens 1_S))\tens_{\gamma_2\tens\id_S}1 \\ \in(\K(\s E_1)\tens S)\tens_{\gamma_2\tens\id_S}1\subset J_1\tens S.
\end{multline*}
Since $M$ is an element of the $\Cstar$-subalgebra $\M(J_1;J)$ of $\M(J_1)$, we have $M^{1/2}\in\M(J_1;J)$. Hence,
\begin{equation}\label{eq34}
(M^{1/2}\tens 1_S)x(\delta_J(F_1\tens_{\gamma_2}1)-q((F_1\tens_{\gamma_2}1)\tens 1_S))\in J\tens S=\K(\s E\tens S). 
\end{equation}
Let $a\in A$ and $s\in S$, we have
\begin{multline*}
(\gamma\tens\id_S)(a\tens s)(\delta_J(M^{1/2}) -  q(M^{1/2}\tens 1_S)) \delta_J(F_1\tens_{\gamma_2}1)
\\=(\gamma(a)M^{1/2}\tens s)\delta_J(M^{1/2})-(\gamma(a)M^{1/2}\tens s)\delta_J(F_1\tens_{\gamma_2}1)
\end{multline*}
(since $[M\tens 1_S,\, q]=0$ we have $[M^{1/2}\tens 1_S,\, q]=0$, cf.\ \ref{lem7}) and 
\[
(1_J\tens s)\delta_J(F_1\tens_{\gamma_2}1),\; (1_J\tens s)\delta_J(M^{1/2})\in J\tens S.
\] 
Hence,
\begin{equation}\label{eq35}
x(\delta_J(M^{1/2}) - q(M^{1/2}\tens 1_S)) \delta_J(F_1\tens_{\gamma_2}1)\in\K(\s E\tens S).
\end{equation}
For all $a\in A$ and $s\in S$, since $[\gamma(a),\,M]\in J$ (recall that ${\rm Ad}(\gamma(a))\in\s F$) we have 
\[
[(\gamma\tens\id_S)(a\tens s),\, M^{1/2}\tens 1_S]=[\gamma(a),\, M^{1/2}]\tens s\in J\tens S \quad \text{{\rm(}cf.\ \ref{lem7}{\rm)}}. 
\]
Thus, we have $[x,\,M^{1/2}\tens 1_S]\in\K(\s E\tens S)$. Hence,
\begin{equation}\label{eq36}
[x,\, M^{1/2}\tens 1_S](\delta_J(F_1\tens_{\gamma_2}1)-q((F_1\tens_{\gamma_2}1)\tens 1_S))\in\K(\s E\tens S).
\end{equation}
By summing up (\ref{eq34}), (\ref{eq35}) and (\ref{eq36}), we have proved that
\[
x(\delta_J(M^{1/2}(F_1\tens_{\gamma_2}1))-q(M^{1/2}(F_1\tens_{\gamma_2}1)\tens 1_S))\in\K(\s E\tens S)
\]
(recall that  $[M^{1/2}\tens 1_S,\,q]=0$). Let  
\[
E:=\{u\in\M(J_1;J)\,;\,\delta_J(u)-q(u\tens 1_S)\in\widetilde{\M}(J\tens S)\, \text{ and }\, [q,\,u\tens 1_S]=0\}.
\] 
Then $E$ is a closed subalgebra of $\M(J_1;J)$. Indeed, $E$ is clearly a closed subspace of $\M(J_1;J)$. Moreover, for all $u,v\in E$ we have
\begin{align*}
\delta_J(uv)-q(uv\tens 1_S) & =\delta_J(u)\delta_J(v)-q(u\tens 1_S)q(v\tens 1_S) \\
& = \delta_J(u)(\delta_J(v)-q(v\tens 1_S)) + (\delta_J(u)-q(u\tens 1_S))q(v\tens 1_S) \\
&= (\delta_J(u) \!-\!q(u\tens 1_S))(\delta_J(v)\!-\!q(v\tens 1_S))\!+\!(u\tens 1_S)(\delta_J(v)\!-\!q(v\tens 1_S))\\
&\qquad +(\delta_J(u)-q(u\tens 1_S))(v\tens 1_S)\in\widetilde{\M}(J\tens S)
\end{align*}
and $[q,uv\tens 1_S]=0$. Hence, we have $\delta_J((1-M)^{1/2})-q((1-M)^{1/2}\tens 1_S)\in\widetilde{\M}(J\tens S)$. Therefore, we have
\begin{equation}\label{eq18}
x(\delta_J((1-M)^{1/2})-q((1-M)^{1/2}\tens 1_S))\in\K(\s E\tens S).
\end{equation}
We also have $(\delta_J((1-M)^{1/2})-q((1-M)^{1/2}\tens 1_S))x\in\K(\s E\tens S)$ by taking the adjoint in (\ref{eq18}). In particular, we have $[x,\,\delta_J((1-M)^{1/2})]=[x,\,q((1-M)^{1/2}\tens 1_S)]$ mod.\ $\K(\s E\tens S)$. Moreover, we have $[x,\,(1-M)\tens 1_S]=-[x,\,M\tens 1_S]\in\K(\s E\tens S)$. It follows from \ref{lem19} that $[x,\,(1-M)^{1/2}\tens 1_S]\in\K(\s E\tens S)$. Since $q$ is a projection such that $[q,\,(1-M)^{1/2}\tens 1_S]=0$, we have $[x,\,q((1-M)^{1/2}\tens 1_S)]=[qxq,\,(1-M)^{1/2}\tens 1_S]\in\K(\s E\tens S)$. Hence,
\begin{equation}\label{eq37}
[x,\,\delta_J((1-M)^{1/2})]\in\K(\s E\tens S).
\end{equation}
We have
\begin{equation}\label{eq38}
\delta_J((1-M)^{1/2})x(\delta_J(T)-q(T\tens 1_S))\in\K(\s E\tens S).
\end{equation}
Indeed, since $k$ is a strictly positive element of $S$, we can assume that $x=x'(1_{\s E}\tens k)$ with $x'\in (\gamma\tens\id_S)(A\tens S)$. In virtue of (\ref{eq37}), we have
\begin{multline*}
\delta_J((1-M)^{1/2})x(\delta_J(T)-q(T\tens 1_S))=x'\delta_J((1-M)^{1/2})(1\tens k)(\delta_J(T)-q(T\tens 1)) \\ \text{ mod.\ } \K(\s E\tens S).
\end{multline*}
Note that $F:=\{u\in\M(J_1)\,;\, \delta_J(u)J_2'\subset J\tens S \, \text{and} \, J_2'\delta_J(u)\in J\tens S\}$ is a $\Cstar$-algebra and $1-M\in F$. Hence, $(1-M)^{1/2}\in F$ and (\ref{eq38}) is proved. By using again (\ref{eq37}), we prove that
\begin{equation}\label{eq19}
x\delta_J((1-M)^{1/2})(\delta_J(T)-q(T\tens 1_S))\in\K(\s E\tens S).
\end{equation}
Therefore, we have
\[
x(\delta_J((1-M)^{1/2}T)-q((1-M)^{1/2}T\tens 1_S))\in\K(\s E\tens S).
\]
Indeed, we have (recall that $[q,\,(1-M)^{1/2}\tens 1_S]=0$)
\begin{align*}
x(\delta_J ((1-M)^{1/2}  T) -& q((1-M)^{1/2}T\tens 1_S))\\
&= x\delta_J((1-M)^{1/2})\delta_J(T) - xq((1-M)^{1/2}\tens 1)q(T\tens 1_S)\\
&= x\delta_J((1-M)^{1/2})(\delta_J(T)-q(T\tens 1_S)) + x(\delta_J((1-M)^{1/2})\\
&\qquad\qquad -q((1-M)^{1/2}\tens 1_S))(T\tens 1_S).
\end{align*}
By (\ref{eq18}) and (\ref{eq19}), we obtain $x(\delta_J ((1-M)^{1/2}  T) - q((1-M)^{1/2}T\tens 1_S))\in\K(\s E\tens S)$. 
\end{proof}

\begin{defin}\label{def8}
Under the notations and hypotheses of \ref{Kasprod}, the class $x$ in ${\sf KK}_{\cal G}(A,B)$ of $(\s E,F)$ where $F\in F_1\#_{\cal G}F_2$ is called the Kasparov product of the class $x_1$ of $(\s E_1,F_1)$ in ${\sf KK}_{\cal G}(A,C)$ and the class $x_2$ of $(\s E_2,F_2)$ in ${\sf KK}_{\cal G}(C,B)$. We denote $x=x_1\tens_{C} x_2$.
\end{defin}

As in the non-equivariant case \cite{Kas2,Kas3} and the equivariant case for actions of quantum groups \cite{BS1}, we have

\begin{thm} The Kasparov product 
\[
{\sf KK}_{\cal G}(A,C)\times{\sf KK}_{\cal G}(C,B)\rightarrow{\sf KK}_{\cal G}(A,B)\; ; \; (x_1,x_2)\mapsto x_1\tens_C x_2
\]
is bilinear, contravariant in $A$, covariant in $B$, functorial in $C$ and associative. 
\end{thm}

\begin{defin} 
Let $A$ and $B$ be $\cal G$-$\Cstar$-algebras.
\begin{itemize}
\item Let $\phi:A\rightarrow B$ be an $\cal G$-equivariant *-homomorphism. If the C*-algebra $B$ is $\sigma$-unital, then the triple $(B,\phi,0)$ is an equivariant Kasparov $A$-$B$-bimodule and we define $[\phi]:=[(B,\phi,0)]\in\kk_{\cal G}(A,B)$. 
\item If the C*-algebra $A$ is $\sigma$-unital, we define $1_A:=[\id_A]=[(A,0)]\in\kk_{\cal G}(A,A)$.\qedhere
\end{itemize}
\end{defin}

The Kasparov product generalizes the composition of equivariant *-homomorphisms. More precisely, we have the following result.

\begin{prop}\label{prop46}
Let $A$, $B$ and $C$ be $\cal G$-C*-algebras with $B$ and $C$ $\sigma$-unital. Let $\phi:A\rightarrow C$ and $\psi:C\rightarrow B$ be $\cal G$-equivariant *-homomorphisms. 
\begin{enumerate}
\item We have $\phi^*([\psi])=[\psi\circ\phi]=\psi_*([\phi])$ in $\kk_{\cal G}(A,B)$.
\item If $A$ is separable, we have $[\psi\circ\phi]=[\phi]\tens_C[\psi]$ in $\kk_{\cal G}(A,B)$.\qedhere
\end{enumerate}
\end{prop}

\begin{prop}\label{prop22}
Let $A$ and $B$ be two $\cal G$-$\Cstar$-algebras. We have:
\begin{enumerate}
\item if $B$ is $\sigma$-unital and $A$ separable, then $x\tens_B 1_B=x$ for all $x\in\kk_{\cal G}(A,B)$;
\item if $A$ is separable, then $1_A\tens_A x=x$ for all $x\in\kk_{\cal G}(A,B)$.\qedhere
\end{enumerate}
\end{prop}

Only statement 2 is not obvious. For the proof, we will need the following easy lemma.

\begin{lem}\label{lem15}
Let $A$ be a $\cal G$-$\Cstar$-algebra and $p\geqslant 2$ an integer. Denote by ${\rm M}_p(A)$ the $\Cstar$-algebra of matrices of size $p$ with entries in $A$. Let $\delta_{{\rm M}_p(A)}:{\rm M}_p(A)\rightarrow\M({\rm M}_p(A)\tens S)$ and $\beta_{{\rm M}_p(A)}:N^{\rm o}\rightarrow\M({\rm M}_p(A))$ be the maps defined by:
\begin{itemize}
\item $\delta_{{\rm M}_p(A)}:=\id_{{\rm M}_p(\GC)}\tens\delta_A\,:\,(a_{ij})\mapsto(\delta_A(a_{ij}))$, up to the identifications 
\begin{align*}
{\rm M}_p(A)&={\rm M}_p(\GC)\tens A \quad \text{and}\\
{\rm M}_p(\GC)\tens\M(A\tens S)&={\rm M}_p(\M(A\tens S))\subset\M({\rm M}_p(A\tens S))=\M({\rm M}_p(A)\tens S);
\end{align*}

\item $\beta_{{\rm M}_p(A)}(n^{\rm o})(a_{ij})=(\beta_A(n^{\rm o})a_{ij})$ and $(a_{ij})\beta_{{\rm M}_p(A)}(n^{\rm o})=(a_{ij}\beta_A(n^{\rm o}))$, for all $n\in N$ and $(a_{ij})\in {\rm M}_p(A)$. 
\end{itemize}
Then, the pair $(\beta_{{\rm M}_p(A)},\delta_{{\rm M}_p(A)})$ is a continuous action of $\cal G$ on ${\rm M}_p(A)$.
\end{lem}

\begin{proof}[Proof of Proposition \ref{prop22} 2.]
The idea of the proof is the same as that of Lemma 3.3 \cite{M} (see also Proposition 17 \cite{Skand}).  Let $(\s E,\gamma,F)$ be an equivariant Kasparov $A$-$B$-bimodule. Consider the Hilbert $B$-submodule $\s E_2:=[\gamma(A)\s E]$ of $\s E$. Let $\s E_1:=\s E$ and $F_1:=F\in\Lin(\s E_1)$. Define maps $\gamma_{ij}:A\rightarrow\Lin(\s E_j,\s E_i)$ for $i,j=1,2$ obtained by range/domain restriction of $\gamma(a)$ (for $a\in A$ fixed) and denote $\gamma_i:=\gamma_{ii}$ for $i=1,2$. Note that $\gamma_{1}=\gamma$ and $\gamma_{2}:A\rightarrow\Lin(\s E_2)$ is a non-degenerate *-representation of $A$ on $\s E_2$. By equivariance of $\gamma$, the action $(\beta_{\s E},\delta_{\s E})$ induces by restriction an action $(\beta_{\s E_2},\delta_{\s E_2})$ of $\cal G$ on $\s E_2$ and it is clear that $(\s E_2,\gamma_{2})$ is an $\cal G$-equivariant Hilbert $A$-$B$-bimodule. It is easily seen that we have an identification of equivariant Hilbert bimodules
$
A\tens_{\gamma}\s E\rightarrow\s E_2 \, ; \, a\tens_{\gamma}\xi \mapsto \gamma(a)\xi.
$
Let $F_2\in 0\,\#_{\cal G}F_1\subset\Lin(\s E_2)$. By combining the maps $\gamma_{ij}$, we obtain an equivariant *-representation 
$
\phi:{\rm M}_2(A)\rightarrow\Lin(\s E_1\oplus\s E_2)\, ; \, (a_{ij}) \mapsto (\gamma_{ij}(a_{ij}))
$
of ${\rm M}_2(A)$ on $\s E_1\oplus\s E_2$ (cf.\ \ref{lem15} and \ref{propdef10}). Hence, the pair $(\s E_1\oplus\s E_2,\phi)$ is an equivariant Hilbert ${\rm M}_2(A)$-$B$-bimodule. We claim that the triple $(\s E_1\oplus\s E_2,\phi,F_1\oplus F_2)$ is an equivariant Kasparov bimodule. For $a\in A$, the operator $T_a\in\Lin(\s E,A\tens_{\gamma}\s E)$ is identified to $\gamma_{21}(a)$ through the identification $A\tens_{\gamma}\s E=\s E_2$. Hence, $F_1\gamma_{12}(a)-\gamma_{12}(a)F_2\in\K(\s E_2,\s E_1)$ and $F_2\gamma_{21}(a)-\gamma_{21}(a)F_1\in\K(\s E_1,\s E_2)$ since $F_2$ is a $F_1$-connection (for $A$). In particular, if $x\in{\rm M}_2(A)$ is an off-diagonal matrix ({\it i.e.\ }the diagonal entries equal zero), then $[F_1\oplus F_2,\,\phi(x)]\in\K(\s E_1\oplus\s E_2)$. Since any element of $A$ is a product of two elements of $A$, any diagonal matrix of ${\rm M}_2(A)$ is a product of two off-diagonal matrices of ${\rm M}_2(A)$. Moreover, if $x,y\in{\rm M}_2(A)$ are off-diagonal, we have
$
[F_1\oplus F_2,\,\phi(xy)]=[F_1\oplus F_2,\,\phi(x)]\phi(y)+\phi(x)[F_1\oplus F_2,\,\phi(y)]\in\K(\s E_1\oplus\s E_2).
$
Hence, $[F_1\oplus F_2,\,\phi(x)]\in\K(\s E_1\oplus\s E_2)$ if $x\in{\rm M}_2(A)$ is diagonal. This relation extend by linearity to all $x\in{\rm M}_2(A)$. The relation $\phi(x)(1-(F_1\oplus F_2)^2)\in\K(\s E_1\oplus\s E_2)$ holds if $x\in{\rm M}_2(A)$ is diagonal. Since we can factorize a finite number of elements of $A$ by a common element of $A$, any matrix of ${\rm M}_p(A)$ factorizes on the right by a diagonal matrix. Hence, this relation extends to all $x\in{\rm M}_p(A)$. The remaning relation of \ref{defkaspbimod} is proved in a same way. With the identifications $\M(\K(\s E_1\oplus\s E_2)\tens S)=\Lin((\s E_1\oplus\s E_2)\tens S)=\Lin((\s E_1\tens S)\oplus(\s E_2\tens S))$, we have 
\begin{center}
$\delta_{\K(\s E_1\oplus\s E_2)}(F_1\oplus F_2)=\delta_{\K(\s E_1)}(F_1)\oplus\delta_{\K(\s E_2)}(F_2)$ \quad and \quad $q_{\beta_{\s E_1\oplus\s E_2}\alpha}=q_{\beta_{\s E_1}\alpha}\oplus q_{\beta_{\s E_2}\alpha}$.
\end{center}
By using the above trick and the identification ${\rm M}_2(A)\tens S={\rm M}_2(A\tens S)$, we prove that 
\[
(\phi\tens\id_S)(x)(\delta_{\K(\s E_1\oplus\s E_2)}(F_1\oplus F_2)-q_{\beta_{\s E_1\oplus\s E_2}\alpha}((F_1\oplus F_2)\tens 1_S))\in\K((\s E_1\oplus\s E_2)\tens S)
\]
for all $x\in{\rm M}_2(A)\tens S$. Let $\tau:=(\s E_1\oplus\s E_2,\phi,F_1\oplus F_2)\in{\sf E}_{\cal G}({\rm M}_2(A),B)$. For $t\in[0,1]$, let $\iota_t:A\rightarrow{\rm M}_2(A)$ be the $\cal G$-equivariant *-homomorphism defined for all $a\in A$ by
\[
\iota_t(a):=\begin{pmatrix} (1-t^2)a & t\sqrt{1-t^2}a\\ t\sqrt{1-t^2}a & t^2a\end{pmatrix}\!.
\]
We have $\iota_0^*(\tau)=(\s E_1,\gamma_1,F_1)\oplus(\s E_2,0,F_2)$ and $\iota_1^*(\tau)=(\s E_1,0,F_1)\oplus(\s E_2,\gamma_2,F_2)$ in ${\sf E}_{\cal G}(A,B)$. Moreover, $(\s E_2,0,F_2)$ and $(\s E_1,0,F_1)$ are degenerate $\cal G$-equivariant Kasparov bimodules. Hence, the triple $(\s E_1\oplus\s E_2, (\phi\circ\iota_t)_{t\in[0,1]},F_1\oplus F_2)$ defines a homotopy between $(\s E_1,\gamma_1,F_1)$ and $(\s E_2,\gamma_2,F_2)$ (cf.\ \ref{prop20} 1 and \ref{ex3} 3). This completes the proof.
\end{proof}

\begin{rk}\label{rk20} 
If $x:=[(\s E,F)]\in\kk_{\cal G}(A,B)$ (with $A$ separable), we can always require the left action $A\rightarrow\Lin(\s E)$ of $A$ on $\s E$ to be non-degenerate (cf.\ \ref{prop22} 2).
\end{rk}

The result below is a generalization of \ref{prop46} 2 and follows straightforwardly from the above remark.

\begin{propdef}
Let $A$, $B$ and $C$ be $\cal G$-C*-algebras with $B$ $\sigma$-unital. Let $\phi:A\rightarrow B$ be a $\cal G$-equivariant *-homomorphism.
\begin{enumerate}
\item If $C$ is separable, we have $x\tens_A[\phi]=\phi_*(x)$ in $\kk_{\cal G}(C,B)$ for all $x\in\kk_{\cal G}(C,A)$.
\item If $A$ is separable, we have $[\phi]\tens_B x=\phi^*(x)$ in $\kk_{\cal G}(A,C)$ for all $x\in\kk_{\cal G}(B,C)$.
\end{enumerate}
In particular, if $A$ is a separable $\cal G$-C*-algebra, then the abelian group $\kk_{\cal G}(A,A)$ endowed with the Kasparov product is a unital ring called the equivariant Kasparov ring of $A$.
\end{propdef}

\subsection{Descent morphisms}

\begin{lem}\label{lem13} (cf.\ 6.13 \cite{BS1})
Let $(\s E,\gamma)$ be a $\cal G$-equivariant $A$-$B$-bimodule with a non-degenerate left action $\gamma:A\rightarrow\Lin(\s E)$ and $F\in\Lin(\s E)$. We assume that
$[\gamma(a),\,F]\in\K(\s E)$ for all $a\in A$, $[F,\beta_{\s E}(n^{\rm o})]=0$ for all $n\in N$ and $(\delta_{\K(\s E)}(F)-q_{\beta_{\s E}\alpha}(F\tens 1_S))(\gamma\tens\id_S)(x)\subset\K(\s E\tens S)$ for all $x\in A\tens S$. Let $D$ be the bidual $\cal G$-C*-algebra of $A$ (cf.\ \ref{not16}).
\begin{enumerate}
\item Through the identification of Hilbert $D$-$B$-bimodules (cf.\ \ref{theo8}, \ref{cor1} and \ref{propdef5} for the definitions and notations)
\begin{equation}\label{eq40}
{\cal E}_{A,R}\tens_{\gamma}\s E\rightarrow{\cal E}_{\s E,R}\; ; \; q_{\beta_A\alpha}(a\tens\xi)\tens_{\gamma}\eta\mapsto q_{\beta_{\s E}\alpha}(\gamma(a)\eta\tens\xi),
\end{equation}
the operator $\restr{\pi_R(F)}{{\cal E}_{\s E,R}}\in\Lin({\cal E}_{\s E,R})$ is identified to a $F$-connection for ${\cal E}_{A,R}$.
\item The operator $F\in\Lin(\s E)$ is identified to a $\restr{\pi_R(F)}{{\cal E}_{\s E,R}}$-connection through the identification of Hilbert $A$-$B$-bimodules ${\cal E}_{A,R}^*\tens_D{\cal E}_{\s E,R}=\s E$.\qedhere
\end{enumerate}
  
\end{lem}

\begin{proof}
1.\ It is clear that the formula 
\begin{equation}\label{eq39}
(A\tens\s H)\tens_{\gamma}\s E\rightarrow\s E\tens\s H \; ;\; (a\tens\xi)\tens_{\gamma}\eta\mapsto\gamma(a)\eta\tens\xi 
\end{equation}
defines an adjointable unitary of Hilbert $B$-modules, which intertwines the left actions of $A\tens\K$. However, we have $\gamma(\beta_A(n^{\rm o})a)=\beta_{\s E}(n^{\rm o})\gamma(a)$ for all $n\in N$ and $a\in A$. Hence, the above unitary induces by restriction the identification of Hilbert $D$-$B$-bimodules ${\cal E}_{A,R}\tens_{\gamma}\s E={\cal E}_{\s E,R}$.
By compactness of the commutators $[\gamma(a),\,F]$ for all $a\in A$, the operator $F\tens 1_{\s H}\in\Lin(\s E\tens\s H)$ is identified to a $F$-connection through (\ref{eq39}). Hence, the operator $\restr{(F\tens 1_{\s H})}{{\cal E}_{\s E,R}}\in\Lin({\cal E}_{\s E,R})$ is identified to a $F$-connection through (\ref{eq40}). Moreover, since $(\delta_{\K(\s E)}(F)-F\tens 1_S)(\gamma\tens\id_S)(q_{\beta_A\alpha}(A\tens S))\subset\K(\s E\tens S)$, the operator $\restr{\pi_R(F)}{{\cal E}_{\s E,R}}-\restr{(F\tens 1_{\s H})}{{\cal E}_{\s E,R}}$ is identified to a $0$-connection (cf.\ 9 (d) \cite{Skand}) through (\ref{eq39}). Hence, the operator $\restr{\pi_R(F)}{{\cal E}_{\s E,R}}\in\Lin({\cal E}_{\s E,R})$ is identified to a $F$-connection through (\ref{eq40}) (cf.\ 9 (c) \cite{Skand}).\newline
2.\ By associativity, we have ${\cal E}_{A,R}^*\tens_D{\cal E}_{\s E,R}=({\cal E}_{A,R}^*\tens_D{\cal E}_{A,R})\tens_A\s E=A\tens_A\s E$ (cf.\ 7.12 \cite{C2}). By non-degeneracy of $\gamma$, we have $A\tens_A\s E=\s E$. We then obtain a canonical identification of Hilbert $A$-$B$-bimodules ${\cal E}_{A,R}^*\tens_D{\cal E}_{\s E,R}=\s E$. For $\xi\in{\cal E}_{A,R}$, the operator $T_{\xi^*}\in\Lin({\cal E}_{\s E,R},{\cal E}_{A,R}^*\tens_D{\cal E}_{\s E,R})$ is identified to $T_{\xi}^*\in\Lin({\cal E}_{A,R}\tens_{\gamma}\s E,\s E)$ up to the identifications ${\cal E}_{A,R}^*\tens_D{\cal E}_{\s E,R}=\s E$ and ${\cal E}_{A,R}\tens_{\gamma}\s E={\cal E}_{\s E,R}$. Therefore, the fact that $F\in\Lin(\s E)$ is a $\restr{\pi_R(F)}{{\cal E}_{\s E,R}}$-connection is just a restatement of 1.
\end{proof}

\begin{lem}\label{lem16}(cf.\ 6.14 \cite{BS2} and 5.3 \cite{Ve})
We follow the notations and hypotheses of \ref{lem13}. Let $\s E_0:={\cal E}_{\s E,R}$ be the $\cal G$-equivariant Hilbert $D$-$B$-bimodule defined in \ref{cor1} and \ref{propdef5}. Let $F_0:=\restr{\pi_R(F)}{\s E_0}\in\Lin(\s E_0)$. Let $\gamma_0:D\rightarrow\Lin(\s E_0)\,;\,d\mapsto\restr{(\gamma\tens\id_{\K})(d)}{\s E_0}$ be the left action of $D$ on $\s E_0$. Then, we have (cf.\ \ref{propdef8})
\[
[\gamma_{0\ast}(D\rtimes{\cal G}),\, \pi_{\K(\s E_0)}(F_0)]\subset\K(\s E_0\rtimes{\cal G}).
\] 
If $(\s E,\gamma, F)$ is a $\cal G$-equivariant Kasparov $A$-$B$-bimodule, then the triple $(\s E_0,\gamma_0,F_0)$ is a $\cal G$-equivariant Kasparov $D$-$B$-bimodule and the triple $(\s E_0\rtimes{\cal G},\gamma_{0\ast},\pi_{\K(\s E_0)}(F_0))$ is a $\widehat{\cal G}$-equivariant Kasparov $D\rtimes{\cal G}$-$B\rtimes{\cal G}$-bimodule.
\end{lem}

\begin{proof}
From the relations $((F\tens 1_S)q_{\beta_{\s E}\alpha}-\delta_{\K(\s E)}(F))(\gamma\tens\id_S)(A\tens S)\subset\K(\s E)\tens S$ and $R(S)\K=\K$, we infer that 
$
((F\tens 1)q_{\beta_{\s E}\widehat{\alpha}}-\pi_R(F))(\gamma\tens\id_{\K})(A\tens\K)\subset\K(\s E\tens\s H).
$
Hence, $(\restr{(F\tens 1)}{\s E_0}-\,F_0)\gamma_0(D)\subset\K(\s E_0)$. We also have $\gamma_0(D)(F_0-\restr{(F\tens 1_{\s H})}{\s E_0})\subset\K(\s E_0)$. Hence, $[\gamma_0(d),\,F_0]=[\gamma_0(d),\,\restr{(F_0\tens 1_{\s H})}{\s E_0}]$ mod.\ $\K(\s E_0)$. By compactness of the commutators $[\gamma(a),\, F]$ for all $a\in A$, we have $[(\gamma\tens\id_{\K})(x),\, F\tens 1_{\s H}]\in\K(\s E\tens\s H)$ for all $x\in A\tens\K$. In particular, $[\gamma_0(d),\,\restr{(F_0\tens 1_{\s H})}{\s E_0}]\in\K(\s E_0)$ for all $d\in D$. Hence, $[\gamma_0(d),\,F_0]\in\K(\s E_0)$ for all $d\in D$. Let $d\in D$ and $x\in\widehat{S}$. We have 
\begin{multline*}
[\gamma_{0\ast}(\pi_{D}(d)\widehat{\theta}_{D}(x)),\, \pi_{\K(\s E_0)}(F_0)]
=\pi_{\K(\s E_0)}([\gamma_0(d),\, F_0])\widehat{\theta}_{\K(\s E_0)}(x) \\+ 
\pi_{\K(\s E_0)}(\gamma_0(d))[\widehat{\theta}_{\K(\s E_0)}(x),\, \pi_{\K(\s E_0)}(F_0)].
\end{multline*}
The first term of the right-hand side belongs to $\K(\s E_0\rtimes{\cal G})$ since $[\gamma_0(d),\, F_0]\in\K(\s E_0)$ (cf.\ \ref{cor3}) and the second one is zero since $[\widehat{\theta}_{\K(\s E_0)}(x),\, \pi_{\K(\s E_0)}(F_0)]=0$ (cf.\ \ref{lem14} and \ref{lem31}). Therefore, we have 
\begin{center}
$[\gamma_{0\ast}(\pi_{D}(d)\widehat{\theta}_{D}(x)),\, \pi_{\K(\s E_0)}(F_0)]\in\K(\s E_0\rtimes{\cal G})$,\quad for all $d\in D$ and $x\in\widehat{S}$.
\end{center} 
Hence, $[\gamma_{0\ast}(D\rtimes{\cal G}),\, \pi_{\K(\s E_0)}(F_0)]\subset\K(\s E_0\rtimes{\cal G})$.
Assume that $(\s E,\gamma,F)$ is a $\cal G$-equivariant Kasparov $A$-$B$-bimodule. By arguing as at the beginning of the proof, we prove the remaining relations of \ref{defkaspbimod} so that the triple $(\s E_0,\gamma_0,F_0)$ is a Kasparov $D$-$B$-bimodule. By invariance of $F_0$, the triple $(\s E_0,\gamma_0,F_0)$ is a $\cal G$-equivariant Kasparov $A$-$B$-bimodule (cf.\ \ref{rk13} 3). We prove the remaining relations of \ref{defkaspbimod} so that the triple $(\s E_0\rtimes{\cal G},\gamma_{0\ast},\pi_{\K(\s E_0)}(F_0))$ is a Kasparov $D$-$B$-bimodule.
For instance, we have 
\begin{align*}
\gamma_{0\ast}(\widehat{\theta}_D(x)\pi_D(d))(&\pi_{\K(\s E_0)}(F_0)^* -\pi_{\K(\s E_0)}(F_0))\\
&=\widehat{\theta}_{\K(\s E_0)}(x)\pi_{\K(\s E_0)}(\gamma_0(d)(F_0^*-F_0))\in\K(\s E_0)\rtimes{\cal G}
\quad \text{(cf.\ \ref{propdef8}),}
\end{align*}
for all $d\in D$ and $x\in\widehat{S}$. Moreover, the triple $(\s E_0\rtimes{\cal G},\gamma_{0\ast},\pi_{\K(\s E_0)}(F_0))$ is a $\cal G$-equivariant Kasparov $D$-$B$-bimodule (cf.\ \ref{lem25} and \ref{rk13} 3).
\end{proof}

Now, we are in position to define the descent morphism.

\begin{thm}\label{descent1} (cf.\ 6.13-6.19 \cite{BS2}, 7.7 b) \cite{BS1} and 5.3-5.4 \cite{Ve})
Let $A$, $B$ and $C$ be $\cal G$-$\Cstar$-algebras.
\newcounter{counter}
\begin{enumerate}
\item If $(\s E,F)$ is a $\cal G$-equivariant Kasparov $A$-$B$-bimodule (with a non-degenerate left action), then $(\s E\rtimes{\cal G},F\tens_{\pi_B}1)$ is a $\widehat{\cal G}$-equivariant Kasparov $A\rtimes{\cal G}$-$B\rtimes{\cal G}$-bimodule. Moreover, if $(\s E_1,F_1)$ and $(\s E_2,F_2)$ are unitarily equivalent (resp.\ homotopic) $\cal G$-equivariant Kasparov $A$-$B$-bimodules, then so are $(\s E_1\rtimes{\cal G},F_1\tens_{\pi_B}1)$ and $(\s E_2\rtimes{\cal G},F_2\tens_{\pi_B}1)$.
\setcounter{counter}{\value{enumi}}
\end{enumerate}
Let $J_{\cal G}:\kk_{\cal G}(A,B)\rightarrow\kk_{\widehat{\cal G}}(A\rtimes{\cal G},B\rtimes{\cal G})$\index[symbol]{jb@$J_{\cal G}$} be the homomorphism of abelian groups defined for all $[(\s E,F)]\in\kk_{\cal G}(A,B)$ (with a non-degenerate left action) by $J_{\cal G}([(\s E,F)]):=[(\s E\rtimes{\cal G},F\tens_{\pi_B}1)]$.
\begin{enumerate}
\setcounter{enumi}{\value{counter}}
\item Let $\phi:A\rightarrow B$ be a $\cal G$-equivariant *-homomorphism. We recall that the equivariance of $\phi$ allows us to define a $\widehat{\cal G}$-equivariant *-homomorphism $\phi_*:A\rtimes{\cal G}\rightarrow B\rtimes{\cal G}$ (cf.\ \ref{FunctorCrossedProd} 1). We have $J_{\cal G}([\phi])=[\phi_*]$. In particular, we have $J_{\cal G}(1_A)=1_{A\rtimes{\cal G}}$.
\item Assume the C*-algebra $A$ to be separable. For all $x_1\in\kk_{\cal G}(A,C)$ and $x_2\in\kk_{\cal G}(C,B)$, we have 
\[
J_{\cal G}(x_1\tens_C x_2)=J_{\cal G}(x_1)\tens_{C\rtimes{\cal G}} J_{\cal G}(x_2).\qedhere
\]
\end{enumerate}
\end{thm}

The following proof is broadly inspired by those of Proposition 5.3 and Lemma 5.4 \cite{Ve}, of which we take some of the notations. In the proof, we will also follow the notations introduced in the lemmas \ref{lem13} and \ref{lem16}.

\begin{proof}1. Let $(\s E,\gamma,F)$ be a $\cal G$-equivariant Kasparov $A$-$B$-bimodule (with a non-degenerate left action). Let $(\s E_0,\gamma_0,F_0)$ be the $\cal G$-equivariant Kasparov $D$-$B$-bimodule defined in \ref{lem16}. Note that $\s E\oplus\s E_0$ is a $\cal G$-equivariant Hilbert $B$-module. We can consider the canonical morphism
$
\pi_{\K(\s E\oplus\s E_0)}:\Lin(\s E\oplus\s E_0)\rightarrow\M(\K(\s E\oplus\s E_0)\rtimes{\cal G}).
$
Up to the canonical identifications $\K(\s E\oplus\s E_0)\rtimes{\cal G}=\K((\s E\oplus\s E_0)\rtimes{\cal G})$ (cf.\ \ref{cor3}) and $(\s E\oplus\s E_0)\rtimes{\cal G}=(\s E\rtimes{\cal G})\oplus(\s E_0\rtimes{\cal G})$, we can consider the following restrictions of $\pi_{\K(\s E\oplus\s E_0)}$:
\[
i:\Lin(\s E)\rightarrow\Lin(\s E\rtimes{\cal G}),\;\;
i_0:\Lin(\s E_0,\s E)\rightarrow\Lin(\s E_0\rtimes{\cal G},\s E\rtimes{\cal G}),\;\;
i^0:\Lin(\s E,\s E_0)\rightarrow\Lin(\s E\rtimes{\cal G},\s E_0\rtimes{\cal G})
\]
\[
\text{and} \;\; i_0^0:\Lin(\s E_0)\rightarrow\Lin(\s E_0\rtimes{\cal G}).
\]
Note that, up to the identifications $\K(\s E\rtimes{\cal G})=\K(\s E)\rtimes{\cal G}$ and $\K(\s E_0\rtimes{\cal G})=\K(\s E_0)\rtimes{\cal G}$, we have $i=\pi_{\K(\s E)}$ and $i_0^0=\pi_{\K(\s E_0)}$. By using the identification ${\cal E}_{A,R}^*\tens_{\gamma_0}{\s E_0}=\s E$ (cf.\ \ref{lem13} 2), for $\zeta\in{\cal E}_{A,R}^*$ let $T_{\zeta}\in\Lin({\cal E}_0,\s E)$ be the operator defined by $T_{\zeta}(\eta)=\zeta\tens_{\gamma_0}\eta$ for all $\eta\in{\cal E}_{\s E,R}$. For $\eta\in\s H$, we denote by $\tau_{\eta}\in\Lin(\s E,\s E\tens\s H)$ the operator defined by $\tau_{\xi}(\eta):=\xi\tens\eta$ for all $\xi\in\s E$. It should be noted that $\tau_{\eta}^*(\xi'\tens\eta')=\langle\eta,\,\eta'\rangle \xi'$ for all $\xi'\in\s E$ and $\eta'\in\s H$.
It follows from $[T_{\zeta}\gamma_0(x)T_{\xi}^*\,;\, \zeta,\,\xi\in{\cal E}_{A,R},\, x\in D]=\gamma(A)$ that 
\[
[i_0(T_{\zeta})\gamma_{0\ast}(x)i^0(T_{\xi}^*) \, ; \, x\in D\rtimes{\cal G},\, \zeta,\,\xi\in{\cal E}_{A,R}^*]=\gamma_*(A\rtimes{\cal G}).
\] 
By combining the non-degeneracy of the canonical morphism $\pi_{D}:D\rightarrow\M(D\rtimes{\cal G})$ with the fact that $\gamma_{0\ast}(\pi_{D}(d))=i_0^0(\gamma_0(d))$ for all $d\in D$, we have that $\gamma_*(A\rtimes{\cal G})$ is the closed linear span of the elements of the form
\[
i_0(T_{\zeta}\gamma_0(b))\gamma_{0\ast}(x)i^0(\gamma_0(c)T_{\xi}^*), \;\; \text{with } x\in D\rtimes{\cal G},\, b,\,c\in D  \text{ and } \zeta,\,\xi\in{\cal E}_{A,R}^*.
\]
Let us prove that $i(F)$ commutes with these elements modulo $\K(\s E\rtimes{\cal G})$. We will carry out the computations modulo $\K(\s E\rtimes{\cal G})$ by using the following inclusions:
\begin{align}
i_0(\K(\s E_0,\s E))\gamma_{0\ast}(A\rtimes{\cal G})&\subset\K(\s E_0\rtimes{\cal G},\s E\rtimes{\cal G})\label{eq25}\\
\gamma_{0\ast}(A\rtimes{\cal G})i^0(\K(\s E,\s E_0))&\subset\K(\s E\rtimes{\cal G},\s E_0\rtimes{\cal G})\label{eq26}
\end{align}
Let us prove (\ref{eq25}) since (\ref{eq26}) will follow by taking the adjoint in (\ref{eq25}). By the relation $\K(\s E_0,\s E)=[\K(\s E_0,\s E)\K(\s E_0)]$, it suffices to prove that $i_0^0(\K(\s E_0))\gamma_{0\ast}(A\rtimes{\cal G})\subset\K(\s E_0\rtimes{\cal G})$. Let $k\in\K(\s E_0)$ and $x\in A\rtimes{\cal G}$. In virtue of the non-degeneracy of the canonical morphism $\widehat{\theta}_A:\widehat{S}\rightarrow\M(A\rtimes{\cal G})$, we can assume that $x=\widehat{\theta}_A(y)x'$ with $y\in\widehat{S}$ and $x'\in A\rtimes{\cal G}$. By the equivariance of $\gamma_{0\ast}$, we have $\gamma_{0\ast}(x)=\widehat{\theta}_{\K(\s E_0)}(y)\gamma_{0\ast}(x')$. Hence, $i_0^0(k)\gamma_{0\ast}(x)\in\K(\s E_0\rtimes{\cal G})$ since $i_0^0(k)\widehat{\theta}_{\K(\s E_0)}(y)\in\K(\s E_0)\rtimes{\cal G}$ and $\gamma_{0\ast}(x)\in\M(\K(\s E_0)\rtimes{\cal G})$.\hfill\break
Let us fix $x\in D\rtimes{\cal G}$, $b,\,c\in D$ and $\zeta,\,\xi\in{\cal E}_{A,R}^*$. We have
\[
i(F)i_0(T_{\zeta}\gamma_0(b))\gamma_{0\ast}(x)i^0(\gamma_0(c)T_{\xi}^*)
=i_0(FT_{\zeta}\gamma_0(b))\gamma_{0\ast}(x)i^0(\gamma_0(c)T_{\xi}^*).
\]
Let $\zeta=(a\tens\zeta'^*)q_{\beta_A\widehat{\alpha}}$ with $a\in A$ and $\zeta'\in\s H$. Let $\eta_0=q_{\beta_{\s E}\widehat{\alpha}}(\eta\tens\chi)\in\s E_0$ with $\eta\in\s E$ and $\chi\in\s H$. Let $(e_{ij}^{(l)})_{1\leqslant l\leqslant k,\, 1\leqslant i,j\leqslant n_l}$ be a system of matrix units of $N$. We have (cf.\ \ref{ProjectionCAlg})
\[
T_{\zeta}(\eta_0)=\sum_{l=1}^k n_l^{-1}\sum_{i,j=1}^{n_l}
\langle\zeta',\, \widehat{\alpha}(e_{ji}^{(l)})\chi\rangle\gamma(a)\beta_{\s E}(e_{ij}^{(l){\rm o}})\eta.
\]
In particular, we have
\begin{align*}
FT_{\zeta}(\eta_0)&=\sum_{l=1}^k n_l^{-1}\sum_{i,j=1}^{n_l}
\langle\zeta',\, \widehat{\alpha}(e_{ji}^{(l)})\chi\rangle F\gamma(a)\beta_{\s E}(e_{ij}^{(l){\rm o}})\eta
=F\gamma(a)\tau_{\zeta'}^*(\eta_0)\\
\intertext{and since $[F,\,\beta_{\s E}(n^{\rm o})]=0$ for all $n\in N$, we also have}
T_{\zeta}(F\tens 1)(\eta_0) & = \sum_{l=1}^k n_l^{-1}\sum_{i,j=1}^{n_l}
\langle\zeta',\, \widehat{\alpha}(e_{ji}^{(l)})\chi\rangle\gamma(a)F\beta_{\s E}(e_{ij}^{(l){\rm o}})\eta
=\gamma(a)F\tau_{\zeta'}^*(\eta_0).
\end{align*}
Hence, we have 
$
FT_{\zeta}-T_{\zeta}\restr{(F\tens 1)}{\s E_0}=[F,\,\gamma(a)]\restr{\tau_{\zeta'}^*}{\s E_0}\in\K(\s E_0,\s E).
$
Thus, we have (cf.\ (\ref{eq25}))
\begin{equation}\label{eq27}
i_0(FT_{\zeta}\gamma_0(b))\gamma_{0\ast}(x)=i_0(T_{\zeta}(F\tens 1)\gamma_0(b))\gamma_{0\ast}(x) \;\; \text{mod.}\;\; \K(\s E_0\rtimes{\cal G},\s E\rtimes{\cal G}).
\end{equation}
We recall (cf.\ proof of \ref{lem16}) that 
\begin{equation}\label{eq28}
(\restr{(F\tens 1)}{\s E_0}-\, F_0)\gamma_0(D)\subset\K(\s E_0).
\end{equation}
We have
\begin{align*}
i(F)i_0(T_{\zeta}\gamma_0(b))\gamma_{0\ast}(x)i^0(\gamma_0(c)T_{\xi}^*)
&=i_0(T_{\zeta}(F\tens 1)\gamma_0(b))\gamma_{0\ast}(x)i^0(\gamma_0(c)T_{\xi}^*) \;\, \text{mod.}\;\, \K(\s E\rtimes{\cal G}) \tag*{(\ref{eq27})}\\
&=i_0(T_{\zeta}F_0\gamma_0(b))\gamma_{0\ast}(x)i^0(\gamma_0(c)T_{\xi}^*) \;\, \text{mod.}\;\, \K(\s E\rtimes{\cal G}) \tag*{(\ref{eq28}),\,(\ref{eq25})}\\
&=i_0(T_{\zeta})i_0^0(F_0)\gamma_{0\ast}(\pi_{D}(b)x\pi_{D}(c))i^0(T_{\xi}^*)\;\, \text{mod.}\;\, \K(\s E\rtimes{\cal G})\\
&=i_0(T_{\zeta})\gamma_{0\ast}(\pi_{D}(b)x\pi_{D}(c))i_0^0(F_0)i^0(T_{\xi}^*) \;\, \text{mod.}\;\, \K(\s E\rtimes{\cal G}) \tag*{\ref{lem16}}\\
&=i_0(T_{\zeta}\gamma_0(b))\gamma_{0\ast}(x)i^0(\gamma_0(c)F_0T_{\xi}^*) \;\, \text{mod.}\;\, \K(\s E\rtimes{\cal G}).
\end{align*}
By using (\ref{eq26}) and \ref{defEqkaspbimod} 3, we prove in a similar way that 
\[
\gamma_{0\ast}(x)i^0(\gamma_0(c)(F\tens 1)T_{\xi}^*)=\gamma_{0\ast}(x)i^0(\gamma_0(c)T_{\xi}^*F)
\;\;
\text{mod. } \K(\s E\rtimes{\cal G},\s E_0\rtimes{\cal G}) \;\;\text{and}
\]
\[
\gamma_0(D)(\restr{(F\tens 1)}{\s E_0}-\, F_0)\subset\K(\s E_0),
\]
which allows us to conclude the above computation by stating that 
\begin{center}
$
i(F)i_0(T_{\zeta}\gamma_0(b))\gamma_{0\ast}(x)i^0(\gamma_0(c)T_{\xi}^*)=i_0(T_{\zeta}\gamma_0(b))\gamma_{0\ast}(x)i^0(\gamma_0(c)T_{\xi}^*)i(F)
$ 
mod.\ $\K(\s E\rtimes{\cal G})$.
\end{center}
The other statements of (\ref{eq20}) are obtained by a direct computation. For instance, for all $x\in A\rtimes{\cal G}$ we have $\gamma_*(x)(i(F)^*-i(F))\in\K(\s E\rtimes{\cal G})$. Indeed, this follows from the fact that $\{\widehat{\theta}(y)i(\gamma(a))\,;\, y\in\widehat{S},\, a\in A\}$ is a total subset of $\gamma_*(A\rtimes{\cal G})$ and the fact that 
$
\widehat{\theta}(y)i(\gamma(a))(i(F)^*-i(F))=\widehat{\theta}(y)i(\gamma(a)(F^*-F))\in\K(\s E\rtimes{\cal G})
$
for all $y\in\widehat{S}$ and $a\in A$.\hfill\break
It follows from the definition of the dual action (\ref{defdualaction}) and the fact that $T_{F\xi}=i(F)T_{\xi}$ for all $\xi\in\s E$ that $i(F)$ is $\delta_{\s E\rtimes{\cal G}}$-invariant. It is also straightforward that $[i(F),\,\alpha_{\s E\rtimes{\cal G}}(n)]=0$ for all $n\in N$. Hence, $(\s E\rtimes{\cal G},\gamma_*,i(F))$ is an equivariant Kasparov bimodule (\ref{rk13} 3).\hfill\break
It is clear that $(\s E,\gamma,F)\in{\sf E}_{\cal G}(A,B)$ defines a unique $(\s E\rtimes{\cal G},\gamma_*,i(F))\in{\sf E}_{\widehat{\cal G}}(A\rtimes{\cal G},B\rtimes{\cal G})$ (cf.\ \ref{prop23}). If $(\s E,F)\in{\sf E}_{\cal G}(A,B[0,1])$ is a homotopy between $(\s E_1,F_1)$ and $(\s E_2,F_2)$, then $(\s E\rtimes{\cal G},i(F))$ is a homotopy between $(\s E_1\rtimes{\cal G},i(F_1))$ and $(\s E_2\rtimes{\cal G},i(F_2))$. This statement makes sense in virtue of the following result:

\begin{lem}\label{lem17}
There exists a unique equivariant *-isomorphism 
\[
B[0,1]\rtimes{\cal G}\rightarrow B\rtimes{\cal G}[0,1] \;\; ;\;\; \pi_{B[0,1]}(f)\widehat{\theta}_{B[0,1]}(x) \mapsto [t\mapsto\pi_B(f(t))\widehat{\theta}_B(x)].\qedhere
\] 
\end{lem}

\begin{proof}[Proof of Lemma \ref{lem17}]
We have the identifications (cf.\ \ref{def3} 4)
\begin{equation*}
\Lin(B[0,1]\tens\s H)\!=\!\M(B[0,1]\tens\K)\!=\!\M((B\tens\K)[0,1])\!=\!\M(B\tens\K)[0,1]\!=\!\Lin(B\tens\s H)[0,1]. 
\end{equation*}
For $f\in B[0,1]$ and $x\in\widehat{S}$, the operator $\pi_{B[0,1],L}(f)(1_{B[0,1]}\tens\rho(x))\in\Lin(B[0,1]\tens\s H)$ is identified to the continuous function $[t\mapsto\pi_{B,L}(f(t))(1_B\tens\rho(x))]\in\Lin(B\tens\s H)[0,1]$. Furthermore, we also have the identifications
\begin{align*}
\Lin({\cal E}_{B[0,1],L})&=\{T\in\Lin(B[0,1]\tens\s H)\,;\, Tq_{\beta_{B[0,1]}\alpha}=T=q_{\beta_{B[0,1]}\alpha}T\} \;\; \text{and}\\
\Lin({\cal E}_{B,L})&=\{T\in\Lin(B\tens\s H)\,;\, Tq_{\beta_B\alpha}=T=q_{\beta_B\alpha}T\}.
\end{align*}
{\setlength{\baselineskip}{1.1\baselineskip}We then obtain an identification between the C*-algebras $\Lin({\cal E}_{B[0,1],L})$ and $\Lin({\cal E}_{B,L})[0,1]$, which identifies $\pi_{B[0,1]}(f)\widehat{\theta}_{B[0,1]}(x)$ with $[t\mapsto\pi_B(f(t))\widehat{\theta}_B(x)]$. By restriction, we obtain an injective *-homomorphism $\phi:B[0,1]\rtimes{\cal G}\rightarrow B\rtimes{\cal G}[0,1]$. Moreover, for all $b\in B$, $x\in\widehat{S}$ and $f\in {\rm C}([0,1])$ we have $f\tens\pi_B(b)\widehat{\theta}_B(y)=\phi(\pi_{B[0,1]}(f\tens b)\widehat{\theta}_B(x))$ (cf.\ $B[0,1]={\rm C}([0,1])\tens B$ and $B\rtimes{\cal G}[0,1]={\rm C}([0,1])\tens B\rtimes{\cal G}$), which proves that the range of $\phi$ is dense. The surjectivity of $\phi$ is then proved. The $\cal G$-equivariance of $\phi$ is a direct consequence of the definition.\qedhere
\par}
\end{proof}

{\it End of the proof of Theorem \ref{descent1}.\ }2. Straightforward.\newline
{\setlength{\baselineskip}{1.15\baselineskip}3. Let $x_1\in\kk_{\cal G}(A,C)$ and $x_2\in\kk_{\cal G}(C,B)$. For $i=1,2$, we consider an equivariant Kasparov bimodule $(\s E_i,\gamma_i,F_i)$ such that $x_i=[(\s E_i,\gamma_i,F_i)]$. Let us consider the $\cal G$-equivariant Hilbert $B$-module $\s E:=\s E_1\tens_{\gamma_2}\s E_2$, the $\cal G$-equivariant *-representation $\gamma:A\rightarrow\Lin(\s E_2)$ defined by $\gamma(a):=\gamma_1(a)\tens_{\gamma_2}1$ for all $a\in A$ and an operator $F\in F_1\#_{\cal G} F_2\subset\Lin(\s E)$ (cf.\ \ref{def7}). Let $y:=x_1\tens_C x_2=[(\s E,\gamma,F)]$ (cf.\ \ref{Kasprod}, \ref{def8}). For $i=1,2$, denote by $(\s E_i',\gamma_i',F_i')$ (resp.\ $(\s E',\gamma',F')$) the equivariant Kasparov bimodules obtained from $(\s E_i,\gamma_i,F_i)$ (resp.\ $(\s E,\gamma,F)$) by the crossed product construction. By definition, we have $J_{\cal G}(x_i)=[(\s E_i',\gamma_i',F_i')]$ for $i=1,2$ and $J_{\cal G}(y)=[(\s E',\gamma',F')]$. We have a canonical identification $\s E'=\s E_1'\tens_{\gamma_2'}\s E_2'$, which intertwines the left actions (cf.\ \ref{prop42}). Let us denote by $\pi:\K(\s E)\rightarrow\M(\K(\s E)\rtimes{\cal G})$ and $\widehat{\theta}:\widehat{S}\rightarrow\M(\K(\s E)\rtimes{\cal G})$ the canonical morphisms. We recall that $F'=\pi(F)$ up to the identification $\Lin(\s E')=\M(\K(\s E)\rtimes{\cal G})$ (cf.\ \ref{prop5}).  We also have $F_1'\tens_{\gamma_2'}1=\pi(F_1\tens_{\gamma_2}1)$ up to the identification $\s E'=\s E_1'\tens_{\gamma_2'}\s E_2'$. We have
$
\gamma(a)[F_1\tens_{\gamma_2}1,\, F]\gamma(a)^*\in\Lin(\s E)_+ + \K(\s E)
$
for all $a\in A$ by assumption. For all $a\in A$ and $x\in\widehat{S}$, we have
\begin{multline*}
\widehat{\theta}(x)\pi(\gamma(a))[F_1'\tens_{\gamma_2'}1,\, F']\pi(\gamma(a)^*)\widehat{\theta}(x)^*\\
=\widehat{\theta}(x)\pi(\gamma(a)[F_1\tens_{\gamma_2}1,\, F]\gamma(a)^*)\widehat{\theta}(x)^*\in\Lin(\s E')_+ + \K(\s E')
\end{multline*}
since $\widehat{\theta}(x)\pi(\Lin(\s E)_+)\widehat{\theta}(x)^*\subset\Lin(\s E')_+$ and $\widehat{\theta}(x)\pi(\K(\s E))\widehat{\theta}(x)^*\subset\K(\s E')$. Hence, the positivity condition is satisfied since $\{\widehat{\theta}(x)\pi(\gamma(a))\,;\, a\in A,\, x\in\widehat{S}\}$ is a total subset of $\gamma'(A\rtimes{\cal G})$. The compatibility with the direct sum is straightforward.\qedhere
\par}
\end{proof}

In a similar way, we prove the following theorem.

\begin{thm}
Let $A$, $B$ and $C$ be $\widehat{\cal G}$-$\Cstar$-algebras.
\newcounter{counter2}
\begin{enumerate}
\item If $(\s F,G)$ is a $\widehat{\cal G}$-equivariant Kasparov $A$-$B$-bimodule (with a non-degenerate left action), then $(\s F\rtimes\widehat{\cal G},G\tens_{\widehat\pi_B}1)$ is a $\cal G$-equivariant Kasparov $A\rtimes\widehat{\cal G}$-$B\rtimes\widehat{\cal G}$-bimodule. Moreover, if $(\s F_1,G_1)$ and $(\s F_2,G_2)$ are unitarily equivalent (resp.\ homotopic) $\widehat{\cal G}$-equivariant Kasparov $A$-$B$-bimodules, then so are $(\s F_1\rtimes\widehat{\cal G},G_1\tens_{\widehat\pi_B}1)$ and $(\s F_2\rtimes\widehat{\cal G},G_2\tens_{\widehat\pi_B}1)$.
\setcounter{counter2}{\value{enumi}}
\end{enumerate}
Let $J_{\widehat{\cal G}}:\kk_{\widehat{\cal G}}(A,B)\rightarrow\kk_{\cal G}(A\rtimes\widehat{\cal G},B\rtimes\widehat{\cal G})$\index[symbol]{jc@$J_{\widehat{\cal G}}$} be the homomorphism of groups defined for all $[(\s F,G)]\in\kk_{\widehat{\cal G}}(A,B)$ (with a non-degenerate left action) by $J_{\widehat{\cal G}}([(\s F,G)])=[(\s F\rtimes\widehat{\cal G},G\tens_{\widehat\pi_B}1)]$.
\begin{enumerate}
\setcounter{enumi}{\value{counter2}}
\item Let $\phi:A\rightarrow B$ be a $\widehat{\cal G}$-equivariant *-homomorphism. We recall that the equivariance of $\phi$ allows us to define a $\cal G$-equivariant *-homomorphism $\phi_*:A\rtimes\widehat{\cal G}\rightarrow B\rtimes\widehat{\cal G}$. We have $J_{\widehat{\cal G}}([\phi])=[\phi_*]$. In particular, we have $J_{\widehat{\cal G}}(1_A)=1_{A\rtimes\widehat{\cal G}}$.
\item Assume the C*-algebra $A$ to be separable. For all $x_1\in\kk_{\widehat{\cal G}}(A,C)$ and $x_2\in\kk_{\widehat{\cal G}}(C,B)$, we have 
\[
J_{\widehat{\cal G}}(x_1\tens_C x_2)=J_{\widehat{\cal G}}(x_1)\tens_{C\rtimes\widehat{\cal G}} J_{\widehat{\cal G}}(x_2).\qedhere
\]
\end{enumerate}
\end{thm}

\begin{nbs}\label{not20}
Before stating the main theorem of this article, we need to specify some further notations. Let $A$ (resp.\ $B$) be a $\cal G$ (resp.\ $\widehat{\cal G}$)-C*-algebra. Let $D$ (resp.\ $E$) be the bidual $\cal G$ (resp.\ $\widehat{\cal G}$)-C*-algebra defined in \ref{not16}. We recall that ${\cal E}_{A,R}$ (resp.\ ${\cal E}_{B,\rho}$) defines a $\cal G$ (resp.\ $\widehat{\cal G}$)-equivariant Morita equivalence between $A$ (resp.\ $B$) and $D$ (resp.\ $E$) (cf.\ \ref{theo8}). Let us define\index[symbol]{ae@${\f a}_A$, $\widehat{{\f a}}_B$}\index[symbol]{be@${\f b}_A$, $\widehat{{\f b}}_B$} 
\begin{align*}
{\f b}_A &:=[({\cal E}_{A,R},0)]\in\kk_{\cal G}(D,A) \quad \text{and} \quad
{\f a}_A:=[({\cal E}_{A,R}^*,0)]\in\kk_{\cal G}(A,D)\\
\text{{\rm(}resp.\ } \widehat{\f b}_B &:=[({\cal E}_{B,\rho},0)]\in\kk_{\widehat{\cal G}}(E,B) \quad \text{and} \quad \widehat{\f a}_B:=[({\cal E}_{B,\rho}^*,0)]\in\kk_{\widehat{\cal G}}(B,E)\text{{\rm)}}.\qedhere
\end{align*}
\end{nbs}

\begin{lem}\label{lem32}
Let $A$ be a separable $\cal G$ {\rm(}resp.\ $\widehat{\cal G}${\rm)}-C*-algebra. Let $D$ (resp.\ $E$) be the bidual $\cal G$ {\rm(}resp.\ $\widehat{\cal G}${\rm)}-C*-algebra of $A$. We have 
\begin{center}
${\f b}_A\tens_A{\f a}_A=1_{D}$ \; and \; ${\f a}_A\tens_{D}{\f b}_A=1_A$ \quad {\rm(}resp.\ $\widehat{\f b}_A\tens_A\widehat{\f a}_A=1_{E}$ \; and \; $\widehat{\f a}_A\tens_{E}\widehat{\f b}_A=1_A${\rm)}.
\end{center}
In particular, if $A$ and $B$ are $\cal G$ {\rm(}resp.\ $\widehat{\cal G}${\rm)}-C*-algebras, then the map 
\begin{align*}
&\kk_{\cal G}(A,B)\rightarrow\kk_{\cal G}(D_{\rm g},D_{\rm d}) \,;\, x\mapsto{\f b}_A\tens_A x\tens_B {\f a}_B\\
\text{{\rm(}resp.\ }&
\kk_{\widehat{\cal G}}(A,B)\rightarrow\kk_{\widehat{\cal G}}(E_{\rm g},E_{\rm d}) \,;\, x\mapsto\widehat{\f b}_A \tens_A x \tens_B \widehat{\f a}_B\text{{\rm)}}
\end{align*}
are isomorphisms of abelian groups {\rm(}cf.\ \ref{conv2} for the writing conventions). 
\end{lem}

\begin{proof}
This is a consequence of 7.12 \cite{C2} and \ref{theo8}.
\end{proof}

We can state the main result of this paragraph. We refer the reader to Théorème 6.20 \cite{BS2}, Remarque 7.7 b) \cite{BS1} and \S 5.1 \cite{Ve} for the corresponding statement in the quantum group framework.

\begin{thm}\label{theo3}
Let $A$ and $B$ be $\cal G$ {\rm(}resp.\ $\widehat{\cal G}${\rm)}-$\Cstar$-algebras. If the C*-algebra $A$ is $\sigma$-unital, then for all $x\in\kk_{\cal G}(A,B)$ {\rm(}resp.\ $x\in\kk_{\widehat{\cal G}}(A,B)${\rm)}, we have 
\[
J_{\widehat{\cal G}}\circ J_{\cal G}(x) ={\f b}_A\tens_A x\tens_B {\f a}_B \quad
\text{{\rm(}resp.\ } J_{\cal G}\circ J_{\widehat{\cal G}}(x) = \widehat{\f b}_A \tens_A x \tens_B \widehat{\f a}_B\text{{\rm)}}
\]
up to the identifications $(A\rtimes{\cal G})\rtimes\widehat{\cal G}=D_{\rm g}$ and $(B\rtimes{\cal G})\rtimes\widehat{\cal G}=D_{\rm d}$ {\rm(}resp.\ $(A\rtimes\widehat{\cal G})\rtimes{\cal G}=E_{\rm g}$ and $(B\rtimes\widehat{\cal G})\rtimes{\cal G}=E_{\rm d}${\rm)} {\rm(}cf.\ \ref{BidualityTheo}{\rm)}.
\end{thm}

\begin{proof}
Let $x\in\kk_{\cal G}(A,B)$. It suffices to prove that $J_{\widehat{\cal G}}(J_{\cal G}(x))\tens_{D_{\rm d}}{\f b}_B={\f b}_A\tens_A x$. Let $(\s E,\gamma,F)\in{\sf E}_{\cal G}(A,B)$ such that $x=[(\s E,\gamma,F)]$. With no loss of generality, we can assume that the *-representation $\gamma$ is non-degenerate. Let us consider the canonical morphisms $\pi:\K(\s E)\rightarrow\M(\K(\s E)\rtimes{\cal G})$ and $\widehat{\pi}:\K(\s E)\rtimes{\cal G}\rightarrow\M((\K(\s E)\rtimes{\cal G})\rtimes\widehat{\cal G})$. We make the identifications $\M((\K(\s E)\rtimes{\cal G})\rtimes\widehat{\cal G})=\Lin((\s E\rtimes{\cal G})\rtimes\widehat{\cal G})$ (cf.\ \ref{prop5}, \ref{cor2}). We have
\[
J_{\widehat{\cal G}}\circ J_{\cal G}(x)=[((\s E\rtimes{\cal G})\rtimes\widehat{\cal G},\widehat{\pi}\circ\pi(F))]
\]
(recall that $(F\tens_{\widehat{\pi}_B}1)\tens_{\pi_B}1$ is identified with $\widehat{\pi}\circ\pi(F)$). Let us compute the Kasparov product $J_{\widehat{\cal G}}(J_{\cal G}(x))\tens_{D_{\rm d}}{\f b}_B$. Denote by $\pi_{D_{\rm d}}:D_{\rm d}\rightarrow\Lin({\cal E}_{B,R})$ the equivariant *-representation given by $\pi_{D_{\rm d}}(d):=\restr{d}{{\cal E}_{B,R}}$ for all $d\in D_{\rm d}$. Recall that we have the identification of equivariant Hilbert $D_{\rm g}$-$B$-bimodules $(\s E\rtimes{\cal G})\rtimes\widehat{\cal G}\tens_{\pi_{D_{\rm d}}}{\cal E}_{B,R}={\cal E}_{\s E,R}$ (cf.\ \ref{cor1}) and the operator $\widehat{\pi}(\pi(F))$ is identified to $\restr{\pi_R(F)}{{\cal E}_{\s E,R}}$. Hence, 
$
J_{\widehat{\cal G}}(J_{\cal G}(x))\tens_{D_{\rm d}}{\f b}_B=[({\cal E}_{\s E,R},\pi_R(F))].
$
Let us compute the Kasparov product ${\f b}_A\tens_A x$. We have an identification of Hilbert $D_{\rm g}$-$B$-bimodules ${\cal E}_{A,R}\tens_{\gamma}\s E={\cal E}_{\s E,R}$ (cf.\ \ref{lem13} 1). It is easily seen that this identification is $\cal G$-equivariant. By Lemma \ref{lem13}, the operator $\restr{\pi_R(F)}{{\cal E}_{\s E,R}}\in\Lin({\cal E}_{A,R}\tens_{\gamma}\s E)$ is a $F$-connection. Since the positivity condition is trivial, we have proved that ${\f b}_A\tens_A x=[({\cal E}_{\s E,R},\pi_R(F))]$.
\end{proof}

\begin{cor}\label{cor6}
The homomorphisms $J_{\cal G}$ and $J_{\widehat{\cal G}}$ are isomorphisms of abelian groups.
\end{cor}

\begin{proof}
This follows from Lemma \ref{lem32} and Theorem \ref{theo3}.
\end{proof}

\begin{cor}
If $A$ is a separable $\cal G$ (resp.\ $\widehat{\cal G}$)-C*-algebra, then the descent morphism $J_{\cal G}:\kk_{\cal G}(A,A)\rightarrow\kk_{\widehat{\cal G}}(A\rtimes{\cal G},A\rtimes{\cal G})$ {\rm(}resp.\ $J_{\widehat{\cal G}}:\kk_{\widehat{\cal G}}(A,A)\rightarrow\kk_{\cal G}(A\rtimes{\widehat{\cal G}},A\rtimes{\widehat{\cal G}})${\rm)} is an isomorphism of rings.
\end{cor}

\begin{proof}
This is a straightforward consequence of Theorem \ref{theo3} and Corollary \ref{cor6}.
\end{proof}

\section{Monoidal equivalence and equivariant KK-theory}

In this chapter, we fix a colinking measured quantum groupoid ${\cal G}:={\cal G}_{\QG_1,\QG_2}$ between two monoidally equivalent regular locally compact quantum groups $\QG_1$ and $\QG_2$.

\subsection{Description of the \texorpdfstring{${\cal G}_{\QG_1,\QG_2}$}{G(G1,G2)}-equivariant Kasparov bimodules}\label{sectionColinkKasp}

In this paragraph, we fix two $\cal G$-$\Cstar$-algebras $A$ and $B$. We also fix a $\cal G$-equivariant Hilbert $A$-$B$-bimodule $(\s E,\gamma)$. We use all the notations and results of \S\S \ref{sectionC*AlgColink} and \ref{sectionHilbModColink} concerning these objects. We assume the C*-algebra $A$ to be separable. In particular, for $j=1,2$ the C*-algebra $A_j$ is separable. 

\begin{lem}\label{lem30}
Let $F\in\Lin(\s E)$. There exist unique operators $F_1\in\Lin(\s E_1)$ and $F_2\in\Lin(\s E_2)$ such that $F=F_1\oplus F_2$. We have the following statements:
\begin{enumerate}
\item the pair $(\s E,\gamma,F)$ is a Kasparov $A$-$B$-bimodule if, and only if, for $j=1,2$ the pair $(\s E_j,\gamma_j,F_j)$ is a Kasparov $A_j$-$B_j$-bimodule;
\item the conditions below are equivalent:
\begin{enumerate}[label=(\roman*)]
\item $(\gamma\tens\id_S)(x)(\delta_{\K(\s E)}(F)-q_{\beta_{\s E}\alpha}(F\tens 1_S))\in\K(\s E\tens S)$ for all $x\in A\tens S$,
\item $(\gamma_k\tens\id_{S_{kj}})(x)(\delta_{\K(\s E_j)}^k(F_j) - F_k\tens 1_{S_{kj}})\in\K(\s E_k\tens S_{kj})$ for all $x\in A_k\tens S_{kj}$ and $j,k=1,2$;
\end{enumerate}
\item if the triple $(\s E,\gamma,F)$ is a $\cal G$-equivariant Kasparov $A$-$B$-bimodule, then the triple $(\s E_j,\gamma_j,F_j)$ is a $\QG_j$-equivariant Kasparov $A_j$-$B_j$-bimodule for $j=1,2$.\qedhere
\end{enumerate}
\end{lem}

\begin{proof}
We recall that $\beta_{\s E}(\GC^2)\subset{\cal Z}(\Lin(\s E))$ (cf.\ \S 3.2.3 (3.11) \cite{BC}). Hence, we have $[F,\,q_{\s E,j}]=0$ for $j=1,2$. Let $F_j:=\restr{F}{\s E_j}\in\Lin(\s E_j)$ for $j=1,2$. We have $F=F_1\oplus F_2$. The equivalence of the first statement follows from the relation $\K(\s E)=\K(\s E_1)\oplus\K(\s E_2)$ and the definitions, for instance we have $[\gamma(a),\,F]\xi=\sum_{j=1,2}[\gamma_j(q_{A,j}a),\, F_j]q_{\s E,j}\xi$ for all $\xi\in\s E$. Condition (ii) is just a straightforward restatement of condition (i). Statement 3 follows by taking $k=j$ in (ii) and by using statement 1.
\end{proof}

\begin{propdef}\label{propdef9}
With the notations and hypotheses of \ref{lem30}, for $j=1,2$ the map
\[
J_{\QG_j,\,{\cal G}}:\kk_{\cal G}(A,B)\rightarrow\kk_{\QG_j}(A_j,B_j)\,;\,[(\s E,\gamma,F)]\mapsto[(\s E_j,\gamma_j,F_j)]
\]
is a homomorphism of abelian groups.
\end{propdef}

\begin{proof}
We first prove that $J_{\QG_j,\,{\cal G}}$ is well defined. By \ref{lem5} 1 and \ref{lem10} 1, we have well-defined maps 
\[
J_{\QG_j,\,{\cal G}}:{\sf E}_{\cal G}(A,B)\rightarrow{\sf E}_{\QG_j}(A_j,B_j)\,;\,(\s E,\gamma,F)\mapsto(\s E_j,\gamma_j,F_j), \quad \text{for }  j=1,2.
\]
The fact that the map $J_{\QG_j,\,{\cal G}}$ factorizes over the quotient map ${\sf E}_{\cal G}(A,B)\rightarrow{\sf KK}_{\cal G}(A,B)$ will follow from the following result:

\begin{lem} Let $C$ be a third $\cal G$-$\Cstar$-algebra. Let $g:B\rightarrow C$ be a $\cal G$-equivariant *-homomorphism. We recall that $g$ induces a $\QG_j$-equivariant *-homomorphism $g_j:B_j\rightarrow C_j$ for $j=1,2$ (cf.\ \ref{lem7bis} 1). Then, the diagram
\begin{center}
$
\begin{tikzcd}
{\sf E}_{\cal G}(A,B) \arrow{r}{g_\ast} \arrow[swap]{d}{J_{\QG_j,\,{\cal G}}} & {\sf E}_{\cal G}(A,C) \arrow{d}{J_{\QG_j,\,{\cal G}}} \\
{\sf E}_{\QG_j}(A_j,B_j) \arrow{r}{(g_j)_\ast} & {\sf E}_{\QG_j}(A_j,C_j)
\end{tikzcd}
$
\end{center}
commutes for all $j=1,2$.
\end{lem}
The proof of the above lemma is effortless and the details are left to the reader. We recall that $B[0,1]$ is a $\cal G$-$\Cstar$-algebra (cf.\ \ref{def3} 4). Then, we apply the notations of \S \ref{sectionC*AlgColink} as follows:
\begin{itemize}
\item $B[0,1]_j:=\beta_{B[0,1]}(\varepsilon_j)B[0,1]=B_j[0,1]$, for $j=1,2$;
\item $\delta_{B_j[0,1]}^k:B_j[0,1]\rightarrow\M(B_k[0,1]\tens S_{kj})=\M(B_k\tens S_{kj})[0,1]$ for $j,k=1,2$, the *-homo\-morphism defined by $\delta_{B_j[0,1]}^k(f)(t):=\delta_{B_j}^k(f(t))$ for all $f\in B_j[0,1]$ and $t\in[0,1]$.
\end{itemize}
We recall that for $t\in[0,1]$ the evaluation map ${\rm e}_t:B[0,1]\rightarrow B$ is $\cal G$-equivariant. Moreover, for $j=1,2$, it is clear that the *-homomorphism $({\rm e}_t)_j:B_j[0,1]\rightarrow B_j$ is the evaluation at $t$. It then follows from the above lemma that the image of a homotopy of ${\sf E}_{\cal G}(A,B[0,1])$ by ${\sf E}_{\cal G}(A,B[0,1])\rightarrow{\sf E}_{\QG_j}(A_j,B_j[0,1])$ is a homotopy, which finally proves that the map $J_{\QG_j,\,{\cal G}}$ is well defined on ${\sf KK}_{\cal G}(A,B)$. The compatibility with the direct sum is straightforward.
\end{proof}

\begin{prop}\label{prop43}
Let $C$ be a third $\cal G$-$\Cstar$-algebra. For $j=1,2$, we have:
\begin{enumerate}
\item $J_{\QG_j,{\cal G}}(1_A)=1_{A_j}$;
\item for all $x\in\kk_{\cal G}(A,C)$ and $y\in\kk_{\cal G}(C,B)$, $J_{\QG_j,{\cal G}}(x\tens_C y)=J_{\QG_j,{\cal G}}(x)\tens_{C_j}J_{\QG_j,{\cal G}}(y)$ in $\kk_{\QG_j}(A_j,B_j)$.\qedhere
\end{enumerate}
\end{prop}

\begin{proof}
The first statement is straightforward. Let us write $x:=[(\s E,F)]\in\kk_{\cal G}(A,C)$ and $y:=[(\s E',F')]\in\kk_{\cal G}(C,B)$. Let $\gamma:C\rightarrow\Lin(\s E')$ be the left action of $C$ on $\s E'$. Let $\s F:=\s E\tens_{\gamma}\s E'$. We have $x\tens_C y=[(\s F,T)]$ for some $T\in F\#_{\cal G} F'\subset\Lin(\s F)$. In the following, we use the notations of \S \ref{sectionHilbModColink} and \ref{lem30} for the Kasparov bimodules $(\s E,F)$, $(\s E',F')$ and $(\s F,T)$. We have a well-defined $B_j$-linear isometric map
$
\Phi :\s E_j\tens_{\gamma_j}\s E'_j \rightarrow \s F_j \; ; \; \xi \tens_{\gamma_j} \eta \mapsto \xi\tens_{\gamma} \eta.
$
Let $\xi\in\s E$ and $\eta\in\s E'$. Let us write $\xi=\zeta c$ with $\zeta\in\s E$ and $c\in C$. We have
\begin{multline*}
q_{\s F,j}(\xi\tens_{\gamma}\eta)=q_{\s E,j}\xi\tens_{\gamma}\eta=(q_{\s E,j}\zeta)c\tens_{\gamma}\eta=(q_{\s E,j}\zeta)q_{C,j}c\tens_{\gamma}\eta=q_{\s E,j}\zeta\tens_{\gamma}\gamma(q_{C,j}c)\eta\\=q_{\s E,j}\zeta\tens_{\gamma}\gamma(c)q_{\s E',j}\eta=q_{\s E,j}\xi\tens_{\gamma}q_{\s E',j}\eta.
\end{multline*}
Therefore, the range of $\Phi $ contains the total subset $\{q_{\s F,j}(\xi\tens_{\gamma}\eta)\,;\,\xi\in\s E,\,\eta\in\s E'\}$ of $\s F_j$. Hence, $\Phi \in\Lin(\s E_j\tens_{\gamma_j}\s E'_j,\s F_j)$ and $\Phi $ is unitary. We have $\Phi ^*(\xi\tens_{\gamma}\eta)=q_{\s E,j}\xi\tens_{\gamma_j} q_{\s E',j}\eta$ for all $\xi\tens_{\gamma}\eta\in\s F_j$. It is clear that $\Phi $ is $\QG_j$-equivariant and intertwines the left actions of $A_j$.

\medbreak

Therefore, $(\s E_j\tens_{\gamma_j}\s E'_j,\Phi ^*T_j\Phi )$ is a $\QG_j$-equivariant Kasparov $A_j$-$B_j$-bimodule unitarily equiva\-lent to $(\s F_j,T_j)$. Hence, $[(\s E_j\tens_{\gamma_j}\s E'_j,\Phi ^*T_j\Phi )]=[(\s F_j,T_j)]$ in $\kk_{\QG_j}(A_j,B_j)$.

\medbreak

Since $T$ is a $F'$-connection for $\s E$, the operator $\Phi ^*T_j\Phi $ is a $F'_j$-connection for $\s E_j$. Indeed, for all $\xi\in\s E$ and $\eta\in\s E'$, we have $q_{\s F,j}T_{\xi}(\eta)=\Phi (q_{\s E,j}\xi\tens_{\gamma_j} q_{\s E',j}\eta)=\Phi T_{q_{\s E,j}\xi}q_{\s E',j}(\eta)$. Hence, $q_{\s F,j}T_{\xi}F'=\Phi T_{q_{\s E,j}\xi}F'_jq_{\s E',j}$. We also have $q_{\s F,j}T T_{\xi}=T_j\Phi T_{q_{\s E,j}\xi}q_{\s E',j}$. Hence, 
$$
T_{\xi}F'_j - \Phi ^*T_j\Phi  T_{\xi}=\Phi ^*q_{\s F,j}\restr{(T_{\xi}F'-TT_{\xi})}{\s E'_j}\in\K(\s E'_j,\s F_j) \quad \text{for all } \xi\in\s E_j.
$$
Let $\phi:A\rightarrow\Lin(\s E)$ and $\pi:A\rightarrow\Lin(\s F)\,;\,a\mapsto\phi(a)\tens_{\gamma} 1$ be the left actions of $A$. For $j=1,2$, we have $\phi_j:A_j\rightarrow\Lin(\s E_j)$ and $\pi_j:A_j\rightarrow\Lin(\s F_j)$. We have $q_{\s F,j}(F\tens_{\gamma}1)=\Phi (F_j\tens_{\gamma_j}1)\Phi ^*q_{\s F,j}$. Hence, we have
\begin{align*}
q_{\s F,j}(F\tens_{\gamma}1)T&=\Phi (F_j\tens_{\gamma_j}1)\Phi ^*T_jq_{\s F,j} \quad \text{and}\\ 
q_{\s F,j}T(F\tens_{\gamma}1)&=T_jq_{\s F,j}(F\tens_{\gamma}1)=T_j\Phi (F_j\tens_{\gamma_j}1)\Phi ^*q_{\s F,j}.
\end{align*}
Hence, 
\[
q_{\s F,j}[F\tens_{\gamma} 1,\,T]=(\Phi (F_j\tens_{\gamma_j}1)\Phi ^*T_j-T_j\Phi (F_j\tens_{\gamma_j}1)\Phi ^*)q_{\s F,j}=\Phi [F_j\tens_{\gamma_j}1,\,\Phi ^*T_j\Phi ]\Phi ^*q_{\s F,j}.
\] 
Therefore, for all $a\in A_j$ we have
\begin{align*}
\restr{\pi(a)[F\tens_{\gamma} 1,T]\pi(a)}{\s E_j}&=\pi_j(a)\Phi [F_j\tens_{\gamma_j}1,\, \Phi ^*T_j\Phi ]\Phi ^*\pi_j(a)\\
&=\Phi (\phi_j(a)\tens_{\gamma_j}1)[F_j\tens_{\gamma_j}1,\, \Phi ^*T_j\Phi ](\phi_j(a)\tens_{\gamma_j}1)\Phi ^*.
\end{align*} 
Hence, the image of
$
(\phi_j(a)\tens_{\gamma_j}1)[F_j\tens_{\gamma_j}1,\Phi ^*T_j\Phi ](\phi_j(a)\tens_{\gamma_j}1)
$
in $\Lin(\s E_j)/\K(\s E_j)$ is positive. Hence, $\Phi ^*T_j\Phi \in F_j\#_{\QG_j} F'_j$ and 
\[
[(\s E_j\tens_{\gamma_j}\s E'_j,\Phi ^*T_j\Phi )]=[(\s E_j,F_j)]\tens_{C_j}[(\s E'_j,F'_j)]=J_{\QG_j,{\cal G}}(x)\tens_{C_j}J_{\QG_j,{\cal G}}(y).\qedhere
\]
\end{proof}

\subsection{Induction of equivariant Kasparov bimodules}

We begin this paragraph with two technical lemmas that will be used in the proof of \ref{prop28}.

\begin{lem}\label{lem18}
Let $\QG$ be a locally compact quantum group. Let $A$, $A'$, ${B}$ and ${B}'$ be $\QG$-$\Cstar$-algebras. Let $({\s E},\phi)$ {\rm(}resp.\ $({\s E}',\phi')${\rm)} be a $\QG$-equivariant Hilbert $A$-${B}$ {\rm(}resp.\ $A'$-${B}')$-bimodule. Denote by $A_0$ {\rm(}resp.\ ${B}_0)$ the $\QG$-$\Cstar$-algebras $A\oplus A'$ {\rm(}resp.\ ${B}\oplus{B}')$. Denote by ${\s E}_0$ the Hilbert ${B}_0$-module ${\s E}\oplus{\s E}'$. Denote by $\phi_0$ the *-representation of $A_0$ on ${\s E}_0$ defined by $\phi_0(a\oplus a'):=\phi(a)\oplus\phi'(a')$, for all $a\in A$ and $a'\in A'$. Fix $T\in\Lin({\s E})$ and $T'\in\Lin({\s E}')$ and denote $T_0:=T\oplus T'\in\Lin({\s E}_0)$.
\begin{enumerate}
\item The triple $({\s E}_0,\phi_0,T_0)$ is a $\QG$-equivariant Kasparov $A_0$-${B}_0$-bimodule if, and only if, the triples $({\s E},\phi,T)$ and $({\s E}',\phi',T')$ are $\QG$-equivariant Kasparov bimodules. 
\item Denote by $p_{\rm g}:A_0\rightarrow A$, $p_{\rm g}':A_0\rightarrow A'$,  $p_{\rm d}:{B}_0\rightarrow{B}$ and $p_{\rm d}':{B}_0\rightarrow{B}'$ the canonical  surjections. Denote also by $i_{\rm g}:A\rightarrow A_0$, $i_{\rm g}':A'\rightarrow A_0$, $i_{\rm d}:{B}\rightarrow{B}_0$, $i_{\rm d}':{B}\rightarrow{B}_0$ the canonical injections. Assume the C*-algebras $A$ and $A'$ to be separable and $B$ and $B'$ to be $\sigma$-unital. If the conditions above hold true, then we have
\[
[i_{\rm g}]\tens_{A_0}[({\s E}_0,F_0)]\tens_{{B}_0} [p_{\rm d}]=[({\s E},F)],\quad
[i_{\rm g}']\tens_{A_0}[({\s E}_0,F_0)]\tens_{{B}_0} [p_{\rm d}']=[({\s E}',F')]
\]
\[
 \text{and} \quad [p_{\rm g}]\tens_A[({\s E},F)]\tens_{B}[i_{\rm d}]=[({\s E}_0,F_0)]=[p_{\rm g}']\tens_{A'}[({\s E}',F')]\tens_{B'}[i_{\rm d}']
\]
in $\kk_{\QG}(A,{B})$, $\kk_{\QG}(A',{B}')$ and $\kk_{\QG}(A_0,{B}_0)$ respectively.\qedhere
\end{enumerate}
\end{lem}

\begin{proof}
1, 3. The result is a straightforward consequence of the canonical identifications 
of C*-algebras $\K({\s E})\oplus\K({\s E}')=\K({\s E}_0)$, $(\K({\s E})\tens\text{C}_0(\QG))\oplus(\K({\s E}')\tens\text{C}_0(\QG))=\K({\s E}_0)\tens\text{C}_0(\QG)$.\hfill\break
2. By definition of the structures of $\QG$-$\Cstar$-algebra on $A_0$ and ${B}_0$, the maps defined above are $\QG$-equivariant *-homomorphisms. We have 
\[
[i_{\rm g}]\tens_{A_0}[({\s E}_0,\phi_0,F_0)]\tens_{{B}_0} [p_{\rm d}] =
[({\s E}_0\tens_{p_{\rm d}}{B},(a\mapsto\phi_0(i_{\rm g}(a))\tens_{p_{\rm d}}1),F_0\tens_{p_{\rm d}}1)].
\]
However, the triples $({\s E}_0\tens_{p_{\rm d}}{B},(a\mapsto\phi_0(i_{\rm g}(a))\tens_{p_{\rm d}}1),F_0\tens_{p_{\rm d}}1)$ and $({\s E},\phi,F)$ are unitarily equivalent via the map ${\s E}_0\tens_{p_{\rm d}}{B}\rightarrow{\s E}\,;\, (\xi\oplus\xi')\tens_{p_{\rm d}} b\mapsto \xi b$. Therefore, we obtain the relation $[i_{\rm g}]\tens_{A_0}[({\s E}_0,\phi_0,F_0)]\tens_{{B}_0} [p_{\rm d}]=[({\s E},\phi,F)]$. With a similar argument, we also prove that $[i_{\rm g}']\tens_{A_0}[({\s E}_0,\phi_0,F_0)]\tens_{{B}_0} [p_{\rm d}']=[({\s E}',\phi',F')]$. The last formula follows from the first two ones since we have $p_{\rm g}\circ i_{\rm g}=\id_{A}$, $p_{\rm g}'\circ i_{\rm g}'=\id_{A'}$, $p_{\rm d}\circ i_{\rm d}=\id_{{B}}$ and $p_{\rm d}'\circ i_{\rm d}'=\id_{{B}'}$ (cf.\ \ref{prop46} 2).
\end{proof}

Before stating the second technical lemma, we need to fix the notion of operators acting by factorization.

\begin{propdef}\label{defOpActFact}
Let $B$ be a C*-algebra. Let $H$ and $K$ be two Hilbert spaces. Let ${\s E}$ be a Hilbert $B$-module. We consider the Hilbert ${B}\tens\K(K)$-module ${\s E}\tens\K(K)$ and the Hilbert $B\tens\K(H)$-module ${\s E}\tens\K(H,K)$. If $F\in\Lin(\s E\tens\K(K))$, then there exists a unique operator $\widetilde{F}\in\Lin({\s E}\tens\K(H,K))$ such that
\[
\widetilde{F}(\xi\tens kT)=F(\xi\tens k)(1_{B}\tens T),\quad \text{for all } \xi\in\s E,\, k\in\K(K) \text{ and } T\in\K(H,K).
\]
The operator $\widetilde{F}$ will be referred to as the operator $F$ acting on ${\s E}\tens\K(H,K)$ by factorization and will sometimes be simply denoted by $F$. Furthermore, for all $F\in\K(\s E\tens\K(K))$ we have $\widetilde{F}\in\K(\s E\tens\K(H,K))$.
\end{propdef}

\begin{proof}
We have $\K(H,K)=\K(K)\K(H,K)$. Indeed, $\K(H,K)$ is Hilbert $\K(K)$-module under the natural left action of $\K(K)$ by composition and the $\K(K)$-valued inner product defined by $\langle T,\, S\rangle:=T\circ S^*$ for $T,S\in\K(H,K)$. Let $(u_i)$ be an approximate unit of the C*-algebra $\K(K)$. Let us fix $\xi\in\s E$ and write $\xi=\eta b$ with $\eta\in\s E$ and $b\in B$. For all $k\in\K(K)$ and $T\in\K(H,K)$, we have
\[
F(\xi\tens k)(1_{B}\tens T)={\rm lim}_i\, F(\eta b\tens u_ik)(1_B\tens T)={\rm lim}_i\, F(\eta\tens u_i)(b\tens kT).
\]
Hence, $F(\xi\tens k)(1_{B}\tens T)=F(\xi\tens k')(1_{B}\tens T')$ for all $k,k'\in\K(K)$ and $T,T'\in\K(H,K)$ such that $kT=k'T'$. Thus, the map $\widetilde{F}$ is well defined. Moreover, it is easily seen that $\widetilde{F}\in\Lin({\s E}\tens\K(H,K))$ with $\widetilde{F}^*=\widetilde{F^*}$. Note also that the map $F\mapsto\widetilde{F}$ is a *-homomorphism. If the Hilbert space $H$ is nonzero, we have $\K(K)=[\K(H,K)\K(K,H)]$. However, for all $\xi,\eta\in\s E$ and $T_1,T_2\in\K(H,K)$ the image of $\theta_{\xi\tens k,\,\eta\tens T_1^{\vphantom{*}}T_2^*}$ by the *-homomorphism $F\mapsto\widetilde{F}$ is the operator $\theta_{\xi\tens kT_2,\,\eta\tens T_1}$.
\end{proof}

\begin{lem}\label{lem28}
Let $A$ and ${B}$ be two C*-algebras. Let $H$ and $K$ be two Hilbert spaces. Let ${\s E}$ be a Hilbert $A$-${B}$-bimodule. We consider the Hilbert $A\tens\K(K)$-${B}\tens\K(K)$-bimodule ${\s E}\tens\K(K)$ and the Hilbert $A\tens\K(K)$-${B}\tens\K(H)$-bimodule ${\s E}\tens\K(H,K)$. Let $F\in\Lin(\s E\tens\K(K))$ such that the pair $({\s E}\tens\K(K),F)$ is a Kasparov bimodule. If $\widetilde{F}\in\Lin({\s E}\tens\K(H,K))$ denotes the operator $F$ acting on ${\s E}\tens\K(H,K)$ by factorization, then the pair $({\s E}\tens\K(H,K),\widetilde{F})$ is a Kasparov bimodule.
\end{lem}

\begin{proof}
Let $\gamma:A\rightarrow\Lin(\s E)$ be the left action of $A$ on $\s E$. The left action of $A\tens\K(K)$ on the Hilbert $B\tens\K(K)$-module $\s E\tens\K(K)$ is $\phi:=\gamma\tens\id_{\K(K)}:A\tens\K(K)\rightarrow\Lin(\s E\tens\K(K))$. It is clear that the left action of $A\tens\K(K)$ on the Hilbert $B\tens\K(H)$-module $\s E\tens\K(H,K)$ is the *-representation $\widetilde{\phi}:A\tens\K(K)\rightarrow\Lin(\s E\tens\K(H,K))$, where for $x\in A\tens\K(K)$ the operator $\widetilde{\phi}(x)$ is the operator $\phi(x)$ acting on $\s E\tens\K(H,K)$ by factorization. The proof follows from the last statement of \ref{defOpActFact} and the fact that the factorization map $F\mapsto\widetilde{F}$ is a *-homomorphism.
\end{proof}

In the following, we fix two $\QG_1$-C*-algebras $A_1$ and $B_1$. Let 
\begin{center}
$A_2:=\ind(A_1)$ \quad and \quad $B_2:=\ind(B_1)$
\end{center} 
be the induced $\QG_2$-C*-algebras. We also fix a $\QG_1$-equivariant Hilbert $A_1$-$B_1$-bimodule $(\s E_1,\gamma_1)$ (with a non-degenerate left action) and denote by 
\begin{center}
$(\s E_2,\gamma_2):=(\ind(\s E_1),\ind\gamma_1)$
\end{center}
the induced $\QG_2$-equivariant Hilbert $A_2$-$B_2$-bimodule.  Let us consider the $\cal G$-C*-algebras $A:=A_1\oplus A_2$ and $B:=B_1\oplus B_2$. We also equip the Hilbert C*-module $\s E:=\s E_1\oplus \s E_2$ with the structure of $\cal G$-equivariant Hilbert $A$-$B$-bimodule defined by the action $(\beta_{\s E},\delta_{\s E})$ of $\cal G$ (cf.\ \ref{prop16}) and the equivariant *-representation $\gamma:A\rightarrow\Lin(\s E)$ (cf.\ \ref{def10}). In what follows, we will make some obvious identifications without always mentioning them explicitly, {\it e.g.}\ $A_1=A_1\oplus\{0\}$. We will also use the notations and results of \S \ref{sectionDoubleCrossedProduct} concerning the objects associated with  $A$, $B$ and $\s E$.

\medbreak

Before recalling the definition of the homomorphisms $J_{\QG_k,\QG_j}:\kk_{\QG_j}(A_j,B_j)\rightarrow\kk_{\QG_k}(A_k,B_k)$ for $j,k=1,2$ with $j\neq k$ (cf.\ \S 4.5 \cite{BC}) we first have to fix some notations.

\begin{nbs} Let $\QG$ be a regular locally compact quantum group. Let ${\cal A}$ be a $\QG$-C*-algebra. By the Baaj-Skandalis duality theorem (cf.\ \cite{BS2}), we identify the $\QG$-C*-algebras $({\cal A}\rtimes\QG)\rtimes\widehat{\QG}$ and ${\cal A}\tens\K({\rm L}^2(\QG))$. Let ${\f b}_{\!\cal A}:=[({\cal A}\tens{\rm L}^2(\QG),0)]\in\kk_{\QG}({\cal A}\tens\K({\rm L}^2(\QG)),{\cal A})$ and ${\f a}_{\!\cal A}:=[({\cal A}\tens{\rm L}^2(\QG)^*,0)]\in\kk_{\cal G}({\cal A},{\cal A}\tens\K({\rm L}^2(\QG)))$ (cf.\ \cite{BS1}). 
\end{nbs} 

\begin{nbs}(cf.\ 3.50 c) and 3.51 \cite{BC})
For all $j,l,l'=1,2$, the Hilbert C*-module ${\cal B}_{ll',j,{\rm g}}$ (resp.\ ${\cal B}_{ll',j,{\rm d}}$) is a $\QG_j$-equivariant imprimitivity ${\cal B}_{l,j,{\rm g}}$-${\cal B}_{l',j,{\rm g}}$-bimodule (resp.\ $\QG_j$-equiva\-riant imprimitivity ${\cal B}_{l,j,{\rm d}}$-${\cal B}_{l',j,{\rm d}}$-bimodule) and we denote by ${\f c}_{ll',j,{\rm g}}$ (resp.\ ${\f c}_{ll',j,{\rm d}}$) the class of $({\cal B}_{ll',j,{\rm g}},0)$ (resp.\ $({\cal B}_{ll',j,{\rm d}},0)$) in $\kk_{\QG_j}({\cal B}_{l,j,{\rm g}},{\cal B}_{l',j,{\rm g}})$ (resp.\ $\kk_{\QG_j}({\cal B}_{l,j,{\rm d}},{\cal B}_{l',j,{\rm d}})$).
\end{nbs}

\begin{prop}(cf.\ 4.30 to 4.33 \cite{BC})
Let $F_1\in\Lin(\s E_1)$ such that the pair $(\s E_1,F_1)$ is $\QG_1$-equivariant Kasparov $A_1$-$B_1$-bimodule. We have:
\begin{enumerate}
\item the pair $({\cal E}_{1,2},(\id_{\K(\s E_2)}\tens R_{21})\delta_{\K(\s E_1)}^2(F_1))$ is a $\QG_1$-equivariant ${\cal B}_{1,2,{\rm g}}$-${\cal B}_{1,2,{\rm d}}$-bimodule (cf.\ \ref{not19} for the definition of $R_{21}$);
\item there exists an operator $F_2\in\Lin(\s E_2)$ such that: 
\smallbreak
a) $(\s E_2,F_2)$ is a $\QG_2$-equivariant Kasparov $A_2$-$B_2$-bimodule,
\smallbreak
b) in $\kk_{\QG_2}(A_2,B_2)$, we have 
\[
{\f b}_{A_2}\!\tens_{A_2}[(\s E_2,F_2)]\tens_{B_2}\!{\f a}_{B_2}={\f c}_{21,2,{\rm g}}\tens_{{\cal B}_{1,2,{\rm g}}}\![({\cal E}_{1,2},(\id_{\K(\s E_2)}\tens R_{21})\delta_{\K(\s E_1)}^2(F_1))]\tens_{{\cal B}_{1,2,{\rm d}}}\!{\f c}_{12,2,{\rm d}}.
\]
\item if $F_2,F_2'\in\Lin(\s E_2)$ satisfy the conditions a) and b) above, then $[(\s E_2,F_2)]=[(\s E_2,F_2')]$ in $\kk_{\QG_2}(A_2,B_2)$.
\end{enumerate}
If $x:=[(\s E_1,F_1)]\!\in\!\kk_{\QG_1}(A_1,B_1)$, let us denote by $J_{\QG_2,\QG_1}(x)$ the unique element $y\!\in\!\kk_{\QG_2}(A_2,B_2)$ satisfying the relation
\[
{\f b}_{A_2}\tens_{A_2}y\tens_{B_2}{\f a}_{B_2}={\f c}_{21,2,{\rm g}}\tens_{{\cal B}_{1,2,{\rm g}}}[({\cal E}_{1,2},(\id_{\K(\s E_2)}\tens R_{21})\delta_{\K(\s E_1)}^2(F_1))]\tens_{{\cal B}_{1,2,{\rm d}}}{\f c}_{12,2,{\rm d}}.
\]
Then, the map $J_{\QG_2,\QG_1}:\kk_{\QG_1}(A_1,B_1)\rightarrow\kk_{\QG_2}(A_2,B_2)$ is a homomorphism of abelian groups.
\end{prop}

In order to define the homomorphism $J_{\QG_1,\QG_2}:\kk_{\QG_2}(A_2,B_2)\rightarrow\kk_{\QG_1}(A_1,B_1)$, we first need to fix further objects.

\medbreak

We consider the induced $\QG_1$-C*-algebras
\begin{center}
$A_1':=\iind(A_2)$ \quad and \quad $B_1':=\iind(B_2)$.
\end{center}
We denote by $\s E'_1:=\iind(\s E_2)$ the induced $\QG_1$-equivariant Hilbert $A_1'$-$B_1'$-bimodule. Let us consider the $\cal G$-C*-algebras $A':=A_1'\oplus A_2$ and $B':=B_1'\oplus B_2$. We also equip the Hilbert C*-module $\s E':=\s E'_1\oplus \s E_2$ with the structure of $\cal G$-equivariant Hilbert $A'$-$B'$-bimodule defined by the action $(\beta_{\s E'},\delta_{\s E'})$ of $\cal G$ (cf.\ \ref{prop16}) and the left action $\gamma':A'\rightarrow\Lin(\s E')$ (cf.\ \ref{def10}). We will use the notations of \S \ref{sectionDoubleCrossedProduct} decorated with a prime concerning the objects associated with $A'$, $B'$ and $\s E'$.

\begin{nbs}
Let $\pi_{1,{\rm g}}:A_1\rightarrow A_1'$ and $\pi_{1,{\rm d}}:B_1\rightarrow B_1'$ be the $\QG_1$-equivariant *-isomorphisms defined in \ref{propind1}. We recall that we have $\cal G$-equivariant *-isomorphisms $A\rightarrow A',\, (a_1,a_2)\mapsto(\pi_{1,{\rm g}}(a_1),a_2)$ and $B\rightarrow B',\, (b_1,b_2)\mapsto(\pi_{1,{\rm d}}(b_1),b_2)$ (cf.\ \S 4.1 \cite{BC}), which then induce a $\cal G$-equivariant *-isomorphism between the bidual $\cal G$-C*-algebras associated with $A$ and $A'$ (resp.\ $B$ and $B'$) by applying the functoriality of the crossed product and the biduality theorem (cf.\ \ref{FunctorCrossedProd}, \ref{FunctorCrossedProdBis}, \ref{IsoTT} and \ref{BidualityTheo}). By restriction, for all $j,l=1,2$ we have two $\QG_j$-equivariant *-isomorphisms $f_{l,j,{\rm g}}:{\cal B}_{l,j,{\rm g}}\rightarrow{\cal B}'_{l,j,{\rm g}}$ and $f_{l,j,{\rm d}}:{\cal B}_{l,j,{\rm d}}\rightarrow{\cal B}'_{l,j,{\rm g}}$.
\end{nbs}

\begin{propdef}
Let $F_2\in\Lin(\s E_2)$ such that the pair $(\s E_2,F_2)$ is a $\QG_2$-equivariant Kasparov $A_2$-$B_2$-bimodule. Let $y:=[(\s E_2,F_2)]\in\kk_{\QG_2}(A_2,B_2)$. Let $J_{\QG_1,\QG_2}(y)$ be the unique element $x\in\kk_{\QG_1}(A_1,B_1)$ satisfying the relation
\[
{\f b}_{A_1}\!\tens_{A_1}\!x\tens_{B_1}\!{\f a}_{B_1}={\f c}_{12,1,{\rm g}}\tens_{{\cal B}_{2,1,{\rm g}}}\![f_{2,1,{\rm g}}]\tens_{{\cal B}'_{2,1,{\rm g}}}\![({\cal E}'_{2,1},(\id_{\K(\s E_1')}\tens R_{12})\delta_{\K(\s E_2)}^1(F_2))]\tens_{{\cal B}'_{2,1,{\rm d}}}\![f_{2,1,{\rm d}}^{-1}].
\]
Then, the map $J_{\QG_1,\QG_2}:\kk_{\QG_2}(A_2,B_2)\rightarrow\kk_{\QG_1}(A_1,B_1)$ is a homomorphism of abelian groups.
\end{propdef}

\begin{lem}\label{lem19} For $j=1,2$, let $F_j\in\Lin(\s E_j)$. Let $F:=F_1\oplus F_2\in\Lin(\s E)$. The pair $(\s E,\gamma)$ is a Kasparov $A$-$B$-bimodule if, and only if, the pair $(\s E_j,F_j)$ is a Kasparov $A_j$-$B_j$-bimodule for $j=1,2$.
\end{lem}

\begin{proof}
For all $a=(a_1,a_2)\in A$, we have the relations $[\gamma(a),\, F]=\oplus_{j=1,2}[\gamma_j(a_j),\, F_j]$, $\gamma(a)(F^2-1)=\oplus_{j=1,2}\gamma_j(a_j)(F_j^2-1)$ and $\gamma(a)(F-F^*)=\oplus_{j=1,2}\gamma_j(a_j)(F_j-F_j^*)$. Therefore, the equivalence follows directly from $\K(\s E)=\K(\s E_1)\oplus\K(\s E_2)$.
\end{proof}

\begin{lem}\label{lem29}
For all $j,l,l'=1,2$, the pair $({\cal E}_{ll',j},(\id_{\K(\s E_j)}\tens R_{jl})\delta_{\K(\s E_l)}^j(F_l))$ is a Kasparov ${\cal B}_{l,j,{\rm g}}$-${\cal B}_{l',j,{\rm d}}$-bimodule.
\end{lem}

\begin{proof}
If $l'=l$, we refer the reader to \cite{BS1} for $l=j$ and 4.30, 4.34 \cite{BC} for $l\neq j$. By applying Lemma \ref{lem28}, the general case follows from the case where $l'=l$. 
\end{proof}

{\setlength{\baselineskip}{1.35\baselineskip}
Actually, we can prove that the pair $({\cal E}_{ll',j},(\id_{\K(\s E_j)}\tens R_{jl})\delta_{\K(\s E_l)}^j(F_l))$ is a $\QG_j$-equivariant Kasparov ${\cal B}_{l,j,{\rm g}}$-${\cal B}_{l',j,{\rm d}}$-bimodule. Indeed, as above the case where $l'=l$ is already known (cf.\ \cite{BS1}, \cite{BC}). Moreover, the operator $(\id_{\K(\s E_j)}\tens R_{jl})\delta_{\K(\s E_l)}^j(F_l))\in\Lin({\cal E}_{l,j})$ is invariant (cf.\ 4.29 \cite{BC}). By a direct computation, we show that the operator $(\id_{\K(\s E_j)}\tens R_{jl})\delta_{\K(\s E_l)}^j(F_l))\in\Lin({\cal E}_{ll',j})$ is invariant (cf.\ \S \ref{sectionDoubleCrossedProduct} for the definitions).
\par}

\medbreak

In the following, we assume the C*-algebra $A_1$ to be separable. Hence, the C*-algebras $A_2$ and $A$ are separable (cf.\ \ref{lem27}).

\begin{prop}\label{prop28}
For $j=1,2$, let $F_j\in\Lin(\s E_j)$ such that $(\s E_j,F_j)$ is a $\QG_j$-equivariant Kasparov $A_j$-$B_j$-bimodule. Let $F:=F_1\oplus F_2\in\Lin(\s E)$. We have:
\newcounter{saveenum2}
\begin{enumerate}
\item the pair $(\s D,\pi_R(F))$ is a $\cal G$-equivariant Kasparov $D_{\rm g}$-$D_{\rm d}$-bimodule;
\item there exists $T\in\Lin(\s E)$ such that:
\begin{enumerate}[label=\alph*)]
\item the pair $(\s E,T)$ is a $\cal G$-equivariant Kasparov $A$-$B$-bimodule,
\item ${\f b}_A\tens_A[(\s E,T)]\tens_B{\f a}_B=
[(\s D,\pi_R(F))]$.
\end{enumerate}
\setcounter{saveenum2}{\value{enumi}}
\end{enumerate}
Moreover, we have:
\begin{enumerate}
\setcounter{enumi}{\value{saveenum2}}
\item if $T,T'\in\Lin(\s E)$ satisfy the conditions a) and b), then $[(\s E,T)]=[(\s E,T')]$ in $\kk_{\cal G}(A,B)$;
\item if $T\in\Lin(\s E)$ satisfies the conditions a) and b), then the class of $(\s E,T)$ in $\kk_{\cal G}(A,B)$ only depends on those of $(\s E_1,F_1)$ and $(\s E_2,F_2)$ in $\kk_{\QG_1}(A_1,B_1)$ and $\kk_{\QG_2}(A_2,B_2)$ respectively;
\item for $j=1,2$ and $T\in\Lin(\s E)$ satisfying the conditions a) and b), let $T_j\in\Lin(\s E_j)$ such that $T=T_1\oplus T_2$ (cf.\ \ref{lem30}), then the pair $(\s E_j,T_j)$ is a $\QG_j$-equivariant Kasparov $A_j$-$B_j$-bimodule and we have $[(\s E_j,T_j)]=[(\s E_j,F_j)]$ in $\kk_{\QG_j}(A_j,B_j)$.\qedhere
\end{enumerate}
\end{prop}

\begin{proof}
{\setlength{\baselineskip}{1.1\baselineskip}
1. Let $X\in\Lin(\s D)$ be the operator defined by $X(\zeta):=\pi_R(F)\circ\zeta$ for all $\zeta\in\s D$. It suffices to prove that $(\s D,X)$ is a Kasparov $D_{\rm g}$-$D_{\rm d}$-bimodule (cf.\ \ref{lem14} and \ref{rk13} 3). This amounts again to proving that $(\s D_j,X_j)$ is a Kasparov $D_{{\rm g},j}$-$D_{{\rm d},j}$-bimodule for $j=1,2$ (cf.\ \ref{lem19} 1). However, this follows straightforwardly from \ref{lem29} and \ref{lem18} 1.\newline
2, 3. These statements are direct consequences of Lemma \ref{lem32}.\newline
4. For $j=1,2$, let $F_j,F_j'\in\Lin(\s E_j)$ such that $(\s E_j,F_j)$ and $(\s E_j,F_j')$ are $\QG_j$-equivariant Kasparov $A_j$-$B_j$-bimodules satisfying $[(\s E_j,F_j)]=[(\s E_j,F_j')]$ in $\kk_{\QG_j}(A_j,B_j)$. Let $F:=F_1\oplus F_2\in\Lin(\s E)$ and  $F':=F_1'\oplus F_2'\in\Lin(\s E)$. Let $T\in\Lin(\s E)$ (resp.\ $T'\in\Lin(\s E)$) be an operator satisfying the conditions a) and b) for $F$ (resp.\ $F'$). Let us prove that $[(\s E,T)]=[(\s E,T')]$ in $\kk_{\cal G}(A,B)$. For $j=1,2$, there exists a degenerate $\QG_j$-equivariant Kasparov $A_j$-$B_j$-bimodule $(\s F_j,X_j)$ such that $(\s E_j\oplus\s F_j,F_j\oplus X_j)$ and $(\s E_j\oplus\s F_j,F_j'\oplus X_j)$ are operator homotopic (cf.\ 5.11 (2) \cite{BS1}). In particular, for $j=1,2$ there exists an operator homotopy $(\s E_j,F_{j,t})_{t\in[0,1]}$ between $(\s E_j,F_j)$ and $(\s E_j,F_j')$. For $t\in[0,1]$, let $F_t:=F_{1,t}\oplus F_{2,t}\in\Lin(\s E)$. For $t\in[0,1]$, let $T_t\in\Lin(\s E)$ be an operator satisfying the conditions a) and b) for $F_t$. Then, $(\s D,\pi_R(T_t))_{t\in[0,1]}$ is an operator homotopy between $(\s D,\pi_R(T))$ and $(\s D,\pi_R(T'))$. Hence, $[(\s D,\pi_R(T))]=[(\s D,\pi_R(T'))]$. It then follows that ${\f b}_A\tens_A[(\s E,T)]\tens_B{\f a}_B={\f b}_A\tens_A[(\s E,T')]\tens_B{\f a}_B$. Hence, $[(\s E,T)]=[(\s E,T')]$ (cf.\ \ref{lem32}).\newline
5. For the fact that $(\s E_j,T_j)$ is a $\QG_j$-equivariant Kasparov $A_j$-$B_j$-bimodule for $j=1,2$, we refer to \S \ref{sectionColinkKasp}. Let $X,Y\in\Lin(\s D)$ be the operators defined by $X(\zeta):=\pi_R(F)\circ\zeta$ and $Y(\zeta):=\pi_R(T)\circ\zeta$ for all $\zeta\in\s D$. It follows from b) and \ref{theo3} that $[(\s D,X)]=[(\s D,Y)]$ in $\kk_{\cal G}(D_{\rm g},D_{\rm d})$. By composing by $J_{\QG_j,{\cal G}}:\kk_{\cal G}(D_{\rm g},D_{\rm d})\rightarrow\kk_{\QG_j}(D_{{\rm g},j},D_{{\rm d},j})$, we have $[(\s D_j,X_j)]=[(\s D_j,Y_j)]$ for all $j=1,2$ (cf.\ \S \ref{sectionColinkKasp}). Hence, we have (cf.\ \ref{lem18} 2)
\begin{center}
$[({\cal E}_{ll',j},(\id_{\K(\s E_j)}\tens R_{jl})\delta_{\K(\s E_l)}^j(F_l))]=[({\cal E}_{ll',j},(\id_{\K(\s E_j)}\tens R_{jl})\delta_{\K(\s E_l)}^j(T_l))]$\quad for all $j,l,l'=1,2$.
\end{center}
In particular, we have 
\begin{center}
$[({\cal E}_{j,j},(\id_{\K(\s E_j)}\tens R_{jj})\delta_{\K(\s E_j)}^j(F_j))]=[({\cal E}_{j,j},(\id_{\K(\s E_j)}\tens R_{jj})\delta_{\K(\s E_j)}^j(T_j))]$\quad  for $j=1,2$.
\end{center}
By composing by the isomorphism $x\mapsto{\f a}_{A_j}\tens_{D_{{\rm g},j}}x\tens_{D_{{\rm d},j}}{\f b}_{B_j}$, we obtain $[(\s E_j,F_j)]=[(\s E_j,T_j)]$ (cf.\ \cite{BS1}).\qedhere
\par}
\end{proof}

In virtue of \ref{prop28} 1-4, the following definition makes sense.

\begin{defin}\label{def11}
Let $j,k=1,2$ such that $j\neq k$. Let $F_j\in\Lin(\s E_j)$ such that the pair $(\s E_j,F_j)$ is a $\QG_j$-equivariant Kasparov $A_j$-$B_j$-bimodule. Let $x:=[(\s E_j,F_j)]\in\kk_{\QG_j}(A_j,B_j)$ and $J_{\QG_k,\QG_j}(x)=[(\s E_k,F_k)]$. Let $F:=F_1\oplus F_2\in\Lin(\s E)$. We denote by $J_{{\cal G},\QG_j}(x)$ the unique element $y\in\kk_{\cal G}(A,B)$ such that ${\f b}_A\tens_A y \tens_B{\f a}_B=[(\s D,\pi_R(F))]$.
For $j=1,2$, we have a well-defined homomorphism of abelian groups $J_{{\cal G},\QG_j}:\kk_{\QG_j}(A_j,B_j)\rightarrow\kk_{\cal G}(A,B)$.
\end{defin}

\begin{lem}\label{lem20}
With the notations and hypotheses of \ref{def11}, if the pair $(\s E,F)$ is a $\cal G$-equivariant Kasparov $A$-$B$-bimodule, then we have $J_{{\cal G},\QG_j}(x)=[(\s E,F)]$.
\end{lem}

\begin{proof}
This is a straightforward consequence of Theorem \ref{theo3}.
\end{proof}

\begin{prop}\label{prop30} 
Let $j,k=1,2$ with $j\neq k$. We have:
\begin{enumerate}
\item $J_{\QG_j,{\cal G}}\circ J_{{\cal G},\QG_j}=\id_{\kk_{\QG_j}(A_j,B_j)}$;
\item $J_{\QG_k,{\cal G}}\circ J_{{\cal G},\QG_j}=J_{\QG_k,\QG_j}$;$\vphantom{\id_{\kk_{\QG_j}(A_j,B_j)}}$
\item $J_{\QG_k,\QG_j}\circ J_{\QG_j,{\cal G}}=J_{\QG_k,{\cal G}}$.\qedhere
\end{enumerate}
\end{prop}

\begin{proof}
The formulas 1 et 2 are immediate consequences of \ref{prop28} 5. The last statement follows by plugging the second formula in the left-hand side and by simplifying with the first one.
\end{proof}

We can state the main results of this paragraph.

\begin{thm}\label{theo4}
Let $j=1,2$. The maps
\[
J_{\QG_j,\,{\cal G}}:\kk_{\cal G}(A,B)\rightarrow\kk_{\QG_j}(A_j,B_j) \quad \text{and} \quad
J_{{\cal G},\QG_j}:\kk_{\QG_i}(A_j,B_j)\rightarrow\kk_{\cal G}(A,B)
\]
are isomorphisms of abelian groups inverse of each other.
\end{thm}

\begin{proof}
Let $j,k=1,2$ with $j\neq k$. It remains to prove that $J_{{\cal G},\QG_j}\circ J_{\QG_j,{\cal G}} =\id_{\kk_{\cal G}(A,B)}$ (cf.\ \ref{prop30} 1). Let $F\in\Lin(\s E)$ such that the pair $(\s E,F)$ is a $\cal G$-equivariant Kasparov $A$-$B$-bimodule. Let $x:=[(\s E,F)]\in\kk_{\cal G}(A,B)$. We have $F=F_1\oplus F_2$ with $F_1\in\Lin(\s E_1)$ and $F_2\in\Lin(\s E_2)$. It follows from \ref{prop30} 3 that $J_{\QG_k,\QG_j}([(\s E_j,F_j)]=[(\s E_k,F_k)]$. By applying \ref{lem20}, we then obtain $J_{{\cal G},\QG_j}(J_{\QG_j,{\cal G}}(x))=J_{{\cal G},\QG_j}([(\s E_j,F_j)])=[(\s E,F)]=x$.
\end{proof}

We then obtain another proof of Th\'eor\`eme 4.36 \cite{BC}.

\begin{cor}\label{cor4}
The map $J_{\QG_2,\QG_1}:\kk_{\QG_1}(A_1,B_1)\rightarrow\kk_{\QG_2}(A_2,B_2)$ is an isomorphism of abelian groups and $(J_{\QG_2,\QG_1})^{-1}=J_{\QG_1,\QG_2}$.
\end{cor}

\begin{proof}
This is an immediate consequence of \ref{prop30} 2 and \ref{theo4}.
\end{proof}

Let us fix a third $\QG_1$-C*-algebra $C_1$. Consider the induced $\QG_2$-C*-algebra $C_2:=\ind(C_1)$ and the $\cal G$-C*-algebra $C:=C_1\oplus C_2$.

\begin{prop}\label{prop44}
For $j=1,2$, we have:
\begin{enumerate}
\item $J_{{\cal G},\QG_j}(1_{A_j})=1_A$;
\item for all $x\in\kk_{\QG_j}(A_j,C_j)$ and $y\in\kk_{\QG_j}(C_j,B_j)$, $J_{{\cal G},\QG_j}(x\tens_{C_j}y)=J_{{\cal G},\QG_j}(x)\tens_C J_{{\cal G},\QG_j}(y)$ in $\kk_{\cal G}(A,B)$.\qedhere
\end{enumerate}
\end{prop}

\begin{proof}
This follows from \ref{theo4} and \ref{prop43}.
\end{proof}

\begin{prop}\label{prop45}
For $j,k=1,2$ with $j\neq k$, we have:
\begin{enumerate}
\item $J_{\QG_k,\QG_j}(1_{A_j})=1_{A_k}$;
\item for all $x\in\kk_{\QG_j}(A_j,C_j)$ and $y\in\kk_{\QG_j}(C_j,B_j)$, $J_{\QG_k,\QG_j}(x\tens_{C_j}y)=J_{\QG_k,\QG_j}(x)\tens_{C_k} J_{\QG_k,\QG_j}(y)$ in $\kk_{\QG_k}(A_k,B_k)$.\qedhere
\end{enumerate}
\end{prop}

\begin{proof}
This is a direct consequence of \ref{prop30} 2, \ref{prop43} and \ref{prop44}.
\end{proof}

\begin{nbs}
We denote by $\kk_{\cal G}$ (resp.\ $\kk_{\QG_j}$ for $j=1,2$) the category of separable $\cal G$ (resp.\ $\QG_j$)-C*-algebras whose set of arrows between two $\cal G$ (resp.\ $\QG_j$)-C*-algebras $A$ and $B$ is the equivariant Kasparov group $\kk_{\cal G}(A,B)$ (resp.\ $\kk_{\QG_j}(A,B)$).
\end{nbs}

\begin{thm}
We have:
\begin{enumerate}
\item for $j=1,2$, the correspondences $J_{\QG_j,{\cal G}}:\kk_{\cal G}\rightarrow\kk_{\QG_j}$ and $J_{{\cal G},\QG_j}:\kk_{\QG_j}\rightarrow\kk_{\cal G}$ are equivalences of categories inverse of each other;
\item the correspondences $J_{\QG_2,\QG_1}:\kk_{\QG_1}\rightarrow\kk_{\QG_2}$ and $J_{\QG_1,\QG_2}:\kk_{\QG_2}\rightarrow\kk_{\QG_1}$ are equivalences of categories inverse of each other.\qedhere
\end{enumerate}
\end{thm}

\begin{proof}
The first (resp.\ second) statement is just a restatement of \ref{theo9}, \ref{prop44} and \ref{theo4} (resp.\ \ref{theo10}, \ref{prop45} and \ref{cor4}).
\end{proof}

\section{Appendix}\label{appendix}

	\subsection[Normal linear forms, weights and operator-valued weights]{Normal linear forms, weights and operator-valued weights on von Neumann algebras \cite{Co2}}\label{integration}

\setcounter{thm}{0}

\numberwithin{thm}{subsection}
\numberwithin{prop}{subsection}
\numberwithin{lem}{subsection}
\numberwithin{cor}{subsection}
\numberwithin{propdef}{subsection}
\numberwithin{nb}{subsection}
\numberwithin{nbs}{subsection}
\numberwithin{rk}{subsection}
\numberwithin{rks}{subsection}
\numberwithin{defin}{subsection}
\numberwithin{ex}{subsection}
\numberwithin{exs}{subsection}
\numberwithin{noh}{subsection}

Let $M$ be a von Neumann algebra. Denote by $M_*$ (resp.\ $M_*^+$) the Banach space (resp.\ positive cone) of the normal linear forms (resp.\ positive normal linear forms) on $M$. Let $\omega\in M_*$ and $a,b\in M$. Denote by $a\omega\in M_*$ and $\omega b\in M_*$ the normal linear functionals on $M$ given for all $x\in M$ by: 
\[
(a\omega)(x):=\omega(xa);\quad (\omega b)(x):=\omega(bx).
\]
We have $a'(a\omega)=(a'a)\omega$ and $(\omega b)b'=\omega(bb')$, for all $a,a,b,b'\in M$. We also denote 
\[
a\omega b:=a(\omega b)=(a \omega) b ; \quad \omega_a:=a^*\omega a.
\]
If $\omega\in M_*^+$, then $\omega_a\in M_*^+$. Note that $(\omega_a)_b=\omega_{ab}$ for all $a,b\in M$. If $\omega\in M_*$ we define $\overline{\omega}\in M_*$\index[symbol]{oa@$\overline{\omega}$} by setting
\[
\overline{\omega}(x):=\overline{\omega(x^*)},\quad \text{for all } x\in M.
\]
Let $\s H$ be a Hilbert space and let us fix $\xi,\eta\in\s H$. Denote by $\omega_{\xi,\eta}\in\B(\s H)_*$\index[symbol]{ob@$\omega_{\xi,\eta}$} the normal linear form defined by
\[
\omega_{\xi,\eta}(x):=\langle\xi,\, x\eta\rangle,\quad \text{for all } x\in\B(\s H).
\]
Note that we have $\overline{\omega}_{\xi,\eta}=\omega_{\eta,\xi}$, $a\omega_{\xi,\eta}=\omega_{\xi,a\eta}$ and $\omega_{\xi,\eta}a=\omega_{a^*\xi,\eta}$ for all $a\in\B(\s H)$.

\begin{noh}{\it Tensor product of normal linear forms.} Let $M$ and $N$ be von Neumann algebras, $\phi\in M_*$ and $\psi\in N_*$. There exists a unique $\phi\tens\psi\in(M\tens N)_*$ such that $(\phi\tens\psi)(x\tens y)=\phi(x)\psi(y)$ for all $x\in M$ and $y\in N$. Moreover, $\|\phi\tens\psi\|\leqslant\|\phi\|\cdot\|\psi\|$. Actually, it is known that we have an (completely) isometric identification $M_*\widehat{\tens}_{\pi}N_*=(M\tens N)_*$, where $\widehat{\tens}_{\pi}$ denotes the projective tensor product of Banach spaces. In particular, any $\omega\in(M\tens N)_*$ is the norm limit of finite sums of the form $\sum_i\phi_i\tens\psi_i$, where $\phi_i\in M_*$ and $\psi_i\in N_*$.
\end{noh}

\begin{noh}{\it Slicing with normal linear forms.} We will also need to slice maps with normal linear forms. Let $M_1$ and $M_2$ be von Neumann algebras, $\omega_1\in (M_1)_*$ and $\omega_2\in(M_2)_*$. Therefore, the maps $\omega_1\odot\id:M_1\odot M_2\rightarrow M_1$ and $\id\odot\omega_2:M_1\odot M_2\rightarrow M_2$ extend uniquely to norm continuous normal linear maps $\omega_1\tens\id:M_1\tens M_2\rightarrow M_2$ and $\id\tens\omega_2:M_1\tens M_2\rightarrow M_1$. 
Let $\s H$ and $\s K$ be Hilbert spaces, for $\xi\in\s H$ and $\eta\in\s K$ we define 
$\theta_{\xi}\in\B(\s K,\s H\tens\s K)$ and $\theta'_{\eta}\in\B(\s H,\s H\tens\s K)$
by setting: 
\begin{center}
$\theta_{\xi}(\zeta):=\xi\tens\zeta$, \quad for all $\zeta\in\s K$;\quad $\theta'_{\eta}(\zeta):=\zeta\tens\eta$, \quad for all $\zeta\in\s H$.
\end{center} 
If $T\in\B(\s H\tens\s K)$, $\phi\in\B(\s K)_*$ and $\omega\in\B(\s H)_*$, then the operators $(\id\tens\phi)(T)\in\B(\s H)$ and $(\omega\tens\id)(T)\in\B(\s K)$ are determined by the formulas:
\begin{align*}
\langle\xi_1,\,(\id\tens\phi)(T)\xi_2\rangle &=\phi(\theta_{\xi_1}^*T\theta_{\xi_2}),\quad \xi_1,\,\xi_2\in\s H;\\[0.3cm]
\langle\eta_1,\,(\omega\tens\id)(T)\eta_2\rangle &=\omega(\theta_{\eta_1}^{\prime *}T\theta'_{\eta_2}),\quad \eta_1,\,\eta_2\in\s K.
\end{align*}
In particular, we have:
\[
(\id\tens\omega_{\eta_1,\eta_2})(T)=\theta_{\eta_1}^{\prime *}T\theta'_{\eta_2},\quad \eta_1,\,\eta_2\in\s K;\quad (\omega_{\xi_1,\xi_2}\tens\id)(T)=\theta_{\xi_1}^*T\theta_{\xi_2},\quad \xi_1,\,\xi_2\in\s H.
\]
Let us recall some formulas that will be used several times. For all $\phi\in\B(\s K)_*$, $\omega\in\B(\s H)_*$ and $T\in\B(\s H\tens\s K)$, we have:
\[
x(\id\tens\phi)(T)y=(\id\tens\phi)((x\tens 1)T(y\tens 1)),\ (y\omega x\tens\id)(T)=(\omega\tens\id)((x\tens 1)T(y\tens 1))
\]
for all $x,y\in\B(\s H)$;
\[
a(\omega\tens\id)(T)b=(\omega\tens\id)((1\tens a)T(1\tens b)), \ (\id\tens b\phi a)(T)=(\id\tens\phi)((1\tens a)T(1\tens b))
\]
for all $a,b\in\B(\s K)$. We also have 
\[
(\id\tens\phi)(T)^*=(\id\tens\overline{\phi})(T^*),\quad (\omega\tens\id)(T)^*=(\overline{\omega}\tens\id)(T^*),
\]
\[
(\phi\tens\id)(\Sigma_{\s H\tens\s K} T\Sigma_{\s K\tens\s H})=(\id\tens\phi)(T),\quad (\id\tens\omega)(\Sigma_{\s H\tens\s K} T\Sigma_{\s K\tens\s H})=(\omega\tens\id)(T),
\]
for all $T\in\B(\s H\tens\s K)$, $\phi\in\B(\s K)_*$ and $\omega\in\B(\s H)_*$.
\end{noh}

\begin{defin}\label{weight}
A weight $\varphi$ on $M$ is a map $\varphi:M_+\rightarrow[0,\infty]$ such that: 
\begin{itemize}
	\item for all $x,y\in M_+$, $\varphi(x+y)=\varphi(x)+\varphi(y)$;
	\item for all $x\in M_+$ and $\lambda\in\GR_+$, $\varphi(\lambda x)=\lambda\varphi(x)$.
\end{itemize}
We denote by $\f N_{\varphi}:=\{x\in M\,;\,\varphi(x^*x)<\infty\}$
\index[symbol]{na@$\f N_{\varphi}$, $\f M_{\varphi}^+$} 
the left ideal of square $\varphi$-integrable elements of $M$, $\f M_{\varphi}^+:=\{x\in M_+\,;\,\varphi(x)<\infty\}$ the cone of positive $\varphi$-integrable elements of $M$ and $\f M_{\varphi}:=\langle\f M_{\varphi}^+\rangle$ the space of $\varphi$-integrable elements of $M$.
\end{defin}	

\begin{defin}
Let $\varphi$ be a weight on $M$. The opposite weight of $\varphi$ is the weight $\varphi^{\rm o}$\index[symbol]{pn@$\varphi^{\rm o}$, opposite weight} on $M^{\rm o}$ given by $\varphi ^{\rm o}(x^{\rm o}):=\varphi(x)$ for all $x\in M_+$. Then, we have $\f N_{\varphi^{\rm o}}=(\f N_{\varphi}^*)^{\rm o}$, $\f M_{\varphi^{\rm o}}^+=(\f M_{\varphi}^+)^{\rm o}$ and $\f M_{\varphi^{\rm o}}=(\f M_{\varphi})^{\rm o}$.
\end{defin}

\begin{defin}\label{nsf}
A weight $\varphi$ on $M$ is called:
\begin{itemize}
	\item semi-finite, if $\f N_{\varphi}$ is $\sigma$-weakly dense in $M$;
	\item faithful, if for $x\in M_+$ the condition $\varphi(x)=0$ implies $x=0$;
	\item normal, if $\varphi(\sup_{i\in\cal I} x_i)=\sup_{i\in\cal I}\varphi(x_i)$ for all increasing bounded net $(x_i)_{i\in\cal I}$ of $M_+$.\qedhere
\end{itemize}
\end{defin}

From now on, we will mainly use normal semi-finite faithful (n.s.f.)\ weights. Fix a \nsf weight $\varphi$ on $M$.

\begin{defin}\label{GNS}
We define an inner product on $\f N_{\varphi}$ by setting
\[
\langle x,\,y\rangle_{\varphi}:=\varphi(x^*y),\quad \text{for all } x,\,y\in\f N_{\varphi}.
\]
We denote by $(\s H_{\varphi},\Lambda_{\varphi})$ the Hilbert space completion of $\f N_{\varphi}$ with respect to this inner product, where $\Lambda_{\varphi}:\f N_{\varphi}\rightarrow\s H_{\varphi}$ is the canonical map. There exists a unique unital normal *-representation $\pi_{\varphi}:M\rightarrow\B(\s H_{\varphi})$ such that
\[
\pi_{\varphi}(x)\Lambda_{\varphi}(y)=\Lambda_{\varphi}(xy),\quad \text{for all } x\in M \text{ and } y\in\f N_{\varphi}.
\]
The triple $(\s H_{\varphi},\pi_{\varphi},\Lambda_{\varphi})$ is called the \GNS construction for $(M,\varphi)$.
\end{defin}

\begin{rks}
The linear map $\Lambda_{\varphi}$ is called the \GNS map. We have that $\Lambda_{\varphi}(\f N_{\varphi})$ is dense in $\s H_{\varphi}$ and $\langle\Lambda_{\varphi}(x),\Lambda_{\varphi}(y)\rangle_{\varphi}=\varphi(x^*y)$ for all $x,y\in\f N_{\varphi}$. In particular, $\Lambda_{\varphi}$ is injective. Moreover, we also call $\pi_{\varphi}$ the \GNS representation.
\end{rks}

We recall below the main objects of the Tomita-Takesaki modular theory.

\begin{propdef}
Let $M$ be a von Neumann algebra and $\varphi$ a \nsf weight on $M$. The anti-linear map
$
\Lambda_{\varphi}(\f N_{\varphi}^*\cap\f N_{\varphi}) \rightarrow \Lambda_{\varphi}(\f N_{\varphi}^*\cap\f N_{\varphi}) \; ; \; \Lambda_{\varphi}(x) \mapsto \Lambda_{\varphi}(x^*)
$
is closable and its closure is a possibly unbounded anti-linear map ${\cal T}_{\varphi}:D({{\cal T}_{\varphi}})\subset\s H_{\varphi}\rightarrow\s H_{\varphi}$ such that $D({{\cal T}_{\varphi}})={\rm im}\,{{\cal T}_{\varphi}}$ and ${\cal T}_{\varphi}\circ{\cal T}_{\varphi}(x)=x$ for all $x\in D({{\cal T}_{\varphi}})$.\newline
Let
$
{\cal T}_{\varphi}=J_{\varphi}\nabla_{\varphi}^{1/2}
$
be the polar decomposition of ${\cal T}_{\varphi}$. The anti-unitary $J_{\varphi}:\s H_{\varphi}\rightarrow\s H_{\varphi}$ is called the modular conjugation for $\varphi$ and the injective positive self-adjoint operator $\nabla_{\varphi}$ is called the modular operator for $\varphi$.
\end{propdef}

\begin{propdef}
There exists a unique one-parameter group $(\sigma_t^{\varphi})_{t\in\GR}$ of automorphisms on $M$, called the modular automorphism group of $\varphi$, such that
\[
\pi_{\varphi}(\sigma_t^{\varphi}(x))=\nabla_{\varphi}^{{\rm i}t}\pi_{\varphi}(x)\nabla_{\varphi}^{-{\rm i}t},\quad \text{for all } t\in\GR \text{ and } x\in M.
\]
Then, for all $t\in\GR$ and $x\in M$ we have $\sigma_t^{\varphi}(x)\in\f N_{\varphi}$ and 
$\Lambda_{\varphi}(\sigma_t^{\varphi}(x))=\nabla_{\varphi}^{{\rm i}t}\Lambda_{\varphi}(x)$.
\end{propdef}

\begin{propdef}
The map $C_M:M\rightarrow M'\,;\, x \mapsto J_{\varphi}\pi_{\varphi}(x)^*J_{\varphi}$ is a normal unital *-antihomomorphism.\index[symbol]{cb@$C_M$}
\end{propdef}

\begin{defin}
Let $N$ be a von Neumann algebra. The extended positive cone of $N$ is the set $N_+^{\ext}$ consisting of the maps $m:N_*^+\rightarrow[0,\infty]$, which satisfy the following conditions:\index[symbol]{nb@$N_+^{{\rm ext}}$, extended positive cone}
\begin{itemize}
	\item for all $\omega_1,\omega_2\in N_*^+$, $m(\omega_1+\omega_2)=m(\omega_1)+m(\omega_2)$;
	\item for all $\omega\in N_*^+$ and $\lambda\in\GR_+$, $m(\lambda\omega)=\lambda m(\omega)$;
	\item $m$ is lower semicontinuous with respect to the norm topology on $N_*$.\qedhere
\end{itemize}
\end{defin}

\begin{nbs}
Let $N$ be a von Neumann algebra.
\begin{enumerate}
	\item From now on, we will identify $N_+$ with its part inside $N_+^{\rm ext}$. Accordingly, if $m\in N_+^{\ext}$ and $\omega\in N_*^+$ we will denote by $\omega(m)$ the evaluation of $m$ at $\omega$.
	\item Let $a\in N$ and $m\in N_+^{\ext}$, we define $a^*m a\in N_+^{\ext}$ by setting
	$\omega(a^*ma):=a\omega a^*(m)$ for all $\omega\in N_*^+$.
If $m,n\in N_+^{\ext}$ and $\lambda\in\GR_+$, we also define $m+n\in N_+^{\ext}$ and $\lambda m\in N_+^{\ext}$ by setting
	$\omega(m+n):=\omega(m)+\omega(n)$ and $\omega(\lambda m):=\lambda\omega(m)$ for all $\omega\in N_*^+$.\qedhere
\end{enumerate}
\end{nbs}

\begin{defin}
Let $N\subset M$ be a unital normal inclusion of von Neumann algebras. An operator-valued weight from $M$ to $N$ is a map $T:M_+\rightarrow N_+^{\rm ext}$ such that:
\begin{itemize}
	\item for all $x,y\in M_+$, $T(x+y)=T(x)+T(y)$;
	\item for all $x\in M_+$, $\forall\lambda\in\GR_+$, $T(\lambda x)=\lambda T(x)$;
	\item for all $x\in M_+$ and $a\in N$, $T(a^*xa)=a^*T(x)a$.
\end{itemize}
Let $\f N_T:=\{x\in M\,;\,T(x^*x)\in N_+\}$, $\f M_T^+:=\{x\in M_+\,;\,T(x)\in N_+\}$ and ${\f M}_T:=\langle{\f M}_T^+\rangle$.\index[symbol]{nc@${\f N}_T$, ${\f M}_T^+$, ${\f M}_T$}
\end{defin}

\begin{defin}
Let $N\subset M$ be a unital normal inclusion of von Neumann algebras. An operator-valued weight $T$ from $M$ to $N$ is said to be:
\begin{itemize}
	\item semi-finite, if $\f N_T$ is $\sigma$-weakly dense in $M$;
	\item faithful, if for $x\in M_+$ the condition $T(x)=0$ implies $x=0$;
	\item normal, if for every increasing bounded net $(x_i)_{i\in\cal I}$ of elements of $M_+$ and $\omega\in N_*^+$, we have $\omega(T(\sup_{i\in\cal I} x_i))=\lim_{i\in\cal I}\omega(T(x_i))$.\qedhere
\end{itemize}
\end{defin}

Note that if $T:M_+\rightarrow N_+^{\ext}$ is an operator-valued weight, it extends uniquely to a semi-linear map $\overline{T}:M_+^{\ext}\rightarrow N_+^{\ext}$. This will allow us to compose \nsf operator-valued weights. Indeed, let $P\subset N\subset M$ be unital normal inclusions of von Neumann algebras. Let $S$ (resp.\ $T$) be an operator-valued weight from $N$ (resp.\ $M$) to $P$ (resp.\ $N$). We define an operator-valued weight from $M$ to $P$ by setting $(S\circ T)(x):=\overline{S}(T(x))$ for all $x\in N_+$.
	
	\subsection[Relative tensor product and fiber product]{Relative tensor product of Hilbert spaces and fiber product of von Neumann algebras}\label{tensorproduct}
	
In this paragraph, we will recall the definitions, notations and important results concerning the  relative tensor product and the fiber product which are the main technical tools of the theory of measured quantum groupoids. For more information, we refer the reader to \cite{Co}.

\medbreak

In the whole section, $N$ is a von Neumann algebra endowed with a \nsf weight $\varphi$. Let $\pi$ (resp.\ $\gamma$) be a normal unital *-representation of $N$ (resp.\ $N^{\rm o}$) on a Hilbert space $\cal H$ (resp.\ $\cal K$).

\paragraph{Relative tensor product.} The Hilbert space $\cal H$ (resp.\ $\cal K$) may be considered as a left (resp.\ right) $N$-module. Moreover, $\s H_{\varphi}$ is an $N$-bimodule whose actions are given by
\begin{center}
$x\xi:=\pi_{\varphi}(x)\xi$ \quad and \quad $\xi y:=J_{\varphi}\pi_{\varphi}(y^*)J_{\varphi}\xi$,\quad for all $\xi\in\s H_{\varphi}$ and $x,\,y\in N$.
\end{center}

\begin{defin}\label{defLeftBoundedVector}
We define the set of right (resp.\ left) bounded vectors with respect to $\varphi$ and $\pi$ (resp.\ $\gamma$) to be:
\begin{align*}
_{\varphi}(\pi,{\cal H})&:=\{\xi\in{\cal H}\,;\,\exists\,C\in\GR_+,\,\forall\,x\in\f N_{\varphi},\,\|\pi(x)\xi\|\leqslant C\|\Lambda_{\varphi}(x)\|\},\\[0.3cm]
\text{(resp.\ }({\cal K},\gamma)_{\varphi}&:=\{\xi\in{\cal K}\,;\,\exists\,C\in\GR_+,\,\forall\,x\in\f N_{\varphi}^*,\,\|\gamma(x^{\rm o})\xi\|\leqslant C\|\Lambda_{\varphi^{\rm o}}(x^{\rm o})\|\}\text{{\rm)}}.
\end{align*}

If $\xi\in{}_{\varphi}(\pi,{\cal H})$, we denote by $R^{\pi,\varphi}_{\xi}\in\B(\s H_{\varphi},{\cal H})$\index[symbol]{r@$R_{\xi}^{\pi}$, $L_{\eta}^{\gamma}$} (or simply $R^{\pi}_{\xi}$ if $\varphi$ is understood) the unique bounded operator such that 
\[
R^{\pi,\varphi}_{\xi}\Lambda_{\varphi}(x)=\pi(x)\xi, \quad \text{for all }x\in\f N_{\varphi}.
\]
Similarly, if $\xi\in({\cal K},\gamma)_{\varphi}$ we denote $L^{\gamma,\varphi}_{\xi}\in\B(\s H_{\varphi},{\cal K})$ (or simply $L_{\xi}^{\gamma}$ if $\varphi$ is understood) the unique bounded operator such that
\[
L^{\gamma,\varphi}_{\xi}J_{\varphi}\Lambda_{\varphi}(x^*)=\gamma(x^{\rm o})\xi, \quad \text{for all }x\in\f N_{\varphi}^*,
\]
where we have used the identification $\s H_{\varphi^{\rm o}}\rightarrow\s H_{\varphi}\,;\,\Lambda_{\varphi^{\rm o}}(x^{\rm o})\mapsto J_{\varphi}\Lambda_{\varphi}(x^*)$.
\end{defin}

Note that $\xi\in\cal K$ is left bounded with respect to $\varphi$ and $\gamma$ if, and only if, it is right bounded with respect to the \nsf weight $\varphi^{\rm c}:=\varphi\circ C_N^{-1}$\index[symbol]{po@$\varphi^{\rm c}$, commutant weight} on $N'$ and the normal unital *-representation $\gamma^{\rm c}:=\gamma\circ C_N^{-1}$ of $N'$. It is important to note that $(\K,\gamma)_{\varphi}$ (resp.\ $_{\varphi}(\pi,\cal H)$) is dense in $\cal K$ (resp.\ $\cal H$) (cf.\ Lemma 2 of \cite{Co}).

\medskip

If $\xi\in{}_{\varphi}(\pi,{\cal H})$ (resp.\ $\xi\in({\cal K},\gamma)_{\varphi}$), we have that $R^{\pi,\varphi}_{\xi}$ (resp.\ $L^{\gamma,\varphi}_{\xi}$) is left (resp.\ right) $N$-linear. Therefore, for all $\xi,\eta\in{}_{\varphi}(\pi,{\cal H})$ (resp.\ $({\cal K},\gamma)_{\varphi}$) we have
\[
(R^{\pi,\varphi}_{\xi})^*R^{\pi,\varphi}_{\eta}\in\pi_{\varphi}(N)'=C_N(N) \;\; \text{and} \;\;
R^{\pi,\varphi}_{\xi}(R^{\pi,\varphi}_{\eta})^*\in\pi(N)'
\]
\[
\text{(resp.\ }(L^{\gamma,\varphi}_{\xi})^*L^{\gamma,\varphi}_{\eta}\in\pi_{\varphi}(N)  \;\; \text{and} \;\; 
L^{\gamma,\varphi}_{\xi}(L^{\gamma,\varphi}_{\eta})^*\in\gamma(N^{\rm o})'\text{{\rm)}}.
\]

\begin{nbs}\label{not18}
(cf.\ 2.1 \cite{E05}) Let\index[symbol]{kf@$\K_{\pi}$, $\K_{\gamma}$}
\[
\K_{\pi,\varphi}:=[R^{\pi,\varphi}_{\xi} (R^{\pi,\varphi}_{\eta})^*\,;\,\xi,\eta\in{}_{\varphi}(\pi,{\cal H})]\quad
\text{(resp.\ }\K_{\gamma,\varphi}:=[L^{\gamma,\varphi}_{\xi}(L^{\gamma,\varphi}_{\eta})^*\,;\, \xi,\eta\in({\cal H},\gamma)_{\varphi}]\text{{\rm)}}.
\]
Note that $\K_{\pi,\varphi}$ {\rm(}resp.\ $\K_{\gamma,\varphi}${\rm )} is a weakly dense ideal of $\pi(N)'$ {\rm(}resp.\ $\gamma(N^{\rm o})'${\rm)} {\rm(}cf. Proposition 3 of \cite{Co}{\rm)}. If $\varphi$ is understood, we denote $\K_{\pi}$ {\rm(}resp.\ $\K_{\gamma})$ instead of $\K_{\pi,\varphi}$ {\rm(}resp.\ $\K_{\gamma,\varphi})$.
\end{nbs}

\begin{nbs}
Let $\xi,\eta\in{}_{\varphi}(\pi,{\cal H})$ {\rm(}resp.\ $({\cal K},\gamma)_{\varphi})$, we denote
\[
\langle\xi,\,\eta\rangle_{N^{\rm o}}:=C_N^{-1}((R^{\pi,\varphi}_{\xi})^*R^{\pi,\varphi}_{\eta})^{\rm o}\in N^{\rm o}\quad 
\text{(resp.\ }\langle\xi,\,\eta\rangle_N:=\pi_{\varphi}^{-1}((L^{\gamma,\varphi}_{\xi})^*L^{\gamma,\varphi}_{\eta})\in N\text{{\rm)}}.\qedhere
\]
\end{nbs}

\begin{prop}
For all $\xi , \eta \in {}_\varphi(\pi,{\cal H})$ {\rm(}resp.\ $\xi , \eta \in ({\cal K}, \gamma)_\varphi)$ and $y\in N$ analytic for $(\sigma_t^\varphi)_{t\in\GR}$,  we have:
\begin{enumerate}
	\item $\langle \xi , \eta \rangle_{N^{\rm o}}^* = \langle \eta, \xi \rangle_{N^{\rm o}}$ {\rm(}resp.\ $\langle \xi , \eta \rangle_{ N}^* = \langle \eta, \xi \rangle_{N}${\rm)};
	\item $\langle \xi , \eta y^{\rm o}\rangle_{N^{\rm o}} =  \langle \xi , \eta \rangle_{N^{\rm o}} \sigma_{{\rm i}/2}^\varphi(y)^{\rm o}$ {\rm(}resp.\ $\langle \xi , \eta y \rangle_N =  \langle \xi , \eta \rangle_N\sigma_{-{{\rm i}/2}}^\varphi(y)${\rm)}.\qedhere
\end{enumerate}
\end{prop}

\begin{lem}
For all $\xi_1 ,\, \xi_2 \in{}_\varphi(\pi,{\cal H})$ and $\eta_1 ,\, \eta_2 \in ({\cal K}, \gamma)_\varphi$, we have 
\[
\langle \eta_1 ,\, \gamma(\langle \xi_1 ,\,  \xi_2 \rangle_{N^{\rm o}})\eta_2 \rangle_{\cal K} = 
\langle \xi_1 , \pi(\langle \eta_1 ,\,  \eta_2 \rangle_N)\xi_2 \rangle_{\cal H}.\qedhere
\]
\end{lem}

\begin{defin}
The relative tensor product 
\[
\cst{\cal K}{\gamma}{\varphi}{\pi}{\cal H}\quad \text{(or simply denoted by }\reltens{\cal K}{\gamma}{\pi}{\cal H}\text{{\rm)}}
\]
is the Hausdorff completion of the pre-Hilbert space $({\cal K}, \gamma)_\varphi \odot {}_\varphi(\pi,{\cal H})$, whose inner product is given by
\[
\langle \eta_1 \otimes \xi_1 ,\, \eta_2 \otimes \xi_2 \rangle := \langle \eta_1 ,\, \gamma(\langle \xi_1 ,\,  \xi_2 \rangle_{N^{\rm o}})\eta_2 \rangle_{\cal K} = 
\langle \xi_1 ,\, \pi(\langle \eta_1 ,\,  \eta_2 \rangle_N)\xi_2 \rangle_{\cal H},
\]
for all $\eta_1,\,\eta_2\in({\cal K}, \gamma)_\varphi$ and $\xi_1,\xi_2\in{}_\varphi(\pi,{\cal H})$. If $\eta\in({\cal K}, \gamma)_\varphi$ and $\xi\in{}_\varphi(\pi,{\cal H})$, we will denote by
\[
\cst{\eta}{\gamma}{\varphi}{\pi}{\xi} \quad \text{(or simply }\reltens{\eta}{\gamma}{\pi}{\xi}\text{{\rm)}}
\]
the image of $\eta\tens\xi$ by the canonical map $({\cal K}, \gamma)_\varphi \odot {}_\varphi(\pi,{\cal H})\rightarrow\reltens{\cal K}{\gamma}{\pi}{\cal H}$ (isometric dense range).
\end{defin}

\begin{rks}
\begin{enumerate}
	\item By applying this construction to $(N^{\rm o},\varphi^{\rm o})$ instead of $(N,\varphi)$ we obtain the relative tensor product $\cst{\cal H}{\pi}{\varphi^{\rm o}}{\gamma}{\cal K}$.
	\item The relative tensor product $\reltens{\cal K}{\gamma}{\pi}{\cal H}$ is also the Hausdorff completion of the pre-Hilbert space $({\cal K},\gamma)_{\varphi}\odot{\cal H}$ (resp.\ ${\cal K}\odot{}_{\varphi}(\pi,{\cal H})$), whose inner product is given by:
	\begin{align*}	
	\langle \eta_1 \otimes \xi_1 ,\, \eta_2 \otimes \xi_2 \rangle & :=
	\langle \xi_1 ,\, \pi(\langle \eta_1 ,\,  \eta_2 \rangle_N)\xi_2 \rangle_{\cal H}\\[0.3cm]
	(\text{resp. }\langle \eta_1 \otimes \xi_1 ,\, \eta_2 \otimes \xi_2 \rangle & := \langle \eta_1 ,\, \gamma(\langle \xi_1 ,\,  \xi_2 \rangle_{N^{\rm o}})\eta_2 \rangle_{\cal K}).
	\end{align*}
	\item Moreover, for all $\eta\in{\cal K} ,\, \xi\in {}_\varphi(\pi,{\cal H})$ and $y\in N$ analytic for $(\sigma_t^\varphi)_{t\in\GR}$ we have
\[
\reltens{\gamma(y^{\rm o})\eta}{\gamma}{\pi}{\xi} =  \reltens{\eta}{\gamma}{\pi}{\pi(\sigma_{-{{\rm i}/2}}^\varphi(y))\xi}.
\qedhere\]
\end{enumerate}
\end{rks}

\begin{noh}
The relative flip map is the isomorphism $\sigma_{\varphi}^{\gamma\pi}$ from $\cst{\cal K}{\gamma}{\varphi}{\pi}{\cal H}$ onto $\cst{\cal H}{\pi}{\varphi^{\rm o}}{\gamma}{\cal K}$ given by:
\[
\sigma_{\varphi}^{\gamma\pi}(\cst{\eta}{\gamma}{\varphi}{\pi}{\xi}):=\cst{\xi}{\pi}{\varphi^{\rm o}}{\gamma}{\eta},\quad \text{for all } \xi\in({\cal K},\gamma)_{\varphi} \text{ and } \eta\in{}_{\varphi}(\pi,{\cal H})\quad \text{(or simply }\sigma_{\gamma\pi}\text{{\rm)}}.
\]
Note that $\sigma_{\varphi}^{\gamma\pi}$ is unitary and $(\sigma_{\varphi}^{\gamma\pi})^*=\sigma_{\varphi^{\rm o}}^{\pi\gamma}$. Then, we can define a relative flip *-homomorphism 
\[
\varsigma_{\varphi}^{\gamma\pi}:\B(\cst{\cal K}{\gamma}{\varphi}{\pi}{\cal H})\rightarrow\B(\cst{\cal H}{\pi}{\varphi^{\rm o}}{\gamma}{\cal K})\quad \text{(or simply denoted by } \varsigma_{\gamma\pi}\text{{\rm)}}
\]
by setting
$
\varsigma_{\varphi}^{\gamma\pi}(X):=\sigma_{\varphi}^{\gamma\pi}X(\sigma_{\varphi}^{\gamma\pi})^*
$
for all $X\in\B(\cst{\cal K}{\gamma}{\varphi}{\pi}{\cal H})$.\index[symbol]{sg@$\sigma_{\gamma\pi}$/$\varsigma_{\gamma\pi}$, relative flip map/*-homomorphism}
\end{noh}

\paragraph{Fiber product of von Neumann algebras.} We continue to use the notations of the previous paragraph.

\begin{propdef}\label{propdef2}
Let ${\cal K}_i$ and ${\cal H}_i$ be Hilbert spaces, and $\gamma_i:N^{\rm o}\rightarrow\B({\cal K}_i)$ and $\pi_i:N\rightarrow\B({\cal H}_i)$ be unital normal *-homomorphisms for $i=1,2$. Let $T\in\B({\cal K}_1,{\cal K}_2)$ and $S\in\B({\cal H}_1,{\cal H}_2)$ such that 
$T\circ\gamma_1(n^{\rm o})=\gamma_2(n^{\rm o})\circ T$ and $S\circ\pi_1(n)=\pi_2(n)\circ S$ for all $n\in N$.
Then, the linear map
\begin{center}
$
({\cal K}_1,\gamma_1)_{\varphi}\odot{}_{\varphi}(\pi_1,{\cal H}_1) \rightarrow \reltens{{\cal K}_2}{\gamma_2}{\pi_2}{{\cal H}_2} \; ; \; \xi\odot\eta \mapsto \reltens{T\xi}{\gamma_2}{\pi_2}{S\eta}
$
\end{center}
extends uniquely to a bounded operator 
$
_{\gamma_2}\reltens{T}{\gamma_1}{\pi_2}{S}_{\pi_1}\in\B(\reltens{{\cal K}_1}{\gamma_1}{\pi_1}{{\cal H}_1},\reltens{{\cal K}_2}{\gamma_2}{\pi_2}{{\cal H}_2})
$
(or simply denoted by $\reltens{T}{\gamma_1}{\pi_2}{S}$), whose adjoint operator is $_{\gamma_1}\reltens{T^*}{\gamma_2}{\pi_1}{S^*}_{\pi_2}$ {\rm(}or simply $\reltens{T^*}{\gamma_2}{\pi_1}{S^*})$. In particular, if $x\in\gamma(N^{\rm o})'$ and $y\in\pi(N)'$, then the linear map
\[
({\cal K},\gamma)_{\varphi}\odot{}_{\varphi}(\pi,{\cal H}) \rightarrow \reltens{\cal K}{\gamma}{\pi}{\cal H} \; ; \; \xi\odot\eta \mapsto \reltens{x\xi}{\gamma}{\pi}{y\eta}
\]
extends uniquely to a bounded operator on $\reltens{\cal K}{\gamma}{\pi}{\cal H}$ denoted by $\reltens{x}{\gamma}{\pi}{y}\in\B(\reltens{\cal K}{\gamma}{\pi}{\cal H})$.
\end{propdef}

\begin{rk}
With the notations of \ref{propdef2}, let $T:{\cal K}_1\rightarrow{\cal H}_2$ and $S:{\cal H}_1\rightarrow{\cal K}_2$ be bounded antilinear maps such that $T\circ\gamma_1(n^{\rm o})^*=\pi_2(n)\circ T$ and $S\circ\pi_1(n)=\gamma_2(n^{\rm o})^*\circ S$ for all $n\in N$.
In a similar way, we define 
$
_{\pi_2}\reltens{T}{\gamma_1}{\gamma_2}{S}_{\pi_1}\in\B(\reltens{{\cal K}_1}{\gamma_1}{\pi_1}{{\cal H}_1},\reltens{{\cal H}_2}{\pi_2}{\gamma_2}{{\cal K}_2})
$ 
(or simply $\reltens{T}{\gamma_1}{\gamma_2}{S}$). Note that these notations are different from those used in \cite{E08,Le}.
\end{rk}

Let $M\subset\B({\cal K})$ and $P\subset\B({\cal H})$ be two von Neumann algebras. Let us assume that $\pi(N)\subset P$ and $\gamma(N^{\rm o})\subset M$.

\begin{defin}
The fiber product 
$
\fprod{M}{\gamma}{\pi}{P}
$
of $M$ and $P$ over $N$ is the commutant of 
$
\{\reltens{x}{\gamma}{\pi}{y}\,;\, x\in M',\,y\in P'\}\subset \B(\reltens{\cal K}{\gamma}{\pi}{\cal H}).
$
Then, $\fprod{M}{\gamma}{\pi}{P}$ is a von Neumann algebra. 
\end{defin}

Note that we have $\varsigma_{\gamma\pi}(\fprod{M}{\gamma}{\pi}{P})=\fprod{P}{\pi}{\gamma}{M}$.
We still denote by $\varsigma_{\gamma\pi}:\fprod{M}{\gamma}{\pi}{P}\rightarrow\fprod{P}{\pi}{\gamma}{M}$ the restriction of $\varsigma_{\gamma\pi}$ to $\fprod{M}{\gamma}{\pi}{P}$.

\begin{noh}\textit{Slicing with normal linear forms.} Now, let us recall how to slice with normal linear forms. For $\xi\in({\cal K},\gamma)_{\varphi}$ and $\eta\in{}_{\varphi}(\pi,{\cal H})$, we consider the following bounded linear maps:
\[
\lambda_{\xi}^{\gamma\pi}:{\cal H}\rightarrow\reltens{\cal K}{\gamma}{\pi}{\cal H},\; 
\zeta \mapsto \reltens{\xi}{\gamma}{\pi}{\zeta}  ; \quad
\rho_{\eta}^{\gamma\pi}:{\cal K}\rightarrow\reltens{\cal K}{\gamma}{\pi}{\cal H},\; 
\zeta \mapsto \reltens{\zeta}{\gamma}{\pi}{\eta}.
\]
Let $T\in\B(\reltens{{\cal K}}{\gamma}{\pi}{{\cal H}})$ and $\omega\in\B({\cal H})_*$ (resp.\ $\omega\in\B({\cal K})_*$). By using the fact that $({\cal K},\gamma)_{\varphi}$ (resp.\ $_{\varphi}(\pi,{\cal H})$) is dense in ${\cal H}$ (resp.\ ${\cal K}$), there exists a unique $(\fprod{\id}{\gamma}{\pi}{\omega})(T)\in\B({\cal K})$ (resp.\ $(\fprod{\omega}{\gamma}{\pi}{\id})(T)\in\B({\cal H})$) such that
\begin{align*}
\langle\xi_1,\,(\fprod{\id}{\gamma}{\pi}{\omega})(T)\xi_2\rangle &=\omega((\lambda_{\xi_1}^{\gamma\pi})^*T\lambda_{\xi_2}^{\gamma\pi}),\quad \text{for all } \xi_1,\,\xi_2\in({\cal K},\gamma)_{\varphi}\\[.5em]
\text{(resp. }\langle\eta_1,\,(\fprod{\omega}{\gamma}{\pi}{\id})(T)\eta_2\rangle &=\omega((\rho_{\eta_1}^{\gamma\pi})^*T\rho_{\eta_2}^{\gamma\pi}),\quad \text{for all } \eta_1,\,\eta_2\in{}_{\varphi}(\pi,{\cal H})\text{{\rm)}}.
\end{align*}
In particular, we have:
\begin{align*}
(\fprod{\id}{\gamma}{\pi}{\omega_{\eta_1,\eta_2}})(T)&=(\rho_{\eta_1}^{\gamma\pi})^*T\rho_{\eta_2}^{\gamma\pi}\in\B({\cal K}),\quad \text{for all } \eta_1,\,\eta_2\in{}_{\varphi}(\pi,{\cal H});\\[.5em]
(\fprod{\omega_{\xi_1,\xi_2}}{\gamma}{\pi}{\id})(T)&=(\lambda_{\xi_1}^{\gamma\pi})^*T\lambda_{\xi_2}^{\gamma\pi}\in\B({\cal H}),\quad \text{for all } \xi_1,\,\xi_2\in({\cal K},\gamma)_{\varphi}.
\end{align*}
If $x\in\fprod{M}{\gamma}{\pi}{P}$, then for all $\omega\in\B(\cal H)_*$ (resp.\ $\omega\in\B(\cal K)_*$) we have
$
(\fprod{\id}{\gamma}{\pi}{\omega})(x)\in M
$
(resp.\ $(\fprod{\omega}{\gamma}{\pi}{\id})(x)\in P$).
We refrain from writing the details but we can easily define the slice maps if $T$ takes its values in a different relative tensor product. Note that we can extend the notion of slice maps for normal linear forms to normal semi-finite weights.
\end{noh}

\paragraph{Fiber product over a finite-dimensional von Neumann algebra.} Now, let us assume that 
\[
N:=\bigoplus_{1\leqslant l\leqslant k}{\rm M}_{n_l}(\GC) \quad \text{and} \quad \varphi:=\bigoplus_{1\leqslant l\leqslant k}{\rm Tr}_l(F_l-),
\]
where $F_l$ is a positive invertible matrix of ${\rm M}_{n_l}(\GC)$ and ${\rm Tr}_l$ is the non-normalized trace on ${\rm M}_{n_l}(\GC)$. Denote by $(F_{l,i})_{1\leqslant i\leqslant n_l}$ the eigenvalues of $F_l$. 

\begin{propdef}($\S 7$ \cite{DC2})
The bounded linear map\index[symbol]{vh@$v_{\gamma\pi}$, canonical coisometry}
\[
v_{\varphi}^{\gamma\pi}:{\cal K}\tens{\cal H}\rightarrow\cst{\cal K}{\gamma}{\varphi}{\pi}{\cal H} \; ;\; \xi\tens\eta \mapsto \cst{\xi}{\gamma}{\varphi}{\pi}{\eta} \quad \text{(or simply denoted by }v_{\gamma\pi}\text{{\rm)}}
\]
is a coisometry if, and only if, we have $\sum_{1\leqslant i\leqslant n_l}F_{l,i}^{-1}=1$ for all $1\leqslant l\leqslant k$. 
\end{propdef}

In the following, we assume the above condition to be satisfied. 

\begin{propdef}\label{propcoiso}($\S 7$ \cite{DC2})
Let us denote\index[symbol]{qd@$q_{\gamma\pi}$, $q_{\pi\gamma}$} 
\[
q_{\varphi}^{\gamma\pi}:=(v_{\varphi}^{\gamma\pi})^*v_{\varphi}^{\gamma\pi} \quad (\text{or simply }q_{\gamma\pi}).
\]
Then, $q_{\varphi}^{\gamma\pi}$ is a self-adjoint projection of $\B({\cal K}\tens{\cal H})$ such that
\[
q_{\varphi}^{\gamma\pi}=\sum_{1\leqslant l\leqslant k}\sum_{1\leqslant i,j\leqslant n_l}F_{l,i}^{-1/2}F_{l,j}^{-1/2}\gamma(e_{ij}^{(l)\,{\rm o}})\tens\pi(e_{ji}^{(l)}),
\] 
where, for all $1\leqslant l\leqslant k$, $(e_{ij}^{(l)})_{1\leqslant i,j\leqslant n_l}$ is a system of matrix units (s.m.u.) diagonalizing $F_l$, {\it i.e.} $F_l=\sum_{1\leqslant i\leqslant n_l}F_{l,i}e_{ii}^{(l)}$. Moreover,
$
\fprod{M}{\gamma}{\pi}{P} \rightarrow q_{\varphi}^{\gamma\pi}(M\tens P)q_{\varphi}^{\gamma\pi} \; ; \; x \mapsto (v_{\varphi}^{\gamma\pi})^*xv_{\varphi}^{\gamma\pi}
$
is a unital normal *-isomorphism.$\vphantom{e_{ij}^{(l)}}$
\end{propdef}

\paragraph{Case of the non-normalized Markov trace.} In this paragraph, we take for $\varphi$ the non-normalized Markov trace on $\displaystyle{N=\oplus_{1\leqslant l\leqslant k}\,{\rm M}_{n_l}(\GC)}$, {\it i.e.\ }$\epsilon=\oplus_{1\leqslant l\leqslant k}\, n_l\cdot{\rm Tr}_l$.
From now on, the operators $q_{\epsilon}^{\gamma\pi}$, $q_{\epsilon^{\rm o}}^{\pi\gamma}$ $q_{\epsilon}^{\pi_1\pi_2}$ and $q_{\epsilon^{\rm o}}^{\gamma_1\gamma_2}$ will be simply denoted by $q_{\gamma\pi}$, $q_{\pi\gamma}$, $q_{\pi_1\pi_2}$ and $q_{\gamma_1\gamma_2}$. As a corollary of \ref{propcoiso}, we have:

\begin{prop}\label{Projection}
For all s.u.m.\ $(e_{ij}^{(l)})_{1\leqslant l\leqslant k,\, 1\leqslant i,j\leqslant n_l}$ of $N$, we have
\[
q_{\gamma\pi}=\sum_{1\leqslant l\leqslant k}n_l^{-1}\!\sum_{1\leqslant i,j\leqslant n_l}\!\gamma(e_{ij}^{(l){\rm o}})\tens\pi(e_{ji}^{(l)}) \;\; \text{and} \;\;
q_{\pi\gamma}=\sum_{1\leqslant l\leqslant k}n_l^{-1}\!\sum_{1\leqslant i,j\leqslant n_l}\!\pi(e_{ij}^{(l)})\tens\gamma(e_{ji}^{(l){\rm o}}).\qedhere
\]
\end{prop}

The following result is a slight generalization of \ref{Projection} to the setting of C*-algebras.

\begin{propdef}\label{ProjectionCAlg}(2.6 \cite{BC})
Let $A$, $B$ be two C*-algebras. We consider two non-degener\-ate *-homomorphisms $\gamma_A:N^{\rm o}\rightarrow\M(A)$ and $\pi_B:N\rightarrow\M(B)$. There exists a unique self-adjoint projection $q_{\gamma_A\pi_B}\in\M(A\tens B)$ {\rm(}resp.\ $q_{\pi_B\gamma_A}\in\M(B\tens A)${\rm)} such that
\begin{align*}
q_{\gamma_A\pi_B}&=\sum_{1\leqslant l\leqslant k}n_l^{-1}\sum_{1\leqslant i,j\leqslant n_l}\gamma_A(e_{ij}^{(l){\rm o}})\tens\pi_B(e_{ji}^{(l)})\\
\text{{\rm(}resp.\ }q_{\pi_B\gamma_A}&=\sum_{1\leqslant l\leqslant k}n_l^{-1}\sum_{1\leqslant i,j\leqslant n_l}\pi_B(e_{ij}^{(l)})\tens\gamma_A(e_{ji}^{(l){\rm o}})\text{{\rm)},}
\end{align*}
for all s.u.m.\ $(e_{ij}^{(l)})_{1\leqslant l\leqslant k,\, 1\leqslant i,j\leqslant n_l}$ of $N$.
\end{propdef}

	\subsection{Unitary equivalence of Hilbert C*-modules}\label{UnitaryEq}

In the following, we recall the notion of morphism between Hilbert modules over possibly different $\Cstar$-algebras. 

\begin{defin}\label{def2} 
Let $A$ and $B$ be two $\Cstar$-algebras and $\phi:A\rightarrow B$ a *-homomorphism. Let $\s E$ and $\s F$ be two Hilbert $\Cstar$-modules over $A$ and $B$ respectively. A $\phi$-compatible operator from $\s E$ to $\s F$ is a linear map $\Phi:\s E\rightarrow\s F$ such that:
\begin{enumerate}[label=(\roman*)]
\item for all $\xi\in\s E$ and $a\in A$, $\Phi(\xi a)=\Phi(\xi)\phi(a)$;
\item for all $\xi,\eta\in\s E$, $\langle \Phi\xi,\, \Phi\eta\rangle=\phi(\langle\xi,\,\eta\rangle)$.
\end{enumerate}
Furthermore, if $\phi$ is a *-isomorphism and $\Phi$ is surjective, we say that $\Phi$ is $\phi$-compatible unitary operator (or a unitary equivalence over $\phi$) from $\s E$ onto $\s F$.
\end{defin}

\begin{rks}\label{rk4}
\begin{enumerate}
\item It follows from (ii) that $\Phi:\s E\rightarrow\s F$ is bounded and even isometric if $\phi$ is faithful. Indeed, we have $\|\langle\Phi\xi,\, \Phi\eta\rangle\|=\|\phi(\langle\xi,\,\eta\rangle)\|=\|\langle\xi,\,\eta\rangle\|$ for all $\xi,\eta\in\s E$. Then,  for all $\xi\in\s E$ we have
	 $
	 \|\Phi\xi\|^2=\|\langle\Phi\xi,\, \Phi\xi\rangle\|=\|\langle\xi,\, \xi\rangle\|=\|\xi\|^2
	 $.
In particular, if $\phi$ is a *-isomorphism and $\Phi$ is a $\phi$-compatible unitary operator, then $\Phi$ is bijective and the inverse map $\Phi^{-1}:\s F\rightarrow\s E$ is a $\phi^{-1}$-compatible unitary operator.
\item It is clear that $\id_{\s E}$ is a $\id_A$-compatible unitary operator. Let $A$, $B$ and $C$ be $\Cstar$-algebras and $\s E$, $\s F$ and $\s G$ be Hilbert modules over $A$, $B$ and $C$ respectively. Let $\phi:A\rightarrow B$ and $\psi:B\rightarrow C$ be *-homomorphisms (resp.\ *-isomorphisms). If $\Phi:\s E\rightarrow\s F$ is a $\phi$-compatible operator (resp.\ unitary operator) and $\Psi:\s F\rightarrow\s G$ a $\psi$-compatible operator (resp.\ unitary operator), then $\Psi\circ \Phi:\s E\rightarrow\s G$ is a $\psi\circ\phi$-compatible operator (resp.\ unitary operator).
\item Let $\Phi:\s E\rightarrow\s F$ be a unitary equivalence over a given *-isomorphism $\phi$. If $T\in\Lin(\s E)$, then the map $\Phi\circ T\circ\Phi^{-1}:\s F\rightarrow\s F$ is an adjointable operator whose adjoint operator is $\Phi^{-1}\circ T^* \circ \Phi$. We define a *-isomorphism $\Lin(\s E)\rightarrow\Lin(\s F)\,;\, T\mapsto\Phi\circ T\circ\Phi^{-1}$. Note that $\Phi\circ\theta_{\xi,\eta}\circ\Phi^{-1}=\theta_{\Phi\xi,\Phi\eta}$ for all $\xi,\eta\in\s E$. In particular,  for all $k\in\K(\s E)$ we have $\Phi\circ k\circ\Phi^{-1}\in\K(\s F)$. More precisely, the map $\K(\s E)\rightarrow\K(\s F)\,;\,k\mapsto\Phi\circ k\circ\Phi^{-1}$ is a *-isomorphism.\qedhere
\end{enumerate} 
\end{rks}

The notion of unitary equivalence defines an equivalence relation on the class consisting of all Hilbert $\Cstar$-modules (cf.\ \ref{rk4} 1, 2). Actually, this notion of morphism between Hilbert modules over possibly different $\Cstar$-algebra can be understood in terms of unitary adjointable operator between two Hilbert modules over the same $\Cstar$-algebra.

\begin{prop}\label{prop8}
Let $A$ and $B$ be two $\Cstar$-algebras and $\phi:A\rightarrow B$ a *-isomorphism. Let $\s E$ and $\s F$ be two Hilbert $\Cstar$-modules over $A$ and $B$ respectively. 
\begin{enumerate}
\item If $\Phi:\s E\rightarrow\s F$ is a surjective $\phi$-compatible unitary operator, then there exists a unique unitary adjointable operator $U\in\Lin(\s E\tens_{\phi}B,\s F)$ such that
$U(\xi\tens_{\phi}b)=\Phi(\xi)b$, for all $\xi\in\s E$ and $b\in B$.
\item Conversely, if $U\in\Lin(\s E\tens_{\phi}B,\s F)$ is a unitary, then there exists a unique $\phi$-compatible unitary operator $\Phi:\s E\rightarrow\s F$ such that $\Phi(\xi)b=U(\xi\tens_{\phi}b)$ for all $\xi\in\s E$ and $b\in B$.\qedhere
\end{enumerate}
\end{prop}

As an application of the above proposition, we can the state the following result.

\begin{propdef}\label{propdef6}
Let $A_1$, $B_1$, $A_2$ and $B_2$ be $\Cstar$-algebras, $\phi_1:A_1\rightarrow B_1$ and $\phi_2:A_2\rightarrow B_2$ *-isomorphisms. Let $\s E_1$, $\s F_1$, $\s E_2$ and $\s F_2$ be Hilbert $\Cstar$-modules over $A_1$, $B_1$, $A_2$ and $B_2$ respectively. Let $\Phi_1:\s E_1 \rightarrow \s F_1$ and $\Phi_2:\s E_2 \rightarrow \s F_2$ be unitary equivalences over $\phi_1$ and $\phi_2$ respectively. Then, the linear map
$
\s E_1 \odot \s E_2 \rightarrow \s F_1\tens \s F_2\; ;\; \xi_1 \tens \xi_2 \mapsto \Phi_1(\xi_1) \tens \Phi_2(\xi_2)
$
extends to a bounded linear map $\Phi_1\tens\Phi_2:\s E_1 \tens \s E_2\rightarrow\s F_1 \tens \s F_2$. Moreover, $\Phi_1\tens \Phi_2$ is a $\phi_1\tens\phi_2$-compatible unitary operator.
\end{propdef}

The notion of unitary equivalence can also be understood in terms of isomorphism between the associated linking $\Cstar$-algebras.

\begin{prop}\label{prop7}
Let $A$ and $B$ be two $\Cstar$-algebras and $\phi:A\rightarrow B$ a *-isomorphism. Let $\s E$ and $\s F$ be two Hilbert $\Cstar$-modules over $A$ and $B$ respectively. 
\begin{enumerate}
\item If $\Phi:\s E\rightarrow\s F$ is a $\phi$-compatible unitary operator, then there exists a unique *-homomorphism 
$
f:\K(\s E\oplus A)\rightarrow\K(\s F\oplus B)
$
such that 
$f\circ\iota_{\s E}=\iota_{\s F}\circ \Phi$ and $f\circ\iota_A=\iota_B\circ\phi$.
Moreover, $f$ is a *-isomorphism.
\item Conversely, let $f:\K(\s E\oplus A)\rightarrow\K(\s F\oplus B)$ be a *-isomorphism such that $f\circ\iota_A=\iota_B\circ\phi$. Then, there exists a unique map $\Phi:\s E\rightarrow\s F$ such that $f\circ\iota_{\s E}=\iota_{\s F}\circ \Phi$. Moreover, $\Phi$ is a $\phi$-compatible unitary operator.\qedhere
\end{enumerate}
\end{prop}

\begin{proof}
1. The *-homomorphism $f:\K(\s E\oplus A)\rightarrow\K(\s F\oplus B)$ is defined by (cf.\ \ref{rk4} 3):
\begin{center}
$
f\begin{pmatrix} k & \xi \\ \eta^* & a\end{pmatrix}:=\begin{pmatrix} \Phi\circ k\circ \Phi^{-1} & \Phi\xi \\ (\Phi\eta)^* & \phi(a)\end{pmatrix}\!,\quad \text{for all } k\in\K(\s E),\; \xi,\,\eta\in\s E \text{ and } a\in A.
$
\end{center}
2. This is a straightforward consequence of \ref{ehmlem1} 1.
\end{proof}

\begin{nb}\label{not3}
Let $A$, $B$ be $\Cstar$-algebras and $\s E$ and $\s F$ be two Hilbert $\Cstar$-modules over $A$ and $B$ respectively. Let $\phi:A\rightarrow B$ be a *-isomorphism and $\Phi:\s E\rightarrow\s F$ a $\phi$-compatible unitary operator. If $T\in\Lin(A,\s E)$, we define the map 
$
\widetilde{\Phi}(T):= \Phi\circ T\circ\phi^{-1}:B\rightarrow\s F.
$
By a straightforward computation, we show that $\widetilde{\Phi}(T)\in\Lin(B,\s F)$ whose adjoint operator is $\widetilde{\Phi}(T)^*=\phi\circ T^*\circ \Phi^{-1}$. We have a bounded linear map 
$
\widetilde{\Phi}:\Lin(A,\s E)\rightarrow\Lin(B,\s F),
$
which is an extension of $\Phi$ up to the canonical injections $\s E\rightarrow\Lin(A,\s E)$ and $\s F\rightarrow\Lin(B,\s F)$.
\end{nb}


\bigbreak

\bigbreak

\noindent{\sc Vrije Universiteit Brussel}, Vakgroep Wiskunde, Pleinlaan 2, B-1050 Brussel (Belgium).\hfill\break 
{\it E-mail addresses}: \texttt{jonathan.crespo@wanadoo.fr}, \texttt{jonathan.crespo@vub.ac.be}\hfill\break
Supported by the FWO grant G.0251.15N

\clearpage
\phantomsection
\addcontentsline{toc}{section}{Index of notations and symbols}
\printindex[symbol]



\begin{thebibliography}{35}


\addcontentsline{toc}{section}{References}

\bibitem{Ba95} {\sc S.\ Baaj}, Rep\'esentation r\'eguli\`ere du groupe quantique des d\'eplacements de Worono\-wicz, \textit{Ast\'erisque} {\bf 232} (1995), 11–49.

\bibitem{BC} {\sc S.\ Baaj} et {\sc J.\ Crespo}, \'Equivalence mono\"idale de groupes quantiques et $K$-théorie bivariante, to appear in {\it Bull.\ Soc.\ Math.\ France}.

\bibitem{BS1}  {\sc S.\ Baaj} et {\sc G.\ Skandalis}, C$^*$-alg\`ebres de Hopf et th\'eorie de Kasparov \'equivariante, {\it K-theory} {\bf 2} (1989), 683-721.

\bibitem{BS2}  {\sc S.\ Baaj} et {\sc G.\ Skandalis}, Unitaires multiplicatifs et dualit\'e pour les produits crois\'es de C$^*$-algèbres, \textit{Ann. Sci. \'Ec. Norm. Sup\'er.} $4^{\rm e}$ s\'erie, {\bf 26} (4) (1993), 425-488.
	
\bibitem{BSV} {\sc S.\ Baaj, G.\ Skandalis} and {\sc S.\ Vaes}, Non-semi-regular quantum groups coming from number theory, {\it Comm. Math. Phys.} {\bf 235} (1) (2003), 139-167.


\bibitem{BRV} {\sc J.\ Bichon}, {\sc A.\ De Rijdt} and {\sc S.\ Vaes}, Ergodic coactions with large multiplicity and monoidal equivalence of quantum groups, {\it Comm. Math. Phys.} {\bf 262} (2006), 703-728.	
	
\bibitem{Co} {\sc A.\ Connes}, On the spatial theory of von Neumann algebras, {\it J. Funct. Anal.} {\bf 35} (1980), 153-164.	

\bibitem{Co2} {\sc A.\ Connes}, Noncommutative Geometry, Academic Press, San Diego, CA, 1994.
	
\bibitem{CS} {\sc A.\ Connes} and {\sc G.\ Skandalis}, The longitudinal index theorem for foliations, {\it Publ. Res. Inst. Math. Sci.} {\bf 20} (6) (1984), 1139-1183.	

\bibitem{C} {\sc J.\ Crespo}, Monoidal equivalence of locally compact quantum groups and application to bivariant K-theory, {\it Ph.D.\ thesis, Université Blaise Pascal} (2015).

\bibitem{C2} {\sc J.\ Crespo}, Actions of measured quantum groupoid on a finite basis (arXiv preprint).

\bibitem{DC2} {\sc K.\ De Commer}, Monoidal equivalence for locally compact quantum groups, preprint: arXiv:math.OA/0804.2405v2.
	
\bibitem{DC} {\sc K.\ De Commer}, Galois coactions for algebraic and locally compact quantum groups, {\it Ph.D. thesis, Leuven, Katholieke Universiteit Leuven} (2009).

\bibitem{DC3} {\sc K.\ De Commer}, Galois coactions and cocycle twisting for locally compact quantum groups, {\it J.\ Operator theory} {\bf 66} (1) (2011), 59-106. 
	

\bibitem{DFY} {\sc K.\ De Commer, A.\ Freslon} and {\sc M.\ Yamashita}, CCAP for Universal Discrete Quantum Groups (with an appendix by S. Vaes), {\it Comm. Math. Phys.} {\bf 331} (2) (2014), 677-701.
	
\bibitem{RV}   {\sc A.\ De Rijdt} and {\sc N.\ Vander Vennet}, Actions of monoidally equivalent compact quantum groups and applications to probabilistic boundaries, {\it Ann.\ Inst.\ Fourier} {\bf 60} (1) (2010), 169-216.
	
\bibitem{E05}   {\sc M.\ Enock}, Quantum groupoids of compact type, {\it J. Inst. Math. Jussieu} {\bf 4} (2005), 29-133.
	
\bibitem{E08}   {\sc M.\ Enock}, Measured Quantum Groupoids in Action, {\it M\'em.\ Soc.\ Math. Fr.} {\bf 114} (2008), 1-150.


\bibitem{Kas1} {\sc G.\ G.\ Kasparov}, Hilbert C*-modules: theorems of Stinespring and Voiculescu, {\it J. Operator Theory} {\bf 4} (1) (1980), 133-150.

\bibitem{Kas3} {\sc G.\ G.\ Kasparov}, The operator $K$-functor and extensions of C$^*$-algebras, {\it Math USSR-Izv.} {\bf 16} (3) (1981), 513-572. Translated from {\it Izv.\ Akad.\ Nauk.\ SSSR.\ Ser.\ Mat.} {\bf 44} (1980), 571-636.

\bibitem{Kas2} {\sc G.\ G.\ Kasparov}, Equivariant KK-theory and the Novikov conjecture, {\it Invent. Math.} {\bf 91} (1) (1988), 147-201.

\bibitem{KV2}  {\sc J.\ Kustermans} and {\sc S.\ Vaes}, Locally compact quantum groups in the von Neumann algebraic setting, {\it Math. Scand.} {\bf 92} (1) (2003), 68-92.

\bibitem{LG} {\sc P.-Y.\ Le Gall}, Th\'eorie de Kasparov \'equivariante et groupoïdes I, {\it K-theory} {\bf 16} (1999), 361-390.

\bibitem{Le}   {\sc F.\ Lesieur}, Measured Quantum Groupoids, {\it M\'em.\ Soc.\ Math.\ Fr.} {\bf 109} (2007), 1-117.
	
\bibitem{M} {\sc R.\ Meyer}, Equivariant Kasparov theory and generalized homomorphisms, {\it K-theory} {\bf 21} (2000), 201–228. 
	
\bibitem{NT14} {\sc S.\ Neshveyev} and {\sc L.\ Tuset}, Deformation of C*-algebras by cocycles on locally compact quantum groups, {\it Adv. Math.} {\bf 254} (2014), 454-496.
	
\bibitem{Rie74} {\sc M.\ A.\ Rieffel}, Induced representations of C*-algebras, {\it Adv. Math.} {\bf 13} (2) (1974), 176-257.
	
\bibitem{Skand} {\sc G.\ Skandalis}, Some remarks on Kasparov theory, {\it J.\ Funct.\ Anal.} {\bf 56} (1984), 337-347.	
	
\bibitem{Tim1} {\sc T.\ Timmermann}, Pseudo-multiplicative unitaries and pseudo-Kac systems on C$^*$-modules, 05/2005, Dissertation, Preprint des SFB 478 M\"unster (394).

\bibitem{Tim2} {\sc T.\ Timmermann}, Coactions of Hopf C$^*$-bimodules, {\it J. Operator Theory} {\bf 68} (1) (2012), 19-66.

\bibitem{Vaes1} {\sc S.\ Vaes}, The unitary implementation of a locally compact quantum group action, {\it J. Funct. Anal.} {\bf 180} (2001), 426-480.
	
\bibitem{VV} {\sc S.\ Vaes} and {\sc N.\ Vander Vennet}, Identification of the Poisson and Martin boundaries of orthogonal discrete quantum groups, {\it J. Inst. Math. Jussieu} {\bf 7} (2008), 391-412.


\bibitem{Val} {\sc J.-M.\ Vallin}, Unitaire pseudo-multiplicatif associ\'e \`a un groupo\"ide. Applications \`a la moyennabilit\'e, {\it J. Operator Theory} {\bf 44} (2) (2000), 347-368.

\bibitem{Ve} {\sc R.\ Vergnioux}, KK-théorie équivariante et opérateur de Julg-Valette pour les groupes quantiques, {\it Ph.D.\ thesis, Université Paris 7 - Denis Diderot} (2002).

\bibitem{Ve2} {\sc R.\ Vergnioux}, $K$-amenability for amalgamated free products of amenable discrete quantum groups, {\it J.\ Funct.\ Anal.\ }{\bf 212} (2004), 206-221.

\bibitem{VeVo} {\sc R.\ Vergnioux} and {\sc C.\ Voigt}, The K-theory of free quantum groups, {\it Math. Ann.} {\bf 357} (1) (2013), 355-400.

\bibitem{V1} {\sc C.\ Voigt}, The Baum-Connes conjecture for free orthogonal quantum groups, {\it Adv. Math.} {\bf 227} (5) (2011), 1873-1913.

\end{thebibliography}
\end{document}